\documentclass[reqno,11pt]{amsart}
\usepackage{amsmath,amssymb,mathrsfs,amsthm,amsfonts}
\usepackage[inline]{enumitem} 
\usepackage[usenames,dvipsnames]{xcolor}
\usepackage{hyperref}
\usepackage{comment}
\usepackage{array,longtable,booktabs,caption}
\DeclareMathAlphabet{\mathpzc}{OT1}{pzc}{m}{it}
\usepackage{stmaryrd}

\hypersetup{%
	colorlinks=true, linkcolor=blue,
	citecolor=ForestGreen
}
\usepackage[paper=letterpaper,margin=1.03in, top=0.9in]{geometry}
\usepackage{graphicx}
\usepackage{mathrsfs}
\usepackage{listings}
\newcommand{\ind}{\mathbf{1}}
\usepackage{bbm}
\usepackage{tikz}
\usetikzlibrary{arrows}

\newtheorem{thm}{Theorem}[section]
\newtheorem{cor}[thm]{Corollary}
\newtheorem{prop}[thm]{Proposition}
\newtheorem{lem}[thm]{Lemma}
\newtheorem{lemma}[thm]{Lemma}

\theoremstyle{definition}
\newtheorem{defn}[thm]{Definition}
\newtheorem{ex}[thm]{Example}
\newtheorem{rk}[thm]{Remark}
\newtheorem{ass}[thm]{Assumption}

\numberwithin{equation}{section}

\newcommand{\ak}{B_{z_0/(2k)}(   0)}
\newcommand{\qdif}{\boldsymbol q^{(k)}_{  \beta}}

\newcommand{\x}{\mathbf x}
\newcommand{\dr}{\mathrm d}
\newcommand{\y}{\mathbf y}
\newcommand{\br}{\boldsymbol{b}}
\newcommand{\brn}{\boldsymbol{b}_N}
\newcommand{\bfa}{\mathbf a}
\newcommand{\Q}{\mathscr Q^{   \beta}_k}
\newcommand{\pdif}{\boldsymbol{p}_{\mathbf{dif}}}
\newcommand{\bdif}{\boldsymbol{b}_{\mathbf{dif}}}
\newcommand{\Pdif}{P_{\mathbf{dif}}}
\newcommand{\z}{\boldsymbol{\zeta}}

\newcommand{\Pb}{ \mathbf P_\mathbf x^{(   \beta,k)}}
\newcommand{\Eb}{ \mathbf E_\mathbf x^{(   \beta,k)} }
\newcommand{\Pbb}{\mathbf P_\mathbf x^{\boldsymbol b}}
\newcommand{\Ebb}{\mathbf E_\mathbf x^{\boldsymbol b}}
\newcommand{\Ebbnn}{\mathbf E_{\mathbf x_N}^{\boldsymbol b_N}}
\newcommand{\Ek}{\mathcal E_{k\times d}^{\mathrm{SRI}}(q,M,F)}
\newcommand{\Ekk}{\mathcal E_{k\times d}^{\mathrm{SRI}}(\Psi_\sigma,M,F)}

\newcommand{\R}{\mathbb{R}}

\renewcommand{\hat}{\widehat}
\newcommand{\Fd}{\mathrm{F_{dec}}}

\renewcommand{\bar}{\overline}

\newcommand{\Ex}{\mathbb{E}}

\newcommand{\Con}{\mathrm{C}}
\newcommand{\mE}{\mathbf{E}}

\newcommand{\e}{\varepsilon_{\mathrm{ic}}}
\newcommand{\sdn}[1]{{\color{red}\ttfamily\upshape\small[#1]\color{black}}}

\title[RWRE fluctuations in all spatial dimensions]{Random walks in space-time random media in all spatial dimensions: the full subcritical fluctuation regime}

\author[H.\ Drillick]{Hindy Drillick}
\address{H.\ Drillick, Simons Laufer Mathematical Sciences Institute,\newline\hphantom{\quad \ \ H. Drillick}17 Gauss Way, Berkeley, CA 94720 USA}

\author[S.\ Parekh]{Shalin Parekh}
\address{S.\ Parekh, Department of Mathematics, University of Maine,\newline\hphantom{\quad \ \ S. Parekh}Neville Hall, Orono, ME 04469 USA}

\subjclass[2020]{
	Primary 60K37, 
        82B21,	
        82C22,	
	Secondary 60G70. 
}

\keywords{Random walks in random environments; Stochastic flows; Additive-noise stochastic heat equation; Additive functionals of Markov chains.
}

\begin{document}
	
	\begin{abstract} In arbitrary spatial dimension $d\ge 1$, we study a generalized model of random walks in a time-varying random environment (RWRE) defined by a stochastic flow of kernels. We consider the quenched probability distribution of the random walker under a scaling where the time is of order $N$ and the spatial window is of size $N^{1/2}$. This spatial window may not necessarily be centered close to the origin.
    
    We show that as $N\to \infty$ there are Gaussian fluctuations up to a certain specific spatial centering radius $\psi_N$ in the tail of the quenched probability distribution, which we call the \textit{critical scale}. This critical scale depends on the spatial dimension of the underlying random walk, specifically $\psi_N = O(N^{3/4})$ when $d=1$, $\psi_N = O( N/\sqrt{\log N})$ when $d=2$, and $\psi_N = O(N)$ when $d\ge 3$. In the particular case of centering the fluctuation window at the origin, our results recover and generalize some known fluctuation results for related models. However, farther from the origin, the previous literature is more sparse. The noise coefficient in the limiting Gaussian field is nontrivial and depends on the invariant measure of the two-point motion of the underlying RWRE model. We furthermore reconcile some of these coefficient formulas with previous works. 

    As part of the proof, we introduce a general class of Markov chains with short-range interactions that admit nice estimates and limit formulas. One of the key technical results for such Markov chains is that in $d\ge 2$, one can propagate test functions backwards in time to obtain precise limiting moment formulas. 
    
    We conjecture that the fluctuations are no longer Gaussian at a certain point within the critical scale, which is already known in $d=1$. In $d=2$, we conjecture that the scaling limit at criticality is given by the 2d critical stochastic heat flow recently constructed in  \cite{CSZ_2d}, and we make various additional conjectures about relating the RWRE crossover to the crossover of directed polymers in arbitrary dimensions.
	\end{abstract}

\maketitle

\setcounter{tocdepth}{1}
\tableofcontents
   
\section{Introduction} \label{sec:intro}

A random walk in a dynamic random environment (RWRE) is a probabilistic model in which a random walker moves through a time-changing random environment that determines the transition probabilities of the walker. The environment may be determined by i.i.d. weights, but often it is nice to allow more general frameworks such as correlated environments that are sufficiently strongly mixing in space or time, or even non-weight-based models. Such frameworks can be used to model scenarios where conditions are not static, such as studying diffusion of many particles in a turbulent medium.

An important object of mathematical interest for these dynamic RWRE models is their \emph{quenched probability density} $P^{\omega}(t,x)$ which describes the random transition probabilities for the random walk given the environment $\omega$. Given a $d$-dimensional vector $x$ and a positive time $t$, the quantity $P^\omega(t,x)$ denotes the probability that the random walker is at position $x$ at time $t$, given the randomness of $\omega$. This quenched density can be studied in different scaling regimes corresponding to different scalings of space and time. In particular, the \emph{bulk regime} corresponds to a diffusive scaling $|x| = O(\sqrt{t})$, and there is a spectrum of \emph{extremal regimes} corresponding to $|x| \gg \sqrt{t}$ going all the way up to the \emph{large deviation regime} $|x| = O(t)$. 

In $d=1$, previous results for these dynamic RWRE models are mostly known for the two ends of this spectrum: the bulk regime where Gaussian fluctuations appear \cite{BMP1,BMP2, timo2, yu,timo, Kot} and the large deviation regime where Tracy-Widom \cite{TW} fluctuations appear \cite{bc, gu, mark,  dom, bc25}. The physics works \cite{ldt2, ldt, bld} therefore predicted the existence of a crossover regime at $x = O(t^{3/4})$  where the fluctuations are given by the Kardar-Parisi-Zhang (KPZ) equation \cite{kpz}. The KPZ equation encapsulates the crossover from the Gaussian to Tracy-Widom regimes. This prediction was mathematically confirmed in our recent series of works \cite{DDP23,DDP+,Par24}. The regime $x = O(t^{3/4})$ can therefore be thought of as a \emph{critical scale} where the fluctuations change from Gaussian to non-Gaussian. In higher dimensions $d >1$, only the fluctuations of the bulk regime have been studied (see for example \cite{BMP1,Kot}), where they are Gaussian. However, we conjecture that there should similarly be a critical scale at which the nature of the fluctuations changes from Gaussian to non-Gaussian.

The goal of this paper is to fill in this spectrum in \emph{all spatial dimensions} up until this conjectured critical scale. In particular do the following.
\begin{enumerate}
    \item We prove new Gaussian fluctuation results for these models in an entire spectrum of scaling locations starting from the bulk regime and going up to a certain critical location scale where we conjecture that non-Gaussian fluctuations will appear. These Gaussian fluctuations will be described by various additive-noise stochastic partial differential equations (SPDEs). We show that there are several different regimes in each dimension, with the law of the SPDE limit changing as one transitions from one regime to another. In $d=1$, our results can be thought of as interpolating the known results in the bulk regimes with the known results in the extremal ones.

    \item We discuss some of the important open problems concerning the fluctuations of the quenched density and tail behavior and comment on some other interesting observables such as the extreme particles in a system of many walkers. Much of this discussion is inspired by the recent physics numerics papers \cite{hass2023b, hass23, hass2024extreme, hass2025superuniversalbehavioroutliersdiffusing, hass2025universal, ark2025universalfluctuationstailprobability}. 

    \item We develop a new framework which we call \emph{Short-Range Interacting Markov Chains (SRI Chains)} to model the environment-averaged law of multiple random walkers in the same random environment. This will be an important tool to help prove the above-mentioned fluctuation results. These are Markov chains that describe families of interacting random walks whose interactions only take place when the random walkers are close enough to each other. We prove various results for such types of Markov chains, including an invariance principle, heat kernel estimates and various limit laws for additive functionals of these chains. 

\end{enumerate}

\subsection{Background}
The study of these dynamical RWRE models began with various works of Bernabei, Boldrighini, Minlos, and Pellegrinotti \cite{BMP1, BMP3,  BMP4, BMP2, BMP5}. They obtain laws of large numbers and central limit theorems for various observables of the random walker \emph{in the bulk regime}. The approach of those papers is to use a cluster expansion method to prove meaningful limit theorems, and relies on rather technical ``small noise" assumptions. Further advances that relaxed those strong assumptions to prove more general quenched central limit theorems and quenched large deviations principles for these dynamic RWRE models came by considering the point of view of ``the environment seen through the random walker," see e.g. the works of \cite{timo3, BZ, timo2, Dolgo1, Dolgo2, ADR, Redig, timo}. Meanwhile, a body of works of \cite{lejan, HW09, HW09b} formulated yet more general and continuum versions of these models that establish deep connections with previous works on classical stochastic flows \cite{kun94a, kun94b}, and further inspired later works of \cite {sss0, sss, sss2} that drew numerous interesting connections between these dynamic RWRE models and the Brownian web and net.

Recently a renewed interest in these dynamic RWRE models has grown from a different direction, namely their connection to the fields of disordered systems, KPZ universality, and singular stochastic PDEs. These objects appear when considering \emph{certain extremal regimes} for the quenched density. These connections were first developed in \cite{bc, gu, mark} and the physics papers \cite{ldt2, ldt, bld} which derived exact formulas for specific RWRE models in spatial dimension $d=1$. They then performed precise asymptotics with those formulas to yield Tracy-Widom type distributions \cite{TW} for various observables of these models \emph{in the large deviation regime}. KPZ universality is the study of physical and probabilistic models of time-varying height interfaces in one spatial dimension that exhibit these limiting distributions as well as anomalous \textit{critical exponents}, usually coming in a ratio of 1:2:3 of fluctuations, space, and time respectively. One of the central objects in the \textit{KPZ universality class} is the so-called KPZ equation, a singular parabolic stochastic PDE of Hamilton-Jacobi type. 
We refer to \cite{FS10, Qua11, Cor12, QS15, CW17, CS20} for some surveys of the mathematical studies of the KPZ equation, and \cite{Wal86, DPZ, Hai13, Hai14, GJ14, GIP15, GP17, GP18, duch} for a treatment of its highly nontrivial existence and uniqueness theory.

In a series of previous papers \cite{DDP23, DDP+, Par24}, we built upon the known connections between RWRE models and KPZ universality by proving convergence of the quenched probability density of one-dimensional dynamic RWRE models to the \emph{multiplicative stochastic heat equation} started with Dirac initial data. This stochastic PDE is the exponentiated version of the KPZ equation. We proved this convergence under a very specific \emph{moderate deviations} scaling regime of time and space, namely $x = O(t^{3/4})$.

In those papers we do not use the previously mentioned techniques of cluster expansions or the environment-seen-from-the-walker. Instead we introduce a different technique, namely deriving a discrete SPDE for the quenched density driven by a martingale noise, inspired by the work of \cite{BG97}. This technique has proven to be fruitful for obtaining the fluctuations in fairly general settings and will be the basis of the proof techniques used in this paper.

\subsection{Definition and assumptions for the model}

Here we introduce a generalized model of random walk in random environment that will be studied throughout the rest of the paper. The paper \cite{Par24} first introduced and studied such a model, but only in spatial dimension $d=1$ and in a very specific scaling window. The model is inspired by the notion of stochastic flows of kernels as introduced by \cite{lejan} and can be thought of as a time discretization of their framework.

\begin{defn}[Markov Kernels] Fix $d\ge 1$. Let $I$ be any locally compact additive subgroup of $\mathbb R^d$, for example $I=\mathbb R^d$ or $I=A \mathbb Z^d$ where $A$ is some nonsingular $d\times d$ matrix. Let $\mathcal B(I)$ denote the Borel $\sigma$-algebra on $I$. 
A \textit{Markov kernel} on $I$ is a continuous function from $I\to \mathcal P(I)$, where $\mathcal P(I)$ is the space of probability measures on $(I,\mathcal B(I))$ equipped with the topology of weak convergence of probability measures. 

Any Markov kernel $K$ on $I$ shall be written as $K(   x,A)$, which is identified with the continuous function from $I\to \mathcal P(I)$ given by $x\mapsto K(   x,\cdot)$. We denote by $\mathcal M(I)$ the space of all Markov kernels on $I$.

We equip $\mathcal M(I)$ with the weakest topology under which the maps $T_{f,   x}:\mathcal M(I)\to \mathbb R$ given by $$K\stackrel{T_{f,   x}}{\mapsto}  \int_I f(   y)\;K(   x,\mathrm d   y)$$ 
are continuous, where one varies over all $x\in I$ and all bounded continuous functions $f:I\to \mathbb R$. This topology endows $\mathcal M(I)$ with a Borel $\sigma$-algebra which allows us to talk about random variables taking values in the space of Markov kernels. For $   a\in I$ define the translation operator $\tau_a:\mathcal M(I)\to \mathcal M(I)$ by $(  \tau_{   a} K)(   x,A):= K(x+   a,A+   a)$ which is a Borel-measurable map on this space. 
\end{defn}
\begin{ass}[Assumptions for the main result] \label{a1}Assume we have a family $\{K_n\}_{n\ge 1}$ of random variables in the space $\mathcal M(I)$, defined on some probability space $(\boldsymbol \Omega, \mathcal F^\omega, \mathbb P)$ such that the following hold true.
\begin{enumerate}
    

    \item \label{a11}(Stochastic flow increments)  $K_1,K_2,K_3,...$ are independent and identically distributed.

    \item \label{a12}(Spatial translational invariance)  $K_1$ has the same law as $  \tau_{a} K_1$ for all $a\in I.$

    \item \label{a22}(Exponential moments for the annealed law) Letting $\mu(A) := \mathbb E[ K_1(   0,A)]$ we have that the moment generating function exists, i.e., $\int_I e^{z_0|   x|}\mu(\mathrm d   x)<\infty$ for some $z_0>0$, where $|\cdot|$ is the Euclidean norm on $\mathbb R^d$. Letting $m_{k_1,...,k_d}:= \int_I x_1^{k_1}\cdots x_d^{k_d} \mu(\mathrm d   x)$ denote the sequence of moments of $\mu$, we also impose that the covariance matrix of $\mu$ is the $d\times d$ identity matrix, i.e., $m_{   e_i+   e_j} -m_{   e_i}m_{   e_j}=\delta_{ij},$ where $   e_i$ are the standard basis of $\mathbb R^d$. We can assume this without loss of generality since the additive subgroup $I$ can be replaced by $\Sigma^{-1} I$ if the covariance matrix is any nondegenerate matrix $\Sigma$. 
    \item \label{a23}(Symmetry up to order $p$) There exists some $p\in \mathbb N$ such that $\int_I y_1 ^{k_1} \cdots y_d^{k_d} K_1(   0,\mathrm d   y)=m_{k_1,...,k_d}$ $\mathbb P$-almost surely for $1\leq k_1+...+k_d \leq p-1$.
    Furthermore $p-1$ is the largest such value, i.e., $\int_I y_1 ^{k_1} \cdots y_d^{k_d} K_1(   0,\mathrm d   y)$ is non-deterministic for some $(k_1,...,k_d)\in \Bbb Z_{\ge 0}^d$ with $k_1+...+k_d=p.$

    \item \label{a24} (Strong decay of correlations)
    Consider a decreasing function $\Fd:\mathbb [0,\infty)\to [0,\infty)$ 
    of exponential decay at infinity (i.e., $\Fd(x)\leq Ce^{-cx}$ for some $c,C>0$). Then assume that the kernel $K_1$ satisfies the following spatial decay of correlations for all $k \in \mathbb N$: there exists $C(k)>0$ such that for every partition $\pi$ of the finite set $\{1,...,k\},$
    \begin{align}\bigg\| \mathbb E \bigg[ \prod_{j=1}^k K_1(x_j,\bullet)\bigg] - \prod_{B\in \pi} \mathbb E \bigg[ \prod_{i\in B} K_1(x_j,\bullet) \bigg] \bigg\|_{TV} \leq C(k) \Fd\bigg( \min_{B\ne B'\in \pi} \min_{i\in B,j\in B'} |   x_i-   x_j|\bigg),  \label{tvb'}
    \end{align}
    uniformly over all $   x_1,...,   x_k \in I$. The products on the left side should be interpreted as the usual tensor products of Borel probability measures.
    \item (Non-degeneracy) \label{a16} Define a deterministic Markov kernel on $I$ by the formula $$\pdif(   x,A):=\int_I \ind_{\{   y_1-   y_2\in A\}} \mathbb E [ K_1(0,\dr y_1) K_1( x, \dr y_2) ],$$ for $   x\in I$ and Borel sets $A\subset I$. We impose the following two conditions on $\pdif$:
    \begin{itemize} 
    \item 
    (Regularity) The family of transition laws $   x\mapsto \pdif(   x,\bullet)$ are continuous in total variation norm.
    \item (Irreducibility) For any two points $   x,   y \in I$ and any $\epsilon>0$ there exists $m\in \mathbb N$ such that $\pdif^m\big(   x,B_\epsilon(   y)\big)>0,$ where $\pdif^m$ is the $m^{th}$ power of this Markov kernel, and $B_\epsilon(   y)$ is the ball of radius $\epsilon$ around $   y \in \mathbb R^d$.
    \end{itemize}
    \end{enumerate}
\end{ass}

We define the \emph{underlying environment} $\omega \in \boldsymbol \Omega$ as the entire realization of the i.i.d. kernels: that is, $\omega=(K_i)_{i=1}^\infty$. 

The crucial part of Assumption \ref{a1} is Item \eqref{a24} which says that, after averaging, random walkers in the same random environment are essentially independent when they are far enough apart. In fact, it says something slightly stronger which is that if we have several clusters of random walkers, the clusters behave independently if they are far enough apart from one another. This can also be viewed as a strong mixing condition on $K_1$ under spatial translation.

Define the $k$-point correlation kernel (a family of measures on $I^k$) by 
\begin{equation}\label{kptkernel}\boldsymbol p^{(k)}\big((   x_1,...,   x_k),(\mathrm d   y_1,...,\mathrm d   y_k)\big):= \mathbb E[ K_1(   x_1,\mathrm d   y_1)\cdots K_1(   x_k,\mathrm d    y_k)].
\end{equation} Taking $\pi = \{ \{1\} ,..., \{k\}\},$ notice that Item \eqref{a24} above implies that 
\begin{align}\big\| \mu^{\otimes k}\big(\bullet-(   x_1,...,   x_{k})\big) \;\;-\;\; \boldsymbol p^{(k)}\big( (   x_1,...,   x_{k}),\;\bullet\; \big)\big\|_{TV} \leq C(k) \Fd\big( \min_{1\le i<j\le k}|   x_i-   x_j|\big).  \label{tvb}
    \end{align}
One might ask if this slightly weaker assumption would suffice to prove our main results, and indeed in a previous work \cite{Par24} in $d=1$ it was sufficient. In higher dimensions, it appears that we need the stronger version above. Checking \eqref{tvb'} in practice is generally straightforward.

While the assumptions look lengthy, they are very general and include many of the well-known cases of random walks in a dynamical random environment such as the discrete random walks of \cite{sss, bc, DDP+}, the sticky Brownian motions of \cite{mark, HW09}, the diffusion in random media of \cite{ew6, dom}, and the KMP model \cite{bc25}. Most physically interesting models satisfy $p=1$ or (rarely) $p=2$, but mathematically it is interesting to study the full hierarchy of $p$-values and the appearance of $p^{th}$ derivatives in the limit theorems below for general values of $p\in \mathbb N$. Our choice to consider general $p\in \mathbb N$ is inspired by the physics work of \cite{hass2025superuniversalbehavioroutliersdiffusing}, who first conjectured based on numerics that different values of $p$ would give rise to different scaling limits and coefficients.

\begin{defn}[The microscopic model of interest] Let $I$ be as in Assumption \ref{a1}, and sample the kernels $K_1,K_2,...$ as described there. Fix a deterministic probability measure $\nu$ on $I$, called the \emph{initial state.} For $r\in \mathbb N$, we define the \emph{quenched probability measure} \begin{equation}P_\nu^\omega(r,\bullet)= \int_I K_1\cdots K_r(   x_0,\bullet)\nu(\dr x_0),\label{kn}\end{equation} where the product of kernels is defined by $K_1K_2(   x,A) = \int_I K_2(   y,A)K_1(   x,\mathrm d   y)$ and so on by associativity. 
\end{defn}
Above and henceforth, the variable $r\in \mathbb N$ will always be used to denote the microscopic time variable. Intuitively, $P_\nu^\omega(r,\bullet)$ is the transition probability at time $r$ of a random walker starting from initial measure $P_\nu^\omega(0,\bullet) = \nu$, such that the walker at position $   x$ at time $n-1$ uses the kernel $K_n(   x,\bullet)$ to determine its position at time $n$, for every $n\ge 0$ and $x\in I$. 

\begin{figure}[t]
    \centering
    \includegraphics[scale = 0.82]{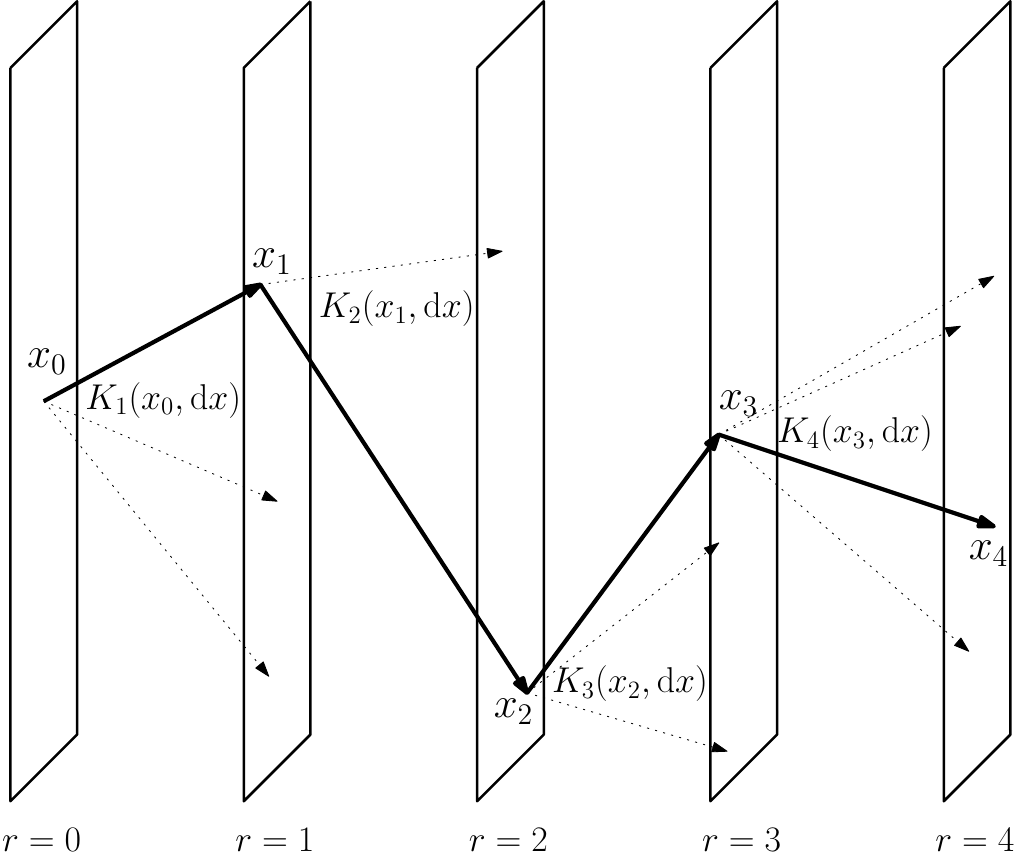}
    \caption{ An illustration of our generalized model of random walk in random environment. The random walker starts at position $x_0$ at time $r=0$, then samples $x_1$ from the \textit{random} probability measure $K_1(x_0,\dr x)$. After landing at position $x_1$ at time $r=1$, the random walker then samples $x_2$ from the probability measure $K_2(x_1,\dr x)$, and continues in this fashion. Each plane represents a copy of the underlying state space $I$ (which in this example is $\mathbb R^2)$. The dotted arrows in the above picture represent any one of many possible jumps that were not actually executed by the random walker in this random environment, while the solid arrows represent the jumps that were actually executed.}
    \label{sticky}
\end{figure}

\subsection{Examples of models satisfying the assumptions}

The following are five simple examples satisfying Assumption \ref{a1} that are useful to have in mind while reading the rest of the paper.

\begin{ex}[Nearest-neighbor random walk in a dynamical random environment in $\mathbb Z^d$]\label{ex:1}
Take $I = \mathbb Z^d$. Let $\hat{e_i}, 1 \leq i \leq d$ denote the unit vectors in $\mathbb Z^d$. We then define a space-time random environment as follows: Consider an i.i.d. family of random vectors $\{\boldsymbol \omega_{n,   x}\}_{n \in \mathbb N,    x \in \mathbb Z^d}$. Each random vector $\boldsymbol \omega_{n,   x}$ has components indexed by the unit vectors in $\mathbb Z^d$, and these components sum up to $1$ so that $\boldsymbol \omega_{n,   x}$ defines jump probabilities for a random walk. In other words we have $$\boldsymbol \omega_{n,   x}= \big(\omega_{n,    x}(\pm \hat{e}_i): 1 \leq i \leq d\big), \;\;\;\;\;\;\sum_{i=1}^d \left( \omega_{n,    x}(\hat{e}_i) +\omega_{n,    x}(-\hat{e}_i) \right) = 1.$$ 

We can then define the kernels 
$$\tilde K_n(   x,    x \pm \hat{e}_i) =  \omega_{n,    x}(\pm \hat{e}_i),$$ and set $\tilde K_n(   x,    y) = 0$ if $|    y -    x| > 1$. In other words, a random walker in this random environment will use the random jump rates given by $\boldsymbol \omega_{n,    x}$ when it is at position $   x$ at time $n$. 

This nearest-neighbor model cannot satisfy the irreducibility condition in Assumption \ref{a1} \eqref{a16}, as the communicating class of the origin under the difference chain $\pdif$ is a proper subgroup $I\subset \mathbb Z^d$. However, this periodicity is not a fatal flaw and can be remedied by considering the two-step chain in which the kernel $K_n$ is defined to be the product $\tilde K_{2n-1} \tilde K_{2n}$. 
The irreducibility condition is then satisfied by these modified $K_n$ on the proper subgroup $I$ as long as we impose that the random variable $\boldsymbol\omega_{0,0} $ 
is not concentrated on a single unit vector a.s. (e.g., Bernoulli if $d=1$) and that it is not supported on a proper-dimensional linear subspace (e.g., there exists some coordinate direction along which both unit vectors have weight zero a.s.). This example illustrates that periodicity issues may cause problems in satisfying the assumptions, however such problems can always be solved by considering the modified kernels that are defined by considering several steps of the random walk.

For these modified two-step kernels $K_n$, Assumption \eqref{a24} is satisfied with $\Fd(x) = \ind_{\{x=0\}}$ as two kernels $K_n(x,\bullet)$ and $K_n(y,\bullet)$ are independent unless $x=y$ for $x,y$ in the lattice $I$, due to the i.i.d. nature of the environment.
\end{ex}

The second example illustrates how we can obtain values of $p > 1$ in Assumption \ref{a1} (\ref{a23}). 

\begin{ex}[Symmetric random walk in random environment]\label{ex:2}
Take $I = \mathbb Z^d$. Let $\{U_{n,    x}\}_{n \in \mathbb N,    x \in \mathbb Z^d}$ be a family of i.i.d. uniform random variables taking values in the interval $[0,1]$. Define 
\begin{align*}
    K_n(   x,    y) = \begin{cases} \frac{1 - U_{n,   x}}{2d} & |   y -    x|=1 \\
   U_{n,    x} &    y =    x  \\
    0 & \text{else}.
    \end{cases}
\end{align*}
The environment is symmetric, so we have that the first moment of the first step of the random walk is deterministically zero, i.e., 
$$\sum_{y\in\mathbb Z^d}    y\;K_1(   0,    y) =   0.$$
We have to go to the second moment to see something random, e.g.
$$\sum_{y\in\mathbb Z^d} y_1 ^2\;K_1(   0,    y) = \frac1d(1 - U_{n,x}).$$ It follows from the definition of $p$ in Assumption \ref{a1} that $p=2$.

\end{ex}

The next example illustrates that despite our discrete-time setup, one can easily fit continuous-time models into our framework.

\begin{ex}[Diffusion in a random medium/Kraichnan flows] \label{diffrm} Consider a vector valued Gaussian space-time white noise $ \vec  \eta = (\eta_1,...,\eta_d)$ with covariance matrix $$C_{ij}(x) = \mathbb E [ \eta_i(t,0) \eta_j (t,x)].$$ Assume that $C$ is compactly supported and that $C(0) = \theta \mathrm{Id}_{d\times d}$ where $\theta \in (0,1)$. As explained in \cite{kun94a, kun94b}, one may sensibly construct It\^o solutions to the SPDE given by 
\begin{equation}\label{fth}\partial_t u (t,x) = \frac12 \Delta u(t,x) + \mathrm{div} \big( u(t,x)  \vec  \eta (t,x)\big), \;\;\;\;\;\; u(0,x) = \delta_0(x),\;\;\;\;\;\;\;\;t\ge 0 , x\in \mathbb R^d.
\end{equation}
A pathwise solution adapted to the filtration generated by $  \vec \eta$ is possible thanks to the spatial smoothness of $\vec \eta $, see \cite{kun94b}. Then \eqref{fth} turns out to be conservative, meaning that $\int_{\mathbb R^d} u(t,x)dx =1$ for all $t>0$. Notice in fact that \eqref{fth} is the forward equation associated to the formal SDE $$d   X(t) =    \vec \eta(t,   X(t))dt + \theta^{1/2} d   W(t).$$ Here $   W(t)$ is a standard Wiener process in $\mathbb R^d$.\footnote{This SDE representation looks misleading because the It\^o-Stratonovich correction in \eqref{fth} is not a constant multiple of $u$, but rather it is given by $\frac12 (1-\theta) \partial_x^2 u(t,x),$ which explains why the viscosity in the It\^o equation is $\frac12$ instead of $\frac12\theta$.} Despite $\vec \eta $ being distribution-valued, \cite{kun94b} shows that such a SDE actually makes sense and that a solution exists for a.e. realization of $\vec \eta$, and furthermore that \eqref{fth} describes the evolution of the density of $   X(t)$ started from 0.

For $t>s\ge 0$ and $x,y\in \mathbb R^d$ we define the ``propagators" $u_{s,t}(x,y)$ associated to the equation \eqref{fth} as follows: let $(t,y)\mapsto u_{s,t}(x,y)$ be the solution of \eqref{fth} started from $\delta_x(y)$ at time $s$. All of these solutions are coupled to the same realization of the driving noise $\vec \eta$. Then we have the convolution property \begin{equation}\label{convo}\int_{\mathbb R^d} u_{s,t}(x,y) u_{t,r}(y,z) dy = u_{s,r}(x,z),\;\;\;\;\; \forall s<t<r\;\;\;\;\;a.s.,\end{equation}
and furthermore if $\{(s_i,t_i]\}_{i=1}^m$ are disjoint intervals then $\{u_{s_i,t_i}\}_{i=1}^m$ are independent (because $u_{s,t}$ is always measurable with respect to the restriction of $\eta$ to $[s,t]\times \mathbb R$) \cite{kun94a,kun94b,ljr02}.

We may thus let $I=\mathbb R^d$ and then define the kernels $$K_n(x,y):= u_{n,n+1}(x,y).$$ The preservation of mass implies that $\int_{\mathbb R^d} K_n(x,y)dy=1$ for all $x\in \mathbb R^d$ and $n\in \mathbb Z_{\ge 0}$. See \cite[Proposition 6.4]{Par24} for a proof that this model satisfies the conditions of Assumption \ref{a1} with $p=1$, thanks to the compact support property of the matrix function $C$.  That proof is for $d=1$, but the same method applies for all $d$. 
\end{ex}

\begin{ex}(The KMP model) Recent work of \cite{bc25} shows that the so called Kipnis-Marchioro-Presutti (KMP) model \cite{MR656869} falls within our framework. To describe it in words, for each \textit{edge} on $\mathbb Z^d$ one has a sequence of Poisson clocks $(T_j)$ and also a sequence of i.i.d. random variables $(U_j)$ taking values in $[0,1]$. Meanwhile, at each \textit{vertex} $x\in \mathbb Z^d$ one has a time-evolving random energy $\eta_t(x)$ that is non-negative. The Markovian dynamics for the temporal evolution of the energy profile are given as follows: when the Poisson clock $T_j$ rings on edge $\{x,y\}$, the total energy of the neighboring two vertices is redistributed according to the $[0,1]$-valued random variable $U_j$; more precisely, $(\eta_{T_j}(x), \eta_{T_j}(y)) $ is replaced by $(U_j( \eta_{T_j}(x)+\eta_{T_j}(y)), (1-U_j) (\eta_{T_j}(x) + \eta_{T_j}(y)))$. 
We refer to \cite[Section 2]{bc25} for a discussion of how this model satisfies our assumptions. That proof is for $d=1$, but the method applies for general $d\ge 1$.
\end{ex}

\begin{ex}[Random landscape model] Let $(\omega_{n,   x})_{n\in \mathbb N,    x\in \mathbb Z^d}$ be any collection of i.i.d. random variables. Let $b: \mathbb Z^d \to [0,1]$ be any deterministic function of finite support, and assume that the additive subgroup of $\mathbb Z^d$ generated by $\{   x: b(   x)>0\}$ is all of $\mathbb Z^d.$
Define the kernels 
\begin{equation}K_n(x,y):= \frac{b(y-x) e^{\omega_{n,y}}}{\sum_{y'\in \mathbb Z} b(y'-x) e^{\omega_{n,y'}}}. 
\end{equation}
These kernels $K_n(x,\bullet)$ are spatially correlated for distinct values of $x$, but it is clear that the dependency is only finite-range due to the finite support of $b$. Thus, these kernels satisfy Assumption \ref{a1}. See \cite[Proposition 6.2]{Par24} for a proof with weaker assumptions on the function $b$ and the weights $\omega$.
\end{ex}

\subsection{Setup and main results}\label{setup}
We now move on to the main results of this paper. The main goal of this paper is to probe the fluctuations of the quenched probability density $P_\nu^\omega(Nt,\bullet)$ from \eqref{kn}. We would like to take $N$ very large and then probe this density around different choices of spatial locations $   d_N\in \mathbb R^d$, and explore how the fluctuations of $P_\nu^\omega(Nt,\bullet)$ change in various different regimes of $ d_N$. Thus, we would like to study the fluctuations of $P_\nu^\omega(Nt,   d_N t + N^{1/2} x).$ The extra $N^{1/2}x$ window around $d_Nt$ turns out to be the correct transversal fluctuation scale in which we can observe a nontrivial space-time limit. This probability will be  very small for spatial locations $   d_N t +N^{1/2}    x $ far from the origin, thus to hope to obtain nontrivial fluctuations, one needs to multiply by a large deterministic value depending on $N, t$, and $   x$. We would also like freedom to choose the initial condition $\nu$ quite generally, not just Dirac.

We will define a deterministic prefactor $C_{N,t,   x,\nu}$ so that $N^{1/2}C_{N,t,   x}P_\nu^\omega(Nt,   d_N t + N^{1/2}    x)$ has nontrivial fluctuations. In particular, we will choose $C_{N,t,x,\nu}$ so that it has the effect of exponentially tilting our random walk. 

We first recall some basic facts about exponential tilting. Let 
\begin{equation}
    M(   \varsigma):=\int_I e^{   \varsigma \bullet    x} \mu(d   x) \label{MGF}
\end{equation} be the moment generating function of $\mu$, where $\mu$ is the annealed measure defined in Assumption \ref{a1} Item \eqref{a22}, and where $\bullet $ is the usual dot product of vectors in $\mathbb R^d.$  Then we can define the exponentially tilted measures 
\begin{equation}
    \mu^{  \varsigma}(\mathrm d    x) := e^{   \varsigma \bullet    x - \log M (   \varsigma)} \mu(\mathrm d   x). \label{eq:expTiltedMeasures}
\end{equation}
The mean of a random vector distributed according to $\mu^{  \varsigma}$ is given by $\int_I    x \mu^{  \varsigma}(d    x) =  \nabla \log M(   \varsigma).$ Let $R_t$ be a random walk starting from initial state $\nu$ with one-step transition probabilities given by $\mu$. Then if we define $C_{N,t,   x,\nu}$ so that $$C_{N, t, \frac{R_{Nt} -    d_N t}{N^{1/2}},\nu} = \exp \left(   \varsigma_N \bullet R_{Nt} - \log M (   \varsigma_N)Nt -\log\int_I e^{\varsigma_N\bullet a} \nu(\dr a)\right)$$ for some vector $   \varsigma_N\in \mathbb R^d$, then this has the effect of tilting our random walk so that $R_{Nt}-R_0$ has mean $N t \nabla \log M (   \varsigma_N)$ under the tilted measure. This means that we should choose the vector $   \varsigma_N\in \mathbb R^d$ so that the original desired spatial scaling location $   d_N$ is given by $   d_N = N \nabla \log M(   \varsigma_N)$.

In practice, it will be more convenient to first choose $   \varsigma_N$ and then choose $   d_N$ accordingly.

\begin{ass}[The location strength vectors] \label{psi_n} Throughout this paper, the sequence
    $\{   \varsigma_N\}_{N\ge 1}$ will denote any deterministic sequence of vectors in $\mathbb R^d$ such that $|N   \varsigma_N| = O(\psi_N(p,d))$ as $N\to \infty$, where 
    $$\psi_N(p,d) = 
    \begin{cases}
     N^{1-\frac1{4p}} ,& d=1\\ 
     N (\log N)^{-\frac1{2p}} , & d=2\\N, &d\ge 3.
    \end{cases}$$
Recall that $p$ denotes the symmetry parameter from Assumption \ref{a1} \eqref{a23}. We then define the \emph{location vector} which is determined from $   \varsigma_N$ according to the relation: \begin{equation}\label{dn}   d_N:= N   \nabla\log  M(    \varsigma_N).\end{equation} 
\end{ass}

We refer to $   d_N$ as the \textbf{location} while $   \varsigma_N$ is the \textbf{location strength}. Each determines the other bijectively, as a consequence of the inverse function theorem. Note that $N^{-1}   d_N$ and $   \varsigma_N$ agree to first order, but their difference may have lower order terms coming from a Taylor expansion. If $\mu$ is a standard Gaussian, which is indeed the case for many interesting continuous-space models where $I=\mathbb R^d$, then $   d_N=N   \varsigma_N$. In the general case, the lower order terms are given by higher cumulants of the one-step measure $\mu$ from Assumption \ref{a1}.

We can now define the prefactor $C_{N, t, x,\nu}$. 
\begin{defn}[The rescaling constants]\label{cntxnu}
Consider a probability measure $\nu$ on $I$, as well as $N\in \mathbb N$, $t\ge 0, x\in\mathbb R^d$. Furthermore, let $  \varsigma_N$ be any sequence of vectors in $\mathbb R^d$ as in Assumption \ref{psi_n}, and let $   d_N$ be the location vector defined from $  \varsigma_N$ as in \eqref{dn}. Define a sequence of deterministic space-time functions $$C_{N,t,   x, \nu}:=\exp\bigg(N^{1/2}   \varsigma_N \bullet    x +\big[   \varsigma_N\bullet    d_N - N\log M(   \varsigma_N)\big]t - \log \int_I e^{\varsigma_N\bullet x} \nu(\dr x)\bigg).$$ We will use the notation $C_{N,t,x}:= C_{N,t,x,\delta_0}$ as a shorthand. We also define the scaling constants $$\mathscr B_N := \begin{cases} N^{\frac{p}2+\frac{d-2}{4}} ,& |   \varsigma_N| = O(N^{-1/2}) \text{ as } N\to\infty \\\frac{N^{\frac{d-2}{4}}}{|   \varsigma_N|^p} , & N^{1/2}|   \varsigma_N| \to \infty \text{ as } N\to \infty.\end{cases}$$ 
\end{defn}

\begin{defn}[Rescaled prelimiting density field]
Choose a sequence of initial states $\nu_N$ (so each $\nu_N$ is a probability measure on $I$). Let $P^\omega_{\nu_N}(r,\dr x)$ be as in \eqref{kn}, with initial state $\nu_N$ and environment $\omega = (K_i)_{i=1}^\infty$ satisfying Assumption \ref{a1}. Furthermore, let $\varsigma_N$ be any sequence of deterministic vectors satisfying Assumption \ref{psi_n} and let $d_N$ be related to $\varsigma_N$ via \eqref{dn}. For $t\in N^{-1}\mathbb Z_{\ge 0}$ we define the \textbf{macroscopic quenched density field} \begin{equation}\label{hn}\mathfrak H^N(t,\phi):= \int_I C_{N,t,N^{-1/2}(   x-   d_Nt), \nu_N} \phi(N^{-1/2}(   x-   d_Nt))P_{\nu_N}^\omega(Nt,\mathrm d   x)
,\;\;\;\;\;\;\; \phi\in C_c^\infty(\mathbb R^d).\end{equation}
Here $N\ge 1$. Then define its recentered and rescaled version as $$\mathfrak F_N:= \mathscr B_N\cdot (\mathfrak H_N-\mathbb E[\mathfrak H_N]).$$
\end{defn}

Thus for each $t\in N^{-1}\mathbb Z_{\ge 0}$ we note that $\mathfrak H^N(t,\cdot)$ is 
a nonnegative random Borel measure on $\mathbb R^d$, which is informally equal to the function given by $   x\mapsto N^{1/2}C_{N,t,   x,\nu_N}P_{\nu_N}^\omega(Nt,   d_Nt +N^{1/2}   x)$
. The definition of $\mathfrak H^N(t,\bullet)$ is extended to $t\in \mathbb R_+$ by linearly interpolating, i.e., taking an appropriate convex combination of the two measures at the two nearest points of $N^{-1}\mathbb Z_{\ge 0}$. 

We explain the meaning of all of the constants. The factor $C_{N, t,    x,\nu}$ encodes the first order behavior of $P_{\nu_N}^\omega(Nt,   d_Nt +N^{1/2}   x)$, and it can be thought of it as a tilting factor that must appear to compensate the choice of taking a general scaling location $\varsigma_N$. Meanwhile $\mathscr B_N$ encodes the scale of its fluctuations, in other words $\mathscr B_N \asymp (\mathrm{Var}(\mathfrak H^N(t,\phi))^{-1/2}$ as we will see. 

We will use the notation $\nu(\dr x) = \mu(c\cdot \dr x)$ to mean that $\int_I f(x) \nu(\dr x) = \int_I f(c^{-1} x) \mu(\dr x).$ In this language, the rescaled density fields have initial conditions $\mathfrak H^N(0,\dr x)= \nu^{\varsigma_N}_N(N^{1/2} \cdot \dr x),$ where $\nu^{\varsigma_N}_N(\dr x) =  e^{  \varsigma_N\bullet x - \log \int_I e^{\varsigma_N\bullet a} \nu_N(\dr a)} \nu_N(\dr x)$. In all of our main results, we will need the sequence $\mathfrak H^N(0,\dr x)$ to satisfy the following regularity condition. 

\begin{defn}[Good sequence of initial data]\label{goodseq} We say that the deterministic sequence of measures $\mathfrak m_N(\dr x)$ on $\mathbb R^d$ is a \textbf{good sequence of measures} if $$\sup_N \mathfrak m_N(\{a\in \mathbb R^d:|a|\ge M\}) \leq Ce^{-c M}$$ for some $C,c>0$ independent of $M\ge 0$, and furthermore there exists $\e>0$ 
for which one has the estimate uniform over $t\in (0,1]$ and $N\ge 1$: $$\sup_{a\in \mathbb R^d} | (\mathfrak m_N, G_t(a-\bullet))_{L^2(\mathbb R^d)} | \leq C(t\vee N^{-1})^{\e-1}$$ for some $\e>0$. Here $G_t(x) = (2\pi t)^{-d/2} e^{-|x|^2/2t}$ is the $d$-dimensional heat kernel. 
\end{defn}

In particular the $\mathfrak m_N$ have uniformly bounded total mass by taking $M=0$ in the first bound. In the second bound, $G_t$ could be replaced by any smooth mollifier of width $t^{1/2}$. We will require $\mathfrak H^N(0,\dr x)$ to be a good sequence in all of our main results, and we will explain in Remark \ref{explain} why this condition is necessary to obtain a nontrivial scaling limit for the field $\mathfrak F_N$ defined above, and in particular why tightness of $\mathfrak H^N$ would fail for a non-good initial sequence. 

The second bound in the good sequence condition is the more important one. Equivalently (and perhaps more intuitively) that bound says that the probability under the tilted initial states $\nu_N^{\varsigma_N}(0, \dr x)$ of sampling a given neighborhood of size 1 around any point of $\mathbb R^d$ is bounded above by $CN^{1-\frac{d}2-\e}$. In dimension 1, \emph{any} sequence of probability measures will satisfy the second bound with $\e=1/2$, while in $d= 2$ this just barely fails as $\e$ would have to be 0 (e.g. $\mathfrak m_N = \delta_0$). In functional-analytic terms, being a good sequence \textit{almost} means having $C^{2(\varepsilon-1)}(\mathbb R^d)$-norm bounded uniformly in $N$, where the latter denotes a H\"older space of negative exponent (see Section \ref{sec:6}). We say ``almost" here because in the definition above, we introduced an extra ``microscopic cutoff" at $t=N^{-1}$ to account for the grid spacing which is $O(N^{-1/2})$ under the macroscopic view---spatial scales below this are sub-microscopic and thus completely irrelevant in all of the analysis. 

\textbf{Invariant measure $\pi^{\mathrm{inv}}(\dr x)$.} An important object appearing in all of our results will be the invariant measure of the Markov kernel $\pdif$ from Assumption \ref{a1} Item \eqref{a24}. Similar to the invariant measure for the simple symmetric random walk, this is an \emph{infinite} measure, not a probability distribution. We will show in Theorem \ref{thm:invMeasure} that this invariant measure exists and is unique (up to a scalar multiple) if we assume constant growth at infinity. We will therefore define the \emph{unit normalization} $\pi^{\mathrm{inv}}(\dr x)$ in Definition \ref{def:unit} to be a specific choice of normalization for this invariant measure. Define the dimension-dependent constants $c_1 =1, c_2=2\pi,$ and $c_d= d(d-2) /\Gamma(1+\frac{d}2)$ for $d\ge 3$. Then informally, the unit normalization is the normalization under which $2 c_d \cdot \pi^{\mathrm{inv}}$ resembles the counting measure (when $I=\mathbb Z^d$) or the Lebesgue measure (when $I=\mathbb R^d$) at infinity.

In all of our main results we will write ``$a(N)\ll b(N)$ as $N\to \infty$" if $a(N)/b(N)\to 0$ as $N\to \infty$. Our main results are then as follows.

\begin{thm}[Main result in spatial dimension $d=1$]\label{main1}
Let $\{K_n\}_{n\ge 1}$ denote an environment satisfying Assumption \ref{a1}, and let $\mathfrak H^N$ be as in \eqref{hn}. Assume as in Definition \ref{goodseq} that one has a good sequence of deterministic initial data $\mathfrak H^N(0,\bullet)$ converging weakly as measures to some limit $\mathfrak H_0$. There is an explicit Banach space $X$ of distributions on $\mathbb R$, which is continuously embedded in $\mathcal S'(\mathbb R)$, such that the rescaled and recentered collection $\{ \mathfrak F^N\}_{N \ge 1}$ is tight with respect to the topology of $C([0,T], X)$. Furthermore any limit point as $N\to \infty$ is as follows. 
\begin{itemize} 
\item (Bulk fluctuations in $d=1$) If $N^{1/2}\varsigma_N \to \boldsymbol{\varsigma} \text{ as } N\to\infty$, then the limit point is unique and coincides with the law of $(t,x)\mapsto e^{\varsigma\bullet x +\frac12 |\varsigma|^2 t} \mathcal U(t,x+\varsigma t)$ where $\mathcal U$ is the solution of the additive-noise SPDE \begin{equation}\label{she3}\partial_t \mathcal U(t, x) = \frac12\partial_x^2 \mathcal U(t, x) + \gamma_{\mathrm{ext}}\cdot \partial_x^p \big((\mathfrak H_0 *G_t)(x)\;\xi(t, x)\big),
\end{equation}started with initial data $\mathcal U(0, x)=0$.

\item (Extremal fluctuations up to location scale $N^{1-\frac1{4p}}$ in $d=1$) If $N^{-1/2}\ll  |\varsigma_N|\ll  N^{-1/4p} \text{ as } N\to \infty,$ then the limit point is unique and coincides with the law of the solution of the additive-noise SPDE \begin{equation}\label{she4}\partial_t \mathcal U(t, x) = \frac12\partial_x^2  \mathcal U(t, x) + \gamma_{\mathrm{ext}}\cdot   (\mathfrak H_0 *G_t)(x)\; \xi(t, x),
\end{equation}started with initial data $\mathcal U(0, x)=0$. 
\end{itemize} Here $G_t(x) = G(t,x)= (2\pi t)^{-1/2}e^{-x^2/2t}$ is the standard heat kernel, $*$ denotes spatial convolution, and $ \xi$ is a standard space-time white noise defined on $\mathbb R_+\times \mathbb R.$ Furthermore in either case we have the noise coefficient $$\gamma_{\mathrm{ext}}^2:=  \frac1{(2p)!}\int_{I} \bigg[ \int_{I^2}( x- y)^{2p}\mu(\mathrm d x)\mu(\mathrm d y)-\int_I ( a- z)^{2p} \pdif( z,\mathrm d a)\bigg]
\pi^{\mathrm{inv}}(\mathrm d z).
$$Here $\pi^{\mathrm{inv}}$ is the invariant measure of the Markov kernel $\pdif$, under its unit normalization (see Definition \ref{def:unit}).
\end{thm}

This theorem will be proved in Section \ref{sec:6}, and $X$ will be some weighted H\"older space of negative exponent.
Also note that for models such as the nearest-neighbor random walk in i.i.d. space-time random environment (Example \ref{ex:1}), the measure $\mu$ is symmetric and $\pdif( z,\mathrm d a)  = (\mu*\mu)(z-\mathrm d a)$ for $z \neq 0$, and therefore the outer integral consists of just a single term as opposed to the full integral written above. 

Interestingly, the fluctuations are different in the bulk region versus anywhere away from the bulk. The fluctuations away from the bulk region do not have a derivative term around the noise, in particular they have better regularity. This can be understood as follows. In the bulk, the limiting SPDE should conserve mass as the prelimiting object is a probability density; however, in the extremal regime, mass is no longer conserved due to multiplying by the prefactor $C_{N, t, x}$ (in the bulk when $   \varsigma_N =0$, note that $C_{N, t,x} = 1$). 

\begin{rk}[Extremal fluctuations at scale $N^{1-\frac1{4p}}$ in $d=1$] \label{critkpz}We remark here on previous results \cite{DDP23,DDP+,Par24} dealing with the regime $\varsigma_N=\beta N^{-1/4p}+o(N^{-1/4p})$. If $\beta\ne 0$, and if $\mathfrak H^N(0,\bullet)$ is a Dirac mass at the origin, then the limit point of $\mathfrak F_N$ is unique in $C([0,T],X)$ and it coincides with the law of $\mathcal U-G$ where $G$ is the deterministic heat kernel and $\mathcal U$ is the solution of the \textit{multiplicative}-noise SPDE \begin{equation}\label{mshe}\partial_t \mathcal U(t, x) = \frac12\partial_x^2  \mathcal U(t, x) + \gamma_{\mathrm{ext}}\cdot \beta^p \cdot   \mathcal U(t,x)\; \xi(t, x),
\end{equation}started with initial data $\mathcal U(0, x)=\frac1{\beta^p}\delta_0(x)$. This was proven in \cite{DDP23,DDP+} for two specific models, and was proven in \cite{Par24} for the full class of models in $d=1$ satisfying Assumption \ref{a1}. 

In other words, when $|  \varsigma_N|$ is \textit{precisely} of order $N^{-1/4p}$ as $N\to \infty$, then no rescaling or recentering are needed, as the fluctuations are of order 1 around the mean field limit, which is consistent with the expression for the rescaling factor $\mathscr B_N$ even at this scale. 

For \textit{even larger scales} $|  \varsigma_N|\gg N^{-1/4p}$ the conjecture is that Tracy-Widom fluctuations (and more generally, the \textit{directed landscape}) will appear under the famous 1:2:3 scaling. See Problem 1 in Subsection \ref{openprob} for more discussion of this.

Note that when $|  \varsigma_N|\ll N^{-1/4p}$ the limiting fluctuations are described by a \textit{linear} stochastic PDE (in particular it is a Gaussian field); however, at the \textit{critical scale} of $  \varsigma_N=N^{-1/4p}$, the fluctuations are now described by a \textit{nonlinear} SPDE that exhibits non-Gaussian marginals and much more complex behaviour. We believe that such a transition from Gaussian to non-Gaussian fluctuations occurs in all dimensions,  but it is currently known only in $d=1$. In higher dimensions the critical scale is much more subtle, and we can only prove results about the fluctuations very slightly but strictly below the true critical threshold. Below this critical threshold, the fluctuations will again be described again by linear SPDEs.
\end{rk}

Next, we discuss the $d=2$ case.

\begin{thm}[Main result in spatial dimension $d=2$]\label{main2}
Let $d=2$, let $\{K_n\}_{n\ge 1}$ denote an environment satisfying Assumption \ref{a1}, and let $\mathfrak H^N$ be as in \eqref{hn}. Assume as in Definition \ref{goodseq} that one has a good sequence of initial data $\mathfrak H^N(0,\bullet)$ converging weakly as measures to some limit $\mathfrak H_0$. There is an explicit Banach space $X$ of distributions on $\mathbb R^2$, which is continuously embedded in $\mathcal S'(\mathbb R^2)$, such that the rescaled and recentered collection $\{ \mathfrak F^N\}_{N \ge 1}$ is tight with respect to the topology of $C([0,T], X)$. Furthermore any limit point as $N\to \infty$ can be described as follows. 
\begin{itemize} \item (Bulk fluctuations in $d=2$) If $N^{1/2}   \varsigma_N \to   {\boldsymbol{\varsigma} }\text{ as } N\to\infty$, then the limit point is unique and coincides with the law of $(t,x)\mapsto e^{\varsigma\bullet x +\frac12 |\varsigma|^2 t} \mathcal U(t,x+\varsigma t)$ where $\mathcal U$ is the solution of the additive-noise SPDE \begin{equation}\label{she1}\partial_t \mathcal U(t,   x) = \frac12\Delta \mathcal U(t,   x) +\sqrt{2\pi} \cdot\eta_p(t,x),
\end{equation}started with initial data $\mathcal U(0,   x)=0$. Here $\eta_p$ is the Gaussian field on $\mathbb R_+\times \mathbb R^2$ that is white in time and for each fixed time $t>0$ has a spatial covariance structure given by $\mathbb E[(\eta_p(t,\bullet),\phi)_{L^2(\mathbb R^2)}^2] = \int_{\mathbb R^2} A_p[\phi,\phi](a)\cdot (\mathfrak H_0 *G_t)(a)^2\dr a$ where
\begin{multline*}A_p[\phi,\psi](a)  := \\
\int_{I} \int_{I^2}  [(   x_1-   y) \bullet \nabla]^p \phi (a)\cdot [   x_2 \bullet \nabla]^p \psi(a) \bigg( \boldsymbol p^{(2)}\big( (   y,   0), (\dr    x_1,\dr    x_2)\big) -\mu(\dr    x_1-   y)\mu(\dr    x_2 )\bigg) \pi^{\mathrm{inv}}(\dr    y). \end{multline*}
\item (Extremal fluctuations up to location scale $N/(\log N)^{\frac1{2p}}$ in $d=2$) If $N^{-1/2}\ll |   \varsigma_N| \ll  (\log N)^{-\frac1{2p}} \text{ as } N\to \infty$ and if $\frac{   \varsigma_N}{|  \varsigma_N|} \to   {\boldsymbol v}_{\mathrm{unit}}$ as $N\to\infty$, then the limit point is unique and coincides with the law of the solution of the additive-noise SPDE \begin{equation}\label{eq:extremal1}\partial_t \mathcal U(t,   x) = \frac12\Delta \mathcal U(t,   x) + \sqrt{2\pi} \cdot \gamma_{\mathrm{ext}}(  {\boldsymbol v}_{\mathrm{unit}})\cdot   (\mathfrak H_0 *G_t)(x)\; \xi(t,   x),
\end{equation}
started from $\mathcal U(0,   x)=0.$ Here $ \xi$ is a standard space-time white noise defined on $\mathbb R_+\times \mathbb R^2$ and $\gamma_{\mathrm{ext}}$ is the homogeneous polynomial of degree $2p$ defined by
\begin{equation}\label{gext}
\gamma_{\mathrm{ext}}(  {\boldsymbol v})^2:= \frac{1}{(2p)!}\int_{I} \bigg[ \int_{I^2}[(   x-   y)\bullet   {\boldsymbol v}]^{2p}\mu(\mathrm d   x)\mu(\mathrm d   y)-\int_I [(   a-   z)\bullet   {\boldsymbol v}]^{2p} \pdif(   z,\mathrm d   a)\bigg]
\pi^{\mathrm{inv}}(\mathrm d   z).
\end{equation}
\item (Extremal fluctuations at scale $N/(\log N)^{\frac1{2p}}$ in $d=2$) There exists $\epsilon_{\mathrm{thr}} \leq 2$ such that the following holds for all $\boldsymbol v$ with $\gamma_{\mathrm{ext}}(  {\boldsymbol v})^2< \epsilon_{\mathrm{thr}}.$ If $   \varsigma_N (\log N)^{\frac1{2p}} \to   {\boldsymbol v} \ne 0\text{ as } N\to \infty,$ then the limit point is unique and coincides with the law of the solution of the additive-noise SPDE \begin{equation}\label{eq:extremal2}\partial_t \mathcal U(t,   x) = \frac12\Delta \mathcal U(t,   x) + \frac{\sqrt{2\pi} }{|  {\boldsymbol{v}}|^{p}}\sqrt{\frac{\gamma_{\mathrm{ext}}(  {\boldsymbol v})^2 }{1 -\frac12\gamma_{\mathrm{ext}}(  {\boldsymbol v})^2}}\cdot   (\mathfrak H_0 *G_t)(x)\; \xi(t,   x),
\end{equation}
started from $\mathcal U(0,   x)=0.$ Here $ \xi$ is a standard space-time white noise defined on $\mathbb R_+\times \mathbb R^2.$
\end{itemize} Above $G_t(x)=G(t,   x) = (2\pi t)^{-1}e^{-|   x|^2/2t}$ is the standard deterministic heat kernel, and $*$ is spatial convolution. Furthermore $\pi^{\mathrm{inv}}$ is the invariant measure of the Markov kernel $\pdif$, under its unit normalization (see Definition \ref{def:unit}).
\end{thm}

This theorem will also be proved in Section \ref{sec:6}. Note that in \eqref{eq:extremal2}, if we take $  {\boldsymbol{v}} \to 0$ then we recover \eqref{eq:extremal1} since  $\gamma_{\mathrm{ext}}(  {\boldsymbol{v}})^2$ is a homogeneous polynomial of degree $2p$ in the variable $  {\boldsymbol{v}}$. Similarly, if we take $|\boldsymbol\varsigma|\to\infty$ in the bulk regime and divide the resulting limiting SPDE by $|\boldsymbol\varsigma|$, then we also recover \eqref{eq:extremal1}. Thus, the limits do commute in some sense, although one cannot go the other way around to recover \eqref{she1} or \eqref{eq:extremal2} from \eqref{eq:extremal1}. This will be the case in all dimensions.

\begin{rk}\label{p=1} If $p=1$ then for the bulk case where the limit is given by \eqref{she1}, the noise $\eta_p$ with $p=1$ can be identified explicitly as $$\eta_{1}(t,x)= \mathrm{div}((\mathfrak H_0 *G_t)(x)\; \Gamma_{\mathrm{bulk}}  \vec  \xi(t,x))$$ for a standard vector-valued space-time white noise $ \vec  \xi$ on $\mathbb R_+\times \mathbb R^2$ and a symmetric $d\times d$ effective matrix $\Gamma_{\mathrm{bulk}}$ defined by the property that $$(\Gamma_{\mathrm{bulk}}^2)_{i,j} = \int_I \int_{I^2} (x_i-y_i)a_j \;\bigg(\boldsymbol p^{(2)} ((   y,0),(\dr    x,\dr   a)) - \mu(\dr   x-   y)\mu(\dr   a)\bigg) \pi^{\mathrm{inv}}(\dr   y). $$ 
Thus even the sheared version of limiting field given by $(t,x)\mapsto e^{\varsigma\bullet x +\frac12 |\varsigma|^2 t} \mathcal U(t,x+\varsigma t)$ can itself be written as an additive noise heat equation with driving noise given by $ \big(\mathrm{div} -  {\boldsymbol{\varsigma}} \bullet \big)\big( (\mathfrak H_0 *G_t)(x)\; \Gamma_{\mathrm{bulk}}  \vec  \xi(t,   x)\big).$ 

Recent work of \cite[Theorem 1.9]{Kot} proves this bulk result \eqref{she1} for a very particular choice of kernels given by a \textit{diffusion in random media} (see Example \ref{diffrm}), where $p=1,   \varsigma_N=0$, and their matrix $V_{\mathrm{eff}}$ is our $\Gamma_{\mathrm{bulk}}$. They obtain a similar topology of convergence. In this special case, $\mu$ is a standard Gaussian and $\boldsymbol p^{(2)} ((y,0),(\dr x,\dr a))$ is a diffusion whose generator can be written explicitly, hence one sees a more explicit formula for $\Gamma_{\mathrm{bulk}}^2$ than the one above.
\end{rk}

\begin{rk}\label{explain} We explain the reason that we need good sequences of initial data in the above result, and why Dirac masses are not permissible in $d\ge 2$. 
Consider in arbitrary dimension the solution of the additive-noise SPDE driven by space-time white noise $\xi:$ $$\partial_t u (t,x)= \frac12 \partial_x^2 u (t,x)+ G (t,x)\cdot \xi(t,x),$$ started from zero initial condition, where $G$ as usual is the deterministic heat kernel and where $(t,x) \in \mathbb R_+\times \mathbb R$. This corresponds to \eqref{eq:extremal1} with $\gamma_{\mathrm{ext}}(  {\boldsymbol v})=\frac1{\sqrt{2\pi}}$ and $\mathfrak H_0=\delta_{(0,0)}$. Then by Duhamel's formula, the solution will \textit{formally} be given by $$u(t,   x) = \int_{\mathbb R_+\times \mathbb R^d} G(s,   y) G(t-s,   x-   y)\xi(ds,d   y),$$ where the latter integral is against the driving white noise $\xi$. In dimension $d=1$, this makes perfect sense and it is even a continuous field in both variables. This is because the integrand is in $L^2(\mathbb R_+\times \mathbb R^d)$. In higher dimensions, the regularity of the noise becomes worse, and one certainly cannot expect function-valued solutions. One might expect that a solution $u$ exists as soon as one allows ``generalized functions," i.e., Schwartz distributions. However if $d>1$, it does not even make sense as a Schwartz-distribution valued field. Indeed the pairing $(u,\phi)_{L^2(\mathbb R_+\times \mathbb R^d)}$ for $\phi \in C_c^\infty(\mathbb R^{d+1})$ will fail to exist because the singularity of square of the heat kernel at the origin exceeds the integrability threshold as soon as $d=2.$ Notice that this bad divergence is logarithmic in $d=2$ and polynomial in $d\ge 3$, with the power getting worse as $d$ grows larger. Hence $u$ as defined above does not exist in any meaningful sense if $d
\ge 2$. Of course, if the continuum limit does not even exist for $d\ge 2$, then we cannot expect any reasonable notion of fluctuations for the prelimiting model.

To fix this issue, one notices that the only real problem is at the origin: the \textit{temporal increments} $u(t,\bullet) - G(t-t_0,\bullet) * u(t_0,\bullet) $ still make sense as Schwartz distributions on $\mathbb R^d$ for $t>t_0>0$, where $*$ is spatial convolution. Furthermore, $t_0$ can be taken to be arbitrarily small. Indeed $u(t,\bullet) - G(t-t_0,\bullet) * u(t_0,\bullet) $ is given by the field $   x\mapsto \int_{[t_0,\infty)\times \mathbb R^d} G(s,   y) G(t-s,   x-   y)\xi(ds,d   y),$ which has removed the singularity of $G(s,   y)$ at $(s,   y) = (0,   0)$ and still agrees with \eqref{eq:extremal1} up to a temporal shift by $t_0$ (e.g. $\mathfrak H_0 = \delta_{(0,0)}$ would get replaced by $\mathfrak H_0 = q_{t_0}$). Hence by shifting time by some small amount $t_0$ in the same manner for the prelimit in our above result, this essentially gets rid of this bad singularity at the origin which would prevent a meaningful scaling limit result. The most general way to fix the singularity issue is to consider a smoother profile of initial conditions so that not all random walk particles start at the same point, which is exactly what we have done. Indeed, a very general class of finite measures $\mathfrak H_0$ for which \eqref{eq:extremal1} makes sense is precisely $C^{2(\e-1)}(\mathbb R^d)$ as needed in Definition \ref{goodseq}, and this threshold is quite sharp as we will see.
\end{rk}

Note that while we only prove \eqref{eq:extremal2} for $\gamma_{\mathrm{ext}}(  {\boldsymbol v})^2< \epsilon_{\mathrm{thr}}$, this is due technical issues arising from our proof. In fact, we expect the results to hold all the way up to $\gamma_{\mathrm{ext}}(  {\boldsymbol v})^2=2$ which is the location where the noise coefficient blows up. Finally we remark that the \textit{critical and }\textit{supercritical regimes} are not covered by Theorem \ref{main2}. In other words, we do not consider what happens if one takes $    \varsigma_N =   {\boldsymbol v}/(\log N)^{\frac1{2p}}$ with $ \gamma_{\mathrm{ext}}(  {\boldsymbol v})^2\ge 2$. This is an important open problem, and we refer to the discussion in Problem 2 of Subsection \ref{openprob} for more context.

Let us now move onto the case $d\ge 3.$

\begin{thm}[Main result in spatial dimensions $d\ge 3$]\label{main3}
 Let $\{K_n\}_{n\ge 1}$ denote an environment satisfying Assumption \ref{a1}, and let $\mathfrak H^N$ be as in \eqref{hn}. Assume that $d\ge 3$, and assume as in Definition \ref{goodseq} that one has a good sequence of initial data $\mathfrak H^N(0,\bullet)$ converging weakly as measures to some limit $\mathfrak H_0$. There is an explicit Banach space $X$ of distributions on $\mathbb R^d$, which is continuously embedded in $\mathcal S'(\mathbb R^d)$, such that the rescaled and recentered collection $\{ \mathfrak F^N\}_{N \ge 1}$ is tight with respect to the topology of $C([0,T], X)$. Furthermore with $c_d:=d(d-2)/\Gamma(1+\frac{d}2)$, any limit point as $N\to \infty$ can be described as follows. 
\begin{itemize} \item (Bulk fluctuations in $d\ge 3$) If $N^{1/2}   \varsigma_N \to   {\boldsymbol{\varsigma} }\text{ as } N\to\infty$, then the limit point is unique and coincides with the law of $(t,x)\mapsto e^{\varsigma\bullet x +\frac12 |\varsigma|^2 t} \mathcal U(t,x+\varsigma t)$ where $\mathcal U$ is the solution of the additive-noise SPDE \begin{equation}\label{she8}\partial_t \mathcal U(t,   x) = \frac12\Delta \mathcal U(t,   x) +\sqrt{c_d} \cdot \eta_p(t,x),
\end{equation}started with initial data $\mathcal U(0,   x)=0$. Here $\eta_p$ is a Gaussian field on $\mathbb R_+\times \mathbb R^d$, whose covariance structure has the same formula as described after \eqref{she1}, i.e., $\mathbb E[(\eta_p(t,\bullet),\phi)_{L^2(\mathbb R^d)}^2] = \int_{\mathbb R^d} A_p[\phi,\phi](a)\cdot (\mathfrak H_0 *G_t)(a)^2da$, with $A_p[\phi,\psi]$ exactly as before. 

\item (Extremal fluctuations up to location scale $N$ in $d\ge 3$) Suppose $N^{-1/2}\ll |   \varsigma_N| \ll  1 \text{ as } N\to \infty$ and that $\frac{   \varsigma_N}{|  \varsigma_N|} \to   {\boldsymbol v}_{\mathrm{unit}}$ as $N\to\infty$. Then the limit point is unique and coincides with the law of the solution of the additive-noise SPDE \begin{equation}\label{she2}\partial_t \mathcal U(t,   x) = \frac12\Delta \mathcal U(t,   x) + \sqrt{c_d} \cdot\gamma_{\mathrm{ext}}(  {\boldsymbol v}_{\mathrm{unit}})\cdot   (\mathfrak H_0 *G_t)(x)\; \xi(t,   x),
\end{equation}
started from $\mathcal U(0,   x)=0$. Here $ \xi$ is a standard space-time white noise defined on $\mathbb R_+\times \mathbb R^d,$ and$$\gamma_{\mathrm{ext}}(  {\boldsymbol v})^2:= \frac{1}{(2p)!}\int_{I} \bigg[ \int_{I^2}[(   x-   y)\bullet   {\boldsymbol v}]^{2p}\mu(\mathrm d   x)\mu(\mathrm d   y)-\int_I [(   a-   z)\bullet   {\boldsymbol v}]^{2p} \pdif(   z,\mathrm d   a)\bigg]
\pi^{\mathrm{inv}}(\mathrm d   z).
$$
\item (Extremal fluctuations at scale $N$ in $d\ge 3$) There exists some $\epsilon_{\mathrm{thr}}>0$ and a bounded measurable function $\nu_{\mathrm{eff}}:B_{\mathbb R^d}(0,\epsilon_{\mathrm{thr}})\to \mathbb R_+$ 
such that the following holds true. Suppose that $     \varsigma_N \to   {\boldsymbol v}\ne 0\text{ as } N\to \infty,$  where $|  {\boldsymbol v}| \in (0,\epsilon_\mathrm{thr})$. Let $H_{  {\boldsymbol v}}$ denote the $d
\times d$ Hessian matrix of $\log M$ at $  {\boldsymbol v}$, where $M$ is the moment generating function of $\mu$ as in \eqref{dn}. 
Then the limit point is unique and coincides with the law of the solution of the additive-noise SPDE \begin{equation}\label{she6}\partial_t \mathcal U(t,   x) = \frac12\mathrm{div} ( H_{  {\boldsymbol v}}  \nabla) \mathcal U(t,   x) + \sqrt{c_d} \cdot\nu_{\mathrm{eff}}(  {\boldsymbol v})\cdot   (\mathfrak H_0 *G_t^{(  {\boldsymbol v})})(x)\; \xi(t,   x),
\end{equation}
started from $\mathcal U(0,   x) = 0$. Here $ \xi$ is a standard space-time white noise defined on $\mathbb R_+\times \mathbb R^d.$ The function $\nu_{\mathrm{eff}}$ depends on the choice of microscopic model through more than just $
\gamma_{\mathrm{ext}}(  {\boldsymbol v})$ and $|v|$.
\end{itemize} Above $G_t(x) = (2\pi t)^{-d/2}e^{-|   x|^2/2t}$ is the standard heat kernel, and $G^{(  {\boldsymbol v})}_t(x)=(\det H_{\boldsymbol v})^{d/2} \cdot G_t(H_{  {\boldsymbol v}}^{-1/2}x)$  
is the heat kernel for the operator $\frac12\mathrm{div} ( H_{  {\boldsymbol v}}  \nabla)$. Furthermore $*$ denotes spatial convolution, and $\pi^{\mathrm{inv}}$ is the invariant measure of the Markov kernel $\pdif$, under its unit normalization (see Definition \ref{def:unit}).
\end{thm}

Along with Theorems \ref{main1} and \ref{main2}, this theorem will be proved in Section \ref{sec:6}. Here we remark that if $  {\boldsymbol v}=0$ then $H_{  {\boldsymbol v}}$ is the $d\times d$ identity matrix, moreover if $\mu$ is the Gaussian measure (which is true for many models of interest) then $H_{  {\boldsymbol v}}$ is equal to the identity matrix for all $  {\boldsymbol v}\in\mathbb R^d$.

In the bulk case, we again emphasize as in Remark \ref{p=1} that if $p=1$, then we have that $\eta_p(t,x)= \mathrm{div}((\mathfrak H_0 *G_t^{(  {\boldsymbol v})})(x)\;\Gamma_{\mathrm{bulk}}   \vec \xi(t,x))$ for a standard vector-valued space-time white noise $  \vec \xi$ on $\mathbb R_+\times \mathbb R^d$ and the same $d\times d$ matrix $\Gamma_{\mathrm{bulk}}$. Thus the limiting field $(t,x)\mapsto e^{\varsigma\bullet x +\frac12 |\varsigma|^2 t} \mathcal U(t,x+\varsigma t)$ can itself be written as an additive noise heat equation with driving noise given by $$ \big(\mathrm{div} -  {\boldsymbol{\varsigma}} \bullet \big)\big( (\mathfrak H_0 *G_t^{(  {\boldsymbol v})})(x)\;\Gamma_{\mathrm{bulk}}  \vec  \xi(t,   x)\big).$$ In the particular case of $   \varsigma_N=0$ and the $K_n$ given by a diffusion in a random media (see Example \ref{diffrm}), this result was recently shown in \cite{Kot}.

\begin{rk}In the extremal case, note in the third bullet point that the appearance of the effective diffusion matrix and the effective environmental variance coefficient in the limiting SPDE is consistent with the result of \cite{GRZ} about the \textit{directed polymer} partition function in $d\ge 3$, which is a different but related model (see Subsection \ref{openprob} below). Their proof also used a tilting method and an invariance principle for the tilted process, though their SPDE approximations were non-Markovian and the reason for their tilting is a different one than the one here. 
One of the main observations of the present work is that for the particular case of RWRE models of the type considered here, a certain type of tilting approach works to obtain the fluctuations in \emph{all} dimensions under the same unified framework, not just for $d\ge 3$. Furthermore, it works for all possible choices of microscopic interactions, as embodied by the generality of Assumption \ref{a1}. In particular these form a nice class of prelimiting models with a rich class of observables, from which one can hope to understand problems concerning critical phenomena for directed systems in all spatial dimensions.
\end{rk}

The constant $\epsilon_{\mathrm{thr}}$ in our result is not a physically meaningful one: it is simply the value at which our methods break down. Beyond this value, we are unsure what to expect and whether there is a critical value analogous to the one in $d=2$, see Problem 3 in Section \ref{openprob} below.

Interestingly, when $d\ge 3$ the function $\nu_{\mathrm{eff}}$ may depend on the model. This model-specific dependence was not the case for the critical scale in the $d=2$ case earlier, when we saw that this effective noise coefficient was ``universal" in the sense that it degenerated to a single function $\frac12 \gamma_{\mathrm{ext}}^2 \mapsto \frac{1}{1 -\frac{1}{2}\gamma_{\mathrm{ext}}^2}$ for \textit{all possible choices of models}. Note that this function ($f(\lambda)= \frac{1}{1-\lambda}$) is the moment generating function of an exponential random variable. If $d\ge 3$ and the model has only nearest-neighbor interactions and independence between distinct spatial transition kernels, then one may show that $\nu_{\mathrm{eff}}^2$ is the moment generating function of some multiple of a geometric random variable of some rate. In general, $\nu_{\mathrm{eff}}^2$ will take a more complicated form, which is the subject of the following theorem. We will now see that it is a kind of moment generating function of a certain additive functional of the tilted 2-point motion generated by the stochastic kernels $K_i$. 

\begin{thm}[Identification of the coefficient at the critical scale in $d\ge 3$] \label{main4} Let $d
\ge 3$. In the context of the third bullet point of Theorem \ref{main3}, we have that $\nu_{\mathrm{eff}}^2(  {\boldsymbol v})$ can be written as follows. Recalling $\boldsymbol p^{(2)}$ from \eqref{kptkernel}, let 
\begin{align*} 
\boldsymbol u_{  {\boldsymbol v}}(   y_1-   y_2)
=\log \int_{I^2} e^{\sum_{j=1,2}   {\boldsymbol v} \bullet (   x_j-   y_j) } \boldsymbol p^{(2)} \big( (   y_1,   y_2),(\mathrm d   x_1,\mathrm d   x_2)\big)   \;\; -\;\; 2 \log M(  {\boldsymbol v}).
\end{align*}
Let $\boldsymbol q_{  {\boldsymbol v}}^{(2)}((x_1,x_2),(\dr y_1,\dr y_2))$ denote the Markov kernel defined by a Radon-Nikodym derivative with respect to $\boldsymbol p^{(2)}((x_1,x_2),(\dr y_1,\dr y_2) )$ that is proportional to $e^{\sum_{j=1,2}   {\boldsymbol v} \bullet (x_j-y_j) }.$ 
 Let $\mathbf E_{y}^{(  {\boldsymbol v},\mathrm{diff})}$ denote the expectation with respect to the Markov chain $(X_r)_{r\ge 0}$ with state space $I$, given by the difference of the two coordinates of the Markov chain associated to the kernel $\boldsymbol q_{  {\boldsymbol v}}^{(2)}((x_1,x_2),(\dr y_1,\dr y_2))$. Define $\pi_{  {\boldsymbol v}}^{\mathrm{inv}}(\dr y )$ to be the unique invariant measure for that Markov chain $(X_r)_r$. Then we have that 
\begin{equation}\nu_{\mathrm{eff}}^2(  {\boldsymbol v}):= \frac1{|  {\boldsymbol v}|^{2p}} \int_I \left(e^{\boldsymbol u_{  {\boldsymbol v}}(   y)} -1\right)\mathbf E_y^{(  {\boldsymbol v},\mathrm{diff})} \bigg[e^{\sum_{s=1}^{\infty} \boldsymbol u_{  {\boldsymbol v}}(X_s)} \bigg] \pi_{  {\boldsymbol v}}^{\mathrm{inv}} (\dr y). \label{nueff}
\end{equation}
\end{thm}
Note that in both \eqref{eq:extremal2} and \eqref{nueff}, the factor $1/|  {\boldsymbol v}|^{2p}$ is somewhat artificial, just coming from the form of the rescaling constants $\mathscr B_N.$

It is important that, unlike all other results, the result in the third bullet point of Theorem \ref{main3} does not depend at all on the value of $p\in\mathbb N$. This is because all other results ``zoom in" on $  {\boldsymbol v}=0$ and thus a Taylor expansion of order $p$ in the variable $  {\boldsymbol v}$ is relevant in those results, whereas for this result there is no such expansion needed, see Section \ref{sec:5} for more details. This is also why in $d\ge 3$ the universality breaks and we obtain a model-dependent effective variance.

\begin{thm}[Convergence of f.d.d's for the entire flow of initial data]\label{main5}
Fix $n\in \mathbb N$. For $1\le j \le n$, let $\mathfrak H^N_j(t,\phi)$ be as in \eqref{hn}, but where the subscript now corresponds to different good sequences of initial conditions $\mathfrak H^N_j(0,\bullet)$, all of which are deterministic and converging weakly to some initial measures $\mathfrak H_j$ $(1\le j \le n$) as $N\to \infty$. Then in all of Theorems \ref{main1}, \ref{main2}, and \ref{main3}, the entire tuple $(\mathfrak H^N_1, ..., \mathfrak H^N_n)$ converges in law in the topology of $C([0,T], X)^n$. The limit is given by the natural coupling of the solutions of the stochastic PDEs written in those theorems.
\end{thm}

This theorem will be proved at the very end of Section \ref{sec:6}. The natural coupling means that all of the solutions are run with the same driving noise $\xi$, but $\mathfrak H_0$ is replaced by $\mathfrak H_j$ in the $j^{th}$ equation. For \eqref{she1} and \eqref{she8}, it is not so clear what we mean by this, since these objects have only been defined in terms of their covariance kernels and not some underlying space-time white noise $\xi$, but in these cases one can explicitly write the joint covariance kernels of the two Gaussian fields, and the proof will make this clear.

\subsection{Classification of regimes}
The three main theorems each have several cases, so that the total number of convergence results is eight if we include all of the separate regimes from Theorems \ref{main1}, \ref{main2}, and \ref{main3}. However, we are not going to prove all eight of these results separately. Instead, it will turn out that all eight results can be consolidated into \textbf{four separate cases} which will be denoted as follows throughout the paper:

\begin{enumerate}
    \item \label{eq:regimeA} \textbf{Regime A (Bulk Regime):} The first case deals with the bulk regime in all dimensions $d \geq 1$. This is the regime where $$N^{1/2}   \varsigma_N \to   {\boldsymbol{\varsigma} }\in\mathbb R^d \;\text{ as } \;\;N\to\infty.$$ The proofs in the bulk regime proceed similarly in all dimensions. There is an important simplification that can be made in this regime and that will be justified in Section \ref{sec:3}. This is to assume that $    \varsigma_N \equiv 0$. Finally note that the SPDE obtained in \eqref{she3} is actually the $d=1$ case of the more complicated formula given in \eqref{she1}, so that the bulk SPDE limit is actually of the same form in every dimension. 
    \item \label{eq:regimeB} \textbf{Regime B (Extremal Regime I) :}  The second case deals with $   \varsigma_N$ in between the bulk and critical scales in all dimensions. In other words, $$N^{-1/2} \ll |  \varsigma_N| \ll  N^{-1}\psi_N(p,d).$$ 

    \item \label{eq:regimeC} \textbf{Regime C (Extremal Regime II in $d=2$):} The third regime deals with $   \varsigma_N$ at the critical scale in $d=2$: 
    $$   \varsigma_N/\psi_N(p,d) =    \varsigma_N(\log N)^{\frac{1}{2p}} \to   {\boldsymbol v} \neq 0.$$

    \item \label{eq:regimeD} \textbf{Regime D (Extremal Regime II in $d \geq 3$):} The last regime deals with $   \varsigma_N$ at the critical scale in $d \geq 3$:   $$   \varsigma_N/\psi_N(p,d) =    \varsigma_N \to   {\boldsymbol v} \neq 0.$$ 
\end{enumerate}

Although we use the term ``critical scale" it should be made clear that these results are all \textit{subcritical} ones in the language of e.g. \cite{u5, CSZ_2d}. True criticality should be reached at the same order of magnitude (hence our terminology), but only past a specific critical location, which our results do not cover. In particular, these four regimes do not encompass the critical and supercritical cases, where we expect non-Gaussian behavior.

\subsection{Intuition and ideas of the proof} Here we list some of the key ideas in the proof. 
\\ 
\\ 
\textbf{1. Martingale Problem}
Take $p=1$ and $  \varsigma_N=0$ for simplicity here. In this particular case, one may prove the theorem by first showing that for any $C([0,T],\mathcal S'(\mathbb R))$ valued random process $u(t,\bullet)$ that is adapted to the canonical filtration, if the processes 
\begin{align*}
    M_t(\phi)&:= u(t,\phi)- \frac12 \int_0^{t}  u(s,\Delta \phi)ds \\ G_t(\phi) &:= M_t(\phi)^2 - \gamma_{\mathrm{ext}}^2 \int_0^{t} \big(G(s,\bullet)^2, |   \nabla \phi
    |^2 \big)_{L^2(\mathbb R^d)}\dr s
\end{align*}
are both martingales for every $\phi\in C_c^\infty(\mathbb R^d)$, then $u(t,\bullet)$ must have the same law as the solution of \eqref{she3}, \eqref{she1}, and \eqref{she8} (depending on $d$) with $p=1$. Likewise for the case where $p=1$ and $|  \varsigma_N|\gg N^{-1/2}$, if one replaces $|   \nabla \phi|^2$ with just $\phi^2$ then one obtains the solution of \eqref{she2} (see Theorem \ref{uniqueness_of_mart_prob} for the general claim). 
Then we will prove tightness of $\mathfrak H^N$ and show that any limit point must have this martingality property.

As mentioned earlier, previous results about the fluctuations of $P^\omega_\nu$ were proved either using cluster expansion techniques or by using the environment seen from the particle. The method of the present work will be quite different, and will instead capitalize on the discrete SPDE technique, in which we show that the prelimiting model satisfies a discretized version of the above martingale problem. This discretized martingale problem is in turn based on Girsanov's formula. This Girsanov technique was first discovered independently in the papers \cite{dom, DDP23}, but in the present work we have numerous additional difficulties coming from the general assumptions and the higher spatial dimensions that will need to be overcome. We refer to Section \ref{sec:2} for a precise discussion of how the Girsanov transform will be used. 
\\
\\
\textbf{2. Exponential Functionals of Random Walks.}
Let us try to give some intuition as to \textit{why} one can go up to location scales of order $\psi_N(p,d)$ but not further in the results. As mentioned in the discussion at the beginning of Subsection \ref{setup}, our results are based on a tilting approach where the underlying random walk obtains a slight bias due to the effect of the rescaling constants $C_{N,t,x,\nu}$. The analysis of the quadratic variations of the martingales appearing in the discrete SPDE will eventually lead us to calculating moment-like quantities of the field $\mathfrak H^0_N$. When calculating the moments of the field $\mathfrak H^N$, the Girsanov tilting will be induced by products of the $C_{N,t,x,\nu}$. This tilting does not come for free, as one has to pay a cost which is seen under the tilted measure. This cost takes the form of an exponential functional of the tilted random walk paths, see \eqref{eq:tiltingExpectations}. 

We now explain how to study such exponential functionals in the case of the simple symmetric random walk. Ultimately, we will need to analyze these kinds of functionals in the context of the tilted chains. Consider independent simple symmetric random walks $  R^1_r,    R^2_r$. Suppose that $   R^1_0 =    R^2_0$, i.e., the two walks start out at the same position. We will be interested in expectations that are (roughly speaking) of the form 
$$\mathbf E_{\mathrm{SSRW}}\left [e^{|\varsigma_N|^{2p} \sum_{r=0}^{s-1}  \ind_{\{R^1_r = R^2_r\}}}\right].$$ Note that the exponent here is a (rescaled) discrete local time for the random walk $R^1 - R^2$. We will therefore review some classical limit laws for such local times.

The asymptotic behavior of such functionals is highly dependent on the dimension that we are working in. These results date back to the works of Chung and Hunt \cite{MR29488}, Kallianpur and Robbins \cite{MR56233}, Darling and Kac \cite{MR84222}, and Erd\"{o}s and Taylor \cite{MR121870}. 

We have the following limit laws (in distribution) as we take $N \to \infty$:
\begin{enumerate}
    \item In $d=1$, $N^{-1/2}\sum_{r=0}^{Nt}\ind_{\{R^1_r  - R^2_r = 0\}} \Rightarrow L^{B}_0(t)$, where $L^B_0(t)$ is the local time at $0$ of a Brownian motion $B$ of rate $2$. This was first proven in \cite{MR29488}.
    \item In $d=2$, $(\log N)^{-1}\sum_{r=0}^{Nt}\ind_{\{   R^1_r  -    R^2_r = 0 \}} \Rightarrow Z$, where $Z$ is an exponential random variable with rate $\pi$. This result was first proven in \cite{MR121870}.
    \item Finally, in $d \geq 3$, $\sum_{r=0}^{\infty}\ind_{\{   R^1_r  -    R^2_r = 0 \}}$ is a geometric random variable, with rate given by the return probability to 0. 
\end{enumerate}

\begin{rk} \label{rk:expTimeIndependence}
    Note that while the local time limit in $d=1$ depends on $t$, the exponential limit in $d=2$ does \emph{not depend on $t$}. This is because $2d$ random walks are just barely recurrent, and therefore all visits of $R^1 - R^2$ to $0$ occur on timescales much smaller than the macroscopic timescale $Nt$.
\end{rk}

We see from the above discussion, that it is precisely at the scale $|\varsigma_N| = O(N^{-1}\psi_N(p,d))$ where the discrete local time converges to a nontrivial random limit. This explains the failure of our methods beyond this location scale as these local times blow up.

Finally, we analyze in more detail what happens at the critical scale $|\varsigma_N| = O(N^{-1}\psi_N(p,d))$. It follows from the above discussion that if we take $|\varsigma_N| = \lambda^{1/2p} N^{-1}\psi_N(p,d)$ for some $\lambda \geq 0$, then 
$$\mathbf E_{\mathrm{SSRW}}\left [e^{|\varsigma_N|^{2p} \sum_{r=0}^{s-1}  \ind_{\{R^1_r = R^2_r\}}}\right ] \approx M(\lambda)$$ where $M(\cdot)$ is the moment generating function of either the folded normal distribution in $d=1$ (the distribution of Brownian local time at zero for fixed $t$), the exponential distribution in $d=2$, or the geometric distribution in $d\ge 3$. 

While the moment generating function for the folded normal distribution is defined for all real numbers, the moment generating function for the exponential distribution, $M(\lambda) = \frac{1}{1- \lambda}$,  is only defined for $\lambda < 1$. Similarly, the moment generating function of the geometric random variable with parameter $p$, $M(\lambda) = 
\frac{p e^{\lambda}}{1-(1-p) e^{\lambda}}$,  is only defined for $\lambda < -\ln(1-p).$ 

This means that in $d \geq 2$, the expectation of such exponential functionals blows up for $\lambda$ large enough. This is the precise reason that at the scale $|\varsigma_N| = O(N^{-1}\psi_N(p,d))$ in $d\geq2$, we expect a phase transition. We will see later that when we translate the above intuition to the tilted chains, this MGF blowup will ultimately be the reason that in $d = 2$, \eqref{eq:extremal2} blows up for $\boldsymbol v$ with $\gamma_{\mathrm{ext}}(  {\boldsymbol v})^2 = 2$ and that in $d \geq 3$, \eqref{she6} blows up for $\boldsymbol{v}$ such that $\nu_{\mathrm{eff}}(  {\boldsymbol v}) = \infty.$

\textbf{3. SRI chains.} We will study these exponential functionals for much more general Markov chains than just simple random walks. To do this, we will introduce a formalism called \emph{Short-range interacting Markov chains (SRI Chains)}. This is a general framework for $k$-dimensional Markov chains whose components only interact when they are close to one another. Such chains can be used to model the $k$-point motions of our RWRE models, as well as their exponentially tilted versions. 

We prove a series of results about such Markov chains including an invariance principle, heat kernel estimates, existence and uniqueness of their invariant measures, and limit laws for additive functionals of sequences of these chains in the spirit of the works \cite{MR29488, MR56233, MR84222, MR121870}. This framework is quite general and may be of independent interest. These results are contained in Section \ref{sec:4} and Appendices \ref{appendix:a}, \ref{appendix:b}, \ref{appendix:c}, and  \ref{appendix:d} and can be read independently of the rest of the paper.

\textbf{4. Backwards propagation of test functions. } We highlight another important technique that is used crucially to obtain results in $d \geq 2$. For simplicity, again consider two independent simple symmetric random walks $R^1,R^2$ in $\mathbb Z^2$ starting from the origin. Rather than just additive functionals, we will need to consider even more general observables 
of the form
\begin{equation} \frac1{\log N}\sum_{s=0}^{Nt}
   \ind_{\{R^1_s - R^2_s=0\}} \phi(N^{-1/2}R^1_s)
   ,
\end{equation}
for some test function $\phi \in C_c^{\infty}(\R^2)$. The idea is that we can propagate this test function back to the initial time. More precisely, take the limit in distribution of the above expression as $N\to \infty$. It turns out that to compute this limiting distribution, one can simply replace $\phi(N^{-1/2}R^1_s)$ by $\phi(R^1_0) = \phi(0)$ and then take the distributional limit. The intuition is that most of the ``weight" of the expectation will be at times that are macroscopically 0 (as explained in Remark \ref{rk:expTimeIndependence}). Similarly in $d\ge 3$, the same holds true without the logarithmic normalization. We develop a formalism to rigorously do this with much more general Markov chains and much more general functions than just indicators at the origin. This will be the content of Theorems \ref{2.11} and \ref{2.12}.

\subsection{The broader picture and some open problems}\label{openprob}

Let us discuss some open problems in this area, first focusing on $d=1$, which is the most well-understood case, then on $d=2,$ and finally $d\ge 3.$ Some of these have already been briefly mentioned above, but we explain them here in more detail for emphasis and clarity, because these problems are the primary motivating force behind the present work.

\textbf{1. The supercritical case in $d=1$.} If $p=1$ note that the critical location scaling is given by $\psi_N = N^{3/4}$. In $d=1$ if one takes $   \varsigma_N = O(N^{3/4})$ it is known from \cite{DDP23,DDP+,Par24} that the fluctuations are no longer Gaussian but instead given by the KPZ equation as explained in Remark \ref{critkpz}. For \textit{even larger scales } $|\varsigma_N|\gg N^{-1/4p}$, it is strongly conjectured that for a broad class of kernels, the one-point fluctuations are governed by the Tracy-Widom distribution \cite{TW}, but this is known only in some exactly solvable cases \cite{bc, mark}. More generally, at location scales exceeding $N^{3/4}$ the conjecture is that one should obtain a \textit{full space-time scaling limit} called the \textbf{directed landscape} \cite{MQR, dov}, but this remains a wide open problem (even for the well-understood exactly solvable cases). To obtain these directed landscape fluctuations, one should not look at the density field $\mathfrak H^N$ as above, but rather one should consider the logarithm of a slightly different but strongly related observable of the RWRE called the \textit{quenched tail field}, see Problem 6 below. We do not expect the density field itself to have nontrivial fluctuations at these location scales because it is too rough to take its logarithm. The reason for taking the logarithm is somewhat similar to the fact that for a Brownian motion $B$ the process $X_t:=e^{B_t-t/2}$ does not exhibit nontrivial $t\to \infty$ fluctuations without taking a logarithm (i.e., $(X_t-\mathbb E[X_t]) / \sqrt{Var(X_t)}$ just converges a.s. to 0, despite the second moment converging to 1). This directed landscape conjecture in the $d=1$ case is consistent with the prediction from \cite{ldt2, ldt, bc} of the \textit{dynamic scaling exponents} of $1/3$ and $2/3$ for the respective height and transversal fluctuations, which goes back to the seminal physics work of \cite{kpz}, see also the surveys \cite{Qua11, Cor12}.

\textbf{2. The critical and supercritical case in $d=2.$ } In $d=2$ if one takes $   \varsigma_N$ such that $\gamma_{\mathrm{ext}}(\varsigma_N)^2 =   \frac2{\log N} + \frac{\theta}{(\log N)^2}$
, where $\theta\in\mathbb R$, then we conjecture that the limit will be given by the \emph{ 2d critical stochastic heat flow} $Z_{s,t}^\theta (\dr x,\dr y)$ recently constructed in \cite{CSZ_2d}. We may pursue this problem in a future work, using a recent moment-based technique developed in \cite{Tsa24}. As in the $d=1$ case, this would indicate a transition from a linear SPDE with Gaussian marginals to a nonlinear object that is much more complicated to describe, as one transitions from the bulk regime to the large deviations regime in the prelimiting RWRE model.

The question of fluctuations at even larger spatial location scales where  $\gamma_{\mathrm{ext}}(  \varsigma_N)^2 \gg  2/\log N$ (e.g. $|  \varsigma_N| = O(1)$ in $d=2$) is very difficult. In this \textit{supercritical} regime, we do not make any conjecture at all, as even the physics literature of what \textit{should} happen in this regime is rather sparse with with conflicting guesses as to what the correct scaling behavior should be, see e.g. \cite{phys1,phys2,phys3,phys4,phys5,phys6}. The work \cite{phys3} theoretically predicts a fluctuation scaling exponent of $1/4$, while strong numerical evidence from \cite{phys2, phys4} predicts a slightly smaller value of $0.241$. It will be abundantly clear that our own techniques for this paper break down beyond the critical location scale of $|\varsigma_N| = O((\log N)^{-1/2p})$ in $d=2.$

\textbf{3. The critical and supercritical case in $d\ge 3.$} For the $d\ge 3$ case, even less is known about what the fluctuations should be at the critical value of the slope and beyond. The possibilities are that there is another ``nonlinear" object that appears at criticality, albeit model dependent. Alternatively, it may be possible that there are still Gaussian fluctuations but with some extra diffusivity or modified noise relative to what one would expect. In $d=3$, the physics paper \cite{phys5} seems to predicts a dynamic scaling exponent of approximately $0.18$ for fluctuations. We make no precise conjecture here. Recent rigorous work of \cite{Junk} may be relevant here, as they show that for related polymer models, there are several different notions of criticality, for which the associated inverse temperatures do not necessarily coincide with one another in $d\ge 3$.

\textbf{4. A connection to polymers in arbitrary dimensions.} A directed polymer in an i.i.d. random environment is a random Gibbs measure on finite length random walk paths where the paths are reweighted by the environment. In all spatial dimensions, these directed polymers have a strong disorder regime and a weak disorder regime, see the monograph \cite{Comets}. For directed polymers in any dimension $d$, it is known that the ``intermediate disorder regime," that is, the transition point from weak to strong disorder, occurs at inverse temperature of order given by $$ \beta_N = \begin{cases} O(N^{-1/4})& \text{ in $d=1$} \\ O(1/\sqrt{\log N})& \text{ in $d=2$} \\ O(1)& \text{ in $d\ge 3$}. \end{cases} $$ The model exhibits Gaussian fluctuations in the weak disorder regime, see e.g. \cite{ marginal, GRZ, u3,u2,u4,u5,u6,CAR,CSZ_2d, kotitsas2025heatequationtimecorrelatedrandom}. See also \cite{MR4413215, MR4709542, dunlap2025edwardswilkinsonfluctuationssubcritical2d} where they prove Gaussian fluctuations for more general nonlinear stochastic heat equations in a similar scaling regime.

In other words, the critical scale for the polymer models is $\beta_N=\psi_N(1,d)/N.$ Given these results for polymers, our Theorems \ref{main1}, \ref{main2}, and \ref{main3} can be viewed as analogous results for the RWRE models by proving Gaussian fluctuations (given by an additive stochastic heat equation) strictly up to location $d_N$ of order $\psi_N(p,d).$

However there are still many differences in those polymer results from our RWRE results here: one observes a smaller number of fluctuation regimes for the polymers than for the RWRE, and much simpler expressions for the noise coefficients that appear in the limiting SPDEs. The smaller number of regimes is due to the fact that there is no analogue of the bulk regime for directed polymers. More importantly, the analogue of the coefficient $\gamma_{\mathrm{ext}}$ in our results is much simpler in the case of polymers. 

For example, for the directed polymer in dimension $d=2$, the intermediate disorder scale is given by $\beta_N = \frac{\hat{\beta}}{\sqrt{\log N}}$. There exists some critical constant $\hat{\beta}_c $ where the transition from Gaussian to non-Gaussian fluctuations occurs. In other words, there are Gaussian fluctuations for all $\hat{\beta} <  \hat{\beta}_c$. The noise coefficient in this regime is given by $\sqrt{\frac{1}{1- \hat{\beta}^2}}.$ We can see by analogy with \eqref{she6} that $\hat{\beta}$ plays the role that $\gamma_{\mathrm{ext}}$ plays in our model. As $\gamma_{\mathrm{ext}}$ is itself a nontrivial function of the underling stochastic flow, this shows that our noise coefficients always have an extra layer of complexity compared to the directed polymer case.   

This connection between the RWRE and the directed polymer is rather suggestive, as our proof makes no use of directed polymers or any seemingly related models. Thus we conjecture that there is a precise way of relating the RWRE models to the polymer ones, such that one has the exact relation $$\beta_N = \gamma_{\mathrm{ext}} (\varsigma_N),$$ which should be valid whenever one is not in the bulk regime, i.e.,  $|\varsigma_N| \gg N^{-1/2}.$ We leave this to future work. We do remark that our proof techniques could be used to derive the directed polymer results up to the critical scale as well.

\textbf{5. Temporal correlations.} Imagine now that Assumption \ref{a1} was weakened so that rather than imposing that the $K_i$ are (temporally) i.i.d. and sampled from the same distribution on Markov kernels, they are instead only stationary in $i$ and satisfy some exponentially-fast mixing condition (or even finite-range correlations). It is clear at an intuitive level that this should not affect any of the \textit{macroscopic} behavior of the model in any significant way, as one only gains some microscopic temporal correlations that are very weak. Thus very similar scaling limit results should hold as in Theorems \ref{main1}-\ref{main3}. However, our proofs would break down completely, as one loses the Markov property of the $k$-point motion (and consequently also the associated Dynkin martingales which are the crucial observables in the proof). Thus it would be nice to find a more robust method of proof that works even for these temporally correlated models. This search for a more robust proof may tie in with the previous open problem of relating the model to polymers by connecting the location parameter $\varsigma_N$ in RWRE to the inverse temperature parameter $\beta_N$ in polymers. We remark that the papers \cite{BZ, Dolgo1, Dolgo2} could consider temporally correlated kernels, but spatially they were independent, which is somewhat opposite to what we consider here. We may pursue this in future work using cumulant-based techniques similar to \cite{Par+}.

\textbf{6. Multi-point convergence of the quenched tail field.} Although we have chosen to study the \textit{quenched density field} \eqref{hn} in this paper, there is another field that is also very interesting and relevant to the extreme value theory. Specifically when $d=1$ one can study the \textit{field of quenched tail probabilities} given by $(t,x) \mapsto P^\omega_\nu ( R(Nt) > x) .$ We studied and proved some rather limited KPZ convergence results about this field in \cite{DDP23, DDP+}, for a very particular regime of location values $x$ that scale in a precise way with $N$. In fact we showed that in the extremal regimes when $d=1$, the \textit{density field agrees with the tail field} up to some vanishing error term. In higher dimensions $d\ge 2$, there is no unique way to define a generalization of this quenched tail field, but one could instead consider several variants: the \textit{tails across a plane} e.g. $(t,x)\mapsto P^\omega_\nu ( R(Nt) \bullet v >x)$), the \textit{tails defined by the cdf} e.g. $(t,x_1,...,x_d)\mapsto P^\omega_\nu(   R^1(Nt) >x_1, ...,R^d(Nt) >x_d)$, and the \textit{tails across a circle} e.g. $(t,x)\mapsto P^\omega_\nu ( |R(Nt)| >x)$. This latter notion has recently been studied numerically in \cite{ark2025universalfluctuationstailprobability} using  computer simulations. It would be very interesting to explore various scaling limits of these tail fields, and see if they more-or-less agree with the density limit theorems proved in this paper, or if they reveal some genuinely new behaviour in $d\ge 2$. 

\textbf{7. Extrema of sticky diffusions.} In the extremal case of $|  \varsigma_N| \gg  N^{-1/2}$, our results have a dual interpretation in terms of the extrema of a large number of independent random walkers in a random environment, which is really the \textit{physical motivation} for studying this problem \cite{hass23, hass2023b}. When $d=p=1$, it is known from \cite[Theorem 1.8]{DDP23} or \cite[Theorem 1.2]{DDP+} that for certain specific models, the maximum of $N$ of these RWRE particles at time scales of order $(\log N)^2$ fluctuates (as $N\to\infty)$ like a randomly shifted Gumbel distribution, where the random shift is distributed as the time-one solution of the (1+1)-dimensional KPZ equation (independent of the Gumbel). Moreover, $(\log N)^2$ is the \textit{smallest time scale} where one can observe important information about the environment as captured by the coefficient $\gamma_{\mathrm{ext}}^2$, which is equivalent to the fact that in $d=1$, $N^{3/4}$ is the \textit{smallest location scale} where one observes $O(1)$ fluctuations of the density as described by the third bullet point in Theorem \ref{main1} (i.e., $\mathscr B_N$ can be taken to be 1). The content of Theorem \ref{main2} suggests that the smallest such timescale in $d=2$ is $(\log N) \cdot \log\log N$ as opposed to $(\log N)^2$. This is because the critical location scaling to observe $\mathscr B_N=1$ has changed from $N^{3/4}$ to $N/\sqrt{
\log N}$. Likewise Theorem \ref{main3} says that in $d=3$, the smallest such timescale should be $\log N.$ One open problem is to make this idea rigorous and calculate the distribution of the random shifts in higher dimensions at these critical timescales. There are numerous open questions related to obtaining a more complete picture in this area, see the physics numerics papers \cite{hass23,ldb,hass2023b, hass2024extreme} and the more theoretical ones \cite{bld, lda}. 

\textbf{8. Improving the topology of convergence to obtain pointwise statistics.} In all of our results, there is always some spatial averaging, that is, the convergence is always in $C([0,T], \mathcal S'(\mathbb R^d))$ where \textit{in the spatial variable} one needs to integrate against a smooth function to observe the fluctuations as described in the main results above. Since the effect of these test functions is to average out the behavior of the density over a large number of microscopic spatial locations---roughly $O(N^{1/2})$ of them---this leaves open the question of what exactly happens at the \textit{pointwise} level without the spatial averaging, even in $d=1$. This is not merely a technical obstacle, as we do expect that there are non-trivial \textit{model-dependent local fluctuations} that appear. In \cite{DDP+} we at least were able to prove the partial result that the convergence cannot be improved to a space of continuous functions. Based on conversations with B. Virag and Y. Gu, we believe that the continuity can be restored if one were to quotient the field $\mathfrak H^N$ from \eqref{hn} by the \textit{local fluctuation field}, that is, the so-called environment process. Proving this seems to be a technical challenge, and it may require some adaptation of the cluster expansion tricks from the early papers in the field \cite{BMP1,BMP2, BMP3}. We leave this to future work.

\subsection{Previously known fluctuations results.} 

In our main results (Theorems \ref{main1}, \ref{main2}, and \ref{main3}), there are several different choices that we are allowed to make in terms of the model parameters: (1) the choice of the dimension $d$ and the location scaling $\varsigma_N \in \mathbb R^d$, (2) the choice of initial data $\mathfrak H_0^N$ for the density profile, and (3) the microscopic details of the model as outlined by Assumption \ref{a1}, and in particular the value of $p\in\mathbb N$. This large amount of freedom is the main novelty of the present work over existing results, as we now explain. 

\textbf{1. Bulk fluctuation results.} In the particular case $p=1, $ $  \varsigma_N=0$ (bulk regime), and Dirac or smooth initial data, a few of our fluctuation results are already known, though many of these results are proved in a weaker topology and for a much more restrictive class of microscopic models. \cite{BMP2} assumes an additional ``small noise assumption" and proves a version of the first bullet point of Theorem \ref{main1}, using a sort of cluster expansion technique. This is further studied by the same authors in \cite{BMP1, BMP3, BMP4, BMP5}. Later work of \cite{timo2, yu,timo} showed more general Gaussian fluctuation results using the \emph{environment seen from particle}. Finally, as mentioned in Remark \ref{p=1}, recent work of \cite{Kot} used martingale arguments to prove bulk convergence of the specific model from Example \ref{diffrm} in arbitrary dimension and in a strong topology.

\textbf{2. Extremal fluctuation results.} For the extremal regime, works of \cite{dom, gu} showed KPZ equation limits under a certain \emph{weak noise scaling} which is different from what we do here.

The papers \cite{ldt, ldt2} conjectured that at the critical scaling, one-dimensional diffusions in random media should exhibit fluctuations given by the KPZ equation. In \cite{bld}, this conjecture was proven at the level of moment convergence for two exactly solvable models: sticky Brownian motion with a uniform characteristic measure and the \emph{Beta random walk in random environment}. In \cite{DDP23} we studied the KPZ equation fluctuations of sticky Brownian motions, a stochastic flow of kernels that falls under the framework of Assumption \ref{a1} if we discretize it in time. We considered sticky Brownian motion with any characteristic measure $\nu$ and we found that $\gamma_{\mathrm{ext}}^2 = 1/(2\nu[0,1])$ in this case. In \cite{DDP23}, we proved the KPZ equation fluctuations for the one-dimensional nearest-neighbor random walk in random environment model introduced in Example \ref{ex:1}, in which case we found that $\gamma_{\mathrm{ext}}^2 = \frac{8\sigma^2}{1-4\sigma^2}$ with $\sigma^2$ being the variance of the random $[0,1]$-valued weights used to define the random transition probabilities. In \cite{Par24}, we proved KPZ equation fluctuations for all models satisfying Assumption \ref{a1}, but we looked specifically in the critical regime $\varsigma_N=N^{-1/4p}$. Recently another work of \cite{bc25} studied the critical regime for the KMP model of energy exchange and showed KPZ equation fluctutations.

Going to the \emph{supercritical} regime, the works of \cite{bc, mark} proved Tracy-Widom limits for the aforementioned exactly solvable models, the Beta RWRE and sticky Brownian motion with a uniform characteristic measure, respectively.

\textbf{3. Comparison with \cite{timo2} and related works.} Next we compare our results with the results of  \cite{timo2} (see also \cite{timo}), where they study the fluctuations of the \emph{quenched mean process} for one-dimensional random walks in random environments in the bulk regime and show that they are given by an additive stochastic heat equation. We now discuss how to go from Theorem \ref{main1} which is a result about the quenched density to their result about the quenched mean.

In $d=p=1$ consider the SPDE \eqref{she3} with $\mathfrak H_0 = \delta_a$ with $a\in\mathbb R$, thus we obtain a three-parameter random field given by the solution $$\partial_t \mathcal U(t,x;a) = \frac12 \partial_x^2 \mathcal U(t,x,a) + \gamma_{\mathrm{ext}} \partial_x \big( G(t,x+a) \xi(t,x)\big),\;\;\;\;\;\;\;\;\; \mathcal U(0,x;a)\equiv 0.$$
By Theorem \ref{main5}, this is exactly the field one would obtain if one varied the starting point along with the diffusive scaling, so that all particles start at $N^{1/2}a$ rather than at 0. By that theorem, the scaling limits are coupled together in the natural way for distinct values of $a$, so that the driving noise $\xi$ is the same for all $a$.
Consider the random process of quenched means for this scaling limit: $$\mathcal V(t,x):= \int_\mathbb R y\cdot \mathcal U(t, y;x ) dy.$$ By writing out the Duhamel formula, one may show that $\mathcal V$ solves the Edwards-Wilkinson SPDE for a different noise $\tilde \xi$
: $$\partial_t \mathcal V(t, x) = \frac12\partial_x^2 \mathcal V(t, x) + \gamma_{\mathrm{ext}} \tilde \xi(t, x),$$ where $\gamma_{\mathrm{ext}}$ is the same constant as in Theorem \ref{main1} and $(t,x) \in \mathbb R_+\times \mathbb R$. This is consistent with the results of \cite{timo2} because formally if one commutes the limit with the quenched mean operation, then it is precisely \cite[Theorem 3.2]{timo2}. This operation may be shown rigorously by using the continuous mapping theorem and the relevant tightness and tail bounds, details of which we will not pursue here. Our fluctuation results above make it clear that this is something rather specific to the bulk regime, and would not be expected to hold elsewhere.

\textbf{4. The noise coefficient.} 
The coefficient $\gamma_{\mathrm{ext}}^2$ appearing in our main results has appeared in various guises in all of the previous works mentioned above, but it is not always defined in a way that is easy to match up with the definition given in this paper. For example, in \cite{timo2}, the coefficient $\gamma_{\mathrm{ext}}^2$ is written in terms of the potential kernels of the associated random walks. See Remark \ref{rk:comparison} for a detailed comparison of the coefficient of the EW equation obtained in \cite{timo2} with the coefficients obtained in the present work (and thus also those obtained in \cite{DDP23, DDP+, Par24, hass2024extreme, hass2025universal}). 
\\
\\
\textbf{Outline.} In Section \ref{sec:2}, we introduce a Markov chain associated to the microscopic model, the so-called $k$-point motion. We then show how to interpret the constants $C_{N,t,x}$ as ``Girsanov tilts" for these Markov chains, which possess the important trait of preserving the Markov property. In Section \ref{sec:3}, we prove that the field $\mathfrak H^N$ from \eqref{hn} satisfies a discrete stochastic PDE driven by a martingale noise. In Section \ref{sec:4}, we introduce a formalism called ``SRI chains" and state strong limit theorems for such chains. In Section \ref{sec:5} we prove various estimates and limiting formulas for the quadratic variation of the martingales defined in Section \ref{sec:3}. In Section \ref{sec:6} we prove  tightness and identify  the limit points using the martingale problem characterization. The appendices contain the proofs of the results for SRI chains stated in Section \ref{sec:4}. 
\\
\\
\textbf{Acknowledgements.} We thank Jacob Hass for helping us to simplify some of the physical constants. We also thank Ivan Corwin, Alex Dunlap, Simon Gabriel, Yu Gu, Timothy Halpin-Healy, Jonathon Peterson, and Ran Tao for helpful discussions about the context. HD’s research was supported by the NSF Graduate Research Fellowship under Grant No. DGE-2036197, Ivan Corwin’s NSF grants DMS-1811143, DMS-1937254, and DMS-2246576, and the W.M. Keck Foundation Science and Engineering Grant on “Extreme diffusion”. This material is partly based upon work supported by the National Science Foundation under Grant No. DMS-2424139, while HD was in residence at the Simons Laufer Mathematical Sciences Institute in Berkeley, California, during the Fall 2025 semester. SP acknowledges support by the NSF Mathematical Sciences Postdoctoral Research Fellowship (DMS-2401884).

\section{The Girsanov transform and the tilted $k$-point Markov chains} \label{sec:2}
In this section, we introduce and study the $k$-point motion of the stochastic flow introduced in Assumption \ref{a1}, and we prove various properties about it under certain exponential Girsanov transforms that are relevant for the proof of Theorem \ref{main2}.

\begin{defn}[$k$-point motion] \label{kmotion}
We define the path measure $\mathbf P_{RW^{(k)}}$ on the canonical space $(I^k)^{\mathbb Z_{\ge 0}}$ to be the \textit{annealed} law of $k$ independent walks (all started from the origin) sampled from kernels $K_n$ satisfying Assumption \ref{a1}. The Markov chain whose law is given by the canonical process $\mathbf R= (\mathbf R_r)_{r\ge 0}$ on $I^{\mathbb Z_{\ge 0}}$ under the measure $\mathbf P_{RW^{(k)}}$ is called the \textbf{$k$-point motion}.
\end{defn}

We remark that the 
$k$-point motion is a Markov chain on $I^k$ whose one-step transition probabilities are precisely $\boldsymbol p^{(k)}(\mathbf x,A) := \int_A \mathbb E[ K_1(   x_1,\mathrm d   y_1)\cdots K_1(   x_k,\mathrm d    y_k)],$ which was also discussed in \eqref{kptkernel}. The most basic property of the family of $k$-point motions is that it is \textit{projective}: the temporal evolution of any deterministic subset consisting of $\ell<k$ coordinates behaves as the $\ell$-point motion, which is immediate from the definition. In particular, each coordinate $R^j$ of the $k$-point motion is marginally distributed as a random walk on $I$ (i.e., has independent increments) with increment distribution $\mu$ as defined in Assumption \ref{a1}\eqref{a22}. 
However, the marginal law of a pair $(   R^i,   R^j)$ of distinct coordinates will not be as simple, for instance, it may not have independent increments. Nonetheless, one can say the following.

\begin{lem}\label{diff0}
    For any $1\le i,j\le k$ with $i\neq j$, the process $   R^i-   R^j$ under $\mathbf P_{RW^{(k)}}$ is a Markov process with transition kernel $\pdif$ as defined in Assumption \ref{a1}\eqref{a16}.
\end{lem}

The proof is straightforward from the definitions. 

\begin{defn}\label{cumu}
    Let $m,k\in \mathbb N$ and consider the multi-index $\mathbf j = (\vec j_1,...,\vec j_m) \in \big( \{1,...,k\}\times \{1,...,d\}\big)^m$. For $\mathbf x, \mathbf y\in I^k$ define the joint cumulants $$\kappa^{(m;k)}(\mathbf j;\mathbf x) := \frac{\partial^m}{\partial\beta_{\vec j_1} \cdots \partial \beta_{\vec j_m}} \log \int_{I^k} \exp \bigg( \sum_{i=1}^k    \beta_i \bullet (   y_i-   x_i) \bigg) \boldsymbol p^{(k)} (\mathbf x, \mathrm d \mathbf y)\;\;\bigg|_{( \vec   \beta_1,...,  \vec  \beta_k)=(\vec    0,...,  \vec 0)},$$ where $  \vec \beta_i = (\beta_{(i,1)},...,\beta_{(i,d)}),$ and the vertical bar denotes evaluation of the function at the origin.
\end{defn}

When taking the above partial derivatives, we remark for clarity that the indices $\vec j_i$ may be repeated as many times as desired.
\begin{lem}\label{prev}
    Suppose that $m,k\geq 2$ and $\mathbf j=(\vec j_1,...,\vec j_m) \in \big( \{1,...,k\}\times \{1,...,d\}\big)^m$ such that all $\vec j_i$ are not the same index in the first coordinate, i.e., $(\vec j_1,...,\vec j_m)$ is not of the form $((i,v_1),(i,v_2),...,(i,v_m))$ for some $i=1,...,k$. Let $p\in \mathbb N$ be as in Assumption \ref{a1} Item \eqref{a23}. 
    \begin{enumerate}\item \label{eq:prev1} If $m< 2p$ then $\kappa^{(m;k)}(\mathbf j;\mathbf x)=0$.
    \item \label{eq:prev2} If $m= 2p$ then write $\vec j_r=(i_r,v_r)$ for each coordinate of $\mathbf j$. Then $\kappa^{(m;k)}(\mathbf j;\mathbf x)=0$, unless the set of indices $\{i_1,...,i_m\}$ is of the form $\{a,b\}$, where exactly $p$ of the indices $i_r$ are equal to $a$, and thus the other $p$ indices are equal to $b$.
    \item \label{eq:prev3} If $m\ge 2p$, then there exists $C(m)>0$ such that $$|\kappa^{(m;k)}(\mathbf j;\mathbf x)| \leq C(m)\cdot \Fd\big( \min_{1\leq i <j \leq k} |   x_i-   x_j|\big). $$
    \end{enumerate}
\end{lem}

Item \eqref{eq:prev3} will only be used when $m=2p$ or $m=2p+1$, since we will usually truncate our Taylor expansions at that point.

\begin{proof}

    For any multi-index $\mathbf j$, note that $\kappa^{(m;k)}(\mathbf j,\mathbf x)$ is actually just a finite linear combination of products of quantities of the form $\int_{I^k} \prod_{j=1}^k (   y_j-   x_j)^{   r_j} \boldsymbol p^{(k)} \big( \mathbf x, d\mathbf y\big), $ where $0\leq \#(   r_1)+...+\#(   r_k)\leq m$. Moreover, it is a very special linear combination with the following property: if the random variables $   y_i-   x_i$ $(1\le i \le k)$ were independent under $\boldsymbol p^{(k)}(\mathbf x,d\mathbf y)$ then in fact this linear combination would vanish identically, simply by the definition of the cumulants and the condition that all indices are not the same (i.e., joint cumulants of independent variables are always zero).
    
    But for $m<2p$, the moments $\int_{I^k} \prod_{j=1}^k (   y_j-   x_j)^{   r_j} \boldsymbol p^{(k)} \big( \mathbf x, d\mathbf y\big), $ agree precisely with those of the independent ones given by $\int_{I^k} \prod_{j=1}^k    u_j^{   r_j}\mu^{\otimes k} (\mathrm d\mathbf u),$ by the deterministicity assumption in Item \eqref{a23} of Assumption \ref{a1}. Therefore we immediately conclude Item (1). The proof of Item (2) is similar, noting that the moments agree precisely with those of the independent family unless the given condition on the indices is satisfied.

    For $m\ge 2p$, the $m^{th}$ order joint moments of $y_i-x_i$ under $\boldsymbol p^{(k)}(\mathbf x,\dr\mathbf y)$ will not necessarily agree precisely with the independent version under $\mu^{\otimes k}$, hence they may not vanish outright. However the moment bound in Item \eqref{a24} of Assumption \ref{a1} guarantees that all of the moments are exponentially close to those of the independent ones, therefore we can immediately conclude Item (3) where $C$ might be larger than the constant in Item \eqref{a24}.
\end{proof}

Note that if $(X_r)_{r\ge 0}$ is any sequence of random variables with finite exponential moments, defined on some probability space $(\Omega,\mathcal F,\mathbf P)$ and adapted to a filtration $(\mathcal F_r)_{r\ge 0}$ on this space, then the exponentiated process \begin{equation}
\label{aft}\exp\bigg( X_r-X_0 -\sum_{s=1}^{r} \log \mathbf E[e^{X_{s}-X_{s-1}}|\mathcal F_{s-1}]\bigg)\end{equation} is a martingale in the variable $r\in \mathbb Z_{\ge 0}$, with respect to the same filtration. Henceforth our probability space will always be $\Omega =(I^k)^{\mathbb Z_{\ge 0}},$ equipped with its canonical filtration.  For $|   \beta|\le z_0/k$ with $z_0$ defined in Assumption \ref{a1} Item \eqref{a22}, if we apply the exponential transform \eqref{aft} to the process $X:=   \beta\bullet (   R^1+...+   R^k)$, we see that if all $R^j$ are started from the origin, then
 \begin{align}\label{m_n}\mathpzc M_r^{\beta}(\mathbf R)&:=\exp\bigg(    \beta\bullet \sum_{j=1}^k    R^j_r - \sum_{\ell =0}^{r-1} f_{\ell}^{   \beta,k}\bigg)
\end{align}
is a $\mathbf P_{RW^{(k)}}$-martingale for the $k$-point motion where $f_\ell^{   \beta,k} := f^{   \beta,k}(   R_\ell^1,\ldots,   R_\ell^k)$, and the latter function $f^{   \beta,k}$ is given by
\begin{align*}f^{   \beta,k}(   x^1,\ldots,    x^k) &:= \log\mathbf E_{(   x^1,\ldots,   x^k)}^{RW^{(k)}}[e^{   \beta\bullet \sum_{j=1}^k(   R^j_1-   x^j)}] = \log \int_{I^k}e^{  \beta \bullet \sum_{i=1}^k (   y_i-   x_i)} \boldsymbol p^{(k)} (\mathbf x,d\mathbf y).
\end{align*}
Here $\mathbf E_{(   x^1,\ldots,   x^k)}^{RW^{(k)}}$ denotes expectation with respect to the $k$-point motion when started from $(   x^1,\ldots,   x^k).$

We then use the martingale \eqref{m_n} to define the following measure.

\begin{defn}\label{q}
    For $|   \beta|\le z_0/k$ with $z_0$ defined in Assumption \ref{a1}\eqref{a22}, define the tilted path measure $\mathbf P^{(   \beta,k)}$ on the canonical space $(I^k)^{\mathbb Z_{\ge 0}}$ to be given by $$\frac{\mathrm d\mathbf P^{(   \beta,k)}}{\mathrm d\mathbf P_{RW^{(k)}}}\bigg|_{\mathcal F_r}(\mathbf R) =  \mathpzc M_r^{\beta}(\mathbf R),$$ where $\mathcal F_r$ is the $\sigma$-algebra on $(I^k)^{\mathbb Z_{\ge 0}}$ generated by the first $r$ coordinate maps, and $\mathpzc M_r(\mathbf R)$ is the martingale given by \eqref{m_n}. We will denote by $\mathbf E^{(   \beta,k)}$ the expectation operator associated to $\mathbf P^{(   \beta,k)}.$
\end{defn}

An obvious but important fact is that $\mathbf P^{(   0,k)}=\mathbf P_{RW^{(k)}}$, which will implicitly be used many times below.
\subsection{Studying the tilted measures $\mathbf P^{(   \beta,k)}$}

\begin{prop}\label{mkov}
    Take $|   \beta|\le z_0/k$ where $z_0$ is as in Assumption \ref{a1}\eqref{a22}, and let $k\in \mathbb N$. The canonical process $(\mathbf R_r)_{r\ge 0}$ on $(I^k)^{\mathbb Z_{\ge 0}}$ under the tilted measure $\mathbf P^{(   \beta,k)}$ of Definition \ref{q} is still a Markov chain on $I^k$, with one-step transition law given by \begin{equation}\label{qk}\boldsymbol q^{(k)}_{  \beta} (\mathbf x,\mathrm d\mathbf y) := \frac{e^{   \beta \bullet \sum_{i=1}^k(   y_i-   x_i)} \boldsymbol p^{(k)}(\mathbf x,\mathrm d\mathbf y)}{\int_{I^k}e^{   \beta\bullet  \sum_{i=1}^k(   a_i-   x_i)} \boldsymbol p^{(k)}(\mathbf x,\mathrm d\mathbf a)}.\end{equation}
\end{prop}

\begin{proof}
    Take any Borel set $A\subset I^k$. By the definition of the measures $\mathbf P^{(   \beta,k)}$, we have \begin{align*}\mathbf P^{(   \beta,k)}(\mathbf R_{r+1} \in A | \mathcal F_r) &= \frac{\mathbf E_{RW^{(k)}} \big[ \ind_A (\mathbf R_{r+1})\mathpzc M^\beta_{r+1}(\mathbf R)|\mathcal F_r \big]}{\mathpzc M^\beta_{r}(\mathbf R)} \\ &=\frac{\mathbf E_{RW^{(k)}} \big[ \ind_A (\mathbf R_{r+1})e^{   \beta \bullet \sum_{i=1}^k (   R^i_{r+1}-   R^i_r) }|\mathcal F_r\big]}{\mathbf E_{RW^{(k)}} \big[ e^{   \beta \bullet \sum_{i=1}^k (   R^i_{r+1}-   R^i_r) }|\mathcal F_r\big]} \\ &= \frac{\int_A e^{   \beta \bullet \sum_{i=1}^k(   y_i-   R^i_r)}\boldsymbol p^{(k)} (\mathbf R_r,\mathrm d\mathbf y) }{\int_{I^k} e^{   \beta \bullet \sum_{i=1}^k(   y_i-   R^i_r)}\boldsymbol p^{(k)} (\mathbf R_r,\mathrm d\mathbf y)} \\ &= \boldsymbol q^{(k)}_{   \beta}(\mathbf R_r,A),
    \end{align*}
    thus proving the claim. In the third line we applied the Markov property of $\mathbf P_{RW^{(k)}}.$ The condition $|   \beta|\leq z_0/k$ ensures that all moment generating functions actually exist, i.e., the integrals are convergent for all $\mathbf x\in I^k$ by e.g. H\"older's inequality which says that $\int_{I^k}e^{   \beta\bullet \sum_{i=1}^k (   a_i-   x_i)} \boldsymbol p^{(k)}(\mathbf x,\mathrm d\mathbf a) \leq \prod_{i=1}^k \big(\int_{I^k}e^{k   \beta \bullet (   a_i-   x_i)} \boldsymbol p^{(k)}(\mathbf x,\mathrm d\mathbf a)\big)^{1/k} = M(k   \beta )$, where $M$ as always denotes the moment generating function of the annealed one-step law as in \eqref{dn}.
\end{proof}

\begin{defn}\label{shfa}
    Define the measure $\mathbf P_\mathbf x^{(   \beta,k)}$ to be the law on the canonical space $(I^k)^{\mathbb Z_{\ge 0}}$ of the Markov chain associated to \eqref{qk} started from $\mathbf x \in I^k.$ We also let $\mathbf E_\mathbf x^{(   \beta,k)}$ be the associated expectation operator. 
\end{defn}

Immediately from Proposition \ref{mkov} we see that the path measures $\mathbf P^{(   \beta,k)}$ from Definition \ref{q} are a special case of the measures $\mathbf P_\mathbf x^{(   \beta,k)}$ with $\mathbf x= (   0,   0,...,   0).$ Thus, the subscript stands for the initial condition, and the superscript denotes the tilting vector $\beta$ and the particle number $k$.

In general the difference of two particles under the tilted processes will no longer be Markov, but we still have the following partial result when the number of particles is limited to $k=2$. 

\begin{prop}[The difference process is still Markov when $k=2$]\label{diffproc}
    In the particular case of $k=2$, the difference $   R^1-   R^2$ of the two coordinate processes is still Markov under $\qdif$, with Markov kernel on $I$ given by $$\boldsymbol q_{   \beta}^{\mathrm{diff}} (x,A):= \int_{I^2} \ind_{\{    y_1-   y_2 \in A\}} \boldsymbol q^{(2)}_{    \beta} \big((x,0), (\dr y_1,\dr y_2)\big).$$
\end{prop}

\begin{proof}
    It suffices to show that for all $a\in I$ we have $$\int_{I^2} \ind_{\{    y_1-   y_2 \in A\}} \boldsymbol q^{(2)}_{    \beta} \big((x_1,x_2), (\dr y_1,\dr y_2)\big) = \int_{I^2} \ind_{\{    y_1-   y_2 \in A\}}\boldsymbol q^{(2)}_{    \beta} \big((x_1+a,x_2+a), (\dr y_1,\dr y_2)\big).$$ In other words we need to show that for all bounded $f$, the following quantity does not depend on $a \in I$: $$\frac{ \int_{I^2} f(   y_1-   y_2) e^{\beta \bullet \sum_{j=1,2} (   y_j-   x_j-   a) } \boldsymbol p^{(2)} \big((   x_1+   a,   x_2+   a), (\dr    y_1,\dr    y_2)\big)}{\int_{I^2} e^{\beta \bullet \sum_{j=1,2} (   y_j-   x_j-   a) } \boldsymbol p^{(2)} \big((   x_1+   a,   x_2+   a), (\dr    y_1,\dr    y_2)\big)} . $$ But that is clear for $k=2$ from the translation invariance in Assumption \ref{a1}.
\end{proof}

The above proof does not work for larger values of $k$ because one needs to consider translations of $I$ of the form $(a,a,b,c,d,...)$ and the marginal laws no longer project onto each other in the same way. In fact, the statement is no longer true for $k\ge 3.$

\begin{lem}\label{grow}
    Fix $k\in \mathbb N$, and recall $z_0$ from Assumption \ref{a1}. With $\boldsymbol q^{(k)}_{   \beta}$ defined in \eqref{qk}, we have $$\sup_{|   \beta|\leq z_0/(2k)}\; \sup_{\mathbf x\in I^k} \max_{1\le i\le k} \int_{I^k} e^{\frac{z_0}2 |   y_i-   x_i|} \boldsymbol q^{(k)}_{   \beta} (\mathbf x,\mathrm d\mathbf y) <\infty.$$ In particular, we have that $$\sup_{|\beta|\leq z_0/(2k)}\; \sup_{\mathbf x\in I^k} \int_{I^k} e^{\frac{z_0}{2k} \sum_{i=1}^k |   y_i-   x_i|} \boldsymbol q^{(k)}_{   \beta} (\mathbf x,\mathrm d\mathbf y) <\infty.$$
\end{lem}

\begin{proof}
    The second bound follows immediately from the first, since by H\"older's inequality we have that $\int_{I^k} e^{\frac{z_0}{2k} \sum_{i=1}^k |   y_i-   x_i|} q^{(k)}_{   \beta} (\mathbf x,\mathrm d\mathbf y) \le \prod_{i=1}^k \big( \int_{I^k} e^{\frac{z_0}{2} |y_i-x_i|} \boldsymbol q^{(k)}_{   \beta} (\mathbf x,\mathrm d\mathbf y) \big)^{1/k}.$
    
    Thus we will prove the first bound.  We find that for $\lambda>0$ and $|   \beta|\le z_0/(2k)$ and $1\le i \le k$, one has that 
    \begin{align*}
        \int_{I^k} e^{\lambda |   y_i-   x_i|} \boldsymbol q^{(k)}_\beta (\mathbf x,\mathrm d\mathbf y) &= \frac{\int_{I^k} e^{\lambda |   y_i-   x_i|} e^{   \beta \bullet \sum_{j=1}^k (   y_j-   x_j)}\boldsymbol p^{(k)} (\mathbf x,\mathrm d\mathbf y) }{\int_{I^k} e^{   \beta \bullet \sum_{j=1}^k (   y_j-   x_j)} \boldsymbol p^{(k)} (\mathbf x,\mathrm d\mathbf y) } 
        \\ &\leq \frac{\big(\int_{I^k} e^{2\lambda  |   y_i-x_i|} \boldsymbol p^{(k)} (\mathbf x,\mathrm d\mathbf y) \big)^{1/2}\prod_{j=1}^k \big(\int_{I^k} e^{2k|   \beta| |   y_j-   x_j|} \boldsymbol p^{(k)} (\mathbf x,\mathrm d\mathbf y) \big)^{1/(2k)}}{e^{\int_{I^k}   \beta \bullet \sum_{j=1}^k (   y_j-   x_j) \boldsymbol p^{(k)} (\mathbf x,\mathrm d\mathbf y)} } \\ &= \frac{\big(\int_I e^{2\lambda |   a|} \mu(\mathrm d   a)\big)^{1/2}\big(\int_I e^{2k |   \beta| |   a|} \mu(\mathrm d   a)\big)^{1/2}}{e^{k    \beta \bullet    m_1}}.
    \end{align*}
    In the first line, we simply applied the definition of the measures $\boldsymbol q^{(k)}_{   \beta}$. In the second line, we used H\"older's inequality for the numerator and Jensen for the denominator. In the last line, we used the fact that the marginal law of each $   y_j-   x_j$ under $\boldsymbol p^{(k)} (\mathbf x,\mathrm d\mathbf y)$ is just $\mu$, and $   m_1:= \int_I    a\mu(\mathrm d   a).$

    By Assumption \ref{a1}\eqref{a22}, in the last expression one can set $\lambda:=z_0/2$, and that expression is clearly bounded independently of $|   \beta|\le z_0/(2k)$, thus completing the proof.
\end{proof}

\subsection{Important functions and Taylor expansions}

In Theorems \ref{main1} - \ref{main4}, the important quantities are given by integrals against some functions against an invariant measure. In this section we introduce some of these functions and prove that they arise naturally by Taylor expansions of cumulant or moment generating functions near the origin, which in turn are used in the tilting of the path measures.

\begin{defn}
    Let us define the difference between the cumulant  generating function of the $k$-point motion and that of $k$ independent motions sampled according to $\mu^{\otimes k}$.
    \begin{equation}\label{kdef}\mathcal K(\x,\boldsymbol\beta; k):= \log \int_{I^k} e^{\sum_{i=1}^k    \beta_i\bullet (   a_i-   x_i)} \boldsymbol p^{(k)}(\x,\dr\bfa)\;\;\;-\;\;\;\sum_{i=1}^k\log M(   \beta_i).
    \end{equation}
\end{defn}

$\mathcal K(\x,\boldsymbol\beta; k)$ arises naturally as a correction factor when tilting the quenched density fields. Recall from Definition \ref{cntxnu} that we defined the rescaling constants $C_{N ,t, x}$ so that 
$$C_{N, N^{-1}s, \frac{R^i_{s} -    N^{-1}d_N s}{N^{1/2}}} = \exp \left(   \varsigma_N \bullet R^i_{s} - \log M (   \varsigma_N)s \right).$$
In other words, this is the correct tilting factor to exponentially tilt a single motion $R^i_s$. However, if we want to tilt the $k$-point motion, we cannot simply take use the product $\prod_{j=1}^k C_{N, t, \frac{R^i_{Nt} -    d_N t}{N^{1/2}}}$ as the $k$ coordinates are not independent. Instead we need to use the martingale $\mathpzc M_s^{\varsigma_N}(\mathbf R)$ defined above to tilt the $k$-point motion. Since the quenched density fields are defined with the prefactors $C_{N, t,x}$, we will need to convert between $\mathpzc M_s^{\varsigma_N}(\mathbf R)$ and $\prod_{j=1}^k C_{N, t, \frac{R^i_{Nt} -    d_N t}{N^{1/2}}}$. We have that 
$$\mathpzc M_s^{\varsigma_N} (\mathbf R) = \left(\prod_{j=1}^k C_{N,N^{-1}s, N^{-1/2} (   R^j_{s}- N^{-1}   d_Ns)}\right)e^{-\sum_{r=0}^{s-1}\mathcal K(\mathbf R_r, (\varsigma_N, \ldots, \varsigma_N);k)}.$$ 
This means that the correction between the correct tilting factor $\mathpzc M_s^{\varsigma_N} (\mathbf R) $ and the product $$\prod_{j=1}^k C_{N,N^{-1}s, N^{-1/2} (   R^j_{s}- N^{-1}   d_Ns)}$$ is precisely given by $e^{\sum_{r=0}^{s-1}\mathcal K(\mathbf R_r, (\varsigma_N, \ldots, \varsigma_N);k)}$. 

This allows us to convert between expectations of the $k$-point motion and expectations with respect to the tilted $k$-point motion as follows. Given a function $f: I^{kd} \to \R$, we have 
\begin{multline}\label{eq:tiltingExpectations}
    \mathbf E_{0}^{(0,k)}\bigg[\prod_{j=1}^k C_{N,N^{-1}s, N^{-1/2} (   R^j_{s}- N^{-1}   d_Ns)} f(  \mathbf R_s)\bigg]  =\mathbf E_{0}^{(0,k)}\bigg[ \mathpzc M_s^{\varsigma_N} (\mathbf R)e^{\sum_{r=0}^{s-1}\mathcal K(\mathbf R_r, (\varsigma_N, \ldots, \varsigma_N); k)} f(  \mathbf R_s)\bigg]\\
    = \mathbf E^{(   \varsigma_N,k)}_{0}\bigg[ e^{\sum_{r=0}^{s-1}\mathcal K(\mathbf R_r, (\varsigma_N, \ldots, \varsigma_N); k)}   f(  \mathbf R_s)\bigg].
\end{multline}

We will use the following definition to Taylor expand $\mathcal K(\x,\boldsymbol\beta; k)$. 

\begin{defn}\label{etadef}
    Fix two multi-indices $   r_1,   r_2 \in \mathbb Z_{\ge 0}^d.$ Define the function $\boldsymbol \eta_{   r_1,   r_2}: I\to \mathbb R$ by $$\boldsymbol \eta_{   r_1,   r_2} (   y_1-   y_2):= \int_{I^2} (   x_1-   y_1)^{   r_1} (   x_2-   y_2)^{   r_2} \bigg(\boldsymbol p^{(2)} \big( (   y_1,   y_2),(\mathrm d   x_1,\mathrm d   x_2)\big) - \mu(\mathrm d   x_1-   y_1)\mu(\mathrm d   x_2-   y_2)\bigg).$$
\end{defn}

The fact that the expression on the right side defines a function of $   y_1-   y_2$ follows from the translation invariance in Assumption \ref{a1}\eqref{a12}. Given the above definition, it will be useful to abbreviate \begin{equation}\label{rho}\rho\big( (   y_1,   y_2), (\mathrm d   x_1,\mathrm d   x_2) \big):= \boldsymbol p^{(2)} \big( (   y_1,   y_2),(\mathrm d   x_1,\mathrm d   x_2)\big) - \mu(\mathrm d   x_1-   y_1)\mu(\mathrm d   x_2-   y_2).\end{equation}

\begin{prop}\label{dbound}
Fix $k\in \mathbb N$. There exists $C=C(k)>0$ such that the following bounds hold uniformly over $   \beta_1,...,   \beta_k \in \ak$ and $\x\in I^k$ 
    \begin{align*}\bigg|
    \mathcal K(\x,\boldsymbol\beta; k) - \sum_{i<j} \sum_{\substack{r_1+...+r_d=p\\s_1+...+s_d=p}} \frac{   \beta_i^{\;   r}   \beta_j^{\;   s}}{   r!    s!} \boldsymbol \eta_{   r,   s}(   x_i-   x_j)
    \bigg|&
    \leq 
C  \min\bigg\{ \bigg(\sum_{i=1}^k|   \beta_i|\bigg)^{2p+1
    } ,\Fd\big(\min_{i<j}|   x_i-   x_j|\big)^{1/2}\bigg\},
       \end{align*} 
    where $   r = (r_1,...,r_d)$ and $   x^{\;   r} = x_1^{r_1}\cdots x_d^{r_d}$ and $   r! = r_1!\cdots r_d!.$  
\end{prop}

\begin{proof}
    For a pair $(\ell,v) \in \mathbb N^2$ we will say that $(1,1)\le (\ell,v)\leq (k,d)$ if $1\le \ell \le k$ and $1\le j \le d$. By Definition \ref{cumu} the $(2p)^{th}$ order Taylor expansion of $K$ in the $   \beta_i$ variables is given by 
    \begin{align*}\mathcal K_{2p}(\x,\boldsymbol\beta; k)&:= \sum_{m=1}^{2p} \frac1{m!} \bigg(\sum_{(1,1)\le \vec j_1,...,\vec j_m\le (k,d)} \beta_{\vec j_1}\cdots \beta_{\vec j_m}\kappa^{(m;k)} (\mathbf j,\mathbf x)\\ &\hspace{1.5 in} - \sum_{\substack{(1,1)\leq \vec j_1,...,\vec j_m\leq (k,d)\\ j_i \text{ all same in first coordinate}}} \beta_{\vec j_1}\cdots \beta_{\vec j_m} \kappa^{(m;k)} (\mathbf j,\mathbf x)\bigg) \\ &= \sum_{m=1}^{2p} \frac1{m!}\sum_{\substack{(1,1)\leq \vec j_1,...,\vec j_m\leq (k,d)\\ j_i \text{ not all same in first coordinate}}}\beta_{\vec j_1}\cdots \beta_{\vec j_m} \kappa^{(m,k)} (\mathbf j; \x)\\&=\frac1{(2p)!}\sum_{\substack{(1,1)\leq \vec j_1,...,\vec j_{2p}\leq (k,d)\\ \vec j_i \text{ not all same in first coordinate}}}\beta_{\vec j_1}\cdots \beta_{\vec j_{2p}} \kappa^{(2p,k)} (\mathbf j; \x)\end{align*}

    In the first equality, the terms with all $j_i$ being the same comes from the $\log M(   \beta_i)$ contribution to $\mathcal K$, and uses the fact that the marginal law of $   a_i-   x_i$ under $\boldsymbol p^{(k)}(\x,\dr\bfa)$ is precisely $\mu$ for every $\x\in I^k$. In the third line, we used Item \eqref{eq:prev1} in Lemma \ref{prev}, which guarantees that all of the joint cumulants up to order $2p-1$ vanish, hence only the $m=2p$ term survives.

    Note by Taylor's theorem and the uniform exponential moment bounds in Lemma \ref{grow} that the difference $\mathcal K-\mathcal K_{2p}$ satisfies the bound 
    $$\sup_{\x\in I^k} \big| \mathcal K(\x,\boldsymbol\beta; k)-\mathcal K_{2p}(\x,\boldsymbol\beta; k)\big|\le C\bigg(\sum_{i=1}^k|   \beta_i|\bigg)^{2p+1} ,$$ uniformly over $   \beta_1,...,   \beta_m \in \ak.$

    On the other hand, since the joint cumulant of two random variables is just their covariance, Item \eqref{eq:prev2} in Lemma \ref{prev} implies that when $m=2p$ we have that $$\frac1{(2p)!}\sum_{\substack{(1,1)\leq j_1,...,j_{2p}\leq (k,d)\\ j_i \text{ not all same in first coordinate}}}\beta_{j_1}\cdots \beta_{j_{2p}} \kappa^{(2p,k)} (\mathbf j; \x) = \sum_{1 \leq i<j \leq k} \sum_{\substack{r_1+...+r_d=p\\s_1+...+s_d=p}} \frac{   \beta_i^{\;   r}   \beta_j^{\;   s}}{   r!    s!}  \boldsymbol \eta_{   r,   s}(   x_i-   x_j),$$ which is enough to complete the proof. This proves the upper bound of $\big(\sum_{i=1}^k|   \beta_i|\big)^{2p+1}$. The upper bound of $\Fd(\min_{i<j}|x_i-x_j|)^{1/2}$ follows from Item \eqref{eq:prev3} of Lemma \ref{prev}.
\end{proof}

\begin{defn}\label{z}
  Let $  {\boldsymbol v}\in \mathbb R^d$. We define the function $\z_{  {\boldsymbol v}}:I\to \mathbb R$ by 
    \begin{equation*}
         \z_{  {\boldsymbol v}}(   y_1-   y_2) :=\sum_{\substack{r_1+...+r_d=p\\s_1+...+s_d=p}}\frac{  {\boldsymbol v}^{   r+   s}}{   r!   s!}\boldsymbol \eta_{   r,   s}(   y_1-   y_2).
    \end{equation*}
  
\end{defn}
The fact that the expression on the right side defines a function of $   y_1-   y_2$ again follows from the translation invariance in Assumption \ref{a1}\eqref{a12}. Using $\frac1{   r!}(x_1-y_1-(x_2-y_2))^{   r} = \sum_{   r_1+   r_2=   r} \frac{(-1)^p}{   r_1!   r_2!}(x_1-y_1)^{   r_1} (x_2-y_2)^{   r_2}$, 
we equivalently have\begin{equation*}
         \z_{  {\boldsymbol v}}(   y) :=(-1)^p\sum_{r_1+...+r_d=2p} \frac{  {\boldsymbol v}^{   r}}{   r!}\bigg[ \int_{I} (   u-   y)^{   r}\pdif (   y,\dr    u) -\int_{I^2} (   a -   b)^{   r} \mu(\dr    a)\mu(\dr    b)\bigg].
    \end{equation*}
 This function $\z_{ \boldsymbol{v}}$ will play a crucial role throughout the paper, as it is precisely the function appearing in the numerator of the noise coefficient $\gamma_{\mathrm{ext}}$ in Theorem \ref{main2}. We remark that $\z_{ \boldsymbol{v}}$ decays quite fast at infinity in the variable $   y$ thanks to Assumption \ref{a1}\eqref{a24}. Notice that by the deterministicity assumption (Assumption \ref{a1} \eqref{a24})
 \begin{equation}\label{pterm}\z_{\boldsymbol{v}}(   y_1-   y_2) = \int_{I^2} \prod_{j=1,2} \frac{\big(\boldsymbol{v}\bullet (   x_j-   y_j)\big)^p}{p!} \rho\big( (   y_1,   y_2), (\mathrm d    x_1,\mathrm d    x_2) \big),
 \end{equation}
where $\rho$ is as in \eqref{rho}. Notice also that $\z_{   v}$ is a homogeneous polynomial of degree $2p$ in the variable $   v$. An important object in our analysis will be $ \mathcal K(\x,(  \varsigma_N,  \varsigma_N);2)$ and we therefore give it its own name:
\begin{defn}\label{def:u} Recall $\mathcal K$ from \eqref{kdef}. For $   \varsigma_N \in \R^d$, define the function $\boldsymbol u_{   \varsigma_N}: \R^d \to \R$ by 
 \begin{align*}
     \boldsymbol u_{   \varsigma_N}(   x_1 -    x_2) := \mathcal K(\x,(  \varsigma_N,  \varsigma_N);2).
 \end{align*}     
\end{defn}

Again, we remark that this is well-defined as $\mathcal K(\x,(  \varsigma_N,  \varsigma_N);2)$ defines a function of $   x_1 -    x_2$ by the translation invariance in Assumption \ref{a1}\eqref{a12}. Note that this is the same $u_{\varsigma}$ that appears in Theorem \ref{main4}. Finally, we will also need the following definition:
\begin{defn} \label{def:vartheta}
    Recall $\mathcal K$ from \eqref{kdef}. For $   \varsigma_N \in \mathbb R^d$ define the function $\boldsymbol{\vartheta}_{   \varsigma_N } :\mathbb R^d\to \mathbb R$ by 
$$\boldsymbol{\vartheta}_{   \varsigma_N }(   y_1-   y_2):= e^{\mathcal K(\y,(   \varsigma_N ,   \varsigma_N );2)} - 1 =  \int_{I^2} \prod_{j=1}^2 e^{   \varsigma_N \bullet(   x_j-   y_j) - \log M(  {\varsigma_N})} \rho\big( (   y_1,   y_2), (\mathrm d   x_1,\mathrm d   x_2) \big).$$
\end{defn}

Putting together all of the above, we obtain the following:
\begin{prop}\label{prop:TaylorG}
  We can expand $$  \boldsymbol u_{   \varsigma_N}(   x)  = \z_{   \varsigma_N}(   x) + \mathscr R_{   \varsigma_N}(   x)$$ where $\z_{   \varsigma_N}(   x)$ is as in Definition \ref{z} and where we have that $\sup_{   x \in I} |\mathscr R_{   \varsigma_N}(   x)| \leq C|  \varsigma_N|^{2p+1}$. 
\end{prop}

\begin{proof}
  This follows from the proof of Proposition \ref{dbound}, where we Taylor expanded  $\mathcal K(\x,(  \varsigma_N,  \varsigma_N);2)$ to the $(2p)^{th}$ order in the variable $  \varsigma_N$ and saw that all terms of order less than $2p$ vanish because of the deterministicity in Assumption \ref{a1} Item \eqref{a24}.
\end{proof}

The following bound will be useful for certain dominated convergence bounds later on. 
\begin{prop} \label{prop:dcbound}
    We have  $|\z_{   v}(   x)|\leq C\;\Fd(|   x|)^{1/2} |   v|^{2p}$.
\end{prop}

\begin{proof}
Use the representation 
\begin{align*}
     \z_{   v}(   x_1-   x_2) =\int_{I^2} \prod_{j=1,2} \frac{\big(   v\bullet (   y_j-   x_j)\big)^p}{p!} \rho \big( (   x_1,   x_2),(\mathrm d   y_1,\mathrm d   y_2)\big).
\end{align*}
Define the function $$G_{   v}(   x_1-   x_2) := \int_{I^2} \prod_{j=1,2} \frac{|v|^{2p} |   y_j-   x_j|^{2p}}{(p!)^2}\big(\boldsymbol p^{(2)} \big( (   x_1,   x_2),(\mathrm d   y_1,\mathrm d   y_2)\big) + \mu(\mathrm d   y_1-   x_1)\mu(\mathrm d   y_2-   x_2)\big),$$ so that $G_{   v}$ contains two separate factors of $|v|^{2p}$ (for a total prefactor of $|v|^{4p}$). 
By Lemma \ref{tbb}, we have that
\begin{align*}
    \z_{   v}(   x) &\leq 2G_{   v}(   x)^{1/2}\big\| \mu^{\otimes 2}\big(\bullet-(   x_1,   x_{2})\big) \;\;-\;\; \boldsymbol p^{(2)}\big( (   x_1,   x_{2}),\;\bullet\; \big)\big\|_{TV} ^{1/2} \\
    &\leq 2G_{   v}(   x)^{1/2}\Fd(|   x_1 -    x_2|)^{1/2} \\&\leq 2 |   v|^{2p} \Fd(|   x_1 -    x_2|)^{1/2} \bigg[ \int_{I^2} \prod_{j=1,2} \frac{|   y_j-   x_j|^{2p}}{(p!)^2}\big(\boldsymbol p^{(2)} \big( (   x_1,   x_2),(\mathrm d   y_1,\mathrm d   y_2)\big) + \mu(\mathrm d   y_1-   x_1)\mu(\mathrm d   y_2-   x_2)\big)\bigg]^{1/2} .
\end{align*}
Finally, note that the 
latter integral is bounded independently of $(   x_1,   x_2)$ by Lemma \ref{grow}. 
\end{proof}

\section{A discrete stochastic heat equation with a martingale noise} \label{sec:3}

In this section, we are going to derive a discrete stochastic heat equation for the rescaled field $\mathfrak H^N$ appearing in Theorem \ref{main2}. This discrete equation will be crucial in proving that theorem, because it will eventually allow us to show that any limit point satisfies the martingale problem characterization of the limiting stochastic PDEs written in Theorems \ref{main1}, \ref{main2}, and \ref{main3}. Using such martingale problem techniques for proving stochastic PDE scaling limits of systems with infinitely many interacting particles was pioneered by papers such as \cite{KS88, BG97}. In the context of RWRE, we developed these techniques in \cite{DDP23,DDP+,Par24}.

First we establish some notations that will be used in this section. For a measure $\nu$ on $I$ we will write $\int_I f(x) \mu (\mathrm d x+y)$ to mean the integral of $f$ with respect to the measure $\nu_y(A) := \nu(A+y),$ where $A+y$ is the translate of the Borel set $A$ by the real number $y$. The notation $\mu (\mathrm d x+y)$ can also be understood as $(\tau_y\mu)(\dr x)$ where $\tau_y$ is the shift operator.

Let $   d_N$ be as in \eqref{dn}, and define the discrete lattice $$\Lambda_N:=\{(t,   x)\in \mathbb N\times\mathbb R^d :   x+tN^{-1}   d_N\in \mathbb N\times I\}.$$ Let $\mu$ be the measure from Assumption \ref{a1}\eqref{a22}, and recall that for $\varsigma \in \mathbb R$ we define the \emph{exponentially tilted} measures $$\mu^{  \varsigma}(\mathrm d    x) = e^{   \varsigma \bullet    x - \log M(   \varsigma)} \mu(\mathrm d   x).$$ Note that $\mu^{   \varsigma}$ makes sense for $|   \varsigma |<z_0$ where $z_0$ is as in  Assumption \ref{a1}\eqref{a22}. 

\begin{defn}Let $\Lambda_N$ be as above, and suppose that $f$ is a signed measure on $\Lambda_N$. We then define the \textbf{discrete heat operator} $$\mathcal L_N f(t,   x) = f(t+1,   x-N^{-1}   d_N) - \int_I f(t,   x-   y) \mu^{   \varsigma_N}(\mathrm d   y).$$ \end{defn}
Since convolution of measures is well-defined, this expression makes sense even if $f(t,\cdot)$ is a finite measure for all $t\in\mathbb N$. The resultant object $g=\mathcal L_N f$ makes sense as a signed measure on $\Lambda_N$
, such that $g(t,\cdot)$ is a \textit{finite }signed measure for all $t \in \mathbb N.$ 
If $f(t,\cdot)$ has exponential tails, then so does $g(t,\cdot)$.

In particular if $P^\omega(r,\cdot)$ is given by \eqref{kn}, then we can define for $r\in \mathbb Z_{\ge 0}$ 
\begin{equation}Z_N(r,\mathrm d   x) = C_{N,N^{-1}r,N^{-1/2}   x, \nu_N} P_{\nu_N}^\omega(r,rN^{-1}   d_N+\mathrm d   x).\label{zn}\end{equation}
Then $\mathcal L_N Z_N$ makes sense as a signed measure on $\mathbb N\times I.$ Furthermore $ \mathcal L_N Z_N(Nt,\cdot)$ is a finite signed measure on $I$ for all $t$, and still has exponential moments. Note that the main object of interest in Theorems \ref{main1}, \ref{main2}, and \ref{main3} is essentially the diffusively rescaled field $\mathscr B_N \cdot Z_N(Nt,N^{1/2}\mathrm d   x),$ and roughly speaking the main goal in the remainder of the section will be to find and study a nice family of martingales for this rescaled family.

The following proposition shows what happens when we apply this heat operator to $Z_N$: 

\begin{prop}\label{prop:heatOperator} We have that
 \begin{equation}\label{diff}(\mathcal L_N Z_N)(Nt,\mathrm d   x) = \int_I e^{  \varsigma_N\bullet (   x-   y) - \log M(  \varsigma_N)}\big[ K_{Nt+1}(y,\mathrm d   x) -\mu(\mathrm d   x-   y)\big]Z_N(Nt,\mathrm d   y) .  \end{equation}   
\end{prop}

\begin{proof}
First note by the convolution property of the kernels $P^\omega$ of \eqref{kn} that if $r=Nt\in \mathbb N$ then $$Z_N(r+1,\mathrm d   x-N^{-1}   d_N) = C_{N,N^{-1}(r+1),N^{-1/2}(   x-N^{-1}   d_N), \nu_N} \int_I C_{N,N^{-1}r,N^{-1/2}   y, \nu_N}^{-1} K_{r+1}(y,\mathrm d   x) Z_N(r,\mathrm d   y).$$ Now notice that 
\begin{align*}C_{N,N^{-1}(r+1),N^{-1/2}(   x-N^{-1}   d_N), \nu_N} C_{N,N^{-1}r,N^{-1/2}   y, \nu_N}^{-1} &= e^{  \varsigma_N\bullet (   x-N^{-1}   d_N-   y) +   \varsigma_N\bullet  N^{-1}    d_N - \log M(   \varsigma_N)}\\ &= e^{  \varsigma_N\bullet (   x-   y) - \log M(   \varsigma_N)}.
\end{align*}On the other hand we also have by definition that $\mu^{   \varsigma_N}(\mathrm d   x-   y) = e^{  \varsigma_N\bullet (   x-   y) - \log M(   \varsigma_N)}\mu(\mathrm d   x-   y).$ Therefore we can write \begin{equation}(\mathcal L_N Z_N)(Nt,\mathrm d   x) = \int_I e^{  \varsigma_N\bullet (   x-   y) - \log M(  \varsigma_N)}\big[ K_{Nt+1}(y,\mathrm d   x) -\mu(\mathrm d   x-   y)\big]Z_N(Nt,\mathrm d   y) .  \end{equation}     
\end{proof}

Recall from the introduction that $$\mathscr B_N := \begin{cases} N^{\frac{p}2+\frac{d-2}{4}} ,& |   \varsigma_N| = O(N^{-1/2}) \text{ as } N\to\infty \\ \frac{N^{\frac{d-2}{4}}}{|   \varsigma_N|^p} , & N^{1/2}|   \varsigma_N| \to \infty \text{ as } N\to \infty.\end{cases}$$
\begin{lem}
    Let $d,N\in \mathbb N$, and let $\mathscr B_N$ be as above. For $\phi: \mathbb R^d\to \mathbb R$ smooth and bounded, we define $M_t^N(\phi)$ for $t\in N^{-1}\mathbb Z_{\ge 0}$ by the formula $M_0^N(\phi)=0$, and \begin{equation}\label{mfield}M^N_{t+N^{-1}}(\phi) - M_t^N(\phi):= \mathscr B_N \int_I \phi(N^{-1/2}   x)\; (\mathcal L_N Z_N)(Nt,\mathrm d   x).\end{equation}
   Then for all $N\in \mathbb N$ and smooth bounded $\phi$, the process $M^N(\phi)$ is a martingale indexed by $N^{-1}\mathbb Z_{\ge 0}$, with respect to the filtration $(\mathcal F^\omega_{Nt})_{t\in N^{-1}\mathbb Z_{\ge 0}}$ where $\mathcal F_t^\omega:=\sigma(K_1,...,K_{t})$ for $t\in \mathbb N$ (and $\omega=(K_i)_{i=1}^\infty$ are the i.i.d. environment kernels as in Assumption \ref{a1}).
\end{lem}

This martingale representation of $Z_N$ will be the most important input to the proof of the main limit theorems of this paper. We remark that in $d=1$, the above martingale representation was successfully used to prove convergence to the KPZ equation in $d=1$, see \cite{Par24} and previous related works \cite{DDP+, DDP23}. 

\begin{proof}Using Proposition \ref{prop:heatOperator} the martingality is clear because $Z_N(Nt,\mathrm d   y)$ is $\mathcal F_{Nt-1}$-measurable, and $K_{Nt+1}(y,\mathrm d   x)$ is independent of $\mathcal F_{Nt-1}$ with $\mathbb E[ K_{Nt+1}(y,\mathrm d   x)] = \mu(\mathrm d   x-   y).$
\end{proof}

\subsection{Calculating the quadratic variations of the martingales $M^N$}
Next we will calculate the predictable quadratic variations of the above martingales. Our final calculations are summarized as Proposition \ref{4.3} below. 

If $\phi: \mathbb R^d\to \mathbb R$ is smooth and bounded, define $$(J^N\phi)(   y_1,   y_2) := \int_{I^2} \prod_{j=1}^2 e^{  \varsigma_N\bullet (   x_j-   y_j) - \log M(  \varsigma_N)} \phi(N^{-1/2}   x_j)  \rho \big( (   y_1,   y_2),(\mathrm d   x_1,\mathrm d   x_2)\big) . $$ 
By \eqref{diff} we find that
    \begin{align}\notag\langle M^N(\phi)\rangle_t &:= \mathscr B_N^2 \sum_{s=1}^{Nt} \mathbb E \big[\big( M^N_{sN^{-1}}(\phi) - M^N_{(s-1)N^{-1}}(\phi)\big)^2 \big| \mathcal F_{s-1}^\omega \big] \\ &= \notag \mathscr B_N^2\sum_{s=1}^{Nt} \mathbb E \bigg[\bigg( \int_I \int_I \phi(N^{-1/2}   x) e^{  \varsigma_N\bullet(   x-   y) - \log M(  \varsigma_N)}\big[ K_{s+1}(   y,\mathrm d   x) -\mu(\mathrm d   x-   y)\big]Z_N(s,\mathrm d   y)\bigg)^2 \bigg| \mathcal F^\omega_{s-1} \bigg] \\ &= \mathscr B_N^2\sum_{s=1}^{Nt} \int_{I^2} (J^N\phi)(   y_1,   y_2) Z_N(s,\mathrm d   y_1)Z_N(s,\mathrm d   y_2).\label{quadvar}
    \end{align}
    In the last line we again use the fact that $Z_N$ is $\mathcal F^\omega_{s-1}$-measurable. Now we let $E^{(\omega, 2)}_{\nu^{\otimes 2}}$ denote a quenched expectation of two independent particles $(   R^1_n,   R^2_n)_{n\ge 0}$ sampled from a fixed realization of the environment kernels $\omega=\{K_n\}_{n\ge 1},$ where the initial data $(R^1_0,R^2_0)$ is independently sampled from $\nu^{\otimes 2}.$ Then the last expression may be rewritten as $$\mathscr B_N^2\sum_{s=1}^{Nt} E^{(\omega, 2)}_{\nu_N^{\otimes 2} } \bigg[\prod_{j=1}^2 C_{N,N^{-1}s, N^{-1/2} (   R^j_{s} - sN^{-1}   d_N), \nu_N} \big(J^N\phi\big) (    R^1_{s} - sN^{-1}   d_N,   R^2_{s} - sN^{-1}   d_N) \bigg] .$$

Unfortunately $J^N\phi$ is not yet in a form that is amenable to asymptotic analysis and tightness arguments, therefore we will need to perform Taylor expansions of some expressions therein, which will yield many ``error" terms that we will need to show are inconsequential in the limit. 

The subsequent discussion will now be split into cases, depending on which regime the location strength $  \varsigma_N$ is in.

\subsection{\hyperref[eq:regimeA]{Regime A:}  $|  \varsigma_N| = O(N^{-1/2})$ as $N\to \infty$} 

In this case, there is a simplification we can make. We will assume that all $  \varsigma_N = 0$ identically for all $N$. The reason we may assume this without loss of generality is the following. If $N^{1/2}\varsigma_N \to \boldsymbol\varsigma$ as $N\to \infty$, then we actually have $C_{N,t,x} \to e^{\boldsymbol\varsigma\bullet x+\frac12|\boldsymbol\varsigma|^2t}.$ This convergence is locally uniform in $t$ and $x$. Thus if we can show that $\mathscr B_N \mathfrak H^N(t,x)$ converges in law to some field $\mathcal U(t,x)$ when all $  \varsigma_N=0$, then for the general case it would immediately follow that $$\mathscr B_N \cdot C_{N,t,x} \mathfrak H^N(t,x+N^{1/2}\varsigma_Nt)\stackrel{d}{\longrightarrow} e^{\boldsymbol\varsigma\bullet x+\frac12|\boldsymbol\varsigma|^2t}\mathcal U(t,x+\boldsymbol \varsigma t),$$ where the convergence in law is with respect to the same topology (which will be a space of tempered distributions). In Theorems \ref{main1}, \ref{main2}, and \ref{main3}, we can thus reduce this entire bulk regime to this case where all $  \varsigma_N=0$. This simplification is not crucial to the proof, but it will certainly be convenient. The convenience of this simplification will be to avoid a delicate combinatorial analysis of certain cross-terms that arise while Taylor expanding various terms.

Henceforth, we therefore say that $  \varsigma_N=0$. Since $\rho$ has mass 0, we have 
\begin{align*}J^N\phi(   y_1,   y_2)&=  \int_{I^2}  \bigg[ \prod_{j=1,2}\phi(N^{-1/2}   x_j)\;\;-\;\prod_{j=1,2}\phi(N^{-1/2}   y_j)\bigg] \rho\big( (   y_1,   y_2), (\mathrm d   x_1,\mathrm d   x_2) \big).
\end{align*}

Again, we emphasize that this is only valid in the present case where $  \varsigma_N=0$ as this kills the exponential term. Taylor expanding $\prod_{j=1,2}\phi(N^{-1/2}   x_j)\;\;-\;\prod_{j=1,2}\phi(N^{-1/2}   y_j)$ to order $p$ around the point $(y_1,y_2)$, we obtain that 
\begin{align*}
    \prod_{j=1,2}&\phi(N^{-1/2}   x_j)\;\;-\;\prod_{j=1,2}\phi(N^{-1/2}   y_j) \\&= \sum_{\substack{0\le \#(   r_1)\le p,0\le  \#(   r_2) \le p\\ \#(   r_1+   r_2)>0}} N^{-\frac{\#(   r_1+   r_2)}2}\frac{\partial^{   r_1}\phi(N^{-1/2}y_1) \partial^{   r_2} \phi(N^{-1/2}y_2)}{   r_1!   r_2!} (x_1-y_1)^{   r_1} (x_2-y_2)^{   r_2} + \mathcal R^1_N(\phi,x_1,x_2,y_1,y_2),
\end{align*}
where the sum is over multi-indices $   r$ and $\#(   r) := r_1+...+ r_d$, and where by Taylor's theorem, the remainder term $\mathcal R^1_N$ satisfies $$\sup_{x_1,x_2,y_2,y_2}\mathcal |   R^1_N(\phi, x_1,x_2,y_1,y_2)| \leq C\|\phi\|_{C^{p+1}(\mathbb R^d)}^2 N^{-p-\frac1{2}} \big( |x_1-y_1|^{2p+2} +|x_2-y_2|^{2p+2}\big) .$$

We now recall from Definition \ref{etadef} that for two multi-indices $   r_1,   r_2 \in \mathbb Z_{\ge 0}^d,$ we defined the function $\boldsymbol \eta_{   r_1,   r_2}: I\to \mathbb R$ by $$\boldsymbol \eta_{   r_1,   r_2} (   y_1-   y_2):= \int_{I^2} (x_1-y_1)^{   r_1} (x_2-y_2)^{   r_2} \rho\big( (   y_1,   y_2), (\mathrm d   x_1,\mathrm d   x_2) \big).$$

Then the above Taylor expansion together with the deterministicity of Item \eqref{a23} in Assumption \ref{a1} yields 
\begin{align*}
    J^N\phi(   y_1,   y_2) = N^{-p}& \sum_{\#(   r_1)= p, \#(   r_2) = p} \frac{\partial^{   r_1}\phi(N^{-1/2}y_1) \partial^{   r_2} \phi(N^{-1/2}y_2)}{   r_1!   r_2!} \boldsymbol \eta_{   r_1,   r_2} (   y_1-   y_2) \\&+ \mathcal \int_{I^2} \mathcal R^1_N(\phi,x_1,x_2,y_1,y_2)\rho\big( (   y_1,   y_2), (\mathrm d   x_1,\mathrm d   x_2) \big).
\end{align*}
Indeed, by Lemma \ref{prev} all terms of order strictly less than $p$ will vanish, thus leaving only terms of order $p$. We now have the following lemma about the latter term, which will be shown to be inconsequential. 
\begin{lemma}\label{45}
There exists $C>0$ such that uniformly over all $N,y_1,y_2,\phi $ we have $$
    \bigg|\mathcal \int_{I^2} \mathcal R^1_N(\phi,x_1,x_2,y_1,y_2)\rho\big( (   y_1,   y_2), (\mathrm d   x_1,\mathrm d   x_2) \big)\bigg|\leq C \|\phi\|_{C^{p+1}(\mathbb R^d)}^2 N^{-p-\frac12} \Fd (|y_1-y_2|)^{1/2}.$$
\end{lemma}
\begin{proof}
     Recall from above that $\sup_{x_1,x_2,y_2,y_2}\mathcal |   R^1_N(\phi, x_1,x_2,y_1,y_2)| \leq C\|\phi\|_{C^{p+1}(\mathbb R^d)}^2 N^{-p-\frac1{2}} \big( |x_1-y_1|^{2p+2} +|x_2-y_2|^{2p+2}\big) .$ Note that Item \eqref{a24} of Assumption \ref{a1} says that $\|\rho((y_1,y_2),\bullet) \|_{TV} \leq \Fd(|y_1-y_2|)$. Now apply e.g. Lemma \ref{tbb} to immediately obtain the claim.
\end{proof}
We thus define the first error term 
\begin{equation}
    \label{errerr} \mathcal V_N^1(t,\phi):= \mathscr B_N^2\sum_{s=0}^{Nt-1} \int_{I^2} \bigg( \mathcal \int_{I^2} \mathcal R^1_N(\phi,x_1,x_2,y_1,y_2)\rho\big( (   y_1,   y_2), (\mathrm d   x_1,\mathrm d   x_2) \big)\bigg) Z_N(s,\dr y_1)Z_N(s,\dr y_2).
\end{equation}
This quantity will be important later. We also define another error term given by 
\begin{align}
    \notag \mathcal V_N^2(t,\phi):= \mathscr B_N^2\sum_{s=0}^{Nt-1} \int_{I^2}\bigg( N^{-p}& \sum_{\#(   r_1)= p, \#(   r_2) = p} \frac{\partial^{   r_1}\phi(N^{-1/2}y_1) \big(\partial^{   r_2} \phi(N^{-1/2}y_2)-\partial^{   r_2} \phi(N^{-1/2}y_1)\big)}{   r_1!   r_2!} \\ &\cdot\boldsymbol \eta_{   r_1,   r_2} (   y_1-   y_2)\rho\big( (   y_1,   y_2), (\mathrm d   x_1,\mathrm d   x_2) \big)\bigg) Z_N(s,\dr y_1)Z_N(s,\dr y_2).\label{errerr1}
\end{align}
These calculations are summarized just below in the first bullet point of Proposition \ref{4.3}.

\subsection{\hyperref[eq:regimeB]{Regime B:} $N^{-1/2} \ll  |  \varsigma_N| \ll  N^{-1} \psi_N(p,d)$ as $N\to\infty$} 

In this case, there is no simple reduction that will allow us to consider some fixed $  \varsigma_N$, thus we simply have to work with the varying sequence. Define$$\mathcal A^1_N(\phi,   y_1,   y_2):=  \int_{I^2} \prod_{j=1}^2 e^{  \varsigma_N\bullet(   x_j-   y_j) - \log M(  \varsigma_N)} \cdot \big[ \prod_{j=1}^2\phi(N^{-1/2}   x_j)\;\;-\;\phi(N^{-1/2}   y_1)^2\big] \rho\big( (   y_1,   y_2), (\mathrm d   x_1,\mathrm d   x_2) \big),$$
 we can then write 
\begin{align}\label{j1}
    J^N\phi(   y_1,   y_2) &= \mathcal A^1_N(\phi,y_1,y_2)+\phi(N^{-1/2}   y_1)^2\cdot \int_{I^2} \prod_{j=1}^2 e^{  \varsigma_N\bullet(   x_j-   y_j) - \log M(  \varsigma_N)} \rho\big( (   y_1,   y_2), (\mathrm d   x_1,\mathrm d   x_2) \big) \\
    &= \mathcal A^1_N(\phi,y_1,y_2)+\phi(N^{-1/2}   y_1)^2\cdot \boldsymbol{\vartheta}_{   \varsigma_N}(   y_1 -    y_2),
\end{align}
where in the last line we used the function $\boldsymbol{\vartheta}_{   \varsigma_N}(   y_1 -    y_2)$ defined in Definition \ref{def:vartheta}. Also recall $\z_{   v}$ from Definition \ref{z}. Now we are going to expand the above exponential in powers of $   \varsigma_N$, truncating after the $(2p)^{th}$ term. We will see shortly that only one of the terms in the expansion will actually be relevant. 

Define a function $\mathcal A^2_N:I\to\mathbb R$ by the formula \begin{align*}
    \boldsymbol{\vartheta}_{   \varsigma_N}(   y_1 -    y_2) =  e^{-2\log M(  \varsigma_N)} \z_{   \varsigma_N}(   y_1-   y_2) +  \mathcal A_N^2(   y_1-   y_2).
\end{align*}
Note that the above expressions appear to be functions of $(y_1,y_2)$, but we have implicitly used the translation invariance condition in Assumption \ref{a1}\eqref{a12} to say that they are actually only dependent on $y_1-y_2$.

We will ultimately show that $\z_{ \boldsymbol{v}}$ is important in the limit, while the terms $\mathcal A^j_N$ have vanishing contribution. The reason for this is the following estimate.

\begin{lem}\label{46}
    There exists $C>0$ such that uniformly over $N,y_1,y_2$ we have the bound $$|\mathcal A_N^2(y)| \leq C\min \big\{ |  \varsigma_N|^{2p+1} , \Fd (|y|)^{1/2}\big\} \leq C |  \varsigma_N|^{2p+\frac12}\Fd (|y|)^{\frac1{4p+2}}.$$
\end{lem}

\begin{proof}
    The second inequality is clear from the first using $\min\{a,b\} \leq a^{\theta}b^{1-\theta}$ where $\theta:= \frac{2p+\frac12}{2p+1}\in [0,1]$. 

     To prove the first inequality, the upper bound of $\Fd (|y|)^{1/2}$ is easier, so let us start with that one. Notice that Item \eqref{a24} of Assumption \ref{a1} says that $\|\rho((y_1,y_2),\bullet) \|_{TV} \leq \Fd(|y_1-y_2|)$. Now apply e.g. Lemma \ref{tbb} and the definition of $\mathcal A^2_N$ to immediately obtain the claim.
     
    The upper bound of $|  \varsigma_N|^{2p+1}$ is trickier, since a naive Taylor expansion only gives an upper bound of $|  \varsigma_N|^{p+1}$ which is not enough. To this end, let us define $$R_{2p}(x) := e^x - \sum_{k=0}^{2p}\frac{x^k}{k!},$$
so that $|R_{2p}(x)| \leq |x|^{2p+1} e^{|x|}.$ 
Next, write 
    \begin{align*}
       \prod_{j=1,2} e^{   \varsigma_N\bullet (   x_j-   y_j) }& = \prod_{j=1,2} R_{2p} \big(   \varsigma_N\bullet (   x_j-   y_j) \big) + R_{2p} ( \varsigma_N\bullet (   x_1-   y_1)) \bigg( \sum_{k=0}^{2p} \frac{(   \varsigma_N\bullet (   x_2-   y_2) )^k}{k!} \bigg)\\&+R_{2p} ( \varsigma_N\bullet (   x_2-   y_2)) \bigg( \sum_{k=0}^{2p} \frac{(   \varsigma_N\bullet (   x_1-   y_1) )^k}{k!}\bigg) + \prod_{j=1,2} \bigg(\sum_{k=0}^{2p} \frac{(   \varsigma_N\bullet (   x_j-   y_j) )^k}{k!}\bigg) 
    \end{align*}
    Out of these four terms, the first term is bounded above by $C|  \varsigma_N|^{4p+2} e^{|  \varsigma_N| \sum_{j=1,2} |x_j-y_j|}$. By Lemma \ref{grow}, we can integrate with respect to $\rho\big( (y_1,y_2), (\mathrm \dr x_1,\mathrm \dr x_2) \big)$ and obtain an upper bound of $C|  \varsigma_N|^{4p+2}$. The second and third terms are bounded above by $C|  \varsigma_N|^{2p+1} e^{|  \varsigma_N| \sum_{j=1,2} |x_j-y_j|}$. By Lemma \ref{grow}, we can integrate with respect to $\rho\big( (y_1,y_2), (\mathrm \dr x_1,\mathrm \dr x_2) \big)$ and obtain an upper bound of $C|  \varsigma_N|^{2p+1}$. Therefore, only the fourth term can be problematic. For this, notice that from the deterministicity in Assumption \ref{a1}\eqref{a23}, we necessarily have 
\begin{equation}\label{vanish}\int_{I^2} \prod_{j=1}^2 \big(   \varsigma_N\bullet (   x_j-   y_j)\big)^{k_j}\rho\big( (y_1,y_2), (\mathrm \dr x_1,\mathrm \dr x_2) \big) = 0,\;\;\;\;\; \text{if} \;\;\;\; k_1<p \;\;\; or \;\;\; k_2<p.\end{equation}
Now use \eqref{vanish} and \eqref{pterm} to write 
    \begin{align*}
        \int_{I^2} \prod_{j=1,2} \bigg(&\sum_{k=0}^{2p} \frac{(   \varsigma_N\bullet (   x_j-   y_j) )^k}{k!}\bigg)  \rho\big( (y_1,y_2), (\mathrm \dr x_1,\mathrm \dr x_2) \big)\\& = \z_{   \varsigma_N} (   y_1-   y_2)+ \int_{I^2} \sum_{\substack{p\leq k_1,k_2\leq 2p\\k_1+k_2>2p}}\frac{ (   \varsigma_N\bullet (x_1-y_1))^{k_1}(   \varsigma_N\bullet(x_2-y_2))^{k_2}}{k_1!k_2!}\rho\big( (y_1,y_2), (\mathrm \dr x_1,\mathrm \dr x_2) \big).
    \end{align*}
    Now notice that $$\big|(   \varsigma_N\bullet (x_1-y_1))^{k_1}(   \varsigma_N\bullet(x_2-y_2))^{k_2}\big| \leq |  \varsigma_N|^{k_1+k_2} |x_1-y_1|^{k_1} |x_2-y_2|^{k_2} ,$$ which then easily yields the claim by Lemma \ref{grow} and the fact that $k_1+k_2 \ge 2p+1$. 
\end{proof}

We therefore define\begin{equation}\label{err1}\mathcal E^1_N(t,\phi) = \mathscr B_N^2\sum_{s=0}^{Nt} 
\int_{I^2}\mathcal A^1_N(\phi,   y_1,   y_2)Z_N(s,\mathrm d   y_1)Z_N(s,\mathrm d   y_2),\end{equation} and 
\begin{align} \mathcal E_N^2(t,\phi)&:= e^{-2\log M(  \varsigma_N)}\mathscr B_N^2 \sum_{s=0}^{Nt-1}\int_{I^2} \phi(N^{-1/2}y_1)^2\mathcal A_N^2(   y_1-   y_2)Z_N(s,\mathrm dy_1)Z_N(s,\mathrm dy_2).\label{err234}
\end{align}

These calculations are summarized just below in the second bullet point of Proposition \ref{4.3}.

\subsection{\hyperref[eq:regimeC]{Regime C:} $d=2,   \varsigma_N/\psi_N(p,d) \to   {\boldsymbol v} \in \mathbb R^d\backslash\{0\}$. }  In this case, the analysis of Case B still applies verbatim.

\subsection{\hyperref[eq:regimeD]{Regime D:} $d=3,   \varsigma_N \to   {\boldsymbol v} \in \mathbb R^d\backslash\{0\}$. } In this case, we still use \eqref{j1} as we did in Case B, but now we do not Taylor expand in the variable $  \varsigma_N.$ Instead we will take into consideration the entire function $\boldsymbol{\vartheta}_{   \varsigma_N}$. We furthermore let $\mathcal E^1_N$ be exactly as defined in \eqref{err1}.

\subsection{Summary of martingale calculations}

We then summarize our calculations in the following proposition, which will be used repeatedly later.
\begin{prop}\label{4.3}
    Let $M^N$ be the martingale from \eqref{mfield}. Given two probability measure $\nu_1,\nu_2$ on $I$, we let $E^{(\omega, 2)}_{\nu_1\otimes \nu_2 }$ denote a \textbf{quenched} expectation of two independent random walks $(   R^1_r,R^2_r)_{r\in \mathbb Z_{\ge 0}}$ in a fixed realization of the environment kernels $\omega= (K_i)_{i=1}^\infty$, with initial conditions $(R_0^1, R_0^2)$ distributed according to $\varrho$. For any bounded function $f:I\to\mathbb R$ define the \textbf{quadratic variation field} \begin{align}Q^f_N(t,\phi; \nu_1\otimes\nu_2):=N^{\frac{d-2}2}&\sum_{s=1}^{Nt} E^{(\omega, 2)}_{\nu_1\otimes \nu_2 }\bigg[\prod_{j=1}^2 C_{N,N^{-1}s, N^{-1/2} (   R^j_{s} - N^{-1}   d_Ns), \nu_j} \cdot \phi\big(N^{-1/2} (   R^1_{s} - N^{-1}   d_Ns)\big) f(   R^1_{s}-   R^2_{s}) \bigg].\label{qfield}\end{align} Then for all $\phi\in C_c^\infty(\mathbb R)$, $N\in \mathbb N$ and $t\in N^{-1}\mathbb Z_{\ge 0}$ we have the following asymptotics: 
    \begin{itemize}
    \item \hyperref[eq:regimeA]{Regime A}: In the case where all $    \varsigma_N =    0$, we have that 
    \begin{align*}\langle &M^N(\phi)\rangle_t= \sum_{\#(   r_1)= p, \#(   r_2) = p} \frac{1}{   r_1!   r_2!} Q_N^{\boldsymbol \eta_{   r_1,   r_2}} (t,\partial^{   r_1} \phi \cdot \partial^{   r_2} \phi; \nu_N^{\otimes 2} )\;\;\;\;+\;\;\; \sum_{j=1,2}\mathcal V_N^j(t,\phi).\end{align*}
    \item \hyperref[eq:regimeB]{Regime B} and \hyperref[eq:regimeC]{Regime C}: In the case where $N^{-1/2} \ll  |  \varsigma_N| \ll  N^{-1} \psi_N(p,d)$, and also in the case where $d=2$ and $   \varsigma_N(\log N)^{\frac{1}{2p}} \to   {\boldsymbol v} \neq 0$, we have that \begin{align*}\langle &M^N(\phi)\rangle_t= e^{-2\log M(   \varsigma_N )} Q^{\z_{\frac{   \varsigma_N}{|  \varsigma_N|}}}_N(t,\phi^2; \nu_N^{\otimes 2})\;\;\;+\;\;\;\sum_{j=1,2} \mathcal E^j_N(t,\phi).\end{align*}
    \item \hyperref[eq:regimeD]{Regime D}: In the case where $d\ge 3$ and $\varsigma_N \to  {\boldsymbol v} \ne 0$ we have that 
    \begin{align*}\langle &M^N(\phi)\rangle_t=  |\varsigma_N|^{-2p}Q^{\boldsymbol{\vartheta}_{   \varsigma_N}}_N(t,\phi^2; \nu_N^{\otimes 2})\;\;\;+\;\;\;\mathcal E^1_N(t,\phi).\end{align*}
    \end{itemize}
    Above $\boldsymbol \eta_{   r_1,   r_2}$ is the function from Definition \ref{etadef}, $\z_{  {\boldsymbol v}}$ is the function from Definition \ref{z}, and $\boldsymbol{\vartheta}_{   \varsigma_N}$ is the function from Defintion \ref{def:vartheta}. Furthermore, 
    $\mathcal V_N^1, \mathcal V_N^2, \mathcal E^1_N, \mathcal E^2_N$ are the ``error terms" that were defined in \eqref{errerr}, \eqref{err1}, and \eqref{err234} respectively. 
\end{prop}

\section{SRI chains: Definitions, estimates, and convergence theorems} \label{sec:4}
In this section, we introduce a specific type of formalism that we call \emph{short-range interacting chains (SRI chains)}.  These are a specific type of Markov chain for which we prove various types of upper bounds and limit theorems. The results of this section are general and self-contained, and can be read independently of the rest of this paper. The ultimate goal of this section is to develop tools to study functionals of the form appearing in the quadratic variation field in \eqref{qfield}.

\subsection{Definition of SRI chains}
\begin{defn}[SRI chain]\label{SRI}  Fix some integers $k,d\ge 1$, and let $I$ as always be an additive locally compact subgroup of $\mathbb R^d$. Consider a Markov kernel $\br(\x,\mathrm d \mathbf y)$ on $I^k$. We say that the $q^{th}$ moments of $\br$ are bounded by $M$ if\begin{align}\sup_{\mathbf x \in I^k} \int_{I^k}|   y_i-   x_i|^q \br(\x,\mathrm d \mathbf y)&\leq M,\text{ for all } 1\le i \le k,
\label{diffs}
\end{align} Let $F: [0,\infty) \to [0,\infty)$ be a decreasing function. Let $\nu$ be some probability measure on $I$, and let $M>0$.  
A $k\times d$ \textbf{SRI chain with base measure $\nu$ and decay function $F$} is defined as a Markov kernel $\br$ on $I^k$, satisfying 
\begin{align}\label{tvb2}\| \br(\x,\bullet) - \nu^{\otimes k} (\mathbf x-\bullet)\|_{TV} &\leq F\big(\min_{1\le i<j\le k} |    x_i-   x_j|\big)\;\;\text{ for all } \x\in I^k .
\end{align} Here $\mathbf x = (   x_1,...,   x_k) \in I^k \subset \mathbb R^{kd}$. 
We denote by $\mathcal E^{\mathrm{SRI}}_{k\times d}(q,M,F)$ the set of all SRI chains on $I^k$ whose $q^{th}$ moments are bounded by $M$, and such that \eqref{tvb2} holds with $F$ on the right side, though we allow the base measure $\nu$ to vary between different elements of $\mathcal E^{\mathrm{SRI}}_{k\times d}(q,M,F)$. 

We furthermore say that the SRI chain is \textbf{centered} if $$\int_{I^k} (y_j-x_j) \br (\x,\dr\y) = \int_I    u \;\nu(\dr    u),\;\; \x\in I^k, 1\le j \le k.$$
We will say that an SRI chain is $\delta$-\textbf{repulsive} 
$$\int_{I^k} |   y_i-   y_j-(   x_i-   x_j)| \br(\x ,\dr \y) 
> \delta
\text{ for all $\x\in I^k$ and all $1\le i <j \le k$.}$$
\end{defn}

Here ``SRI" refers to the fact that the $k$ different coordinates or ``particles" only have short-range interactions with one another and behave independently when far apart. In other words, an SRI chain is just a Markov chain whose one-step transitions very closely resemble those of a random walk as one gets further away from the boundary of the ``$(k-1)d$-dimensional Weyl chamber" in $\mathbb R^{kd}$. An SRI chain being $\delta$-repulsive implies (but is stronger than imposing) that these Weyl chamber boundaries are not absorbing: in some sense these boundaries are uniformly repelling with strength determined by $\delta$. We take the convention that the minimum of the empty set is $+\infty$, thus if $k=1$, the above definition \eqref{tvb2} is vacuous and the Markov chain associated to the kernel $\br$ is simply the random walk in $\mathbb R^d$ with increment distribution $\nu.$

\begin{ex} The quintessential example of a \textit{centered} SRI chain is the random walk on $I^k$ with increment distribution $\nu^{\otimes k}$. However these are certainly not the only ones, and the family of $k$-point motions $\boldsymbol p^{(k)}$ associated to any given random walk in random environment model provides its own distinct example of a centered SRI chain, see Definition \ref{kmotion} above. Centered SRI chains will \textit{not} be the only ones we consider. Very important examples of \textit{non-centered} SRI chains include the Girsanov tilts of the $k$-point motions, see \eqref{qk} above. These tilted chains will be instrumental in the analysis of proving our main results, and we will prove that the tilts are still SRI chains that are ``close enough" to centered chains to inherit similar properties, e.g. diffusive behavior. 
\end{ex}

\begin{defn}[Markov operator, canonical process, and canonical measure]
    Let $\br \in \Ek$ be an SRI chain with base measure $\nu$ and decay function $F$. Define the associated \textbf{Markov operator} $P_{\br}$ to act on functions $f$ by the formula $$(P_{\br} f)(\mathbf x) := \int_{I^k} f(\y) \br ( \mathbf x,\dr \y),$$ defined for all Borel-measurable functions $f:I^k\to\mathbb R$ for which the integral converges absolutely for all $\x\in I^k.$

    We also denote by $(\mathbf R_r)_{r\ge 0} = (   R^1_r,...,   R^k_r)_{r\ge 0}$ the \textbf{canonical coordinate process} on $(I^k)^{\mathbb Z_{\ge 0}},$ and we denote by $\Pbb$ the \textbf{canonical probability measure} on $(I^k)^{\mathbb Z_{\ge 0}}$ under which the canonical coordinate process is distributed as the Markov chain with transition kernel $\br$ started from initial state $\x$. We also denote by $\Ebb$ the associated \textbf{expectation operator.}
\end{defn}

\begin{defn}
    Define the \textbf{distance} between two SRI chains $\br,\br'\in \mathcal E_{k\times d}^{\mathrm{SRI}}(q,M,F)$ by $$d_{\mathrm{SRI}} (\br,\br'):= \sup_{\x\in I^k} \| \br(\x,\bullet) - \br' (\x,\bullet)\|_{TV}.$$
    We do not impose that $\br,\br'$ have the same base measure $\nu$, though $(q,M,F)$ will always be the same and fixed for any two SRIs that we compare. We say that a sequence of SRI chains $\brn$ \textbf{converges} as $N\to \infty$ to an SRI chain $\br$ if $d_{\mathrm{SRI}} (\br,\brn) \to 0$ as $N\to \infty$.
\end{defn}

\begin{defn}\label{ekk}
    Fix $\sigma>0$, $M>0$, and a decreasing function $F:[0,\infty)\to[0,\infty)$ of exponential decay at infinity. 
    We say that $\br \in \Ekk$ if $\br$ is an SRI chain with decay function $F$, and the assumption \eqref{diffs} is replaced by the (much stronger) assumption $$\int_{I^k} e^{\sigma |y_j-x_j|} \br(\x,\dr\y)\leq M,\;\;\;\; \x\in I^k, 1\le i\le k.$$
\end{defn}

In other words, $\Ekk$ is just a class of SRI chains where the $q^{th}$ moment assumption is replaced by an exponential moment assumption. The following quantity will play an extremely important role going forward: for $N>1$ define \begin{equation}\label{omega_n}\omega_N(d):= (N^{-1}\psi_N(p,d))^{2p} = \begin{cases}
    N^{-1/2} , & d=1 \\ (\log N)^{-1} ,& d=2 \\ 1, & d=3.
\end{cases}\end{equation}

We now state the several theorems about SRI chains that will be most relevant for us. This includes 
an invariance principle to $kd$-dimensional Brownian motion, a sharp anticoncentration estimate and heat kernel bound for SRI chains, an existence and uniqueness theorem for invariant measures, and limit laws for additive functionals of these chains. The proofs of these theorems will be left to the appendices.

\subsection{Invariance principle}
\begin{thm}[Invariance principle]\label{inv00}
    Fix $q>4$ and $M>0$, and $F:[0,\infty)\to[0,\infty)$ of exponential decay at infinity. 
    Consider a sequence $\brn \in\mathcal E^{\mathrm{SRI}}_{k\times d}(q,M,F)$, and say $\nu_N$ is the base measure of $\brn$. Define the $\mathbb R^{kd}$-valued process $$\mathbf X_r^N:= \mathbf R_r - r\mathbf s_N,$$ where $\mathbf s_N: = \mathbf s(\nu_N)$ and where $\mathbf s(\nu):= \int_{I^k} \mathbf u \;\nu^{\otimes k} (\dr \mathbf u)$ is deterministic. Assume that $\brn$ converge as $N\to \infty$ to some limiting SRI chain $\br$ that is centered 
    and $\delta$-repulsive. Then we have a diffusive bound $$\sup_{N\ge 1} \sup_{\x\in I^k} \mathbf E^{\brn}_{\x} [ |\mathbf X^N_r - \mathbf X^N_{r'}|^q]^{1/q} \leq C|r-r'|^{1/2}.$$
    
    Furthermore, let $\mathbf x_N$ be a sequence of vectors in $I^k$ such that $N^{-1/2}\x_N\to \boldsymbol x\in \mathbb R^{kd}$. Then as $N\to\infty$, the sequence of processes $(N^{-1/2} \mathbf X^N_{Nt})_{t\in [0,T]}$, viewed under $\mathbf P^{b^{\nu_N}}_{\x_N}$, converges in law in the topology of $C[0,T]$ for any $T>0$, to $(\Sigma B(t)+\boldsymbol x)_{t\in [0,T]}$ where $B$ is a standard $kd$-dimensional Brownian motion and $\Sigma^T \Sigma$ is the covariance matrix of $\nu^{\otimes k}$, with $\nu$ being the base measure of $\br$.
\end{thm}

Note that if $\brn=\br_0$ is a constant sequence, then centeredness is equivalent to saying that the process $\mathbf X^N_r$ is a martingale. Then the $L^q$ bound is trivial from Burkholder-Davis-Gundy, and the invariance principle is also a fairly easy consequence of the martingale central limit theorem. The above theorem can thus be understood as saying that we can slightly perturb this martingality so that it only holds for the limiting chain, and the invariance principle still holds.

This Theorem is proven in Appendix \ref{appendix:a}. 

\subsection{Heat kernel estimates}
\begin{thm}[Anti-concentration/heat kernel estimates]\label{anti}
    Fix $\sigma>0$, $M>0$, $K > 0$. Let $F:[0,\infty)\to[0,\infty)$ be decreasing and of exponential decay at infinity. Let $\br_0 \in \Ekk$ be centered and $\delta$-repulsive. Then
    \begin{itemize} 
    \item Case $d\ge 2$: there exists $\epsilon=\epsilon(F,\sigma,M,\delta)>0$ and $C=C(F,\sigma,M,\delta,K)>0$ such that uniformly over $\x\in I^k$ and $\y\in \mathbb R^{kd}$ and integers $r\ge 0$ one has the bound $$
       \sup_{\substack{\br\in \Ekk\\d_{\mathrm{SRI}}(\br,\br_0) <\epsilon
       }}
       \; \Pbb ( |\mathbf R_r-r\mathbf s(\nu_{\br}) -\y|\leq 1) < C r^{-kd/2} \big( 1+ r^{-1/2} |\x-\y|\big)^{-K}. $$
       \item $d=1$. We have the weaker estimate just for the centered chain: $$\mathbf P_{\x}^{\br_0} ( |\mathbf R_r - \y|\le 1) \leq Cr^{-kd/2}  \big( 1+ r^{-1/2} |\x-\y|\big)^{-K}.$$ We also have the following time-averaged bound for non-centered chains when $q\ge 4$: there exists $\epsilon>0$ so that for all $\lambda>0$ one has 
        $$  \sup_{\substack{\br \in \Ek\\d(\br,\br_0)<\epsilon }} \sup_{\x\in I^k} \sup_{r\in\mathbb Z_{\ge 0}} \Ebb [ e^{ \lambda r^{-1/2} \sum_{s=0}^{r-1} \sum_{1\le i'<j'\le k} \ind_{\{|R^{i'}_s-R^{j'}_s|\le 1 \}} }]<\infty.$$ 
       \end{itemize} 
\end{thm}

This theorem is proven in Appendix \ref{appendix:b}. The factor with the $-K$ power represents a sort of replacement for the Gaussian decay kernel $e^{-\frac{|\x-\y|^2}r}$. With Cramer's theorem as intuition, the optimal decay should be exponential decay like $e^{-r^{-1/2} |\x-\y|}$, but we could not obtain as strong a decay as that in our proofs, and large negative powers will suffice for our purposes.

In particular, for $d \geq 2$, by a union bound over $r^{(k-1)d/2}$ balls, we have for each $i<j$ the bound uniformly over $   y\in\mathbb R^d$ and $\mathbf x=(x_1,...,x_k) \in I^k$ that 
\begin{equation}\label{eq:acBound2}
    \sup_{\substack{\br\in \Ekk\\d_{\mathrm{SRI}}(\br,\br_0) <\epsilon
       }}
       \; \Pbb ( |R_r^i-   R^j_r -   y|\leq 1) < C r^{-d/2} \big( 1+ r^{-1/2} |   x_i -   x_j  -    y|\big)^{-K}.
\end{equation}

We remark that this bound is fairly sharp; e.g., for a non-degenerate random walk on $\mathbb R^d$, it is sharply obtained from the local central limit theorem that the density at time $r$ has $L^\infty$ norm bounded above by $Cr^{-d/2}$, and the above estimate generalizes this fact to more complicated Markov chains. Thus the above theorem is fairly intuitive, essentially saying that any Markov chain that does not ``get stuck" near the boundaries of the Weyl chamber and resembles a random walk far from those boundaries still satisfies a similar bound. 

We will often apply Theorem \ref{anti} to bound expectations of certain functionals of $   R^i_r -   R^j_r$. In particular, we will make extensive use of the following result: 
\begin{cor}\label{cor:expectationBoundGeneral}
     Let $F:[0,\infty) \to[0,\infty)$ be a decreasing function of exponential decay at infinity. Fix $\sigma>0$, $M>0$, $K > 0$, and $d\ge 2$. Let $H:[0,\infty)\to[0,\infty)$ be decreasing and satisfy $H(x) \leq Ae^{-bx}$ where $A,b>0$. Let $\br_0 \in \Ekk$ be centered and $\delta$-repulsive. Then there exists $\epsilon=\epsilon(F,\sigma,M,\delta,d)>0$ and $C=C(F,\sigma,M,\delta,K, A,b,d)>0$ such that uniformly over $\x\in I^k$, $y\in \mathbb R^d$, $r\ge 0$ and  all $f: I\to [0,\infty)$ satisfying $|f(x)| \leq H(|x|)$ one has 
        $$ \sup_{\substack{\br\in \Ekk\\d_{\mathrm{SRI}}(\br,\br_0) <\epsilon
       }}\Ebb [f(   R^i_r -    R^j_r -y)]<C \cdot A b^{-d} \cdot r^{-\frac{d}2}  \big( 1+ r^{-1/2} |x_i-x_j-y|\big)^{-K} .$$
    \end{cor}

    In most applications, $H$ will be fixed, so we will use the form of the bound with the right side given by $Cr^{-d/2}(1+r^{-1/2}|x_i-x_j-y|)^{-K}$, disregarding $A,b$.

    \begin{proof}

We have
        \begin{align*}
\Ebb [f(   R^i_r - R^j_r -y)] &\leq  \Ebb [H(|   R^i_r - R^j_r -y|)] \\
&\leq \sum_{a\in\mathbb Z^d}  H(|a|) \Pbb(|   R^i_r - R^j_r -y -a| \le 1)
\end{align*}
We can use Theorem \ref{anti} (more precisely, Equation \eqref{eq:acBound2}) to obtain $\Pbb(|   R^i_r - R^j_r -y -    a| \leq 1) \leq C r^{-d/2} \cdot (1+r^{-1/2} |x_i-x_j-y-a|)^{-2K},$ uniformly over all the desired parameters. We therefore obtain
\begin{align*}
 \sup_{\substack{\br\in \Ekk\\d_{\mathrm{SRI}}(\br,\br_0) <\epsilon}}\Ebb [f(   R^i_r - R^j_r -y)] &\leq  \sup_{\substack{\br\in \Ekk\\d_{\mathrm{SRI}}(\br,\br_0) <\epsilon}}  \sum_{a\in \mathbb Z^d}^{\infty} H(|a|)\Pbb(|   R^i_r - R^j_r -y -    a| \leq 1)   \\
       &\leq  Cr^{-d/2} \sum_{a\in \mathbb Z^d} H(|a|) (1+r^{-1/2} |x_i-x_j-y-a|)^{-2K}.
\end{align*}
On one hand, $H$ decays exponentially which gives the upper bound of $\frac{A}{b}\cdot \big( 1+r^{-1/2} |x_i-x_j-y|\big)^{-2K}$. On the other hand, the $-K$ power in the sum is bounded above by $1$, thus the sum is also bounded by $\sum_{a\in \mathbb Z^d} H(|a|) \leq C \sum_{ \ell =1}^{\infty} \ell^{d-1} H(\ell) \leq C \frac{A\cdot d!}{b^d} .$ Now just use $\min\{ \frac{A}{b}, \frac{A}{b^d}\} = Ab^{-d} .$
\end{proof}

\subsection{Existence of invariant measures}
\begin{defn}
We say that a Borel measure $\pi$ on $\mathbb R^d$ has \textbf{constant growth at infinity} if $$\sup_{y\in \mathbb R^d} \pi(\{x:|x-y|\le M\}) \le CM^d$$ where $C$ is independent of $M.$
\end{defn}

If $\pi$ is a measure of constant growth of infinity on $\mathbb R^d$, then any measurable function $f:\mathbb R^d\to\mathbb R$ such that $|f(   x)|\leq H(|   x|)$ is integrable as long as $H:[0,\infty)\to[0,\infty)$ is non-increasing and satisfies $\sum_{\ell=0}^\infty \ell^{d-1} H(\ell)<\infty$ or equivalently $\sum_{\ell=0}^\infty H(\ell^{1/d})<\infty.$ This is fairly straightforward to show. Examples of such functions include any $f$ satisfying $|f(   x)| \leq C |   x|^{-d-\epsilon}$ for $|   x|>1$. 
This will be a very important and useful fact later.

\begin{thm}[Invariant measure] \label{thm:invMeasure} 
Assume that one has an SRI chain $\br\in \Ekk$ with $k=2$. 
Furthermore assume that $\br$ has the property that $\br (\x, \dr \y) = \br(\x',\dr\y')$ whenever $\x'=\x+(a,a)$ and $\y'=\y+(a,a)$ for some $a\in I$. 

Define the associated difference kernel $\bdif(   x,A):= \int_{I^k} \ind_{\{   y_1-   y_2\in A\}}\br( (   x,   0),\dr\y), $ for Borel sets $A\subset I$. Furthermore assume that $\bdif$ is ``Strong Feller" in the sense that $x\mapsto \bdif(x,\bullet)$ is continuous in total variation norm, and furthermore that for all $x\in I$ and all open sets $U\subset I$ there exists $m\in \mathbb N$ such that $\bdif^m(x,U)>0$. Then the Markov kernel $\bdif$ has a unique (up to scalar multiple) invariant measure $\pi^{\mathrm{inv}}$ on $I$ that has constant growth at infinity. 
\end{thm}

This theorem will be proven in Appendix \ref{appendix:c}. 

Note that $\pi^{\mathrm{inv}}$ is not a probability measure or even a finite measure---typically it will look like Lebesgue measure or counting measure far from the origin, but will have some impurities near the origin. In particular, the origin will often have extra mass to account for the extra time spent there by the process $\Pdif$. Since the invariant measure is only unique up to scalar multiple, we will need to fix a normalization of the measure $\pi^{\mathrm{inv}}$. 

\begin{defn}\label{def:f}
    Define the function $$\mathfrak f:=(\Pdif - \mathrm{Id})u \;\;\;\;\;\;\;\text{ where }\;\;\;\; u(x) := \begin{cases}1+|x|, &d=1 \\ \log(1+|x|) , & d=2\\ -(1+|x|)^{2-d}, &d\ge 3 .
\end{cases}$$
\end{defn}
We can think of $u(x)$ as an approximate (up to rescaling by a constant plus some lower order terms) fundamental solution for the generator $(\Pdif - \mathrm{Id})$. 
It follows from \eqref{eq:intfdu} and the constant growth condition of $\pi^{\mathrm{inv}}$ that $\int_I \mathfrak f\; \dr\pi^{\mathrm{inv}} < \infty.$ This allows us to define the following:
\begin{defn}\label{def:unit}
We define the \textit{unit normalization} of $\pi^{\mathrm{inv}}$ as the one in which $\int_I \mathfrak f\; \dr\pi^{\mathrm{inv}} =1,$ where $\mathfrak f$ is as in Definition \ref{def:f}. We further define the dimension dependent constants $c_1 =1, c_2=2\pi,$ and $c_d= d(d-2) /\Gamma(1+\frac{d}2)$ for $d\ge 3$.
\end{defn}

For the definition to make sense, we must show that $\int_I \mathfrak f \;\dr \pi^{\mathrm{inv}}\ne 0$, which will be clear from e.g. \eqref{theorem417} and \eqref{eq:12limit}. Informally, the unit normalization is the one such that $2c_d\cdot \pi^{\mathrm{inv}}$ would agree with Lebesgue measure at infinity if $I=\mathbb R^d$, or counting measure at infinity if $I=\mathbb Z^d$. Henceforth, all integrals against $\pi^{\mathrm{inv}}$ should be understood as being under the unit normalization.

\begin{rk}[Comparison with noise coefficients in \cite{timo2} and \cite{DDP+}] \label{rk:comparison}
    We discuss how our noise coefficient in $d=1$ matches with the coefficient obtained in \cite{timo2} for finite range $1$-d random walks in random environments. This discussion gives another interpretation of the noise coefficient and unit normalization in terms of the potential kernel (or Green's function if we extended this to the case $d \geq 3)$ for $\Pdif$. In \cite{timo2} they consider a model where $I = \mathbb Z$ and the kernels $K_n(x, \cdot)$ are i.i.d. in $x$ and $n$. Furthermore, the kernels have finite range, meaning that $K_n(x, y) \neq 0$ for only finitely many $y$, and they are aperiodic.  For simplicity we will assume that $p=1$.

    Following the notation in \cite{timo2}, let $q(x,y) = \Pdif(x,y)$, i.e. the transition probabilities for the difference of the two walkers sampled from the two-point motion, and let $\overline{q}(x,y)$ denote the transition probabilities for the difference of two independent random walks each sampled from the one-point motion.
    
   In $d=1$, the noise coefficient in Theorem \ref{main1} is given by 
   \begin{align*}
    \gamma_{\mathrm{ext}}^2 &:= \frac12 \sum_{z} \bigg[ \sum_{x,y}( x- y)^{2}\mu(x)\mu(y)-\sum_{a} ( a- z)^{2} q( z,a)\bigg] \pi^{\mathrm{inv}}(z) \\
    &= \frac12 \bigg [\sum_{x,y}( x- y)^{2}\mu(x)\mu(y)-\sum_{a} a^{2} q( 0,a)\bigg] \pi^{\mathrm{inv}}(0).
\end{align*}
Here $\pi^{\mathrm{inv}}$ is the invariant measure of the Markov kernel $q$, under its unit normalization. The second equality is due to the fact that for $z \neq 0$, $q(z,a) = \overline{q}(z,a) = \sum_{k} \mu(k)\mu(k + z-a)$. The inner sum can then be shown to equal zero. This can be further simplified by noting that 
\begin{align*}
    \frac12 \left[\sum_{x,y}( x- y)^{2}\mu(x)\mu(y)-\sum_{a} a^{2} q( 0,a)\right]= \text{Var}\left(\sum_{x} x K_1(0, x)\right).
\end{align*}
In the language of \cite{timo2}, this right-hand side is denoted by $\sigma_D^2$. Further define $D = \sum_{x} x K_1(0, x)$ and let 
$ \sigma_a^2: = \sum_{x \in \mathbb Z} (x - D)^2 \mu(x).$ By Assumption \ref{a1} Item \eqref{a22}, we will assume here (WLOG) that $\sigma_a^2 = 1$.

Let $G_n(x,y)$ and $\overline{G}_n(x,y)$ be the $n$-step Green's functions for $q$ and $\overline{q}$ respectively. Then define the potential kernels 
\begin{equation}\label{eq:potential}
a(x):= \lim_{n \to \infty}\{G_n(0,0) - G_n(x,0)\}; \  \ \ \ \ \  \overline{a}(x):= \lim_{n \to \infty}\{\overline{G}_n(0,0) - \overline{G}_n(x,0)\}.    
\end{equation} These potential kernels are fundamental solutions to their respective generators, i.e. 
$$(q - \mathrm{Id})a(x) = \tfrac12 \ind_{\{x = 0\}}; \  \ \ \ \ \   (\overline{q} - \mathrm{Id})\overline{a}(x) = \tfrac12 \ind_{\{x = 0\}}.$$
Next define $$\beta:= \sum_x q(0,x)\overline{a}(x) = \frac{\overline{a}(x)}{a(x)}.$$ This equivalence is explained in \cite{timo2}. In \cite{timo2}, the noise coefficient is then given by $\frac{\sigma_D^2}{\beta}$ so it remains to show that $\beta^{-1} =\pi^{\mathrm{inv}}(0).$ As shown in \cite[Theorem 4.4.8]{lawler2010random}, we have that $\overline{a}(x) = \frac{|x|}{2\sigma_a^2} + C + O(e^{-\beta x})$ for some constants $C$ and $\beta$. It follows that 
$$\int_I (q - \mathrm{Id})u \dr\pi^{\mathrm{inv}} = 2\int_I (q - \mathrm{Id})\overline{a} \dr\pi^{\mathrm{inv}}$$ as the $O(e^{-\beta |x|})$ term of the integral vanishes due to $\pi^{\mathrm{inv}}$ being an invariant measure. 

Finally, we have defined the unit normalization so that 
\begin{align*}
    1&= \int_I (q - \mathrm{Id})u \dr\pi^{\mathrm{inv}}= 2\sigma_a^2\int_I (q - \mathrm{Id})\overline{a} \dr\pi^{\mathrm{inv}} \\
    &= 2\sigma_a^2 \beta \int_I (q - \mathrm{Id})a\dr\pi^{\mathrm{inv}} =\beta \pi^{\mathrm{inv}}(0),
\end{align*}
where the last line follows from \eqref{eq:potential}. This concludes the proof. 

In \cite{DDP+}, we considered a special case of the above model, which is the $d=1$ case of the nearest-neighbor random walk in random environment described in Example \ref{ex:1}. For that model, we defined $\sigma^2: = \frac{1}{4}\sigma_D^2$. We then showed that  $\gamma_{\mathrm{ext}}^2 = \frac{8\sigma^2}{1 - 4\sigma^2} = \frac{2\sigma_D^2}{1 - 4\sigma^2}$. This denominator of $1 - 4\sigma^2$ is precisely equal to $\beta = \frac{1}{\pi^{\mathrm{inv}}(0)}$ and represents how likely the two walkers in the two-point motion are to stick together when they are at the same location. Note that the extra factor of $2$ in the numerator is due to the periodicity of the nearest-neighbor random walk. 

For further discussion of the noise coefficient for these discrete random walks in random environments, see \cite{hass2024extreme,hass2025universal}. 

\end{rk}

\subsection{Limit theorems for additive functionals} \label{sec:additivefunctionals}
We now consider additive functionals of the difference chain $R_r^1 - R^2_r$. The asymptotic behavior of such functionals is highly dependent on the dimension that we are working in as well as how close together the two walks are at the initial time. 

 We will first recall some classical results for simple symmetric random walks that will illustrate the dimension-dependent behavior that occurs for SRI chains as well. 
Consider independent simple symmetric random walks $   R^1_r,    R^2_r$. First suppose that $   R^1_0 =    R^2_0$, i.e., the two walks start out at the same position. Then as described in the introduction, we have the following limit laws (in distribution) as we take $N \to \infty$:

\begin{enumerate}
    \item In $d=1$, $N^{-1/2}\sum_{r=0}^{Nt}\ind_{\{R^1_r  - R^2_r = 0\}} \Rightarrow L^{B}_0(t)$, where $L^B_0(t)$ is the local time at $0$ of a Brownian motion $B$ of rate $2$. This was first proven in \cite{MR29488}.

    \item In $d=2$, $(\log N)^{-1}\sum_{r=0}^{Nt}\ind_{\{   R^1_r  -    R^2_r = 0 \}} \Rightarrow Z$, where $Z$ is an exponential random variable with rate $\pi$. This result was first proven in \cite{MR121870}.

    \item Finally, in $d \geq 3$, $\sum_{r=0}^{\infty}\ind_{\{   R^1_r  -    R^2_r = 0 \}}$ is a geometric random variable, with rate given by the return probability to 0. 
\end{enumerate}


More generally, for any random walk on $\mathbb Z^d$ with increments of mean zero and identity covariance matrix, we can consider the asymptotic behavior of $\sum_{r=0}^{N}f(   R^1_r -    R^2_r)$ for any nonnegative bounded function $f: \mathbb Z^d \to \R$ with $\sum_{x \in \mathbb Z^d} f(x) < \infty$, and we would get the following:
\begin{enumerate}
    \item In $d=1$, $N^{-1/2}\sum_{r=0}^{Nt}f(   R^1_r -    R^2_r) \Rightarrow (\sum_{x \in \mathbb Z} f(x)) \cdot L^{B}_0(t)$.

    \item In $d=2$, $(\log N)^{-1}\sum_{r=0}^{Nt}f(   R^1_r -    R^2_r) \Rightarrow (\sum_{x \in \mathbb Z^2} f(x)) \cdot Z,$ with $Z$ exponentially distributed and not depending on $f$ or $t$. 

    \item In $d \geq 3$, $\sum_{r=0}^{\infty}f(   R^1_r -    R^2_r)$ is a random variable depending both on $f$ and the model.
\end{enumerate}

These types of results are part of the general theory for asymptotic behaviors of additive functionals of Markov chains called the Darling-Kac theorem \cite{MR84222}, which shows that under suitable renormalization, the limiting distribution for a large class of additive functionals of a Markov process is given by a Mitag-Leffler distribution. Both the folded normal distribution (the distribution of Brownian local time at zero for fixed $t$) and the exponential distribution are special cases of the Mitag-Leffler distribution. See also \cite{MR56233} for analogous result for 2d Brownian motion.


Next, we consider the case where the random walkers start out far apart. In particular, we consider the case when $| R^1_0-   R^2_0| = O(\sqrt N)$. In $d=1$, we still have $$N^{-1/2}\sum_{r=0}^{Nt}\ind_{\{R^1_r  - R^2_r = 0\}} \Rightarrow L^{B}_0(t),$$ where the only difference is that Brownian motion $B$ may not start at 0. 

In $d \geq 2$, things change dramatically, as these additive functionals are now dominated by the (very small) probability that the two walkers meet up at least once before time $Nt$. We must normalize appropriately if we hope to see a nontrivial limit. The expected number of collisions between the random walkers is given by the $d$-dimensional Green's function whose behavior is of order $|   R^1_0 -    R^2_0|^{2-d} = N^{\frac{2-d}{2}}$. Therefore, $N^{\frac{2-d}{2}}$ should be the correct normalization if we expect to see something nonzero in the limit.

However, by Donsker's theorem, we still have that 
$N^{\frac{d-2}{2}}\sum_{r=0}^{Nt}\ind_{\{R^1_r  - R^2_r = 0\}}$ converges a.s. to $0$, 
because a $d$-dimensional Brownian motion with $d \geq 2$ starting macroscopically away from $0$ will remain bounded away from $0$ 
on the interval $[0,t]$. While the almost sure limit is zero, we still obtain a nontrivial limit in expectation.  
If the starting location is $(x^1_N,x^2_N)$ with $N^{-1/2} (x^1_N-x^2_N) \to x\ne 0$, then we have
$$ \lim_{N \to \infty} N^{\frac{d-2}{2}}\mathbb E_{(x^1_N,x^2_N)}\bigg[ \sum_{r=0}^{Nt} \ind_{\{R^1_r  - R^2_r = 0\}}\bigg] = c_d \int_0^t G(2s, x) ds .$$ 
Here $c_d$ are the dimension dependent constants from Defintion \ref{def:f}, and $G$ is the standard heat kernel on $\R^d$.


For SRI chains, we will essentially have the same set of results, with $\sum_{x \in \mathbb Z} f(x)$ replaced everywhere with $\int f\dr\pi^{\mathrm{inv}}$, which captures the local structure of the SRI chain near the origin. This is the subject of the next several theorems. We begin by stating an assumption that will be used throughout the remainder of this subsection. 

\begin{ass}\label{ass:SRI}
    Fix $\sigma>0$, $M>0$, $d\ge 1$ and a decreasing function $F:[0,\infty)\to[0,\infty)$ of exponential decay at infinity. 
Consider a sequence $\brn \in \Ekk$ with $k=2$, 
satisfying $\brn (\x, \dr \y) = \brn(\x',\dr\y')$ whenever $\x'=\x+(a,a)$ and $\y'=\y+(a,a)$. Assume that $\brn$ converges as $N\to \infty$ to some limiting SRI chain $\br$, which is $\delta$-repulsive for some $\delta>0$. We do not assume that $\br$ is centered, but we impose that $\br$ satisfies $d_{\mathrm{SRI}} (\brn, \br_0)\leq \epsilon$ for some centered chain $\br_0$ where $\epsilon$ is as in Theorem \ref{anti}. This is required in order for the anti-concentration estimates of Theorem \ref{anti} to hold, and these estimates will be crucial in the forthecoming proofs. We also assume that  the limiting chain $\br$ satisfies the strong Feller conditions in Theorem \ref{thm:invMeasure}. 

Consider any sequence of functions $f_N:I\to\mathbb R$ such that $|f_N(x)|\leq H(|x|)$ for some decreasing $H:[0,\infty)\to[0,\infty)$ of exponential decay at infinity. Assume $f_N\to f$ uniformly as $N\to\infty$. Let $   s_N:= \int_I    u \;\nu_N(\dr u)$, where $\nu_N$ is the base measure of $\brn$. Also assume (without loss of generality) that the covariance matrix of the base measure $\nu$ of $\br$ is $\mathrm{Id}_{d\times d}$. Let  $\pi^{\mathrm{inv}}$ be the invariant measure of $\bdif$ under its unit normalization.
\end{ass}

\begin{rk} The exponential decay assumption on $F$ is stronger than is necessary. It could be relaxed to a power decay of sufficient degree depending on $d$, but this would clutter the notation and require additional proof details, without adding much value to the main results, and we therefore do not pursue it.\end{rk}

\begin{thm}\label{inv01} Let $d \geq 1$ and consider the setting of Assumption \ref{ass:SRI}.
For any $t\ge 0$, and any $\phi\in C_c^\infty(\mathbb R^d)$, we have that 
\begin{multline}\label{eq:invlimit}\lim_{N\to\infty} N^{\frac{d-2}2} \sum_{r=0}^{Nt} \Ebbnn \big[\phi(N^{-1/2} (   R^1_r -    s_N r)\big)f_N(   R^1_r-   R^2_r) ] \\ =\bigg( \int_0^t \int_{\mathbb R^d} G(2s,x_1-x_2) G(2s,y) \phi(\tfrac12(y+x_1+x_2)) \dr y\dr s\bigg) c_d \int_I f \;\dr\pi^{\mathrm{inv}}
\end{multline}
whenever $N^{-1/2}\mathbf x_N \to \boldsymbol x= (   x_1,   x_2)\in \mathbb R^{2d},$ where $   x_{1}\ne    x_{2}$. Here $G$ appearing on the right side is the standard heat kernel on $\mathbb R^d$, that is $G(t,y) = (2\pi t)^{-d/2} e^{-|y|^2/(2t)}.$ Furthermore, $c_d$ and $\pi^{\mathrm{inv}}$ are as in Definition \ref{def:unit}.
\end{thm}

When $N^{-1/2}\mathbf x_N \to \boldsymbol x= (   x_1,   x_2)\in \mathbb R^{2d},$ where $   x_{1}\ne    x_{2}$, we say that $x^1_N$ and $x_N^2$ are \emph{well-separated}. We now address the case where $x_N^1$ and $x_N^2$ are close to each other, rather than well-separated as in Theorem \ref{inv01}. 

\begin{thm}\label{inv02} Assume Assumption \ref{ass:SRI}.
\begin{enumerate}\item If $d=1$ then Theorem \ref{inv01} remains valid even when $x_1=x_2,$ under the same assumptions.

    \item If $d=2$ then in the same setting as Theorem \ref{inv01}, if $\x_N$ is any sequence in $I^k$ such that $x^1_N-x^2_N$ remains bounded as $N\to\infty$, then we have that  \begin{equation} \label{theorem417}
    \lim_{N\to\infty}   \frac1{\log N}\sum_{r=0}^N\Ebbnn \bigg[f_N(   R^1_r-   R^2_r)\bigg] = \frac12 \int_I f\;\dr \pi^{\mathrm{inv}} . \end{equation}
    \end{enumerate}
\end{thm}

The above two theorems will be proven in Appendix \ref{appendix:d}, but we sketch the important ideas of the proofs here. The main idea is that by using the Krylov-Bogoliubov trick, all subsequential limits of these additive functionals can be shown to be equal to some constant multiple of the unique invariant measure for the limiting Markov chain $\bdif$, yielding the term $\int_I f\;\dr \pi^{\mathrm{inv}}$ in each of the results. We then identify the constant uniquely, so that the entire sequence converges to some unique multiple of the invariant measure. To nail down this constant, we compute the value of the limit for a specific choice of $f$. In particular, we take $f = \mathfrak f$ as in Definition \ref{def:f}. This yields the coefficient $\big( \int_0^t \int_{\mathbb R^d} G(2s,x_1-x_2) G(2s,y) \phi(\tfrac12(y+x_1+x_2)) \dr y\dr s\big)c_d$ in Theorem \ref{inv01} and the simpler coefficient $\frac{1}{2}$ in Theorem \ref{inv02}. 

We explain here how one can obtain the coefficient $\frac{1}{2}$ as the proof is illustrative. For an explanation of how the coefficient  $\big( \int_0^t \int_{\mathbb R^d} G(2s,x_1-x_2) G(2s,y) \phi(\tfrac12(y+x_1+x_2)) \dr y\dr s\big)c_d$ arises in $d=1$, see Remark \ref{rk:localtime}. 

The invariance principle of Theorem \ref{inv00} implies that for $\mathfrak f = (\Pdif - \mathrm{Id})u$ where $ u(x) =\log(1+|x|)$,
    \begin{align}
        \notag (\log N)^{-1}\sum_{r=0}^{N-1} \Ebbnn[\mathfrak f(   R^1_r-   R^2_r)] &= (\log N)^{-1} \sum_{r=0}^{N-1} \Ebbnn [ \log (1 + |R^i_{r+1} - R^j_{r+1}|) - \log ( 1+ |R^i_r-R^j_r|) ] \\ &= \notag \frac{ \Ebbnn[\log (1 + |R^1_{N}-R^2_{N}| )] - \log(1+ |x^i_N-x^j_N| ) }{\log N} \\&= \notag \frac12 + \frac{ \Ebbnn[\log (N^{-1/2} + N^{-1/2}|R^1_{N}-R^2_{N}| )] - \log(1+|x^i_N-x^j_N| ) }{\log N} \\& \stackrel{N\to\infty}{\longrightarrow} \frac12.\label{eq:12limit}
    \end{align} 
    In the last line, we used the invariance principle and the assumption that $|x^i_N-x^j_N|$ remains bounded to conclude that the numerator remains bounded. Finally, note that $\frac{1}{2} = \frac12 \int_I \mathfrak f \dr \pi^{\mathrm{inv}}$ when $\pi^{\mathrm{inv}}$ is under the unit normalization.

\begin{rk} \label{rk:localtime}
    For $d=1$, there is an alternative approach to proving Theorem \ref{inv01} (and the $d=1$ case of Theorem \ref{inv02}) in which one can use the theory of local times to take the limit of the expectations, using the invariance principle. This uses the approach of \cite{DDP23,DDP+, Par24}. For instance, assuming all particles start from 0, we have that 
     \begin{align*}\lim_{N\to\infty}  N^{-\frac{1}{2}} \sum_{r=0}^{Nt} \mathbf{E}_{(0,0)}^{\brn}\big[\phi(N^{-1/2} &(   R^1_r -    s_N r)\big)f_N(   R^1_r-   R^2_r) ] = \pi^{\mathrm{inv}}(f) \mathbf E_{BM^{\otimes 2}}\left[ \int_0^t \phi(X_s)L_0^{X-Y}(\dr s) \right]\\
     &=\pi^{\mathrm{inv}}(f) \mathbf E_{BM^{\otimes 2}}\left[ \int_0^t \phi\left(\frac{X_s + Y_s}{2}\right)L_0^{X-Y}(\dr s) \right]
     \\&= \pi^{\mathrm{inv}}(f) \int_0^t \bigg[\int_\mathbb R  G\big(\tfrac{s}{2},x\big) \phi(x)\dr x\bigg]G(2s,0)\dr s.
     \end{align*}
   The second equality is due to the fact that on the support of the local time measure $L_0^{X-Y}(\dr s)$, we have $X_s = Y_s$. The reason we want to change $X \to \frac{X +Y}{2}$ is so that it becomes independent of $X-Y$. Doing a further change of variables $x \to \frac{x}{2}$ matches  this expression with the right-hand side of \eqref{eq:invlimit} and completes the proof. 
     
     However, such an approach does not work in higher dimensions. Indeed if $d\ge 2$, then the family of processes $\big(N^{\frac{d-2}2} \sum_{s=N\varepsilon}^{Nt} f(   R^i_s-   R^j_s)\big)_{t\ge \varepsilon}$ viewed under the measures $\Ebbnn $ do not satisfy any nice $L^q$ moment bounds for \textit{any} $q>1$. Indeed if $|x_N^i-x_N^j|$ is of order $N^{1/2}$, the process itself converges in $C[0,T]$ to the zero process (by the invariance principle of Theorem \ref{inv00}), however the expectation converges to a finite and positive value, while the second moments and beyond blow up to infinity. This is in sharp contrast to $d=1$ where the process converges to the local time and satisfies moment bounds of every order. 
 \end{rk}

\begin{rk}In the above convergence theorems, the main difficulty comes from the fact that $d_{\mathrm{SRI}}(\br,\brn)$ can tend to zero extremely slowly (as slowly as one desires). Indeed if $d_{\mathrm{SRI}}(\br,\brn)$ tends to zero at a fast rate of $O(N^{-1})$, then many of the proofs of the appendix can be substantially simplified by a direct coupling with the centered chain. However such a case is irrelevant to us, since all of the interesting cases of models we actually study will have much slower convergence rates. For example in $d=2$, by Theorem \ref{main2} the most interesting case is when it goes to 0 like $1/\sqrt{\log N}$.
\end{rk}

The following theorem is an extension of the Erd\"os-Taylor theorem for 2d SRI chains and will be most important for the $  \varsigma_N$ at the critical scale in $d=2$ (that is, the third bullet point of Theorem \ref{main2}). 

\begin{thm}[Erd\"os-Taylor theorem for SRI chains in $d=2$]\label{2.10}
    Let $d=2$ and assume Assumption \ref{ass:SRI}. If $\x_N$ is any sequence in $I^2$ such that $x^1_N-x^2_N$ remains bounded as $N\to\infty$, then $(\log N)^{-1}\sum_{r=0}^N f_N(   R^1_r-   R^2_r) $ converges in law (viewed under $\Ebbnn$) as $N\to \infty$ to an exponential random variable of rate $2/\big(\int f\dr\pi^{\mathrm{inv}}\big)$. 

    Furthermore, for any $t>\varepsilon>0$ the random variable $\sum_{r=N\varepsilon}^{Nt}f_N(   R^1_r-   R^2_r)$ goes to 0 in probability, thus $(\log N)^{-1}\sum_{r=0}^{Nt} f_N(   R^1_r-   R^2_r), $ viewed as a process in the variable $t$, converges in f.d.d.'s to a constant process equal to the same exponential random variable. 
\end{thm}

\begin{rk}
    This theorem can be thought of as an Erd\"{o}s-Taylor theorem for sequences of SRI chains. The original Erd\"{o}s-Taylor theorem \cite{MR121870} proves that the occupation times of a 2d simple symmetric random walk are asymptotically exponential when rescaled by $\log  N$. A series of recent papers \cite{MR4616647,MR4750558} extends the Erd\"{o}s-Taylor theorem to the pairwise collisions of a family of i.i.d. simple symmetric random walks on $\mathbb Z^2$.
\end{rk}

We include the proof of Theorem \ref{2.10} in this section because it is extremely instructive for the proofs in later sections.

\begin{proof}
    Note that the last sentence is immediate from the invariance principle of Theorem \ref{inv00}, and the fact that the law of Brownian motion in $d=2$ stays positive distance from the origin on any time interval $[\varepsilon,t]$. 

    It suffices to prove the statement for the case where $f_N = f$ for all $N$. We can then bound the error terms involving $f_N - f$ using Corollary \ref{cor:expectationBoundGeneral}, noting e.g. that $|f_N(x)-f(x)| \leq H_N(|x|)$ for some functions of the form $H_N \downarrow 0$ with $H_N$ having some uniform exponential decay bound. In that case $b$ can be taken to be fixed and $A \to 0$, thus $Ab^{-d} \to 0$.

    We now prove the first statement. We will prove convergence of the integer moments of $(\log N)^{-1}\sum_{r=0}^N f(   R^1_r-   R^2_r) $ to the correct sequence of values, and one may easily show that this is enough. Thus we need to show that for all $m\in\mathbb N$:
    \begin{equation}\label{renew}
    \lim_{N\to\infty} \Ebbnn \bigg[ \bigg(\frac1{\log N}\sum_{r=0}^N f(   R^1_r-   R^2_r)\bigg)^m\bigg] = m! 2^{-m} \bigg(\int_I f\;\dr \pi^{\mathrm{inv}} \bigg)^m. \end{equation}
    
    Without loss of generality, we will assume that $f$ is compactly supported. This suffices by an approximation argument because for general nonnegative $f$ of exponential decay, one may take a sequence $f^{(\ell)} \uparrow f$ with each $f^{(\ell)}$ of compact support. Note that again by taking the factor $A$ appearing in the right side of Corollary \ref{cor:expectationBoundGeneral} to zero, $$\lim_{\ell \to \infty} \sup_{N\ge 1}\Ebbnn \bigg[ (\log N)^{-1} \sum_{r=0}^{N} |f^{(\ell)}-f|(   R^1_r-   R^2_r)\bigg] = 0.$$ Clearly any limit in $L^1$ of exponential random variables is still exponential, with rate given by the limit of the rate, thus justifying the reduction to $f$ of compact support.

    Let us now proceed by induction on $m\in \mathbb N$. The base case for $m=1$ just follows from Theorem \ref{inv02}.
    The inductive step simplifies due to the compact support assumption on $f$. The method that we use to complete the inductive step will henceforth be called the \textbf{renewal trick.} Suppose that the support of $f$ is contained in a ball of radius $K$ about the origin. Suppose the claim holds up to $m-1$. Then there exists some decreasing deterministic sequence $c_{m-1,f}(N)$ with $\lim_{N\to\infty} c_{m-1,f}(N)\to 0$ such that $$\sup_{\x: |x_i-x_j|\leq K} \bigg| \mathbf E_{\x}^{\brn} \bigg[\bigg( \frac1{\log N} \sum_{s=0}^N  f(   R^1_s-   R^2_s) \bigg)^{m-1}\bigg] - (m-1)!\bigg( \frac12 \int_I f\dr\pi^{\mathrm{inv}} \bigg)^{m-1}\bigg| \leq c_{m-1,f}(N).$$

    We can expand
    \begin{align*}\Ebbnn \bigg[ \bigg(\frac1{\log N}\sum_{r=0}^N f(   R^1_r-   R^2_r)\bigg)^m\bigg] &= m! (\log N)^{-m} \sum_{0\leq s_1<...<s_m \leq N} \Ebbnn \bigg[ \prod_{\ell=1}^m f(   R^1_{s_\ell}-   R^2_{s_\ell})\bigg] \\& \;\;+ (\log N)^{-m} \sum_{\substack{0\le s_1,...,s_m\le N\\ \text{repeated index}}} \Ebbnn \bigg[\prod_{\ell=1}^m f(   R^1_{s_\ell}-   R^2_{s_\ell})\bigg].
    \end{align*}
One may check that the terms with repeated indices are negligible as they yield a term of order $(\log N)^{-1}$ by the inductive hypothesis. As for the first term with no repeated indices, we can write it using the Markov property as 
\begin{align*}
    m! (\log & N)^{-m} \sum_{0\leq s_1<...<s_m \leq N} \Ebbnn \bigg[ \prod_{\ell=1}^m f(   R^1_{s_\ell}-   R^2_{s_\ell})\bigg] = \\&= \frac{m!}{\log N} \sum_{s_1=0}^N \Ebbnn \bigg[ f(   R^1_{s_1} -   R^2_{s_1}) \cdot \frac{g_N(s_1)^{m-1}}{(m-1)!} \mathbf E^{\brn}_{\mathbf R_{s_1}} \bigg[\bigg( \frac1{\log (N-s_1)} \sum_{s=0}^{N-s_1}f(   R^1_s-   R^2_s) \bigg)^{m-1} \bigg]\bigg]+O(1/\log N)
\end{align*}
where $g_N(s):= \frac{\log(N-s)}{\log N}$ and the $O(1/\log N)$ term again consists of collection of terms that have repeated indices and are negligible in the limit. By the inductive hypothesis and compact support of $f$, we know that 
\begin{align*}f(   R^1_{s_1} -   R^2_{s_1}) \bigg| \bigg( \frac1{\log (N-s_1)} \sum_{s=0}^{N-s_1} \mathbf E^{\brn}_{\mathbf R_{s_1}} [f(   R^1_{s}-   R^2_{s})]& \bigg)^{m-1} - (m-1)! \bigg( \frac12 \int_I f\dr\pi^{\mathrm{inv}} \bigg)^{m-1}\bigg| \\&\leq c_{m-1,f}(N-s_1)\|f\|_\infty \cdot \ind_{\{|   R^1_{s_1}-   R^2_{s_1}|\le K\}},\end{align*} where the right side is a deterministic quantity. 

Note that $$\lim_{N\to\infty} \frac1{\log N} \sum_{s=0}^N c_{m-1,f}(N-s)\Ebbnn [ \ind_{\{|   R^1_{s}-   R^2_{s}|\le K\}}] = 0.$$ This follows from splitting the sum in half and observing that by Corollary \ref{cor:expectationBoundGeneral}, we have
\begin{align*}
     \frac1{\log N}\sum_{s=0}^{N/2} c_{m-1,f}(N-s)\Ebbnn [ \ind_{\{|   R^1_{s}-   R^2_{s}|\le K\}}]  &\leq   \frac{C}{\log N} \sum_{s=0}^{N/2} \frac{c_{m-1, f}(N/2)}{1 \wedge s} \\
    & \leq C \cdot c_{m-1, f}(N/2)
\end{align*}
and
\begin{align*}
     \frac1{\log N}\sum_{s=N/2}^{N} c_{m-1,f}(N-s)\Ebbnn [ \ind_{\{|   R^1_{s}-   R^2_{s}|\le K\}}] 
    &\leq  \frac{C}{\log N} \sum_{s=N/2}^{N} \frac{1}{s} \\
    &\leq \frac{C \log 2}{\log N}.
\end{align*}
Both of these expressions tend to $0$ as $N \to \infty$. 

Putting the above together, it remains to show that 
\begin{align*}
\lim_{N \to \infty} \frac{1}{\log N} \sum_{s_1=0}^N \Ebbnn \bigg[ f(   R^1_{s_1} -   R^2_{s_1}) \cdot g_N(s_1)^{m-1}\bigg]  =  \frac12 \int_I f\dr\pi^{\mathrm{inv}} .
\end{align*}
We note that $g_N(s)\in[0,1]$ so that $$|g_N(s)^{m-1} - 1 | \leq m |g_N(s)-1| \leq Cm \cdot \frac{s}{N\log N}.$$ We therefore have 
\begin{multline*}
    \frac{1}{\log N} \sum_{s_1=0}^N \Ebbnn \bigg[ f(   R^1_{s_1} -   R^2_{s_1}) \cdot g_N(s_1)^{m-1}\bigg]  \\ \leq   \frac{1}{\log N} \sum_{s_1=0}^N \Ebbnn \bigg[ f(   R^1_{s_1} -   R^2_{s_1})\bigg] +\frac{C m}{N(\log N)^2} \sum_{s_1=0}^N s\Ebbnn \bigg[ f(   R^1_{s_1} -   R^2_{s_1})  \bigg].   
\end{multline*}
The first term on the right-hand side converges to $ \frac12 \int_I f\dr\pi^{\mathrm{inv}}$ by the $m=1$ case of this theorem, and the second term converges to $0$ by Corollary \ref{cor:expectationBoundGeneral}. 
\end{proof}

\begin{rk}
    Note that the reason this proof technique fails in $d= 1$ is due to the crucial fact noted in Remark \ref{rk:expTimeIndependence} that while in $d=1$, the limit of the discrete local time up to time $Nt$ depends on $t$, the exponential limit in $d=2$ does not depend on $t$. This allowed us to apply the inductive step to $\sum_{s=0}^{N-s_1} \mathbf E^{\brn}_{\mathbf R_{s_1}} [f(   R^1_{s}-   R^2_{s})]$ as $s_1$ is of order $N$ in the sum.
\end{rk}

We now prove a version of the last theorem with a test function appearing in the sum, which will be important later.
\begin{thm}[Backwards-time propagation of test functions in $d=2$]\label{2.11} 
Let $d=2$. Assume $\x_N=(x^1_N,x^2_N)$ is any sequence such that $N^{-1/2}x_N^1 \to x_1$ and $x^1_N-x^2_N$ remains bounded, and consider the same setting as Assumption \ref{ass:SRI}. 
      Consider smooth bounded functions $\phi_1, ..., \phi_m : \mathbb R^d\to \mathbb R$ with globally bounded and continuous first and second partial derivatives, and consider $f^1,..., f^m : I\to \mathbb R$ of exponential decay at infinity. Then \begin{equation}\label{eq:expMultiplef} 
    \lim_{N\to\infty} \Ebbnn \bigg[ \frac1{(\log N)^m}\sum_{0\leq r_1 < \ldots < r_m \leq Nt} \prod_{\ell =1}^m f^{\ell}(   R^1_{r_{\ell}}-   R^2_{r_{\ell}})\phi_{\ell}(N^{-1/2} (R^1_{r_\ell}-   s_N r_\ell))\bigg] = \prod_{\ell = 1}^m \bigg( \phi_\ell (x_1) \int_I f^\ell \; \dr \pi^{\mathrm{inv}} \bigg)\end{equation}
    The claim remains true if we replace $(f^1,...,f^m)$ on the left side by some sequence $(f^1_N,...,f^m_N)$ converging uniformly on compact sets as $N\to\infty$ to some $(f^1,...,f^m)$, as long as $\sup_{N} |f^j_N|$ decays exponentially fast at infinity for $j=1,...,m$.
\end{thm}

The reason we call this backwards propagation in time is because we begin with a test function $\phi$ whose input value is the Markov chain evaluated at some large time $r$, and then this input value gets propagated backwards in time to obtain the test function evaluated at time zero in the limit, i.e., the expression $\phi(x_1)$ appearing on the right-hand side of \eqref{eq:expMultiplef}. 

In $d=1$, this type of backward-time propagation fails completely, and the joint limit will be a nontrivial one given by a local time of Brownian motion and some Lebesgue-Stieltjes integral against the local time measure. The point is that in higher dimensions, all of the information is already concentrated at the initial time, as we already observed in Theorem \ref{2.10}.

\begin{proof}
    We proceed by induction on $m$. The induction step follows exactly as in the proof of Theorem \ref{2.10}, using the Markov property and applying the same exact renewal trick. It only remains to prove the $m=1$ case which is the base case for the induction argument. For the $m=1$ case, we will show that 
    \begin{equation}\label{backprop}(\log N)^{-1}  \sum_{r=0}^N \Ebbnn \bigg[ f(   R^1_r-   R^2_r) \cdot \bigg(\phi \big(N^{-1/2} (   R^1_r-   s_Nr)\big) - \phi(N^{-1/2}x_1^N) \bigg)\bigg]  \to 0.
    \end{equation}
    We remark that this is not true if one were to try to bring an absolute value inside the sum, so we need to be more precise.
    

    To prove the above limit, fix the test function $\phi$ and consider the sequence of (signed) measures $\gamma_N(f):= (\log N)^{-1}  \sum_{r=0}^N \Ebbnn \big[ f(   R^1_r-   R^2_r) \cdot \phi \big(N^{-1/2} (   R^1_r-   s_Nr)\big)\big] $ as defined above. On one hand, it is simple to show that any subsequential limit of $\gamma_N$ must be an invariant measure (see e.g., \hyperref[step4]{Step 4} in Appendix D), thus by Theorem \ref{thm:invMeasure} a multiple of $\pi^{\mathrm{inv}}$. If we can uniquely identify the constant as being equal to $\frac12 \cdot \phi(x_1)$, this would uniquely pin down every subsequential limit of $\gamma_N$. To do that, we just need to pick some specific $f$ and then explicitly calculate the limit of the expectation. 

    In the proof of Theorem \ref{2.10} we picked the specific $f$ given by $\mathfrak f = (\Pdif - \mathrm{Id})u(\x)$ where $u(\x):= \log(1+|x_1-x_2|)$, and here we will pick the same $f$. The calculation for this will be based on ``discrete Ito calculus" and is very similar to ideas used later in Appendix C. Henceforth abbreviate $V:=   R^1-   R^2$. We have by the Markov Property that 
    \begin{align*}
        \gamma_N&\big((\Pdif - \mathrm{Id})u\big) = (\log N)^{-1} \sum_{r=0}^N \Ebbnn \bigg[ \bigg( \log (1+|V_{r+1}| ) - \log (1+ |V_r|) \bigg) \cdot \phi \big(N^{-1/2} (   R^1_r-   s_Nr)\big)\bigg] \\
        &= (\log N)^{-1} \sum_{r=0}^N \Ebbnn \bigg[ \bigg( \log (N^{-1/2}+N^{-1/2}|V_{r+1}| ) - \log (N^{-1/2}+ N^{-1/2}|V_r|) \bigg) \cdot \phi \big(N^{-1/2} (   R^1_r-   s_Nr)\big)\bigg].
        \end{align*}
        Summing this expression by parts, we get that 
        \begin{align*}
        &\gamma_N\big((\Pdif - \mathrm{Id})u\big)
        = (\log N)^{-1} \Ebbnn \bigg[ \log (N^{-1/2}+N^{-1/2}|V_{N+1}|) \phi(N^{-1/2} (   R^1_{N+1} -    s_N(N+1))) \bigg] \\
        &- (\log N)^{-1} \Ebbnn \bigg[\log (N^{-1/2} +N^{-1/2} |x^1_N - x^2_N|) \phi(N^{-1/2} x_1^N)  \bigg] \\
        &\;\; -(\log N)^{-1} \sum_{r=0}^N \Ebbnn \bigg[ \log (N^{-1/2}+ N^{-1/2}|V_{r}|) \cdot \bigg(  \phi \big(N^{-1/2} (   R^1_{r+1}-   s_N(r+1)\big)- \phi \big(N^{-1/2} (   R^1_r-   s_Nr)\big)\bigg)\bigg] \\& + (\log N)^{-1} \sum_{r=0}^N \Ebbnn \bigg[ \bigg(\log (1+|V_{r+1}|) - \log(1+|V_r|) \bigg) \bigg( \phi \big(N^{-1/2} (   R^1_{r+1}-   s_N(r+1)\big)- \phi \big(N^{-1/2} (   R^1_r-   s_Nr)\big)\bigg) \bigg].
    \end{align*}

We now take the limits of each of these four terms.

\textbf{Term 1:} By the invariance principle of Theorem \ref{inv00}, the expression $\Ebbnn [ \log (N^{-1/2} +N^{-1/2} |V_{N+1}|)]$ approaches a finite value as $N \to \infty$ and hence $$(\log N)^{-1} \Ebbnn \bigg[ \log (N^{-1/2}+N^{-1/2}|V_{N+1}|) \phi(N^{-1/2} (   R^1_{N+1} -    s_N(N+1))) \bigg] \to 0.$$ 

\textbf{Term 2:} Using the assumption that $x^1_N-x^2_N$ remains bounded and the fact that $N^{-1/2} X^1_N\to x_1$, we have 
$$- (\log N)^{-1} \Ebbnn \bigg[\log (N^{-1/2} +N^{-1/2} |x^1_N - x^2_N|) \phi(N^{-1/2} x_1^N)  \bigg] \to \frac12 \phi(x_1).$$

\textbf{Term 3:} 
Taylor expand $\phi$ to get 
\begin{align*}
    \phi& \big(N^{-1/2} (   R^1_{r+1}-   s_N(r+1)\big)- \phi \big(N^{-1/2} (   R^1_r-   s_Nr)\big) \\& = N^{-1/2} (   R^1_{r+1}-R^1_r-   s_N) \bullet \nabla\phi  \big(N^{-1/2} (   R^1_r-   s_Nr)\big) +O(1/N).
\end{align*}
The $O(1/N)$ term can be disregarded since there are only $N$ summands and the sum above includes a further multiplying factor of $1/\log N$. On the other hand the first term is also inconsequential since the SRI property yields $|\Ebbnn [ R^1_{r+1} - R^1_r -    s_N| \mathcal F_r]| \leq \tilde f(   R^1_r-   R^2_r)$ for some rapidly decreasing function $\tilde f:I\to \mathbb R$, thus by Theorem \ref{anti} and the extra multiplying factor of $N^{-1/2}$ one obtains the vanishing in the associated term above.

\textbf{Term 4:} Note that 
\begin{align*}|\log (1+|V_{r+1}|) - \log(1+ |V_r|)| &\leq \bigg(\frac{1}{1+|V_r|} + \frac1{1+|V_{r+1}|} \bigg) |V_{r+1}-V_r| \\& \le  N^{-1/2} \bigg(\frac{1}{ N^{-1/2}(1+ |V_r|)} + \frac1{ N^{-1/2} (1+|V_{r+1}|)} \bigg) \sum_{j=1,2} |R^j_{r+1}-   R^j_r-   s_N|.
\end{align*} 
Likewise $$|\phi \big(N^{-1/2} (   R^1_{r+1}-   s_N(r+1)\big)- \phi \big(N^{-1/2} (   R^1_r-   s_Nr)\big)| \leq CN^{-1/2} \|\nabla \phi\|_{L^\infty} |   R^1_{r+1} - R^1_r -    s_N|.$$
Thus Term 4 is bounded above by something of the form $$C(\log N)^{-1} N^{-1} \|\nabla \phi\|_{L^\infty} \sum_{r=0}^N \Ebbnn \bigg[ \bigg(\frac{1}{ N^{-1/2}(1+ |V_r|)} + \frac1{ N^{-1/2} (1+|V_{r+1}|)} \bigg) \cdot \sum_{j=1,2} |R^j_{r+1}-   R^j_r -    s_N|^2 \bigg].$$ Consider conjugate exponents $q',q$ with $q'<2$, and use the Young's inequality $ab \le \frac1p a^{q'} + \frac1q b^q $, then note that the last expression is bounded above by $$C(\log N)^{-1} N^{-1} \|\nabla \phi\|_{L^\infty} \sum_{r=0}^N \Ebbnn \bigg[ \bigg(\frac{1}{ N^{-1/2}(1+ |V_r|)} + \frac1{ N^{-1/2} (1+|V_{r+1}|)} \bigg)^{q'} +  \sum_{j=1,2} |R^j_{r+1}-   R^j_r -    s_N|^{2q}  \bigg].$$ 


The power $q'$ term can be shown to vanish like $O(1/\log N)$ using the invariance principle of Theorem \ref{inv00} (and the anticoncentration estimate of Theorem \ref{anti} to establish the uniform integrability, noting $q'<2$), while the power $2q$ term vanishes thanks to uniform boundedness of the $(2q)^{th}$ moments in the definition of SRI chains which ensures that $ \sum_{r=0}^N \Ebbnn \big[ \sum_{j=1,2} |R^j_{r+1}-   R^j_r -    s_N|^{2q}  \big] \leq C N$ for some uniform constant $C$. Dividing this sum by $N \log N$ takes this term to $0$.

Combining the above four term-wise limits, we finally conclude that   $\gamma_N\big((\Pdif - \mathrm{Id})u\big) \to \frac12 \phi(x_1).$ As this is also the limit of $(\log N)^{-1}  \sum_{r=0}^N \Ebbnn \big[ \mathfrak f(   R^1_r-   R^2_r) \cdot \phi(N^{-1/2}x_1^N)\big] $ by Theorem \ref{inv02}, we conclude the proof of \eqref{backprop}. 
\end{proof}

\begin{ex}[Triviality of the exponential field] Let $d=2$ and assume Assumption \ref{ass:SRI}.
    Assume $N^{-1/2}x_N^1 \to x_1$ and $x^1_N-x^2_N$ remains bounded. Then for any $\phi\in C_c^\infty(\mathbb R^2)$, the pair $$(\log N)^{-1}\;\;\cdot\;\; \bigg (  \sum_{r=0}^{N} f(   R^1_r-   R^2_r)\;\;\;\; ,\;\;\;\; \sum_{r=0}^{N} f(   R^1_r-   R^2_r) \cdot \phi \big(N^{-1/2} (   R^1_r-   s_Nr)\big)\bigg) $$ when sampled from the chain $\brn$ converges in law to $(Z,\phi(x_1)Z)$ where $Z$ is an exponential random variable of rate $2 / \big( \int f\;\dr \pi^{\mathrm{inv}}\big).$ This follows immediately from Theorem \ref{2.11} which immediately yields
\begin{equation}(\log N)^{-1}  \sum_{r=0}^N   f(   R^1_r-   R^2_r) \cdot \bigg(\phi \big(N^{-1/2} (   R^1_r-   s_Nr)\big) - \phi(N^{-1/2}x_1^N) \bigg)  \to 0
    \label{backprop2} 
\end{equation}
in $L^2$ under the measures $\Ebbnn$, in the sense that the sequence of second moments goes to 0.
\end{ex}

    

\begin{thm}[Backwards-time propagation of test functions in $d \geq 3$] \label{2.12} Let $d\ge 3$. Assume $\x_N=(x^1_N,x^2_N)$ is any sequence such that $N^{-1/2}x_N^1 \to x_1$ and $x^1_N-x^2_N$ converges to $y\in I$ (\textit{without} any normalization), and consider the same setting as Assumption \ref{ass:SRI}. 
      Consider smooth bounded functions $\phi_1, ..., \phi_m : \mathbb R^d\to \mathbb R$ with globally bounded and continuous first and second partial derivatives, and consider $f^1,..., f^m : I\to \mathbb R$ continuous of exponential decay at infinity. Then \begin{align}\label{eq:expMultiplef1} 
    \lim_{N\to\infty}\notag  \Ebbnn &\bigg[\sum_{0\leq r_1 < \ldots < r_m \leq Nt} \prod_{\ell =1}^m f^{\ell}(   R^1_{r_{\ell}}-   R^2_{r_{\ell}})\phi_{\ell}(N^{-1/2} (R^1_{r_\ell}-   s_N r_\ell))\bigg] \\&=  \phi_1 (x_1) \cdots \phi_m(x_1) \sum_{0\le r_1 < ...< r_m<\infty } \; \mathbf E^{\boldsymbol b_{\mathrm{dif}}}_y \big[ f^1(X_{r_1}) \cdots f^m(X_{r_m}) ].\end{align}
    The claim remains true if we replace $(f^1,...,f^m)$ on the left side by some sequence $(f^1_N,...,f^m_N)$ equicontinuous and converging uniformly on compact sets as $N\to\infty$ to some $(f^1,...,f^m)$, as long as $\sup_{N} |f^j_N|$ decays exponentially fast at infinity for $j=1,...,m$.
\end{thm}

Unlike the $d=2$ case above, we remark that there is no logarithmic factor here, but more importantly the terms do not decouple as a product of individual terms in the limit, instead nesting inside one another due to the transience and the Markov property. Thus summarizing, we see that for $d=1$, the backwards propagation simply fails, while for $d=2$ and $d\ge 3$ it remains true but takes on somewhat different forms.

\begin{proof}
Again, we proceed by induction on $m$. The inductive step follows from the same renewal trick we used in the proof of Theorem \ref{2.10}, so we just need to prove the base case $m=1$. 

    For the $m=1$ case, we must show that
    \begin{equation} \label{ref:backpropd3} \sum_{r=0}^N \Ebbnn \bigg[ f(   R^1_r-   R^2_r) \cdot \bigg(\phi \big(N^{-1/2} (   R^1_r-   s_Nr)\big) - \phi(N^{-1/2}x_1^N) \bigg)\bigg]  \to 0.
    \end{equation}
   
   Fix the test function $\phi$ and consider the sequence of signed measures $\gamma_N(f):= \sum_{r=0}^N \Ebbnn \big[ f(   R^1_r-   R^2_r) \cdot \phi \big(N^{-1/2} (   R^1_r-   s_Nr)\big)\big] $. By the usual arguments (see \hyperref[step4]{Step 4} in Appendix D), any subsequential limit of $\gamma_N$ must be a multiple of $\pi^{\mathrm{inv}}$.  We will show that this multiple equals $u(y) \phi(x_1)$, which uniquely pins down every subsequential limit of $\gamma_N$. As usual, we will show this by considering the specific choice $\mathfrak f = (\Pdif - \mathrm{Id})u$ with $u(x) := (1+|x|)^{2-d}$. Recall that $V:=   R^1-   R^2$. We have by the Markov Property that 
    \begin{align*}
        \gamma_N&\big((\Pdif - \mathrm{Id})u\big) = \sum_{r=0}^N \Ebbnn \bigg[ \bigg( u(V_{r+1}) - u(V_r) \bigg) \cdot \phi \big(N^{-1/2} (   R^1_r-   s_Nr)\big)\bigg].
        \end{align*}
        Summing this expression by parts, we get that 
        \begin{align*}
        &\gamma_N\big((\Pdif - \mathrm{Id})u\big)
        =  \Ebbnn \bigg[ u(V_{N+1}) \phi(N^{-1/2} (   R^1_{N+1} -    s_N(N+1))) \bigg] \\
        &-  \Ebbnn \bigg[u(x^1_N - x^2_N) \phi(N^{-1/2} x_N^1)  \bigg] \\
        &\;\; -\sum_{r=0}^N \Ebbnn \bigg[ u(V_r) \cdot \bigg(  \phi \big(N^{-1/2} (   R^1_{r+1}-   s_N(r+1)\big)- \phi \big(N^{-1/2} (   R^1_r-   s_Nr)\big)\bigg)\bigg] \\
        & + \sum_{r=0}^N \Ebbnn \bigg[ \bigg(u(V_{r+1}) - u(V_r) \bigg) \bigg( \phi \big(N^{-1/2} (   R^1_{r+1}-   s_N(r+1)\big)- \phi \big(N^{-1/2} (   R^1_r-   s_Nr)\big)\bigg) \bigg].
    \end{align*}

We now take the limits of each of these four terms.

\textbf{Term 1:} By the invariance principle of Theorem \ref{inv00}, the expression $\Ebbnn [ u(V_{N+1})] =\Ebbnn [ (1 + |V_{N+1}|)^{2-d}] $ approaches $0$ as $N \to \infty$.

\textbf{Term 2:} Using the assumption that $x^1_N-x^2_N \to y$ and the fact that $N^{-1/2} x^1_N\to x_1$, we have 
$$ \Ebbnn \bigg[u(x^1_N - x^2_N) \phi(N^{-1/2} x_N^1)  \bigg] \to  u(y) \phi(x_1) .$$

\textbf{Term 3:} 
Taylor expand $\phi$ to get 
\begin{align*}
    \phi& \big(N^{-1/2} (   R^1_{r+1}-   s_N(r+1)\big)- \phi \big(N^{-1/2} (   R^1_r-   s_Nr)\big) \\& = N^{-1/2} (   R^1_{r+1}-R^1_r-   s_N) \bullet \nabla\phi  \big(N^{-1/2} (   R^1_r-   s_Nr)\big) +O(1/N).
\end{align*}
The $O(1/N)$ term goes to $0$ since by Theorem \ref{inv00},
\begin{align*}
   \frac{1}{N} \sum_{r=1}^N \Ebbnn \big[ u(V_r)\big] \leq  C\frac{1}{N} \sum_{r=1}^N r^{1- \frac{d}{2}} \to 0.
\end{align*}
The first term is also inconsequential since the SRI property yields $|\Ebbnn [ R^1_{r+1} - R^1_r -    s_N| \mathcal F_r]| \leq \tilde f(   R^1_r-   R^2_r)$ for some rapidly decreasing function $\tilde f:I\to \mathbb R$, thus by Theorem \ref{anti} and the extra multiplying factor of $N^{-1/2}$ one obtains the vanishing in the associated term above. 

\textbf{Term 4:} 
We have 
$$|\phi \big(N^{-1/2} (   R^1_{r+1}-   s_N(r+1)\big)- \phi \big(N^{-1/2} (   R^1_r-   s_Nr)\big)| \leq CN^{-1/2} \|\nabla \phi\|_{L^\infty} |   R^1_{r+1} - R^1_r -    s_N|.$$
Thus Term 4 is bounded above by $$C N^{-1/2} \|\nabla \phi\|_{L^\infty} \sum_{r=0}^N \Ebbnn \bigg[ \bigg(\frac{1}{ (1+ |V_r|)^{d-2}} - \frac1{  (1+|V_{r+1}|)^{d-2}} \bigg) \cdot  |   R^1_{r+1}-R^1_r -    s_N| \bigg].$$ 

Consider conjugate exponents $q',q$ with $q'<2$, and use H{\"o}lder's inequality to obtain
$$C N^{-1/2} \|\nabla \phi\|_{L^\infty} \sum_{r=0}^N \Ebbnn \bigg[ \bigg|\frac{1}{(1+ |V_r|)^{d-2}} - \frac1{  (1+|V_{r+1}|)^{d-2}} \bigg| ^{q'} \bigg]^{1/q'} \cdot  \Ebbnn \bigg[|   R^1_{r+1}-R^1_r -    s_N|^{q'}\bigg]^{1/q'}.$$ 

Due to uniform boundedness of the $q^{th}$ moments in the definition of SRI chains which ensures that $ \Ebbnn \big[ |   R^1_{r+1}-R^1_r -    s_N|^{q}  \big] \leq C $ for some uniform constant $C$.

The remaining sum can be shown to vanish using the invariance principle of Theorem \ref{inv00} (and the anti-concentration estimate of Theorem \ref{anti} to establish the uniform integrability, noting $q'<2$), while the power $2q$ term vanishes thanks to uniform boundedness of the $(2q)^{th}$ moments in the definition of SRI chains which ensures that $ \sum_{r=0}^N \Ebbnn \big[ \sum_{j=1,2} |R^j_{r+1}-   R^j_r -    s_N|^{2q}  \big] \leq C N$ for some uniform constant $C$. Dividing this sum by $N \log N$ takes this term to $0$.

Combining the above four term-wise limits, we finally conclude that   $\gamma_N\big((\Pdif - \mathrm{Id})u\big) \to  -u(y) \phi(x_1).$  As this is also the limit of $ \sum_{r=0}^N \Ebbnn \big[ \mathfrak f(   R^1_r-   R^2_r) \cdot \phi(N^{-1/2}x_1^N)\big] $, we conclude the proof of \eqref{ref:backpropd3}. 
\end{proof}

\begin{ex}
    Let $d \geq 3$ and assume Assumption \ref{ass:SRI}. Further assume that $f$ is continuous. Assume $N^{-1/2}x_N^1 \to x_1$ and $x^1_N-x^2_N$ converges to $y \in I$. Then for any $\phi\in C_c^\infty(\mathbb R^2)$, we have that 
\begin{equation} \sum_{r=0}^N   f(   R^1_r-   R^2_r) \cdot \bigg(\phi \big(N^{-1/2} (   R^1_r-   s_Nr)\big) - \phi(N^{-1/2}x_1^N) \bigg)  \to 0.
\end{equation}
in $L^2$ under the measures $\Ebbnn$. This is immediate from Theorem \ref{2.12}. In particular the sequence $$ \sum_{r=0}^{N} f(   R^1_r-   R^2_r) \cdot \phi \big(N^{-1/2} (   R^1_r-   s_Nr)\big)$$ converges in law as $N\to \infty$ to a random variable distributed as $\phi(x_1) \cdot \sum_{r=0}^\infty f(X_r)$ where $X$ is the limiting centered chain representing the difference of the two coordinates $R^1,R^2,$ with $X_0=y$.
\end{ex}

\subsection{Proof of SRI property for the tilted $k$-point Markov chains}

In this section, we will show that the tilted measures $\Pb$ that we defined in Section \ref{sec:2}  are SRI chains.

\begin{prop}\label{SRI2}
    The tilted measures $\Pb$ are SRI chains in the sense of Definition \ref{SRI}. More precisely, we have $\boldsymbol p^{(k)}\in \Ekk $ with $F= \Fd$ as in Assumption \ref{a1}, and $\boldsymbol p^{(k)}$ is centered and $\delta$-repulsive for some $\delta>0$. 
    
    Moreover, $d_{\mathrm{SRI}}(\boldsymbol p^{(k)}, \qdif) \leq C|   \beta|,$ thus there exists $\beta_0>0$ such that for $|   \beta|<\beta_0$ we have $\qdif \in \mathcal E^{\mathrm{SRI}}_{k\times d} (\Psi_{\sigma/2}, 2M, F^{1/2})$ and $\qdif$ is $(\delta/2)$-repulsive. 
    
    Furthermore, for any sequence $\beta^N\to v$ with $|v|<\beta_0$ the conditions in Assumption \ref{ass:SRI} are satisfied by the sequence $\boldsymbol q_{\beta^N}^{(2)}.$
\end{prop}

\begin{proof}
    The reason why $\boldsymbol p^{(k)}$ is centered is that the marginal law of each coordinate of the associated Markov chain is simply a random walk with increment distribution $\mu$. 
    
    The reason why $\boldsymbol p^{(k)}$ is $\delta$-repulsive is because of the final item in Assumption \ref{a1}. Next, one has by definition that 
    $$\|\boldsymbol p^{(k)}(\x,\bullet) - \qdif(\x,\bullet) \|_{TV} = \int_{I^k} \bigg| e^{\beta \bullet \sum_{i=1}^k (y_i-x_i) - \log \int_{I^k} e^{\beta \bullet \sum_{i=1}^k (a_i-x_i) }\boldsymbol p^{(k)}(\x,\dr \bfa)}\;-\;1 \bigg| \boldsymbol p^{(k)}(\x,\dr\y) .$$ Taylor expand everything in $\beta$, then note that the zeroth order terms cancel since it is 1, thus for small $\beta$ the first order term is the first nontrivial term. Factor out $\beta$ in this term and take the norm, then use Lemma \ref{grow} to bound the remaining integral by an absolute constant, and we arrive at the bound $d_{\mathrm{SRI}}(\boldsymbol p^{(k)}, \qdif) \leq C|   \beta|.$ The fact that $\qdif \in \mathcal E^{\mathrm{SRI}}_{k\times d} (\Psi_{\sigma/2}, 2M, F^{1/2})$ follows from e.g. Lemmas \ref{grow} and Lemma \ref{tbb}. Next, we need to make sure that the regularity conditions of Theorem 

    To prove the repulsivity statement, note more generally that there exists $\epsilon_0=\epsilon_0(\sigma,M,\delta)>0$, so that $\br\in\Ekk$ and $d_{\mathrm{SRI}}(\br,\br_0)<\epsilon_0$ together imply that $\br$ is $(\delta/2)$-repulsive. To prove this, notice (for example by Lemma \ref{tbb}) that $$\bigg|\int_{I^k} |y_i-y_j-(x_i-x_j)| \big( \br - \br_0 )(\x,\dr\y) \bigg| \leq C \cdot d_{\mathrm{SRI}}(\br,\br_0)^{1/2},$$
for some $C=C(q,M)>0.$ Thus taking $\epsilon_0:= \delta^2/(4C) $ gives the claim.

To prove the final statement, note that the strong Feller condition on the limiting chain is already provided by Assumption \ref{a1} Item (6).
\end{proof}

    \begin{cor}\label{cor:expectationBound}
        Let $d\ge 2$. Let $H:[0,\infty) \to[0,\infty)$ be decreasing with $H(x) \leq Ae^{-bx}$. There exists $C , \beta_0>0$ such that uniformly over all $r\in\mathbb Z_{\ge 0}$ and all $f: I\to [0,\infty)$ satisfying $|f(x)| \leq H(|x|)$ one has 
        $$ \sup_{   \beta: |   \beta|\leq \beta_0}\sup_{\x,\y \in I^k} \Eb [f(   R^i_r - R^j_r -y)]<Cr^{-\frac{d}2} \cdot Ab^{-d} \cdot  \big( 1+ r^{-1/2} |x_i-x_j-y|\big)^{-K}.$$
        The constant $C$ only depends on $H$ through $A,b$.
    \end{cor}

    \begin{proof} This follows from Proposition \ref{SRI2} and Corollary \ref{cor:expectationBoundGeneral}. The exponential decay assumption on $\Fd$ ensures that the conditions of Corollary \ref{cor:expectationBoundGeneral} are satisfied.
    \end{proof}

The following lemma allows us to apply Corollary \ref{cor:expectationBound} in conjunction with the Markov property to bound expectations of products of functions of exponential decay.
   \begin{lem}\label{prop:productEstimate} Let $d\ge 2$.
Suppose that $F:\mathbb [0,\infty)\to [0,\infty)$ 
    is a decreasing function of exponential decay at infinity. 
    Let $K > 0$. Then there exists some $\epsilon>0$ and a constant $C = C(K, F) >0$ such that uniformly over $|\varsigma|<\epsilon$, $   a_1,   a_2\in I$, and all $0 = r_0 < r_1 \leq \cdots \leq r_m$, we have  \begin{align*}
    \mathbf E^{(\varsigma,2)}_{(   a_1,   a_2)} \bigg[ \prod_{j=1}^{m} F(|   R^1_{r_j}-   R^2_{r_j}|)\bigg] \leq C \left( 1+ r_1^{-1/2} |   a_1 -    a_2|\right)^{-K} \prod_{j=1}^{m} 1 \wedge (r_j - r_{j-1})^{-d/2}. 
\end{align*}
In particular, for fixed $r\in \mathbb Z_{\ge 0}$ we have that \begin{equation}\label{!!!}\sum_{1\le s_1\le ...\le s_{m-1} \le r}\mathbf E^{(\varsigma,2)}_{(   a_1,   a_2)} \bigg[ F(|R^1_r -R^2_r|) \prod_{j=1}^{m} F(|   R^1_{r_j}-   R^2_{r_j}|)\bigg] \leq \frac{C^m}{\omega_r(d)^{m-1} } r^{-d/2} (1+r^{-1/2} |   a_1-   a_2| \big)^{-K} .
\end{equation}
\end{lem}

One might think of $F$ here as a place-holder for functions like $\Fd$ from Assumption \ref{a1}, as well as for particular important functions like $\z, \boldsymbol\eta_{\sigma,\beta},\boldsymbol u_{   v}$ from Section \ref{sec:2}.

\begin{proof}
The second bound follows readily from the first one. To prove the first one, consider $1 < k \leq m$. If $r_k \neq r_{k-1}$, it follows from Corollary \ref{cor:expectationBound} that 
    \begin{align*}
        \mathbf E^{(\varsigma,2)}_{(   a_1,   a_2)} \bigg[ \prod_{j=1}^{k} F(|   R^1_{r_j}-   R^2_{r_j}|) \bigg] &=  \mathbf E^{(\varsigma,2)}_{(   a_1,   a_2)} \bigg[ \prod_{j=1}^{k-1} F(|   R^1_{r_j}-   R^2_{r_j}|)  \mathbf E^{(\varsigma,2)}_{(   a_1,   a_2)}\left[F( |   R^1_{r_k}-   R^2_{r_k}|)| \mathcal F_{r_{k-1}}  \right]\bigg] \\
        &\leq C(r_k - r_{k-1})^{-d/2} \mathbf E^{(\varsigma,2)}_{(   a_1,   a_2)} \bigg[ \prod_{j=1}^{k-1} F(|   R^1_{r_j}-   R^2_{r_j}|) \bigg].
    \end{align*}
    On the other hand, if $r_k = r_{k-1}$, we can absorb the $k$th term of the product into the ${k-1}$st term so that the $(k-1)^{st}$ term is $F^2(|   R^1_{r_{k-1}}-   R^2_{r_{k-1}}|)$. We can then repeat the above argument  for the index $r_{k-1}$, as $F^2$ is still a function of exponential decay, so Corollary \ref{cor:expectationBound} still applies.  
   
    Therefore, iterating over the product, we obtain that 
    \begin{align*}
        \mathbf E^{(\varsigma,2)}_{(   a_1,   a_2)} \bigg[ \prod_{j=1}^{m} F(|   R^1_{r_j}-   R^2_{r_j}|)  f(    R^1_{s}-   R^2_{s})\bigg] &\leq C \prod_{j=2}^{m+1} 1 \wedge (r_j - r_{j-1})^{-d/2} \mathbf E^{(\varsigma_N,2)}_{(   a_1,   a_2)}\bigg[ F(|   R_{r_1}^1-   R^2_{r_1}|)\bigg].
    \end{align*}
    Finally, since $r_1 \neq 0$, we can apply Theorem \ref{anti} to the remaining expectation, taking into account the initial positions of the two particles to obtain that 
    \begin{align*}
        \mathbf E^{(\varsigma,2)}_{(   a_1,   a_2)}\bigg[ F(|   R_{r_1}^1-   R^2_{r_1}|) \bigg] \leq C r_1^{-d/2} \left( 1+ r_1^{-1/2} |   a_1 -    a_2|\right)^{-K}.
    \end{align*}
   
\end{proof}

The following corollary illustrates how we will apply the previous results.
    \begin{cor}[Markov property allows us to upgrade first-moment bounds to exponential moment bounds] Fix $t\ge 0$. Recall $\psi_N(p,d)$ from Assumption \ref{psi_n}, and define $$\omega_N(d):= (N^{-1}\psi_N(p,d))^{2p}.$$ Consider $f: I\to [0,\infty)$ of exponential decay at infinity. Then there exists $\epsilon_0>0$ and $C=C(d,k,f)>0$ such that we have the exponential moment bound
        $$\sup_{N\ge 1} \sup_{|   \beta|<\epsilon_0}\sup_{\x\in I^k} \Eb [ e^{\frac1{C} \omega_N(d)\sum_{r=1}^{Nt} f(   R^i_r-   R^j_r)} ] <\infty. $$
    \end{cor}

    Before the proof, we remark that this is the crucial estimate that explains why the $\psi_N(p,d)$ has the form that it has.

    \begin{proof}
      Taylor expand the exponential inside the expectation, and commute the infinite sum with the expectation. We will obtain an infinite sum over $m\in \mathbb N$ of terms of the form $$\sum_{m=1}^\infty \frac{\omega_N(d)^m}{C^m m!}\sum_{1\le r_1\le ... \le r_m \le Nt}\Eb \bigg[ \prod_{\ell=1}^m f(   R^i_{r_\ell} -   R^j_{r_\ell}) \bigg]. $$ Corollary \ref{prop:productEstimate} implies that this last sum is upper bounded by $$\sum_{m=1}^\infty \frac{\omega_N(d)^m}{C^m m!}\sum_{1\le r_1\le ... \le r_m \le Nt}\prod_{\ell=1}^m 1\wedge (r_\ell - r_{\ell-1})^{-d/2}.$$
      Note that in dimension $d$, this inner sum is precisely of order $m! \omega_N(d)^m$, so that this entire expression is bounded by 
     \begin{align*}
       \sum_{m=1}^\infty \frac{(C')^m}{C^m}. 
     \end{align*}
      Now choose $C<C'$ and the series converges, completing the proof.
    \end{proof}

\section{Limit formulas and estimates for the quadratic variations} \label{sec:5}


The quadratic variation field (QVF) from \eqref{qfield} will be extremely useful when proving tightness in a certain topology and for identifying the limit points. In particular, it will help us to obtain the correct noise coefficient for the limiting fields. 

\subsection{The limit formulas}

\begin{defn} \label{def:extendedQVF}
Given $   a_1,    a_2 \in I$, we define $Q^f_N(t,\phi;    a_1,   a_2): = Q^f_N(t,\phi; \delta_{a_1} \otimes \delta_{a_2})$, where $Q_N^f$ is given by \eqref{qfield}. In other words, the parameters $   a_1,   a_2\in I$ denote the starting locations of the two random walkers $R^1,   R^2$.
\end{defn}

Note the identity 
\begin{equation}\label{tilt0}\Ex\big[ Q_N^f(t,\phi; \nu_1 \otimes \nu_2)\big]  = \frac{\int_{I^2} e^{\varsigma_N\bullet(a_1+a_2)} \Ex\big[ Q_N^f(t,\phi;a_1,a_2)\big] \nu_1(\dr a_1) \nu_2(\dr a_2)}{\int_{I^2} e^{\varsigma_N \bullet(a_1+a_2)} \nu_1(\dr a_1)\nu_2(\dr a_2)},
\end{equation}
which will be important later. This identity is immediate from the definition of $Q_N$ in \eqref{qfield} and the form of the constants $C_{N,t,x,\nu}$ from Definition \ref{cntxnu}, in particular their dependence on $\nu.$

\begin{ass} \label{ass:QVBounds}
         Let $\phi \in C_c^\infty(\mathbb R)$. Let $f_N:I\to\mathbb R$ be any sequence of continuous functions such that $\sup_N|f_N(x)| \leq F(|x|)$ for some function $F: [0, \infty) \to [0, \infty)$ of exponential decay. Furthermore assume that $f_N$ converge uniformly as $N\to \infty$ to some $f: I\to\mathbb R$. Consider sequences $   a_i^N\in I \; (i=1,2)$ such that $(N^{-1/2}    a^N_1,N^{-1/2}   a^N_2) \to (   a_1,   a_2)$ respectively as $N\to \infty$. If $d \geq 2$, we further assume that $a_1^N$ and $a_2^N$ are well-separated in the sense that the limits satisfy $   a_1\ne    a_2$.
     \end{ass}

With this assumption, we will now prove some limit formulas for this quadratic variation field in each dimension. These quadratic variation formulas will be the key to identifying limit points later in Section \ref{sec:6}.

Fix $\phi\in C_c^\infty(\mathbb R^2)$. Notice that 
\begin{multline}\label{eq:quadVarSeriesOrig}
    \mathbb E[Q^{f_N}_N(t,\phi;    a_1,   a_2)]= \\N^{\frac{d-2}{2}}\sum_{s=0}^{Nt}\mathbf E_{(   a_1^N,   a_2^N)}^{(0,2)}\bigg[ \prod_{j=1}^2 C_{N,N^{-1}s, N^{-1/2} (   R^j_{s}-a_j^N - N^{-1}   d_Ns)} \phi\big(N^{-1/2} (   R^1_{s} - N^{-1}   d_Ns)\big)\cdot f_N(    R^1_{s}-   R^2_{s})\bigg].
\end{multline}
We would like to take the $N\to\infty$ limit of this expectation.  
Recall from Theorem \ref{main4} that we have $$\boldsymbol u_{   \varsigma_N}(   x_1 -    x_2)= \log\int_{I^2} e^{  \varsigma_N \bullet \sum_{j=1,2} (   y_j-   x_j)} \boldsymbol p^{(2)} (\x,\dr \y) - 2 \log M(   \varsigma_N).$$
Using \eqref{eq:tiltingExpectations} to rewrite \eqref{eq:quadVarSeriesOrig} in terms of the tilted measures introduced in Definition \ref{shfa}, we obtain 
     $$N^{\frac{d-2}{2}}\sum_{s=0}^{Nt}\mathbf E^{(   \varsigma_N,2)}_{(   a_1^N,   a_2^N)}\bigg[ e^{\sum_{r=0}^{s-1}\boldsymbol u_{   \varsigma_N}(   R^1_{r}-   R^2_{r})}  \phi \big(N^{-1/2} (   R^1_{s} - N^{-1}   d_Ns)\big)\cdot f(    R^1_{s}-   R^2_{s})\bigg].$$ 
Finally, we Taylor expand the exponential into an infinite series and we obtain an infinite series of the form 
\begin{equation}\label{eq:quadVarSeries}
N^{\frac{d-2}{2}}\sum_{s=0}^{Nt}\sum_{m=0}^\infty \frac1{m!} \mathbf E^{(   \varsigma_N,2)}_{(   a_1^N,   a_2^N)} \bigg[ \bigg(\sum_{r=0}^{s-1}\boldsymbol u_{   \varsigma_N}(   R^1_{r}-   R^2_{r})\bigg)^m  \phi \big(N^{-1/2} (   R^1_{s} - N^{-1}   d_Ns)\big)\cdot f(    R^1_{s}-   R^2_{s})\bigg].
\end{equation}

The limit of \eqref{eq:quadVarSeries} as $N \to \infty$ depends heavily on the dimension as shown in the following propositions. The proofs will involve taking the $N\to\infty$ limit for each fixed $m$ and then using dominated convergence to sum over $m$. The dominated convergence will be justified in Proposition \ref{dcbd}.


     \begin{prop}[Limit formula for the QVF in $d=1$] \label{prop:QVd1}Let $d=1.$  Assume that we are in the setting of Assumption \ref{ass:QVBounds}. Then for all $t>0$, if $ N^{1/4p}\varsigma_N  \to 0$ then 
		\begin{align}
			\label{eq:QLimitd1}
			\lim_{N \to \infty}\;  \Ex\big[ Q_N^{f_N}(t,\phi;    a^N_1,   a^N_2)\big] = \int_I f\;\dr\pi^{\mathrm{inv}} \cdot \int_0^t \int_{\mathbb R^d} G(2s,   a_1-   a_2) G(2s,   y) \phi(\tfrac12(   y+   a_1+   a_2)) \dr    y\dr s.
		\end{align}
	\end{prop}

    The case where $d=1$ and $ N^{1/4p}\varsigma_N  \to   {\boldsymbol v}\ne 0$ was addressed in \cite{Par24} and leads to a more complicated limit, but this regime is not covered in this paper. 

\begin{prop}[Limit formula for the QVF in $d=2$] \label{4.1b}Let $d=2.$ Assume that we are in the setting of Assumption \ref{ass:QVBounds}. Then there exists $\epsilon_{\mathrm{thr}} > 0$ (possibly depending on $F$) such that for all $t>0$ and $  {\boldsymbol v}\in \mathbb R^2$ with $|   {\boldsymbol{v}}| < \epsilon_{\mathrm{thr}}$, if $(\log N)^{1/2p}   \varsigma_N\to   {\boldsymbol v}$ then 
		\begin{align}
			\label{eq:QLimitd2}
			\lim_{N \to \infty}\;  \Ex\big[ Q_N^{f_N}(t,\phi;    a^N_1,   a^N_2)\big] = \frac{\int_I f\;\dr\pi^{\mathrm{inv}}}{1- \frac12\gamma_{\mathrm{ext}}(  {\boldsymbol v})^2 }\cdot 
          2 \pi \cdot \int_0^t \int_{\mathbb R^d} G(2s,   a_1-   a_2) G(2s,   y) \phi(\tfrac12(   y+   a_1+   a_2)) \dr    y\dr s.
		\end{align}
	\end{prop}

     \begin{prop}[Limit formula for the QVF in $d\ge 3$] \label{4.1c}Let $d \geq 3$. Assume that we are in the setting of Assumption \ref{ass:QVBounds}. Then there exists $\epsilon_{\mathrm{thr}}>0$ (possibly depending on $F$ and $d$) and a function $\Theta_{\mathrm{eff}}^2(f,\bullet):B_{\mathbb R^d}(0,\epsilon_{\mathrm{thr}})\to \mathbb R$ equal to $\pi^{\mathrm{inv}}(f)$ at $  {\boldsymbol v}= 0$, such that for all $t>0$ and $  {\boldsymbol v}\in \mathbb R^d$ with $|  {\boldsymbol v}|<\epsilon_{\mathrm{thr}}$, if $  \varsigma_N\to   {\boldsymbol v}$ then 
		\begin{align}
			\label{eq:QLimitd3}
			\lim_{N \to \infty}\;  \Ex\big[ Q_N^{f_N}(t,\phi;    a^N_1,   a^N_2)\big] = \Theta_{\mathrm{eff}}^2(f;  {\boldsymbol v}) \cdot c_d \cdot \int_0^t \int_{\mathbb R^d} G^{(\boldsymbol{v})}(2s,   a_1-   a_2) G^{(\boldsymbol{v})}(2s,   y) \phi(\tfrac12(   y+   a_1+   a_2)) \dr    y\dr s,
		\end{align}
        where $G^{(\boldsymbol{v})}(t,x)$ is the heat kernel with respect to the operator $\mathrm{div}(H_{  {\boldsymbol v}}\nabla)$, 
where $H_{  {\boldsymbol v}}$ is the Hessian matrix of $\log M$ at $  {\boldsymbol v}$, and where $M$ is the moment generating function of $\mu$ as in \eqref{dn}.
        The explicit expression for $\Theta_{\mathrm{eff}}^2(f;  {\boldsymbol v})$ is given by $$\int_I  \left[f(y) +  \left(e^{\boldsymbol u_{  \boldsymbol v}(   y)} -1\right) \mathbf E_y^{(\boldsymbol v,\mathrm{diff})} \left[\sum_{n=0}^{\infty} \sum_{1 \leq s_1\le ...\le s_{n} < \infty} g_{\mathrm{count}}( \vec s) \boldsymbol u_{  \boldsymbol v}(X_{s_1}) \cdots \boldsymbol u_{  \boldsymbol v}(X_{s_{n}}) \left(\sum_{s = s_{n} +1}^{\infty} f(X_{s})\right) \right] \right]\pi_{\boldsymbol v}^{\mathrm{inv}} (\dr y). $$ 
        Here $\pi_v^{\mathrm{inv}}$ is the invariant measure for the difference process (see Proposition \ref{diffproc}) of the tilted two-point chain $\boldsymbol q^{(2)}_{ v}$ from \eqref{qk}, and 
        \begin{equation}\label{gcount} g_{\mathrm{count}}( \vec r):= \prod_{n \in \mathbb Z} \frac{1}{\left(\#\{i: r_i = n\}\right)!}.\end{equation}
        In the particular special case of $f=\boldsymbol{\vartheta}_{   v}:= e^{\boldsymbol u_{   v}}-1,$ 
        we have the simpler form which matches with Theorem \ref{main4}: \begin{equation} \label{eq:d3specialization}\Theta_{\mathrm{eff}}^2(\boldsymbol{\vartheta}_{   v};    v) = \int_I \big( e^{\boldsymbol u_{   v }(y) }-1 \big) \mathbf E_y^{(v,\mathrm{diff})} \bigg[ e^{\sum_{s=1}^\infty \boldsymbol u_{   v} (X_s)} \bigg] \pi_{   v}^{\mathrm{inv}} (\dr y).\end{equation} 
        \end{prop}

        \begin{rk}
        Note that the above quadratic variation formulas are computed for general functions $f_N$, which will be useful for certain upper bounds. 
        However, to identify the limiting SPDEs, we will ultimately be choosing 
        $f_N = \boldsymbol{\eta}_{   r_1,    r_2}, f_N = \z_{\frac{   \varsigma_N}{|   \varsigma_N|}}$ or $f_N = \boldsymbol{\vartheta}_{   \varsigma_N}$ 
        as explained in Proposition \ref{4.3}. The choice $f_N = \boldsymbol{\eta}_{   r_1,    r_2}$ 
        will correspond to  \hyperref[eq:regimeA]{Regime A}, and will be addressed later. In the case where 
        $f_N = \z_{\frac{   \varsigma_N}{|   \varsigma_N|}}$ 
        and $\frac{   \varsigma_N}{|   \varsigma_N|} \to \frac{\boldsymbol v}{|\boldsymbol v|}$, 
        we have that  $f_N\to f:=\z_{\frac{\boldsymbol{   v}}{|\boldsymbol v|}}$. 
        Therefore $ \int_I f\;\dr\pi^{\mathrm{inv}} = \frac{\gamma_{\mathrm{ext}}(  {\boldsymbol v})^2}{|\boldsymbol{   v}|^{2p}}$.
        This choice is what ultimately gives us the noise coefficient in \hyperref[eq:regimeB]{Regime B} and \hyperref[eq:regimeC]{Regime C}. Finally, the case where $f_N = \boldsymbol{\vartheta}_{   \varsigma_N}$ corresponds with \hyperref[eq:regimeD]{Regime D} and was addressed in \eqref{eq:d3specialization}.
    \end{rk}


\subsection{Proofs of the limit formulas}

A key idea to prove the limit formulas will be to write exponential functionals of the relevant Markov processes as infinite series, then commute the infinite series with the limit $N\to\infty$. To do this, we will need various uniform estimates independent of the initial data as well as $N$. 

\begin{prop}[Bounds for Uniform integrability] Fix $f:I\to \mathbb R$ of exponential decay at infinity. \label{prop:DomConvergenceBound} 
\begin{enumerate}

\item 
We have a bound uniform over all $N\ge 1$, all initial conditions $   a_1,   a_2\in I$ and all $  {\boldsymbol{v}}$ in some neighborhood of the origin of $\mathbb R^d$:
\begin{multline}
    \left| N^{\frac{d-2}{2}}\sum_{s=0}^{Nt}\frac1{m!} \mathbf E^{(  {\boldsymbol{v}},2)}_{(   a_1,   a_2)} \bigg[ \bigg(\sum_{r=0}^{s-1}\boldsymbol u_{  {\boldsymbol{v}}}(   R^1_{r}-   R^2_{r})\bigg)^m  \phi \big(N^{-1/2} (   R^1_{s} - N^{-1}   d_Ns)\big)\cdot f(    R^1_{s}-   R^2_{s})\bigg]\right| \\\leq 
    \begin{cases}
      C^m \|\phi\|_{L^\infty} |  {\boldsymbol{v}}|^{2pm}\frac{(Nt)^{\frac{m}{2}}}{(m/2)!} & d=1  \\
      C^m \|\phi\|_{L^\infty}|  {\boldsymbol{v}}|^{2pm}  (\log Nt)^{2pm} \left[ 1+\big|\log \left(\frac{Nt}{  1+|a_1-a_2|^2}\right)\big|\right]  & d=2\\
      C^m \|\phi\|_{L^\infty}  |  {\boldsymbol{v}}|^{2pm}  \big(N^{-1/2} (1+|a_1-a_2|)\big)^{2-d} & d\ge 3\\
    \end{cases}\label{geombound}
\end{multline}
Here $C$ is an absolute constant not depending on $v,t,m,\phi, a_1,a_2$.

\item Summing \eqref{geombound} over $m\in\mathbb Z_{\ge 0}$, we have a bound uniform over all $N\ge 1$, all initial conditions $   a_1,   a_2\in I$ and all $  {\boldsymbol{v}}$ in some neighborhood of the origin of $\mathbb R^d$:
\begin{multline}\label{eq:geomBound1}
\left| N^{\frac{d-2}{2}}\sum_{s=0}^{Nt} \mathbf E^{(  {\boldsymbol{v}},2)}_{(   a_1,   a_2)} \bigg[ e^{\sum_{r=0}^{s-1}\boldsymbol u_{  {\boldsymbol{v}}}(   R^1_{r}-   R^2_{r})}  \phi \big(N^{-1/2} (   R^1_{s} - N^{-1}   d_Ns)\big)\cdot f(    R^1_{s}-   R^2_{s})\bigg]\right| \\\leq 
    \begin{cases}
      C\|\phi\|_{L^\infty} e^{N |  {\boldsymbol{v}}|^{4p} t} & d=1  \\
      \frac{C\|\phi\|_{L^\infty}}{1-C|  {\boldsymbol{v}}|^{2p}\log Nt}   \left[ 1+\big|\log \left(\frac{Nt}{  1+|a_1-a_2|^2}\right)\big|\right] & d=2\\
     \frac{C\|\phi\|_{L^\infty}}{1-C|  {\boldsymbol{v}}|^{2p}}  \left(N^{-1/2} (1+|a_1-a_2|)\right)^{2-d}  & d\ge 3\\
      \end{cases}
\end{multline}
The right side should be understood to be infinite if the denominator is non-positive. Here $C$ is an absolute constant not depending on $v,t,m,\phi, a_1,a_2$.
\end{enumerate}

\end{prop}

These bounds are quite sharp and they are illuminating, because their specific forms show where and why the critical scaling arises in each separate dimension. Notice that in $d=1$, the bounds on the right side is independent of the initial condition $(a_1,a_2)$, whereas in $d\ge 3$ the bound on the right side is independent of $t$. Meanwhile in $d=2$, the bound depends on both $t$ and $(a_1,a_2)$.

\begin{proof}
Since $\phi$ is a bounded function, we can just upper bound it by $\|\phi\|_{L^\infty}$ which respects the bounds, so we can just assume $\phi\equiv 1$ henceforth. Now since Item (1) implies Item (2) by summing over $m$, we focus on proving Item (1). 

We can then rewrite the remaining terms on the left-hand side as follows: 

\begin{align*}
 &C N^{\frac{d-2}{2}}\sum_{s=0}^{Nt}\frac1{m!} \mathbf E^{(  {\boldsymbol{v}},2)}_{(   a_1,   a_2)} \bigg[ \bigg(\sum_{r=0}^{s-1}\boldsymbol u_{  {\boldsymbol{v}}}(   R^1_{r}-   R^2_{r})\bigg)^m  f(    R^1_{s}-   R^2_{s})\bigg]\\
 &= C N^{\frac{d-2}{2}}\sum_{s=0}^{Nt}\sum_{0 \leq r_1 \leq \ldots \leq r_m \leq s-1} g_{\mathrm{count}}( \vec r)\mathbf E^{(  {\boldsymbol{v}},2)}_{(   a_1,   a_2)} \bigg[ \prod_{j=1}^{m} \boldsymbol u_{  {\boldsymbol{v}}}(   R^1_{r_j}-   R^2_{r_j})  f(    R^1_{s}-   R^2_{s})\bigg]
\end{align*}
where $g_{\mathrm{count}}$ is as in \eqref{gcount}. 
 We will use the bound $g_{\mathrm{count}}( \vec r) \leq 1$ to omit the term $g_{\mathrm{count}}( \vec r)$ going forward. It follows from Propositions \ref{dbound}, \ref{prop:TaylorG}, and \ref{prop:dcbound} that for $|  {\boldsymbol v}| < z_0/8$ we have 
\begin{align*}
    |\boldsymbol u_{  {\boldsymbol{v}}}(   R^1_{r}-   R^2_{r})| 
    &\leq C \Fd(|(   R^1_{r}-   R^2_{r})|)^{1/2} |  {\boldsymbol{v}}|^{2p}.
\end{align*}
Abbreviating $F(x):= \max\{|f(x)|, \Fd(|x|)^{1/2}\}$, it remains to show that the quantity
\begin{align}
\label{lhs1}C^m |  {\boldsymbol{v}}|^{2pm} \cdot N^{\frac{d-2}{2}}\ \sum_{s=0}^{Nt}\sum_{0 \leq r_1 \leq \ldots \leq r_m \leq s-1} \mathbf E^{(  {\boldsymbol{v}},2)}_{(   a_1,   a_2)} \bigg[ F(R^1_s-R^2_s) \prod_{j=1}^{m+1} F(   R^1_{r_j}-   R^2_{r_j})\bigg],
\end{align}
can be upper-bounded by the right side of \eqref{geombound}. The key to proving this will of course be \eqref{!!!}. 

We break this proof into cases based on dimension.
\\
\\
\textbf{$d= 2$ case.} Using \eqref{!!!} we see that 
\begin{align*}
         \eqref{lhs1} & \le  C^m |  {\boldsymbol{v}}|^{2pm}(\log (Nt))^m  \sum_{s=1}^{Nt} s^{-1} (1+s^{-1/2}|a_1-a_2|)^{-K} 
         \\& \le C^m |  {\boldsymbol{v}}|^{2pm}  (\log(Nt))^m \cdot [1+|\log (Nt / (1+|a_1-a_2|^2))|],
     \end{align*}
where the last bound follows by taking any $K>2$ so that those terms with $s$ of order less than $|a_1-a_2|^2$ become negligible in the sum.
\\
\\
\textbf{$d\ge 3$ case.} Using \eqref{!!!} we see that 
\begin{align*}
    \eqref{lhs1} & \le  C^m |  {\boldsymbol{v}}|^{2pm} N^{\frac{d-2}2}  \sum_{s=1}^\infty s^{-d/2} (1+s^{-1/2}|a_1-a_2|)^{-K} 
         \\& \le C^m |  {\boldsymbol{v}}|^{2pm} N^{\frac{d-2}2}  \cdot (1+ |a_1-a_2|)^{2-d}.
     \end{align*}
where the last bound follows by taking any $K>d$ so that those terms with $s$ of order less than $|a_1-a_2|^2$ become negligible in the sum. 
\\
\\
\textbf{$d=1$ case.} This case is somewhat different from the case of $d\ge 2$ as \eqref{!!!} has only been proved for $d\ge 2$. However, for $d=1$ the desired bound for \eqref{lhs1} is instead implied by \eqref{f2b} of the appendix (noting that $\sqrt{m!}$ and $(m/2)!$ are of the same order up to some constant $C^m$), completing the proof.
\end{proof}

As an easy corollary, the following proposition will allow us to apply the dominated convergence theorem to interchange the limit in $N$ and the sum in $m$ in \eqref{eq:quadVarSeries}. Recall $\omega_N$ from \eqref{omega_n}.

\begin{cor}[Bounds for dominated convergence]\label{dcbd} Assume that we are in the setting of Assumption \ref{ass:QVBounds}. Then there exists some $C,\epsilon > 0$ such that for each $m \geq 0$, each $N \geq 0$ and every $|  {\boldsymbol v}| < \epsilon$, we have  
\begin{align}\label{eq:geomBound}
\left| N^{\frac{d-2}{2}}\sum_{s=0}^{Nt}\frac1{m!} \mathbf E^{(\omega_N^{1/2p}  {\boldsymbol{v}},2)}_{(   a_1^N,   a_2^N)} \bigg[ \bigg(\sum_{r=0}^{s-1}u_{\omega_N^{1/2p}  {\boldsymbol{v}}}(   R^1_{r}-   R^2_{r})\bigg)^m  \phi \big(N^{-1/2} (   R^1_{s} - N^{-1}   d_Ns)\big)\cdot f(    R^1_{s}-   R^2_{s})\bigg]\right| \leq C^{m+1} |  {\boldsymbol{v}}|^{2pm}.
\end{align}In $d=1$ the right side can be improved to $C^{m+1}|  {\boldsymbol{v}}|^{2pm} / (m/2)!$.
\end{cor}
By choosing $|  {\boldsymbol{v}}|$ small enough, the right-hand side of \eqref{eq:geomBound} is summable. Furthermore, we see that if $  {\boldsymbol{v}}= 0$, then only the $m=0$ term is nonzero.  

\begin{proof}
    This is immediate from the previous proposition, specializing to the case where $(a_1,a_2) = (a_1^N,a_2^N)$ and using their well-separatedness to argue that $|a^1_N-a^2_N| \ge cN^{1/2}$, and also replacing $v$ by $\omega_N^{1/2p} v$. 
\end{proof}

We now turn to proving Propositions \ref{prop:QVd1}, \ref{4.1b} and \ref{4.1c}. Theorem \ref{inv01} will be the key to calculating the relevant limits, therefore we introduce a shorthand notation for the quantities appearing there. Let $$\pi^{\mathrm{inv}}(f):= \int_I f\;\dr\pi^{\mathrm{inv}},$$ where $\pi^{\mathrm{inv}}$ is the invariant measure for the difference process of the two-point motion $\boldsymbol p^{(2)}$, and has its unit normalization. Also let
\begin{align*}
    \mathcal G^{(\boldsymbol{v})} (t, \phi;   a_1,   a_2) :=  c_d\cdot \int_0^t \int_{\mathbb R^d} G^{(\boldsymbol{v})}(2s,   a_1-   a_2) G^{(\boldsymbol{v})}(2s,   y) \phi(\tfrac12(   y+   a_1+   a_2)) \dr    y\dr s,
\end{align*}
where $G^{(\boldsymbol{v})}$ is the heat kernel with respect to the operator $\mathrm{div}(H_{  {\boldsymbol v}}\nabla)$. We will finally use the notation $  \mathcal G (t, \phi;   x_1,   x_2) = \mathcal G^{(0)} (t, \phi;   x_1,   x_2)$ to denote the case where $G^{(0)} = G$ is the standard heat kernel without any tilting. 

\textbf{} 

\begin{proof}[Proof of Proposition \ref{prop:QVd1}] Thanks to the anticoncentration estimate of Theorem \ref{anti}, we can assume all $f_N\equiv f$, as the difference is easily bounded and shown to vanish by e.g. Corollary \ref{cor:expectationBoundGeneral}.

We can apply Proposition \ref{prop:DomConvergenceBound} to interchange the sum in $m$ in \eqref{eq:quadVarSeriesOrig} with the limit in $N$. Also note that it follows from Proposition \ref{prop:DomConvergenceBound} that for $m > 0$, the $m$th term converges to $0$ as $N \to \infty$ since we are in the case $  {\boldsymbol{v}} = 0$. Therefore only the $m=0$ term of the series contributes in the limit. The limit of this term is computed via Theorem \ref{inv01}, giving the desired result. 
\end{proof}

\textbf{}

 \begin{proof}[Proof of Proposition \ref{4.1b}]
Again, thanks to the anticoncentration estimate of Theorem \ref{anti}, we can assume all $f_N\equiv f$, as the difference is easily bounded and shown to vanish.

It follows from Proposition \ref{prop:TaylorG} that we can Taylor expand $\boldsymbol u_{   \varsigma_N}$ in the ${   \varsigma_N}$-variable to obtain
$$ \boldsymbol u_{   \varsigma_N}(   R^1_{r}-   R^2_{r}) = \z_{   \varsigma_N}(   R^1_{r}-   R^2_{r}) + \mathscr R_{   \varsigma_N}(   R^1_{r}-   R^2_{r}),$$ where $\z_{   \varsigma_N}(   x_1-   x_2)$ is a polynomial of degree $2p$ in the $  \varsigma_N$-variable and where we have that $\sup_{   x\in I} |\mathscr R_{   \varsigma_N}(   x)| \leq C|  \varsigma_N|^{2p+1}$.

Using this and interchanging the sum over $m$ with the sum over $s$, the $m^{th}$ term of the series in \eqref{eq:quadVarSeries} becomes $$\sum_{s=0}^{Nt} \sum_{0\le r_1< ...< r_m \le s-1} \mathbf E^{(   \varsigma_N,2)}_{(   a_1^N,   a_2^N)} \bigg[ \prod_{j=1}^m\z_{   \varsigma_N}(   R_{r_j}^1-   R^2_{r_j}) \cdot  \phi \big(N^{-1/2} (   R^1_{s} - N^{-1}   d_Ns)\big)\cdot f(    R^1_{s}-   R^2_{s}) \bigg] + O(1/\log N)$$where the $O(1/\log N)$ term is uniform over all variables and consists of collection of terms that either have repeated indices, or depend on the remainder term $\mathscr R_N$, and are thus negligible in the limit.
     
Let $  {v}_N:=(\log N)^{\frac{1}{2p}}    \varsigma_N$. We will now show that for all $m\in \mathbb N$ one has \begin{align}\notag
    \lim_{N\to\infty}& \sum_{s=0}^{Nt} \frac1{(\log N)^m}\sum_{0\le r_1< ...< r_m \le s-1} \mathbf E^{(   \varsigma_N,2)}_{(   a_1^N,   a_2^N)} \bigg[ \prod_{j=1}^m \z_{   v_N}(   R_{r_j}^1-   R^2_{r_j}) \cdot  \phi \big(N^{-1/2} (   R^1_{s} - N^{-1}   d_Ns)\big)\cdot f(    R^1_{s}-   R^2_{s}) \bigg] \\&= 2^{-m} \pi^{\mathrm{inv}}(\z_{   v})^m    \pi^{\mathrm{inv}}(f) \cdot \mathcal G (t, \phi;   a_1,   a_2) .\label{renew2} \end{align}
Recall the defintion of $\gamma_{\mathrm{ext}}(\boldsymbol{v})$ from \eqref{gext}. Observe that $\gamma_{\mathrm{ext}}(\boldsymbol{v}) =\pi^{\mathrm{inv}}(\z_{   v})$. Then summing \eqref{renew2} over $m$ and using the dominated convergence theorem to commute the sum over $m$ with the limit as $N\to\infty$ will yield the claim as long as the limiting vector $  {\boldsymbol v}$ satisfies $|  {\boldsymbol v}|< \epsilon_{\mathrm{thr}}$ where $\epsilon_{\mathrm{thr}}$ is chosen so that $C \epsilon_{\mathrm{thr}}^{2p} < 1$ with $C$ as appearing on the right side in Corollary \ref{dcbd}. This last condition ensures that the left-hand-side of \eqref{eq:geomBound} is summable in $m$, even with a supremum over all $N$ so that dominated convergence may be applied.
     
     Let us first consider $m=0$. This corresponds to finding the limit of $$\sum_{s=0}^{Nt}\mathbf E^{(   \varsigma_N,2)}_{(   a_1^N,   a_2^N)} \bigg[ \phi\big(N^{-1/2} (   R^1_{s} - N^{-1}   d_Ns)\big)\cdot f(    R^1_{s}-   R^2_{s})\bigg].$$ We can apply Theorem \ref{inv01}, since the tilted Markov kernels $\qdif$ have the SRI property as noted in Proposition \ref{SRI2} to obtain that this converges to 
     \begin{equation*}
    \pi^{\mathrm{inv}}(f) \cdot \mathcal G (t, \phi;   a_1,   a_2)
     \end{equation*}
    as desired. 

    Next we deal with the case $m \geq 1$.   
    We rearrange the indices on the left-hand side of \eqref{renew2} and applying the Markov property, we get that it equals
    \begin{multline}
\sum_{r_1=0}^{Nt-1} \mathbf E^{(   \varsigma_N,2)}_{(   a_1^N,   a_2^N)} \Bigg[ \z_{   v_N}(   R^1_{r_1} -    R^2_{r_1})  \\ \cdot
\mathbf E^{(   \varsigma_N,2)}_{(   R^1_{r_1},   R^2_{r_1})} \bigg[\frac{1}{(\log N)^m} \sum_{1 \leq s_1 <...< s_{m}  \le Nt}  \prod_{j=1}^{m-1}  \z_{   v_N}(   R^1_{s_j} -    R^2_{s_j}) \cdot f(    R^1_{s_m}-   R^2_{s_m})\phi \big(N^{-1/2} (   R^1_{s_m} - N^{-1}   d_Ns_m)\big)\bigg] \Bigg] \label{mp1}
\end{multline}
Here we take the convention that the empty sum is equal to $1$ and that $s_0 = 0$.
The main idea to proceed is now to apply the expectation limit formula of Theorem \ref{inv01} in conjunction with the backwards-in-time propagation of test functions from Theorem \ref{2.11}.


Fix $T>0$. We now claim that thanks to the backwards-time-propagation from Theorem \ref{2.11}, one has
\begin{multline}
\lim_{N \to \infty} N\cdot \sup_{u \in N^{-1} \mathbb Z \cap [0,T]} \mathbf E^{(   \varsigma_N,2)}_{(a^1_N,a^2_N)} \Bigg[\z_{   v_N} (   R^1_{Nu}-   R^2_{Nu}) \;\;\cdot \\ \bigg|\mathbf E^{(   \varsigma_N,2)}_{(   R^1_{Nu},   R^2_{Nu})} \bigg\{\frac{1}{(\log N)^m} \sum_{1 \leq s_1 <...< s_{m}  \le Nt}  \prod_{j=1}^{m-1}  \z_{   v_N}(   R^1_{s_j} -    R^2_{s_j}) \cdot f(    R^1_{s_m}-   R^2_{s_m})\phi \big(N^{-1/2} (   R^1_{s_m} - N^{-1}   d_Ns_m)\big)\bigg\}\\ - \phi \big(N^{-1/2} (   R^1_{Nu} - N^{-1}   d_N Nu)\big)\cdot 2^{-m} \pi^{\mathrm{inv}}(f) \pi^{\mathrm{inv}}(\z_v)^{m-1} \bigg|\Bigg]=0.\label{th2}
\end{multline} 

In other words, we can replace the inner expectation of \eqref{mp1} with the much simpler quantity given by some constant multiple of $\phi \big(N^{-1/2} (   R^1_{r_1} - N^{-1}   d_N r_1)\big),$ where the constant is precisely $2^{-m} \pi^{\mathrm{inv}}(f) \pi^{\mathrm{inv}}(\z_v)^{m-1}$. The proof of \eqref{th2} follows readily from the backwards propagation of Theorem \ref{2.11}, and for completeness we will sketch the argument at the end of this proof. For now we assume it is true. Then we may divide by $N$ in \eqref{th2}, and sum over all $u\in N^{-1}\mathbb Z\cap [0,t]$. 
This means that $N\cdot \sup_{u\in (N^{-1}\mathbb Z)\cap [0,T]}$ can be replaced by $\sum_{u\in (N^{-1}\mathbb Z)\cap [0,T]}$ in the expression \ref{th2}, and the limit is still 0. Finally we can use this replacement to calculate the original limit \eqref{mp1}. Indeed, it just remains to take the limit of 
\begin{multline} 
2^{-m} \pi^{\mathrm{inv}}(f) \pi^{\mathrm{inv}}(\z_v)^{m-1} \sum_{r_1=0}^{Nt-1} \mathbf E^{(   \varsigma_N,2)}_{(   a_1^N,   a_2^N)} \Bigg[ \z_{   v_N}(   R^1_{r_1} -    R^2_{r_1}) \phi \big(N^{-1/2} (   R^1_{r_1} - N^{-1}   d_Nr_1)\big)
\Bigg].
\end{multline}
We can now use Theorem \ref{inv01} to immediately obtain that this converges to \begin{align*}
&2^{-m} \pi^{\mathrm{inv}}(\z_{   v})^{m}    \pi^{\mathrm{inv}}(f) \cdot \mathcal G (t, \phi;   a_1,   a_2),
\end{align*}
as desired. 
\\
\\
\textbf{Proof of \eqref{th2}.} Consider any sequence $u_N \in (N^{-1} \mathbb Z_{\ge 0} )\cap [0,T]$, 
with $\liminf_N u_N>0$. Define the change of measure $\tilde {\mathbf E}_N [F] := N \cdot \mathbf E^{(\varsigma_N,2)}_{(a^1_N,a^2_N)} \big[|\z_{   v_N} (   R^1_{Nu_N}-   R^2_{Nu_N})|\cdot F \big].$ Such $\tilde {\mathbf E}_N $ may not be probability measures, but their mass is bounded uniformly in $N$ by the anticoncentration estimate of Theorem \ref{anti} (since $N^{d/2} = N$ if $d=2$), so we can assume without loss of generality that they are probability measures. Now the same anticoncentration estimate and the fact that the initial states $a^1_N,a^2_N$ are assumed to be well-separated, one easily obtains $\tilde {\mathbf E}_N [R^1_{Nu_N}] \leq CN^{1/2}$, so that $N^{-1/2} R^1_{Nu_N}$ is tight under $\tilde {\mathbf E}_N .$
Similarly $|R^1_{Nu_N} -R^2_{Nu_N}|$ is tight under $\tilde{\mathbf E}_N$ since $\z_{v_N}$ decays rapidly. This tightness ensures that the random walkers are close together at time $Nu_N$ so that we can apply Theorem \ref{2.11} with $\phi_\ell\equiv 1$ for $1\le \ell \leq m-1$, and $\phi_m:=\phi$. This yields that under $\tilde {\mathbf E}_N $, one has convergence in law as $N\to \infty$: 

\begin{align*}
    \mathbf E^{(   \varsigma_N,2)}_{(   R^1_{Nu_N},   R^2_{Nu_N})} \bigg\{\frac{1}{(\log N)^m} \sum_{1 \leq s_1 <...< s_{m}  \le Nt}  \prod_{j=1}^{m-1}  \z_{   v_N}(   R^1_{s_j} -    R^2_{s_j}) \cdot f(    R^1_{s_m}-   R^2_{s_m})\phi \big(N^{-1/2} (   R^1_{s_m} - N^{-1}   d_Ns_m)\big)\bigg\}\\ - \phi \big(N^{-1/2} (   R^1_{Nu_N} -    d_N u_N)\big)\cdot 2^{-m} \pi^{\mathrm{inv}}(\z_{   v})^{m-1}    \pi^{\mathrm{inv}}(f) \to 0.
\end{align*}
From here, \eqref{th2} follows readily by e.g., bounded convergence. 
\end{proof}

\begin{proof}[Proof of Proposition \ref{4.1c}] As usual, thanks to the anticoncentration estimate of Theorem \ref{anti}, we can assume all $f_N\equiv f$, as the difference is easily bounded and shown to vanish. 

Interchanging the sum over $m$ with the sum over $s$, the $m^{th}$ term of the series in \eqref{eq:quadVarSeries} becomes 
\begin{equation}\label{wttl}N^{\frac{d-2}2} \sum_{s=0}^{Nt} \sum_{0\le r_1\le  ...\le r_m \le s-1} g_{\mathrm{count}}( \vec r) \mathbf E^{(   \varsigma_N,2)}_{(   a_1^N,   a_2^N)}  \bigg[ \prod_{j=1}^m \boldsymbol u_{   \varsigma_N}(   R^1_{r_j}-R^2_{r_j}) \cdot  \phi \big(N^{-1/2} (   R^1_{s} - N^{-1}   d_Ns)\big)\cdot f(    R^1_{s}-   R^2_{s}) \bigg].
\end{equation}
For each $m\in \mathbb N$, we would like to take the limit of this expression as $N\to\infty$. Since $  \varsigma_N\to    v$ as $N\to \infty$, notice that $\boldsymbol u_{   \varsigma_N}$ converges to $\boldsymbol u_{   v}$, the function defined in Theorem \ref{main4}. Similarly, $\boldsymbol q^{(2)}_{    \varsigma_N}  \to \boldsymbol q_{  {\boldsymbol v}}^{(2)}$ as defined in Theorem \ref{main4}. The expectation with respect to the difference between the coordinates of this limiting chain is denoted by $\mathbf E_y^{(v,\mathrm{diff})}.$ Unlike the proof of Proposition \ref{4.1b}, the terms with repeated indices will \textbf{not} be irrelevant here since $   \varsigma_N$ is not necessarily going to $0$ here. 

By Theorem \ref{inv01}, the limit of the $m=0$ term is 
$\pi^{\mathrm{inv}}(f) \cdot \mathcal G^{(\boldsymbol{v})} (t, \phi;   a_1,   a_2).$
Here $\pi_v^{\mathrm{inv}}$ is the invariant measure for $\boldsymbol q_{  {\boldsymbol v}}^{(2)}$, and it has its unit normalization as in Theorem \ref{inv01}. Similarly, the heat kernel $G^{(v)}$ is with respect to the operator $\mathrm{div}(H_{  {\boldsymbol v}}\nabla)$
where $H_{  {\boldsymbol v}}$ is the Hessian matrix of $\log M$ at $  {\boldsymbol v}$, where $M$ is the moment generating function of $\mu$ as in \eqref{dn}.

Next we consider the case $m \geq 1$. As in the proof of Proposition \ref{4.1b}, we will assume without loss of generality that $\Fd(\cdot)$ is compactly supported so that the sequence of functions $\boldsymbol u_{   \varsigma_N}(\cdot)$ has uniformly bounded support in $N$. This ensures that at time $r_1$, the distance $|   R^1_{r_1} -    R^2_{r_1}|$ is uniformly bounded in $N$ on the event where $\boldsymbol u_{   \varsigma_N}(   R^1_{r_1} -    R^2_{r_1}) > 0$. We can then use Theorem \ref{2.12} to propagate the function $\phi$ backwards in time to time $r_1$ in the expression \eqref{wttl}. The details are similar to the proof of Proposition \ref{4.1b} for the $d=2$ case, and follow the same lines as the proof of \eqref{th2} above (using Theorem \ref{2.12} as opposed to Theorem \ref{2.11}). The effect of this backwards propagation is that we can replace $\phi(N^{-1/2} (   R^1_s - N^{-1}    d_Ns))$ in \eqref{wttl} with the quantity $\phi \big(N^{-1/2} (   R^1_{r_1} - N^{-1}   d_Nr_1)\big)$, without changing the value of the $N\to\infty$ limit. More precisely, we just need to  compute the limit of 
\begin{align*}
    &N^{\frac{d-2}2} \sum_{s=0}^{Nt} \sum_{0\le r_1\le  ...\le r_m \le s-1} g_{\mathrm{count}}( \vec r) \mathbf E^{(   \varsigma_N,2)}_{(   a_1^N,   a_2^N)}  \bigg[ \prod_{j=1}^m \boldsymbol u_{   \varsigma_N}(\mathbf R_{r_j}) \cdot  \phi \big(N^{-1/2} (   R^1_{r_1} - N^{-1}   d_Nr_1)\big)\cdot f(    R^1_{s}-   R^2_{s}) \bigg].
\end{align*}
Rearranging the indices and applying the Markov property, this equals
\begin{multline*}
N^{\frac{d-2}2} \sum_{r_1=0}^{Nt-1}\mathbf E^{(   \varsigma_N,2)}_{(   a_1^N,   a_2^N)} \Bigg[ \sum_{l = 1}^m \frac{1}{l!} \boldsymbol u_{   \varsigma_N}(   R^1_{r_1} -    R^2_{r_1})^l \phi \big(N^{-1/2} (   R^1_{r_1} - N^{-1}   d_Nr_1)\big) \\ \cdot
\mathbf E^{(   \varsigma_N,2)}_{(   R^1_{r_1},   R^2_{r_1})} \left[\sum_{1\le s_1\le  ...\le s_{m-l}  \le Nt-1} g_{\mathrm{count}}( \vec s)  \prod_{j=1}^{m-l} \boldsymbol u_{   \varsigma_N}(\mathbf R_{s_j}) \cdot \sum_{s = s_{m - l}+1}^{Nt} f(    R^1_{s}-   R^2_{s})\right] \Bigg]
\end{multline*}
As usual, we take the convention that the empty sum is equal to $1$ and that $s_0 = 0$.

Applying Theorem \ref{inv01}, we immediately obtain that this converges to 

\begin{multline} \label{eq:fixedmd3}
 \mathcal G^{(\boldsymbol{v})} (t, \phi;   a_1,   a_2) \cdot \\ 
\int_I   \sum_{l = 1}^m \frac{1}{l!} \boldsymbol u_{ \boldsymbol  v}(   y)^l \mathbf E_y^{(\boldsymbol v,\mathrm{diff})} \bigg[ \sum_{1 \leq s_1\le ...\le s_{m-l} < \infty} g_{\mathrm{count}}( \vec s) \boldsymbol u_{ \boldsymbol  v}(X_{s_1}) \cdots \boldsymbol u_{  \boldsymbol v}(X_{s_{m-l}}) \sum_{s = s_{m-l} +1}^{\infty} f(X_{s}) \bigg] \pi_{\boldsymbol v}^{\mathrm{inv}} (\dr y). 
\end{multline}

Finally, if the limiting vector $  {\boldsymbol v}$ satisfies $|  {\boldsymbol v}|< \epsilon_{\mathrm{thr}}$ where $\epsilon_{\mathrm{thr}}$ is chosen so that $C |  {\boldsymbol{v}}|^{2p} < 1$ in Proposition \ref{dcbd}, then the left-hand-side of \eqref{eq:geomBound} is absolutely summable in $m$ (even with a supremum over $N$). We can therefore use the dominated convergence theorem to compute \eqref{eq:quadVarSeries} by commuting the sum in $m$ with the limit as $N \to \infty$. This then reduces to summing the expression \eqref{eq:fixedmd3} over all $m\ge 0$. Letting $n = m-l$, this yields the limiting value of
\begin{multline*}
 \mathcal G^{(\boldsymbol{v})} (t, \phi;   a_1,   a_2) \cdot  \\
\int_I  \left[f(y) +  \left(e^{\boldsymbol u_{   v}(   y)} -1\right) \mathbf E_y^{(v,\mathrm{diff})} \left[\sum_{n=0}^{\infty} \sum_{1 \leq s_1\le ...\le s_{n} < \infty} g_{\mathrm{count}}( \vec s) \boldsymbol u_{   v}(X_{s_1}) \cdots \boldsymbol u_{   v}(X_{s_{n}}) \left(\sum_{s = s_{n} +1}^{\infty} f(X_{s})\right) \right] \right]\pi_v^{\mathrm{inv}} (\dr y). 
\end{multline*}
The second factor is precisely $\Theta(f,v)$ as defined in the statement of the proposition.
\\
\\
\textbf{Specialization to the specific case $f=\boldsymbol \vartheta_{   v}$.} Proposition \ref{4.1c} gives the limit of the quadratic variation field for a general function $f$. We now plug in $f(   y) = \boldsymbol{\vartheta}_{   v}(   y) = e^{\boldsymbol u_{   v}(   y)}  -1 = \sum_{k=1}^{\infty}\frac{1}{k!} \boldsymbol u_{   v}(X_s)^k.$  We can now further simplify the term involving $f$ as follows: 
\begin{align*}
&f(y) +  \left(e^{\boldsymbol u_{   v}(   y)} -1\right) \mathbf E_y^{(v,\mathrm{diff})} \bigg[\sum_{n=0}^{\infty} \sum_{1 \leq s_1\le ...\le s_{n} < \infty} g_{\mathrm{count}}( \vec s) \boldsymbol u_{   v}(X_{s_1}) \cdots \boldsymbol u_{   v}(X_{s_{n}}) \sum_{s = s_{n} +1}^{\infty} f(X_{s}) \bigg] \\
&= \left(e^{\boldsymbol u_{   v}(   y)} -1\right) \left[1 + \mathbf E_y^{(v,\mathrm{diff})} \bigg[\sum_{n=0}^{\infty} \sum_{1 \leq s_1\le ...\le s_{n} < \infty} g_{\mathrm{count}}( \vec s) \boldsymbol u_{   v}(X_{s_1}) \cdots \boldsymbol u_{   v}(X_{s_{n}}) \sum_{s = s_{n} +1}^{\infty} \left(\sum_{k=1}^{\infty}\frac{1}{k!} \boldsymbol u_{   v}(X_s)^k\right) \bigg] \right] \\
&=  \left(e^{\boldsymbol u_{   v}(   y)} -1\right)\left[1 + \mathbf E_y^{(v,\mathrm{diff})} \bigg[\sum_{n=0}^{\infty}\sum_{k=1}^{\infty} \sum_{1 \leq s_1\le ...\le s_{n} < s_{n+1}  = \cdots = s_{n+k}  < \infty} g_{\mathrm{count}}( \vec s) \boldsymbol u_{   v}(X_{s_1}) \cdots \boldsymbol u_{   v}(X_{s_{n+k}}) \bigg]\right]. \\
\end{align*}
In the last line, we absorbed the $\frac{1}{k!}$ into the function $g_{\mathrm{count}}( \vec s)$. Finally, note that we are now summing over all possible multi-indices $   s$ such that $s_1  \geq 1$. Therefore, this equals 
\begin{align*}
\left(e^{\boldsymbol u_{   v}(   y)} -1\right)\mathbf E_y^{(v,\mathrm{diff})} \bigg[e^{\sum_{s=1}^{\infty} \boldsymbol u_{   v}(X_s)} \bigg]. 
\end{align*}
Putting everything together, we arrive at 
\begin{equation*}
 \mathcal G^{(\boldsymbol{v})} (t, \phi;   a_1,   a_2) \cdot 
\int_I \left(e^{\boldsymbol u_{   v}(   y)} -1\right)\mathbf E_y^{(v,\mathrm{diff})} \bigg[e^{\sum_{s=1}^{\infty} \boldsymbol u_{   v}(X_s)} \bigg] \pi_v^{\mathrm{inv}} (\dr y). 
\end{equation*}
\end{proof}

  \begin{rk}[The important difference between $d=1$ and $d\ge 2$ when looking at the critical location scale] We explain here why we cannot use these techniques to study the $d=1$ case in the critical scale $|  \varsigma_N|\sim \psi_N(p,d)$. This is because the renewal trick of Theorem \ref{2.10} that was used to prove the limit theorem for $Q_N$ in $d \geq 2$ no longer works. This is due to the existence of a \textit{nontrivial} limit in distribution for the intersection processes in $d=1$ (the limit being the Brownian local time process), which does not happen for $d\ge 2$. As indicated by Theorem \ref{main1}, this is related to the fact that the fluctuations of the quenched density $P^\omega(t,x)$ are \textbf{not} Gaussian in $d=1$, for $x$ at the critical scale. Letting $Q^\infty$ denote a limit point of $Q_N$ as $N\to \infty$, and letting $\mathcal F_s$ denote the canonical filtration on $C([0,T],\mathcal S'(\mathbb R)),$ this boils down to the fact that at the critical scale $|  \varsigma_N|\sim \psi_N(p,d)$, the following are true: if $d\ge 2$, then from Propositions \ref{4.1b} and \ref{4.1c}, the quantity $\mathbb E[ Q^\infty(t,\phi) - Q^\infty(s,\phi)| \mathcal F_s]$ is a deterministic quantity; but if $d=1$ then $\mathbb E[ Q^\infty(t,\phi) - Q^\infty(s,\phi)| \mathcal F_s]$ is a genuinely random quantity (see \cite[Proposition 5.3]{Par24}).
\end{rk}
 
 Up until now, we have studied the quadratic variation field with initial conditions given by $\delta_{a_1^N} \otimes \delta_{a_2^N}.$ We now consider the quadratic variation field with initial conditions given by $\nu^{\otimes 2}_N.$ 
 \begin{thm}[The key estimate with general initial data in all dimensions]\label{gen.ic}
     Consider any (microscopic) sequence of measures $\nu_N$ on $I$ for which the associated (macroscopic) sequence of initial conditions $\mathfrak H^N(0,\bullet)$ as defined in \eqref{hn} is a good sequence in the sense of Definition \ref{goodseq}, converging weakly as measures on $\mathbb R^n$ to a measure $\mathfrak H_0$. Let $Q_N(t,\bullet; \nu_N^{\otimes 2})$ be as in \eqref{qfield}. Then, we have the limits:
     \begin{itemize}
        \item If $d=1$ and $N^{1/4p}   \varsigma_N\to   0$ then we have 
            $$\lim_{N \to \infty}\;  \Ex\big[ Q_N^f(t,\phi; \nu_N^{\otimes 2})\big] =\pi^{\mathrm{inv}}(f)\int_0^t \int_{\mathbb R} \big((G_s * \mathfrak H_0)(a)\big)^2 \phi(a)\dr a\dr s.$$
     
         \item If $d=2$ and $(\log N)^{1/2p}   \varsigma_N\to   {\boldsymbol v}$ with $|\boldsymbol v|$ sufficiently small, then we have 
            $$\lim_{N \to \infty}\;  \Ex\big[ Q_N^f(t,\phi; \nu_N^{\otimes 2})\big] =\frac{\pi^{\mathrm{inv}}(f)}{1- \frac12 \gamma_{\mathrm{ext}}(  {\boldsymbol v})^2}\cdot 2\pi \cdot \int_0^t \int_{\mathbb R^2} \big((G_s * \mathfrak H_0)(a)\big)^2 \phi(a)\dr a\dr s.$$
         \item if $d\ge 3$ and if $  \varsigma_N\to   {\boldsymbol v}$ with $|\boldsymbol v|$ sufficiently small, then 
         $$\lim_{N \to \infty}\;  \Ex\big[ Q_N^f(t,\phi; \nu_N^{\otimes 2})\big]= \Theta_{\mathrm{eff}}^2(f;   {\boldsymbol v}) \cdot c_d\cdot \int_0^t \int_{\mathbb R^d} \big((G_s^{(  {\boldsymbol v})} * \mathfrak H_0)(a)\big)^2 \phi(a)\dr a\dr s,$$ where $\Theta_{\mathrm{eff}}^2( f,  {\boldsymbol v})$ is as in Proposition \ref{4.1c}.
     \end{itemize}
 \end{thm}

 We remark that the integrals on the right side are in fact convergent due to $\mathfrak H_0^N$ being a good sequence as in Definition \ref{goodseq}, which is not obvious but will be shown as an artifact of the proof.

 \begin{proof}
     The main idea is to integrate all of the expressions from Propositions \ref{prop:QVd1} $-$ \ref{4.1c} over the initial conditions $(   a_1^N,   a_2^N),$ and use the assumption of goodness to obtain a vanishing bound on the part of the integral where $   a^N_1,   a^N_2$ are close to each other, using Proposition \ref{prop:DomConvergenceBound}. 
     
     Recall from \eqref{tilt0} that \begin{align}\notag \Ex\big[ Q_N^f(t,\phi; \nu_N^{\otimes 2})\big] & = \frac{\int_{I^2} e^{\varsigma_N \bullet(a_1+a_2)} \Ex\big[ Q_N^f(t,\phi;a_1,a_2)\big] \nu_N^{\otimes 2}(\dr a_1, \dr a_2)}{\int_{I^2} e^{\varsigma_N \bullet(a_1+a_2)} \nu_N^{\otimes 2}(\dr a_1, \dr a_2)} \\&= \int_{I^2} \Ex\big[ Q_N^f(t,\phi;a_1,a_2)\big] (\nu_N^{\varsigma_N})^{\otimes 2} (\dr a_1,\dr a_2), \notag
     \end{align}
     where $\nu^{\varsigma_N}_N(\dr x) :=  e^{  \varsigma_N\bullet x - \log \int_I e^{\varsigma_N\bullet a} \nu_N(\dr a)} \nu_N(\dr x)$ are the tilted initial states. We can reinterpret this as 
     a nested expectation given by \begin{equation}\label{though}\Ex \big[\Ex\big[ Q_N^f(t,\phi; \mathfrak a^N_1, \mathfrak a^N_2)| \mathfrak a_1^N,\mathfrak a_2^N \big]\big]
     \end{equation}where $\mathfrak a_i^N$ are independent $I$-valued random variables, each distributed as $\nu_N^{\varsigma_N}$ (and also independent of the kernels $K_i$). The \textit{distributional limit as $N\to\infty$} of the conditional expectation $\Ex [Q_N^f(t,\phi;\mathfrak a_1^N,\mathfrak a_2^N)|\mathfrak a_1^N,\mathfrak a_2^N]$ is given precisely by the expressions in 
     Propositions \ref{prop:QVd1} $-$ \ref{4.1c}, where $(a_1,a_2)$ there are replaced by $(\mathfrak a_1,\mathfrak a_2)$ which are distributed as $\mathfrak H_0^{\otimes 2}$. This is because $\mathfrak H_0^{\otimes 2}$ is the limiting law of $N^{-1/2} (\mathfrak a_1^N,\mathfrak a_2^N)$, simply by the assumptions in the statement of the theorem.
     
     We claim that the sequence $\Ex [Q_N^f(t,\phi;\mathfrak a_1^N,\mathfrak a_2^N)|\mathfrak a_1^N,\mathfrak a_2^N]$ is uniformly integrable, which will be justified further below. 
     Thus we can interchange the outer expectation in \eqref{though} 
     with the distributional limit at $N\to \infty$ as calculated by Propositions \ref{prop:QVd1} $-$ \ref{4.1c}. Thus, the limit as $N\to \infty$ of \eqref{though} would given by the $(f,v)$-dependent prefactor which depends on the dimension (e.g. $\pi^{\mathrm{inv}}(f)$, or $\frac{\pi^{\mathrm{inv}}(f)}{1-\gamma_{\mathrm{ext}}(v)^2}\cdot 2\pi$, or $\Theta_{\mathrm{eff}}(f;v)\cdot c_d$), in turn multiplied by $$\int_{\mathbb R^{2d}} \mathcal G(t,\phi, a_1,a_2) \mathfrak H_0^{\otimes 2}(\dr a_1,\dr a_2) =\int_0^t \int_{\mathbb R^d} \int_{\mathbb R^{2d}} G(2s,x) G(2s, a_1-a_2) \phi(\tfrac12 (x+a_1+a_2)) \mathfrak H_0^{\otimes 2}( \dr a_1,\dr a_2)  \dr x\dr s,$$ 
     where $G$ should be replaced by $G^{(v)}$ if $d\ge 3$ and $v\ne 0$ (i.e., \hyperref[eq:regimeD]{Regime D}). Substitute $x\to x-a_1-a_2$, and we obtain that the last integral can be rewritten as $$\int_0^t \int_{\mathbb R^d} \int_{\mathbb R^{2d}} G(2s,x-a_1-a_2) G(2s, a_1-a_2)  \mathfrak H_0^{\otimes 2}( \dr a_1,\dr a_2)  \cdot \phi(\tfrac12 x) \dr x\dr s.$$ 
     Making the change of variables $x \to \frac12 x$, we obtain that the last expression is the same as
     
     \begin{equation}\label{eq5}\int_0^t \int_{\mathbb R^d} \int_{\mathbb R^{2d}} G(\tfrac{s}{2},x-\tfrac12(a_1+a_2)) G(2s, a_1-a_2)  \mathfrak H_0^{\otimes 2}( \dr a_1,\dr a_2)  \cdot \phi(x) \dr x\dr s.
     \end{equation}
    Finally, use the identity $G(\frac{s}{2},x-\frac12(a_1+a_2)) G(2s, a_1-a_2) = G(s, x- a_1)G(s,x-a_2),$ and notice (for \hyperref[eq:regimeD]{Regime D}) that this identity also remains true for the $v$-tilted heat kernel, because $G^{(v)}(s,x) = G(s,H_v^{-1/2} x)$. Thus we obtain that the above integral equals
     $$\int_0^t \int_{\mathbb R^d} (G_{s} * \mathfrak H_0(x))^2 \cdot \phi(x) \dr x\dr s,$$ 
   and likewise for \hyperref[eq:regimeD]{Regime D} but with $G$ replaced by $G^{(v)}$. 
   It is not entirely obvious a priori that this integral converges, as one only has $G_s * \mathfrak H_0 (a)^2 \leq t^{2\e -2}$ which is not necessarily an integrable singularity for $\e\in (0,1/2)$. However, using \eqref{eq5}, 
   the convergence follows by first performing the integral over $x$ first, and noting that $\int_{\mathbb R^d} G(\frac{s}2, x-\frac12(a_1+a_2) )\phi(x)\dr x \leq C$ for some absolute constant $C$ not depending on $s$, then noting that $\int_{\mathbb R^{2d}}G(2s,a_1-a_2) \mathfrak H_0^{\otimes 2}(\dr a_1,\dr a_2)\leq Cs^{\e-1}$ thanks to the good sequence property from Definition \ref{goodseq}.
\\
\\   
     \textbf{Justification of the uniform integrability.} Definition \ref{goodseq} will be crucial here. 
     To first give some rough idea of how we need to apply that definition, notice first that if we have some fixed finite measure $\mathfrak H$ on $\mathbb R^d$ of Holder class $C^{2(\e-2)}$ and exponential tails, then for $d\ge 2$ we can form the convolution $\int_I |x-y|^{2-d} \mathfrak H(\dr y)$ by writing $|x|^{2-d} \le \phi(x) + \sum_{n=1}^\infty 2^{n(d-2)} \psi( 2^{n} x)$ for some smooth bounded $\phi\ge 0$ and compactly supported $\psi\ge 0$. Performing the sum shows that the convolution is (just barely) well defined: writing $\delta_n(x):= 2^{nd}\psi(2^nx)$ one sees that $\int|x|^{2-d}\mathfrak H(\dr x) \le \sum_n 2^{-2n} |(\delta_n,\mathfrak H)| \le \sum_n 2^{-2\e n}<\infty.$ To rigorously apply these ideas, we are going to break up the proof into cases depending on the dimension. 
\\
\\
     \textbf{$d=1$ case.} This is the easiest case. By Proposition \ref{prop:DomConvergenceBound}, one has $\sup_{a_1,a_2, N} \mathbb E [ Q^f_N(t,\phi,a_1,a_2)] < \infty$ for each fixed $t>0$, which immediately implies the required uniform integrability, in fact the expectation with the randomized initial conditions $\mathfrak a_N^1 ,\mathfrak a_N^2$ will just converge by the bounded convergence theorem.
\\
\\
     \textbf{ $d\ge 3$ case.} By Proposition \ref{prop:DomConvergenceBound}, the definition of $Q_N$, and $|\varsigma_N|<\epsilon_{\mathrm{thr}}$  we have that \begin{align*}
         \mathbb E[ Q_N^f(t,\phi;   a_1,   a_2)] & \le N^{\frac{d-2}2} \|\phi\|_{L^\infty} \sum_{s=0}^\infty \mathbf E^{(\varsigma_N,2)}_{(a_1,a_2)} [e^{\sum_{r=0}^{s-1} \boldsymbol u_{   v} (R^1_r-R^2_r)} f(R^1_s-R^2_s)] 
         \\& \le C N^{\frac{d-2}2} \|\phi\|_{L^\infty} \cdot (1+ |a_1-a_2|)^{2-d}.
     \end{align*}

     Now, we are going to prove uniform integrability by proving that there exists some $q>1$ such that $\sup_N \mathbb E \big[ \mathbb E[Q_N^f(t,\phi; \mathfrak a^N_1, \mathfrak a^N_2)|\mathfrak a_1^N,\mathfrak a_2^N]^{q}\big] <\infty.$ It suffices to show that for $q>1$ sufficiently small one has that 
     \begin{align*}\sup_N \int_{I^2} \big( N^{-1/2} (1+|x-y|)\big)^{q(2-d)} \nu_N^{\varsigma_N}(
     \dr x) \nu_N^{\varsigma_N}(\dr y) <\infty,
     \end{align*}
     where $\nu^{\varsigma_N}_N(\dr x) :=  e^{  \varsigma_N\bullet x - \log \int_I e^{\varsigma_N\bullet a} \nu_N(\dr a)} \nu_N(\dr x)$ are the tilted initial states. 
     Make the change of variable $(x,y)= N^{1/2}(a,b)$, and recall (from e.g.\eqref{hn}) that $\mathfrak H_0^N(N^{-1/2}\dr x) =\nu_N^{\varsigma_N}(\dr x)$. Thus we get that\begin{align}\notag \int_{I^2} \big( N^{-1/2} (1+|x-y|)\big)^{q(2-d)} \nu_N^{\varsigma_N}(
     \dr x) \nu_N^{\varsigma_N}(\dr y) &=\int_{I^2} \big( N^{-1/2} +|a-b| \big)^{q(2-d)} \mathfrak H_0^N (\dr a) \mathfrak H_0^N (\dr b) \\ \notag &= \int_I \bigg[ \int_I\big( N^{-1/2} +|a-b| \big)^{q(2-d)} \mathfrak H_0^N (\dr a) \bigg] \mathfrak H_0^N (\dr b) \\&\leq C\sup_{b} \int_I \big( N^{-1/2} +|a-b| \big)^{q(2-d)} \mathfrak H_0^N (\dr a), \label{su1}
     \end{align}
     where we used the fact that $\mathfrak H_0^N$ has uniformly bounded total mass thanks to the first bound in Definition \ref{goodseq}. Now as long as $q< 1 + \frac{2\e}{d-2} $ we will show that this last quantity is bounded independently of $N$ thanks to the second bound of Definition \ref{goodseq}. Choose some smooth bounded $\phi\ge 0$ and smooth compactly supported $\psi\ge 0$ with the property that $\big( N^{-1/2} +|x| \big)^{q(2-d)} \leq \phi(x) +\sum_{n=1}^{\lfloor \frac12 \log_2 N\rfloor} 2^{nq(d-2)}\psi(2^nx).$ Using the second bound in the good sequence condition, one sees that $\sup_b\int_I 2^{nd} \psi(2^n(a-b))\mathfrak H_0^N(\dr a) \leq 2^{-2n (\e-1)}$ for $n\le \frac12\log_2 N$. The uniform and at-worst-exponential tail decay of $\mathfrak H_0^N$ also implies that $(\phi,\mathfrak H_0^N) \leq C$. Hence \eqref{su1} is bounded above by $C\sum_{n\ge 0} 2^{nq(d-2) -nd -2n(\e-1)}$, which converges for $q< 1 + \frac{2\e}{d-2} $, with an upper bound independent of $N$, proving the uniform integrability in the case where $d\ge 3$. 
     \\
     \\
     \textbf{$d=2$ case.} We proceed very similarly to the $d\ge 3$ case above. By Proposition \ref{prop:DomConvergenceBound} and $|\varsigma_N| <\epsilon_{\mathrm{thr}} (\log N)^{-1/2p},$
     \begin{align*}
         \mathbb E[ Q_N^f(t,\phi;   a_1,   a_2)] & \le  \|\phi\|_{L^\infty} \sum_{s=0}^{Nt} \mathbf E^{(\varsigma_N,2)}_{(a_1,a_2)} [e^{\sum_{r=0}^{Nt} \boldsymbol u_{   v} (R^1_r-R^2_r)} f(R^1_s-R^2_s)] 
         \\& \le C  \|\phi\|_{L^\infty} \cdot (1+|\log (Nt / (1+|a_1-a_2|^2))|).
     \end{align*}
     Again it suffices to show that for some $q>1$ one has that 
     \begin{align*}\sup_N \int_{I^2} |\log ( N / (1+|x-y|^2) )|^q  \nu_N^{\varsigma_N}(
     \dr x) \nu_N^{\varsigma_N}(\dr y) <\infty.
     \end{align*}
     Since $d=2$ one can actually show that this is true for all $q>1$. Write $|\log x| = \log_+ x + \log_- x$. Dealing with $\log_+$ is very easy thanks to the uniform-in-$N$ exponential tail decay on $\mathfrak H_0^N$ coming from the first bound in Definition \ref{goodseq}. The negative part $\log_-$ is more subtle and requires the second bound of Definition \ref{goodseq}. Make change of variable $(x,y)\to N^{1/2}(x,y)$ and we get that the negative part of the above integral is bounded by \begin{align*}\int_{I^2} \log_- \big( N^{-1/2} +|a-b| \big)^{q} \mathfrak H_0^N (\dr a) \mathfrak H_0^N (\dr b) &= \int_I \bigg[ \int_I \log_- \big( N^{-1/2} +|a-b| \big)^{q} \mathfrak H_0^N (\dr a) \bigg] \mathfrak H_0^N (\dr b) \\&\leq C \sup_{b} \int_I \log_- \big( N^{-1/2} +|a-b| \big)^{q} \mathfrak H_0^N (\dr a) \\ &\leq C_{\e,q} \cdot \sup_{b} \int_I (N^{-1/2} +|a-b| \big)^{-\e} \mathfrak H_0^N (\dr a),
     \end{align*}
     where we used the fact that $\mathfrak H_0^N$ has uniformly bounded total mass and $|\log x|^q \leq C_{\e,q} x^{-\e}$ for $|x|<1$. Now this last quantity is finite thanks to the good sequence condition of Definition \ref{goodseq}, using a similar dyadic decomposition idea as in the $d\ge 3$ case, thus proving the uniform integrability in the case where $d=2$. 
 \end{proof}

 \subsection{Important moment estimates for the QVF}

 So far in this section, we proved \textit{limit formulas} that will be crucial to identify limit points later. However, we still need to prove \textit{upper bound estimates} that will be important to obtain the tightness (i.e., existence of limit points) in the first place. 

  \begin{prop}[Estimates for moments of the increments of QVF in all dimensions]\label{tight1} Fix $k\in\mathbb N$, $T>0$, and $F:[0,\infty)\to[0,\infty)$ decreasing and of exponential decay at infinity. Let $f(x):=F(|x|)$. 
If the sequence of initial data $\mathfrak H^N(0,\bullet)$ is a good sequence in the sense of Definition \ref{goodseq}, then we have that there exists $C = C(\varepsilon_{\mathrm{i.c.}},k,d,f)$ such that \begin{align}
			\label{e.tight7}
			\sup_{N\ge 1}\Ex\big[ (Q_N^f(t,\phi)-Q_N^f(s,\phi))^{k}\big]\leq C\|\phi\|^{k}_{L^\infty(\mathbb R)}\cdot (t-s)^{k\e}.
		\end{align}
        where $\e$ is as in Definition \ref{goodseq}.
        \end{prop}


        \begin{proof}
Using the definition of $Q_N$ together with the trivial bound $|\phi(N^{-1/2}(   R^j(Nt_j)-   d_N t_j))|\leq \|\phi\|_{L^\infty},$ we obtain that 
		\begin{align*}
			 \notag &\Ex[\big(Q_N^f(t,\phi)-Q_N^f(s,\phi)\big)^k]  \leq \|\phi\|_{L^\infty}^k \Ex[\big(Q_N^f(t,\ind_{\mathbb R})-Q_N^f(s,\ind_{\mathbb R} )\big)^k]\\&= N^{\frac{k(d-2)}2} k!\|\phi\|_{L^\infty}^k\sum_{Ns\leq r_1\leq ... \leq r_k\leq Nt} g_{\mathrm{count}}(\vec r) \mathbf E_{\nu_N^{\otimes 2k}}^{(   0,2k)}\bigg[\prod_{j=1}^k\bigg\{ f(R^{2j-1}_{r_j} -R^{2j}_{r_j}) \prod_{i=0,1} C_{N,N^{-1}r_j,N^{-1/2}(R^{2j-i}_{{r_j}}-  N^{-1} d_N r_j),\nu_N}\bigg\} \bigg]
             \\ & \le  N^{\frac{k(d-2)}2} k!\|\phi\|_{L^\infty}^k\sum_{Ns\leq r_1\leq ... \leq r_k\leq Nt} \mathbf E_{\nu_N^{\otimes 2k}}^{(   0,2k)}\bigg[\prod_{j=1}^k\bigg\{ f(R^{2j-1}_{r_j} -R^{2j}_{r_j}) \prod_{i=0,1} C_{N,N^{-1}r_j,N^{-1/2}(R^{2j-i}_{{r_j}}-  N^{-1} d_N r_j),\nu_N}\bigg\} \bigg].
		\end{align*}

Fix $d\ge 2$, then we claim that there exists some $K>d$ such that uniformly over $N\ge 1$ and  $t>s\ge 0$ and $\bfa \in I^{2k},$ one has the bound
\begin{align}\notag  
\mathbf E_{ \bfa }^{(0,2k)}  \bigg[\prod_{j=1}^k\bigg\{ f(R^{2j-1}_{r_j} -R^{2j}_{r_j}) \prod_{i=0,1} &C_{N,N^{-1}r_j,N^{-1/2}(R^{2j-i}_{{r_j}}- a_j - N^{-1}d_N r_j)}\bigg\} \bigg] \\&\le C \prod_{j=1}^{k} r_j^{-d/2}\big( 1+  r_j^{-1/2} |a_{2j-1} -a_{2j} |\big)^{-K}.\label{uniform1}
\end{align}
Here $\bfa = (a_1,...,a_{2k})$ can be taken as large as desired, $\e$ is as in Definition \ref{goodseq}, and $C$ may depend on $\varepsilon,k,d,f,K$ but not on $s,t,\bfa, N$.

For $d=1$, we claim that the bound still holds with time averaging, that is, with $\sum_{1\le r_1 \le ... \le r_k \le r}$ on both sides and ignoring the factor $\big( 1+  r_j^{-1/2} |a_{2j-1} -a_{2j} |\big)^{-K}$ on the right side, which is unnecessary.

The remainder of the proof will consist of two separate parts: establishing \eqref{uniform1} and then using it to prove the claim. First, we assume it is true and use that to establish the claim. 
\\
\\
\textbf{Proof of \eqref{e.tight7} assuming \eqref{uniform1}. } This is where Definition \ref{goodseq} will be crucial. 
Recall that $\nu^{\varsigma_N}_N(\dr x) :=  e^{  \varsigma_N\bullet x - \log \int_I e^{\varsigma_N\bullet a} \nu_N(\dr a)} \nu_N(\dr x)$ are the tilted initial states, and recall (from e.g. \eqref{hn}) that $\mathfrak H_0^N(N^{-1/2}\dr x) =\nu_N^{\varsigma_N}(\dr x)$. Use the second bound in the definition of good sequence (Definition \ref{goodseq}), taking $t:=u/N$, and we obtain that 
\begin{align*}(N/u)^{d/2} &\int_{I^{2}} \big(1 + u^{-1/2}|x_1-x_2|\big)^{-K} (\nu_N^{\varsigma_N})^{\otimes 2}(\dr x)\\&=  (N/u)^{d/2} \int_{I^{2}} \big(1 + (u/N)^{-1/2}|x_1-x_2|\big)^{-K} (\mathfrak H_0^N)^{\otimes 2}(\dr x) \\& \leq C (u/N)^{\e-1} .
\end{align*}
Applying this iteratively and using \eqref{tilt0}, we have that 
\begin{align*}N^{\frac{k(d-2)}2} &\sum_{Ns\leq r_1\leq ... \leq r_k\leq Nt} \mathbf E_{\nu_N^{\otimes 2k}}^{(   0,2k)}\bigg[\prod_{j=1}^k\bigg\{ f(R^{2j-1}_{r_j} -R^{2j}_{r_j}) \prod_{i=0,1} C_{N,N^{-1}r_j,N^{-1/2}(R^{2j-i}_{{r_j}}-  N^{-1} d_N r_j),\nu_N}\bigg\} \bigg]\\ &= N^{\frac{k(d-2)}2}\sum_{Ns\le r_1\le ... \le r_k \le Nt} (r_1\cdots r_k)^{-d/2} \int_{I^{2k}} \prod_{j=1}^{k} \big( 1+  r_j^{-1/2} |a_{2j-1} -a_{2j} |\big)^{-K} (\nu_N^{\varsigma_N})^{\otimes 2k}(\dr \bfa) \\& \le N^{-k} \sum_{Ns \le r_1 \le ... \le r_k \le Nt} \prod_{j=1}^k\bigg( \frac{r_j}{N} \bigg)^{\e-1} \\& \leq C_{k,\e} \cdot (t-s)^{k\e}.
\end{align*}
The last bound follows by e.g. interpreting it as a Riemann sum and iteratively integrating. This is the desired result. 
\\
\\
\textbf{Proof of \eqref{uniform1}. } 
Take $d\ge 2$. We prove this bound by induction on $k$. 
The claim is vacuously true when $k=0$, let us assume that the above bound has been proved up to $k-1$ where $k\in \mathbb N.$ By conditioning on time $r_1$, using the Markov property, and applying the inductive hypothesis for the conditioned expression, it suffices to show that 
\begin{align} \notag &\mathbf E_{ \bfa }^{(0,2k)}  \bigg[f(R^{1}_{r_1} -R^{2}_{r_1})\prod_{j=1}^{2k} C_{N,N^{-1} r_1, N^{-1/2} (R^j_{r_1} -a_j-N^{-1}d_N r_1)} \\&\times \prod_{j=2}^k (r_j-r_1)^{-d/2} \big( 1 + (r_j-r_1)^{-1/2} |R^{2j-1}_{r_1} - R^{2j}_{r_1}| \big)^{-K} \bigg] \le C \prod_{j=1} ^{k} r_j^{-d/2}\big( 1+  r_j^{-1/2} |a_{2j-1} -a_{2j}|\big)^{-K},\notag \end{align}
where in the left side everything beyond time $r_1$ has already been reduced to time $r_1$ using the inductive hypothesis. Cancelling out like terms which are deterministic, we just need to prove that 
\begin{align}\notag \mathbf E_{ \bfa }^{(   0,2k)} & \bigg[f(R^{1}_{r_1} -R^{2}_{r_1})\prod_{j=1}^{2k} C_{N,N^{-1} r_1, N^{-1/2} (R^j_{r_1} -a_j-N^{-1}d_N r_1)}   \prod_{j=2}^k \big( 1 + (r_j-r_{1})^{-1/2} |R^{2j-1}_{r_1} - R^{2j}_{r_1}| \big)^{-K} \bigg] \\& \le C \prod_{j=1}^k(r_j-r_1)^{d/2} r_j^{-d/2} \big( 1+  r_j^{-1/2} |a_{2j-1} -a_{2j}|\big)^{-K}.\label{uniform2}\end{align}
Now we will use the constants to rewrite the left side in terms of the tilted measures $\mathbf P^{(\beta,k)}_\x$ of Definition \ref{shfa}. Define $\mathcal H_N: I^{2k} \to \mathbb R_+$ by \begin{equation} \label{hfunction}\mathcal H_N(\mathbf R_{s-1}):= \mathcal K(\mathbf R_{s-1}, (\varsigma_N, \ldots \varsigma_N); 2k) =   \log \mathbf E_{\x}^{(0,2k)} \big[ e^{  \varsigma_N\bullet \sum_{j=1}^{2k} (   R^j_s -   R^j_{s-1})} \big| \mathcal F_{s-1}\big] \; - 2k \log M(   \varsigma_N) \end{equation}
where by the Markov property the equality is true for all $\x\in I^{2k}$. Here $\mathcal K$ is as defined in Definition \ref{kdef}. Then using \eqref{eq:tiltingExpectations}, the previous bound \eqref{uniform2} can be rewritten in terms of the tilted measures $\Pb$ of Definition \ref{shfa} as \begin{align}\notag \mathbf E_{ \bfa }^{(  \varsigma_N,2k)} & \bigg[f(R^{1}_{r_1} -R^{2}_{r_1})e^{\sum_{s=0}^{r_1-1} \mathcal H_N( \mathbf R_s)} \times \prod_{j=2}^k \big( 1 + (r_j-r_{1})^{-1/2} |R^{2j-1}_{r_1} - R^{2j}_{r_1}| \big)^{-K} \bigg] \\ &\le C \prod_{j=1}^k(r_j-r_1)^{d/2} r_j^{-d/2} \big( 1+  r_j^{-1/2} |a_{2j-1} -a_{2j}|\big)^{-K}.\label{uniform3}\end{align}
This is what we wish to prove. We begin by Taylor expanding (with respect to $  \varsigma_N$) the expression inside of the exponential and using the result of Lemma \ref{prev} one will obtain for each $t,N>0$ that $$\sup_{\x\in I^k} |\mathcal H_N(\x) |\leq C |  \varsigma_N|^{2p}\sum_{1\le i \le j \le k}  f_0(   R^i_r-   R^j_r)  $$ deterministically for some function $f_0:I\to \mathbb R$ decaying to 0 exponentially fast at infinity. 
Now use the very crude bound $\prod_{1\le i<j \le 2k} a_{ij} \le \sum_{1\le i <j\le 2k} a_{ij}^{k(2k+1)}$ (valid if $a_{ij}\ge 0$), we obtain 
\begin{align*}e^{C\sum_{1\le i<j\le 2k} \sum_{s=1}^{r}  |  \varsigma_N|^{2p} f_0(   R^i_s-   R^j_s)} &\leq \sum_{1\le i<j\le 2k} e^{Ck(2k+1) \sum_{s=1}^{r}  |  \varsigma_N|^{2p} f_0(   R^i_s-   R^j_s)} \\&\le \sum_{1\le i<j\le 2k} \sum_{m=0}^\infty C^m |  \varsigma_N|^{2pm} \sum_{s_1 \leq ... \leq s_m\le r} \prod_{\ell =1}^m f_0(   R^i_{s_\ell} -R^j_{s_\ell}). 
\end{align*}
We Taylor expanded the exponential into powers in the last line, and canceled the $m!$ coming from the ordering of coordinates. Thus to prove \eqref{uniform3}, we are now reduced to proving the bound for each $a<b$ 
\begin{align}\notag \sum_{s_1 \leq ... \leq s_m\le r_1} \mathbf E_{ \bfa }^{(  \varsigma_N,2k)} & \bigg[f(R^{1}_{r_1} -R^{2}_{r_1})\prod_{\ell =1}^m f_0(   R^a_{s_\ell} -R^b_{s_\ell}) \times \prod_{j=2}^k \big( 1 + (r_j-r_{1})^{-1/2} |R^{2j-1}_{r_1} - R^{2j}_{r_1}| \big)^{-K} \bigg] \\ &\le C^m  \prod_{j=1}^k(r_j-r_1)^{d/2} r_j^{-d/2} \big( 1+  r_j^{-1/2} |a_{2j-1} -a_{2j}|\big)^{-K}.\label{uniform4}\end{align}
This can be shown directly using the anticoncentration estimate in Theorem \ref{anti} combined with the Markov property of the measures $\mathbf P^{(\beta,k)}_\bfa$. The key here is to iteratively apply the elementary convolution bound $$r^{-d/2}s^{-d/2} \sum_{x\in \mathbb Z^d} \big( 1+ r^{-1/2} |x-y| \big)^{-K} \big( 1+ s^{-1/2} |x|\big)^{-K} \leq C (r+s)^{-d/2} \cdot \big( 1+ (r+s)^{-1/2} |y|\big)^{-K}, $$ with $C$ independent of $r,s,y$. To prove this bound, take e.g. $K=d+1$. The bound follows immediately from the fact that the radially symmetric Cauchy distribution on $\mathbb R^d$ is stable, i.e., \begin{multline*}\int_{\mathbb R^d} a^{-d} (1+a^{-1} |x|^2)^{-\frac{d+1}2} \cdot b^{-d} (1+ b^{-1} |x-y|^2)^{-\frac{d+1}2} \dr x =\frac{\pi^{d/2}\Gamma(\frac12)}{\Gamma(1+ \frac{d}2)}  (a+b)^{-d} (1+ (a+b)^{-1} |y|^2)^{-\frac{d+1}2}\end{multline*} for all $a,b>0$ and $y\in\mathbb R^d$.

For the $d=1$ case, one simply notes that at every step, the bounds are still true after time-averaging by the final bound of Theorem \ref{anti}. 
        \end{proof}

\begin{prop}[Controlling the error terms]\label{errbound} Let $\mathcal E^j_N, \mathcal V^j_N$ $(1\le j \le 2)$ be as in Proposition \ref{4.3}. Let $q\ge 1$ be an even integer, and let $T>0$. Then there exists $C>0$ and a function $f:I\to \mathbb R_+$ such that $|f(x)| \leq F(|x|)$ where $F:[0,\infty)\to[0,\infty)$ is decreasing and 
of exponential decay at infinity, such that uniformly over all $N\ge 1$ and $s,t\in [0,T]\cap (N^{-1}\mathbb Z_{\ge 0})$ we have the following:
    \begin{enumerate}\item \hyperref[eq:regimeA]{Regime A:} In the case where $|  \varsigma_N|=O(N^{-1/2})$ we have the bound 
    \begin{align*}
        \sum_{j=1,2}\mathbb E[ |\mathcal V^j_N(t,\phi)-\mathcal V^j_N(s,\phi)|^q]^{1/q} &\leq C N^{-1/2}\|\phi\|_{C^{p+2}(\mathbb R^d)}^2 \bigg(\mathbb E[ |Q^f_N(t,\ind_\mathbb R)-Q^f_N(s,\ind_\mathbb R)|^q]^{1/q}+|t-s|\bigg).
\end{align*}
The integer $p$ appearing on the right side is the one from Assumption \ref{a1}.
\item \hyperref[eq:regimeB]{Regime B:} In the case where $N^{-1/2} \ll  |  \varsigma_N| \ll  \psi_N(p,d)N^{-1}$ we have \begin{align*}\sum_{j=1,2}\mathbb E[ |\mathcal E^j_N(t,\phi)-\mathcal E^j_N(s,\phi)|^q]^{1/q}& \leq C |  \varsigma_N|^{1/2} \|\phi\|_{C^3(\mathbb R^d)}^2\bigg(\mathbb E[ |Q^f_N(t,\ind_\mathbb R)-Q^f_N(s,\ind_\mathbb R)|^q]^{1/q}+|t-s|\bigg).
    \end{align*}
    \item \hyperref[eq:regimeC]{Regime C:} In the case where $d=2$ and  $ (\log N)^{1/2p}   \varsigma_N\to   {\boldsymbol v}\ne 0 $, we have 
    \begin{align*}\sum_{j=1,2}\mathbb E[ |\mathcal E^j_N(t,\phi)-\mathcal E^j_N(s,\phi)|^q]^{1/q}& \leq C |  \varsigma_N|^{1/2} \|\phi\|_{C^3(\mathbb R^d)}^2\bigg(\mathbb E[ |Q^f_N(t,\ind_\mathbb R)-Q^f_N(s,\ind_\mathbb R)|^q]^{1/q}+|t-s|\bigg).
    \end{align*}
    \item \hyperref[eq:regimeD]{Regime D:} In the case where $d\ge 3$ and $  \varsigma_N  \to   {\boldsymbol v}\ne 0$ we have that $$\mathbb E[ |\mathcal E^1_N(t,\phi)-\mathcal E^1_N(s,\phi)|^q]^{1/q} \leq C|\varsigma_N-  {\boldsymbol v}| \|\phi\|_{C^3(\mathbb R^d)}^2\bigg(\mathbb E[ |Q^f_N(t,\ind_\mathbb R)-Q^f_N(s,\ind_\mathbb R)|^q]^{1/q}+|t-s|\bigg).$$
    \end{enumerate}
 Note that the second and third cases have the same error bounds; however, we distinguish these two cases because they will need to be analyzed separately later on. 
\end{prop}
We remark that these bounds show that the ``error terms" $\mathcal E^1,\mathcal E^2,\mathcal V^1,\mathcal V^2$ in Lemma \ref{4.3} vanish relative to $Q_N(t,\phi^2)$, which will be a crucial observation in showing tightness and identifying limit points in Section \ref{sec:6} just below. 

\begin{proof}
\textbf{Proof of (2) and (3). }Let us first deal with the $\mathcal E^j_N$ (the second bullet point) which corresponds to the extremal regime but below the critical scale. 
First consider $\mathcal E_N^2$. In the expression for $\mathcal E^2_N$ given by \eqref{err234}, use Lemma \ref{46} to obtain immediately that that $$|\mathcal E^2_N(t,\phi)-\mathcal E^2_N(s,\phi)| \leq C|  \varsigma_N|^{\frac12} \|\phi\|_{C^{p+2}(\mathbb R)}^2  |Q_N^f(t,\ind_{\mathbb R}) - Q_N^f(s,\ind_{\mathbb R})|, $$ where $f(x):= \Fd(x)^{1/(2p+2)}$ which decays exponentially fast at infinity by Assumption \ref{a1}. 

    Finally let us bound $\mathcal E^1_N$. To do this we will need to work with $\mathcal A^1_N(\phi,y_1,y_2)$ as defined just before \eqref{err1}. Let us perform a second-order Taylor expansion of $\phi$ centered at $N^{-1/2}y_1$. 
    Then the expression for $\mathcal A^1_N(\phi,y_1,y_2)$ can be written as a sum $\mathcal A^{1,a}_N(\phi,y_1,y_2)+\mathcal A^{1,b}_N(\phi,y_1,y_2)$ where
    \begin{align*}\mathcal A^{1,a}_N(\phi,y_1,y_2)&:= \sum_{\substack{0\leq \#k_1,\#k_2\leq 2\\\#k_1+\#k_2>0}}\frac1{k_1!k_2!}N^{-(k_1+k_2)/2} (\partial^{k_1}\phi)(N^{-1/2}y_1) (\partial^{k_2}\phi)(N^{-1/2} y_1)\\& \hspace{2 cm} \cdot \int_{I^2} \prod_{j=1}^2 e^{  \varsigma_N\bullet (x_j-y_j) - \log M(   \varsigma_N)} \cdot(x_1-y_1)^{k_1} (x_2-y_1)^{k_2}\rho\big( (y_1,y_2), (\mathrm \dr x_1,\mathrm \dr x_2) \big) 
    \end{align*} where the sum is over multi-indices $k_1,k_2$, where $\#j = j_1 +... +j_d$, and where $x^j = x_1^{j_1} \cdots x_d^{j_d}$ and $j! = j_1!\cdots j_d!$ for vectors $x=(x_1,...,x_d)$ and multi-indices $j = (j_1,...,j_d)$. By Taylor's theorem the other term $\mathcal A^{1,b}_N$ is a ``remainder" satisfying $$|\mathcal A^{1,b}_N(\phi,y_1,y_2)| \leq N^{-3/2}\|\phi\|_{C^3}^2  \int_{I^2} \prod_{j=1}^2 e^{|  \varsigma_N||x_j-y_j| - \log M(   \varsigma_N)} \cdot|x_1-y_1|^3 |x_2-y_2|^3 \big|\rho\big( (y_1,y_2), (\mathrm \dr x_1,\mathrm \dr x_2) \big)\big|.$$
    As in \eqref{err1} we can write $\mathcal E^1_N = \mathcal E^{1,a}_N+ \mathcal E^{1,b}_N$ corresponding respectively to the contributions of $\mathcal A^{1,a}_N$ and $\mathcal A^{1,b}_N$ respectively.
    
    The bound for $\mathcal E^{1,b}_N$ is straightforward: since the power of $N$ is $-3/2$, and since the number of summands in the expression for $\mathcal E^{1,b}_N$ is $N(t-s)$, we can use brutal absolute value bounds to obtain the desired estimate using arguments very similar to the bound for $\mathcal E_N^2$ above. By Lemma \ref{grow}, the integral is bounded independently of $y_1,y_2, N$.

    Let us discuss the bound for $\mathcal E^{1,a}_N$. In the expression for $\mathcal A^{1,a}_N$, each of the terms $(x_2-y_1)^{k_2}$ can be further expanded out into a sum of terms of the form $(x_2-y_2)^{i}(   y_1-   y_2)^{k_2-i}$ where $0\leq \#i \leq \#k_2.$ Ultimately $\mathcal A^{1,a}_N$ will be written as a sum of terms of the form \begin{align*}(   y_1-   y_2)^{k_2-i}N^{-(k_1+k_2)/2} &(\partial^{k_1}\phi)(N^{-1/2}y_1) \cdot (\partial^{k_2}\phi)(N^{-1/2} y_1)\\&\cdot \int_{I^2} \prod_{j=1}^2 e^{  \varsigma_N\bullet (x_j-y_j) - \log M(   \varsigma_N)} \cdot(x_1-y_1)^{k_1} (x_2-y_2)^{i}\rho\big( (y_1,y_2), (\mathrm \dr x_1,\mathrm \dr x_2) \big)\end{align*} where $0\leq k_1\le 2$ and $0\le i\le k_2\leq 2$. 
    Assuming that $\phi$ is supported on $[-J,J]$ we can then take the absolute value to bound the last expression by \begin{align*}|y_1-y_2|^{k_2-i}N^{-(k_1+k_2)/2} &\|\phi\|_{C^2}^2\ind_{[-J,J]}(N^{-1/2} y_1)\\&\cdot \bigg|\int_{I^2} \prod_{j=1}^2 e^{  \varsigma_N\bullet(x_j-y_j) - \log M(   \varsigma_N)} \cdot(x_1-y_1)^{k_1} (x_2-y_2)^{i}\rho\big( (y_1,y_2), (\mathrm \dr x_1,\mathrm \dr x_2) \big)\bigg|\end{align*}

    At this point, apply a Taylor expansion of the quantity $  \varsigma_N\mapsto \prod_{j=1}^2 e^{  \varsigma_N\bullet(x_j-y_j) - \log M(   \varsigma_N)}$, then use Lemma \ref{prev} and the decay assumptions on $\Fd$, and (using e.g. Lemma \ref{tbb} and the uniform moment bounds of Lemma \ref{grow}) we will see that the absolute value of the integral is bounded above by $\Fd(|y_1-y_2|)^{1/2} \cdot |  \varsigma_N|^{2p-\#k_2-\#i} . $ By considering all of the different cases of $(k_1,k_2,i)$, one verifies that the decay conditions on the function $\Fd$ in Item \eqref{a24} of Assumption \eqref{a1} are strong enough so that one may find $f$ as in the proposition statement to bound $\mathcal E^{1,a}_N$. 
\\
\\
    \textbf{Proof of (1).} The bounds for $\mathcal V^1,\mathcal V^2$ in Regime A are done in an extremely similar fashion, using the Taylor expansions of $\phi$ in the $y_i$-variable as well as the Taylor expansions of the exponential in the $  \varsigma_N$ variable to bound things correctly. One instead uses Lemma \ref{45} rather than \ref{46}, and one replaces $\phi$ by $\partial^k\phi$ with $\#k= p$ where appropriate, but the remaining details are extremely similar.
\\
\\
    \textbf{Proof of (4).} The final bullet point is also very similar to the proof of the second bullet point, noting here that $   v\mapsto \boldsymbol{\vartheta}_{   v}$ is a smooth function, more precisely since $  \varsigma_N\to   {\boldsymbol v}$ we have $|\boldsymbol{\vartheta}_{   \varsigma_N}-\boldsymbol{\vartheta}_{   v}| \leq |\varsigma_N-v| e^{C \sum_{j=1,2} |y_j-x_j|} $ where $C$ is independent of $N$. Then an application of Lemma \ref{grow} will allow to absorb the integral over $I^2$ into a constant and yield the desired bound as above while dealing with.
\end{proof}


\begin{cor}[Predictable quadratic variation bound]\label{predbound}
    Let $q$ be a positive even integer, and let $T>0$. Let $M^N(\phi)$ be the martingale defined in \eqref{mfield}, and let $\langle M^N(\phi)\rangle$ denote its predictable quadratic variation as in \eqref{quadvar}. Then there exists $C>0$ and exists $C>0$ and a function $f:I\to \mathbb R_+$ such that $|f(x)| \leq F(|x|)$ where $F:[0,\infty)\to[0,\infty)$ is decreasing and of exponential decay at infinity, such that we have the bound uniformly over all $N\ge 1$, all $s,t\in [0,T]\cap (N^{-1}\mathbb Z_{\ge 0}),$ and all $\phi\in C_c^\infty(\mathbb R)$ supported on $[-J,J]$
    \begin{align*}\mathbb E\big[ \big|\langle M^N(\phi)\rangle_{t}-\langle M^N(\phi)\rangle_s\big|^q\big]^{1/q}&\leq C  \|\phi\|_{C^{p+2}}^2 |\varsigma_N-v|^{1/2} \bigg( \mathbb E[ |Q^f_N(t,\ind_{\mathbb R})-Q^f_N(s,\ind_{\mathbb R})|^q]^{1/q}+ 
     |t-s|\bigg)
        \end{align*}
        In particular, if the sequence of initial data $\mathfrak H^N(0,\bullet)$ is a good sequence in the sense of Definition \ref{goodseq}, then we have that the above quantity is bounded above by $C  \|\phi\|_{C^{p+2}}^2 |t-s|^{\e},$ with $\e$ as in Theorem \ref{goodseq}.
\end{cor}

\begin{proof}
    This is immediate from Propositions \ref{4.3} and \ref{errbound}. The last sentence then follows from \eqref{e.tight7}. 
\end{proof}


    \begin{prop}[Martingale tightness bound] \label{optbound}
        Fix $q\ge 1$ and 
        $T>0$. For all $s,t\in (N^{-1}\mathbb Z_{\ge 0})\cap [0,T]$ and all $\phi\in C_c^\infty(\mathbb R)$ we have that 
         \begin{equation*}\mathbb E\big[ \big|M_t^N(\phi)-M_s^N(\phi)\big|^q\big]^{1/q}\leq C \|\phi\|_{C^{p+2}(
         \mathbb R^d)} |t-s|^{\e}.
        \end{equation*}
        Here $C$ is independent of $N,\phi,s,t$. Moreover $\e$ is as in Definition \ref{goodseq}.
    \end{prop}

    \begin{proof}       
    By \cite[Theorem 2.11]{Hall}, we have the Burkholder-type bound 
\begin{align*}
    \mathbb E\big[ \big|M_t^N(\phi)-M_s^N(\phi)\big|^q\big]^{1/q} & \leq C_p \mathbb E \big[ \big( \langle M^N(\phi)\rangle_t - \langle M^N(\phi)\rangle_s \big)^{q/2} ]^{1/q} + \mathbb E \big[ \sup_{u \in N^{-1} \mathbb Z \cap [s,t]} |M_{u+N^{-1}}(\phi) - M_u(\phi)|^q    \big]^{1/q}\\ &\leq C\|\phi\|_{C^{p+2}} \cdot |t-s|^{\e} + \mathbb E \big[ \sup_{u \in N^{-1} \mathbb Z \cap [s,t]} |M_{u+N^{-1}}(\phi) - M_u(\phi)|^q    \big]^{1/q}.
\end{align*}
Recall from \eqref{diff} that
\begin{align*} \notag (M^N_{t+N^{-1}}(\phi)- M^N_t(\phi))^2 = \mathscr B_N^2 \bigg( \int_I \int_I \phi(N^{-1/2}x) e^{  \varsigma_N\bullet (x-y) - \log M(   \varsigma_N)}\big[ K_{r+1}(y,\mathrm \dr x) -\mu(\mathrm \dr x-y)\big]Z_N(r,\mathrm dy)\bigg)^2
\end{align*}

   It suffices to prove the claim when $q=2k$ for some positive integer $k$. We have that 
   \begin{align*}
      \mathbb E \big[ \sup_{u \in N^{-1} \mathbb Z \cap [s,t]} |M_{u+N^{-1}}(\phi) - M_u(\phi)|^{2k}    \big]^{1/{2k}} & \leq \mathbb E \bigg[ \sum_{r = Ns}^{Nt}|M_{N^{-1}(r+1)}(\phi) - M_{N^{-1}r}(\phi)|^{2k}    \bigg]^{1/{2k}} \\
      &= \bigg( \sum_{r = Ns}^{Nt}\mathbb E \big[ |M_{N^{-1}(r+1)}(\phi) - M_{N^{-1}r}(\phi)|^{2k}    \big]\bigg)^{1/{2k}}
   \end{align*}

   For $r\in \mathbb Z_{\ge 0}$ let us now define the signed (random) measure $H_{r+1}(y,\dr x):= K_{r+1}(y,\mathrm \dr x) -\mu(\mathrm \dr x-y). $ Note by Assumption \ref{a1} Item \eqref{a22} that $\int_I (y-x)^j H_{r+1}(y,\dr x) =0$ a.s. for all multi-indices $j$ with $0\le \#j \le p-1$. 

   Recall from \eqref{diff} that
\begin{align*} \notag (M^N_{t+N^{-1}}(\phi)- M^N_t(\phi))^{2k}= \mathscr B_N^{2k} \bigg( \int_I \int_I \phi(N^{-1/2}x) e^{  \varsigma_N\bullet (x-y) - \log M(   \varsigma_N)}H_{r+1}(y,\dr x)Z_N(r,\mathrm dy)\bigg)^{2k}
\end{align*}

Then define the signed (random) measure on $I^2$ $$\Phi_{r+1}\big( (y_1,y_2) , (\dr x_1, \dr x_2)\big):= H_{r+1}(y_1,\dr x_1) H_{r+1} (y_2,\dr x_2) .$$
Let us therefore define a shorthand notation for the inner integral: \begin{equation}\label{ann}\mathpzc A_N(r,\phi,y_1,y_2) := \int_{I^2} \prod_{j=1}^2 \phi(N^{-1/2}x_j)e^{  \varsigma_N\bullet (x_j-y_j) - \log M(   \varsigma_N)} \Phi_{r+1} \big( (y_1,y_2) , (\dr x_1, \dr x_2)\big).\end{equation} 

To prove the desired bound, we will proceed similarly to previous propositions by first Taylor expanding $\prod_{j=1}^2 \phi(N^{-1/2}x_j)$ around the point $(y_1,y_2)$ so we get that $\mathpzc A_N(r,\phi,y_1,y_2)$ can be written as a sum $\mathpzc A^{a}_N(r,\phi,y_1,y_2)+\mathpzc A^{b}_N(r,\phi,y_1,y_2)$ where
    \begin{align*}\mathpzc A^{a}_N(r,\phi,y_1,y_2)&:= \sum_{\substack{0\leq \#k_1,\#k_2\leq kd\\\#k_1+\#k_2>0}}\frac1{k_1!k_2!}N^{-(k_1+k_2)/2} (\partial^{k_1}\phi)(N^{-1/2}y_1) (\partial^{k_2}\phi)(N^{-1/2} y_1)\\& \hspace{2 cm} \cdot \int_{I^2} \prod_{j=1}^2 e^{  \varsigma_N\bullet (x_j-y_j) - \log M(   \varsigma_N)} \cdot(x_1-y_1)^{k_1} (x_2-y_1)^{k_2}\Phi_{r+1}\big( (y_1,y_2), (\mathrm \dr x_1,\mathrm \dr x_2) \big) 
    \end{align*} where the sum is over multi-indices $k_1,k_2$, where $\#j = j_1 +... +j_d$, and where $x^j = x_1^{j_1} \cdots x_d^{j_d}$  and $j! = j_1!\cdots j_d!$ for vectors $x=(x_1,...,x_d)$ and multi-indices $j = (j_1,...,j_d)$. By Taylor's theorem the other term $\mathpzc A^{b}_N$ is a ``remainder" satisfying $$|\mathpzc A^{b}_N(r,\phi,y_1,y_2)| \leq N^{-\frac{kd+1}2}\|\phi\|_{C^3}^2  \int_{I^2} \prod_{j=1}^2 e^{|  \varsigma_N||x_j-y_j| - \log M(   \varsigma_N)} \cdot|x_1-y_1|^{kd+1} |x_2-y_2|^{kd+1} \big|\Phi_{r+1}\big( (y_1,y_2), (\mathrm \dr x_1,\mathrm \dr x_2) \big)\big|.$$

Thus we obtain that for $r:=Nt$ one has 
\begin{align*}\mathbb E[ (M^N_{t+N^{-1}}(\phi) - M_t^N(\phi))^{2k} ] & = \mathbb E\bigg[\bigg( \mathscr B_N^2 \int_{I^2} \mathpzc A_N( r,\phi, y_1,y_2) \prod_{j=1,2} Z_N(r, \dr y_j) \bigg)^k \bigg]
\\&= \mathscr B_N^{2k} \int_{I^{2k}} \mathbb E [ \prod_{j=0}^{k-1} \mathpzc A_N(r,\phi, y_{2j+1}, y_{2j+2})] \; \mathbb E [\prod_{j=1}^{2k} Z_N(r, \dr y_j) ]
\end{align*}
where in the last line we used the independence of $\mathpzc A_N(r,\bullet)$ and $Z_N(r,\bullet)$ to split the expectations. 

Now we have explicitly that 
\begin{align*}\mathbb E [\prod_{j=1}^{2k} Z_N(r, \dr y_j) ] &= \mathbf E^{(\varsigma_N,2k)}_{ (\nu_N^{\varsigma_N})^{\otimes 4k}}[ e^{\sum_{s=0}^{r-1} \mathcal H_N(\mathbf R_s)} \ind_{\{\mathbf R_r \in \dr \y\}} ],
\end{align*}
where $\mathcal H_N$ is exactly as in \eqref{hfunction}. Meanwhile we have a bound coming from Assumption \ref{a1} which says that : 
$$\bigg\| \mathbb E \bigg[ \bigotimes_{j=1}^k \Phi_{r+1} ( y_{2j-1},y_{2j}, \bullet) \bigg]\bigg\|_{TV} \leq C \Fd (\max_{i<j} |y_i-y_j|).$$
\begin{align*}|\mathbb E & [ \prod_{j=0}^{k-1} \mathpzc A_N(r,\phi, y_{2j+1}, y_{2j+2}) | \le C |\varsigma_N|^{2kp} \Fd ( \max_{i<j} |y_i-y_j|)^{1/2} .
\end{align*}
Note here the max as opposed to the min which usually appears, which appears by considering all possible partitions in Assumption \ref{a1} Item \eqref{a24} and taking the mimimum. Noting that $\mathscr B_N^{2k} \cdot |\sigma_N|^{2kp} \leq C N^{\frac12 k(d-2)},$ we thus see that 
\begin{align*}\mathscr B_N^{2k} &\int_{I^{2k}} \mathbb E [ \prod_{j=0}^{k-1} \mathpzc A_N^a(r,\phi, y_{2j+1}, y_{2j+2})] \; \mathbb E [\prod_{j=1}^{2k} Z_N(r, \dr y_j) ]\\ &\leq C N^{\frac12 k(d-2)} \mathbf E^{(\varsigma_N,2k)}_{ (\nu_N^{\varsigma_N})^{\otimes 2k}}[ e^{\sum_{s=0}^{r-1} \mathcal H_N(\mathbf R_s)} \Fd (\max_{i<j}|R^i_r-R^j_r|)^{1/2} ].
\end{align*}
These are the worst terms, as products with ``$b$" indices will be similar but have extra factors of $N^{-\frac{kd+1}2}$ which may be bounded brutally, enough to counteract the growth of $\mathscr B_N^{2k}$. 
As $\Fd$ has exponential decay at infinity, one has the bound which can be derived from Theorem \ref{anti}: 
$$\mathbf E^{(\varsigma_N,2k)}_\bfa [e^{\sum_{s=0}^{r-1} \mathcal H_N(\mathbf R_s)} \Fd (\max_{i<j}|R^i_r-R^j_r|)^{1/2} ] \leq C  r^{-(2k-1)d/2}  \big( 1+r^{-1/2} \max_{i<j} |a_i-a_j|))^{-K}. $$
Integrating that over the initial conditions, fibering $I^{2k}$ over $I^{2k-1}$ by the subspace $\{(a,...,a):a\in I\}$, and using the second bound in Definition \ref{goodseq} we find that 
$$(r/N)^{-(2k-1)d/2} \int_{I^{2k}} \big( 1+r^{-1/2} \max_{i<j} |a_i-a_j|))^{-K} (\nu_N^{\varsigma_N})^{\otimes 2k} (\dr \bfa) \leq C(r/N)^{(2k-1)(\e-1)}.$$
Summarizing, we find that 
\begin{align*}
    \sum_{r=Ns}^{Nt-1} \mathbb E[ (M^N_{t+N^{-1}}(\phi) - M_t^N(\phi))^{2k} ]&\le N^{\frac12 k(d-2) -(2k-1) (\frac{d}2 - 1+\e)} \sum_{r=Ns}^{Nt} r^{(2k-1) (\e-1)} 
\end{align*}
Note that for $k=1$, this bound does recover the one for the predictable variation from earlier. Without loss of generality, $k$ is so large that $(2k-1)(1-\e)>1.$ Thus the whole sum is bounded above by $N^{(k-1)(1-\frac{d}2-2\e) -\e } $ which is upper bounded by $|t-s|^{(k-1)(\frac{d}2+2\e-1)}$ as long as $t\ne s$ (since $|t-s|\ge N^{-1})$. Taking the $2k$th root gives an upper bound of $|t-s|^{(\frac12-\frac1{2k}) (\frac{d}2 +2\e -1)},$ which is enough (recalling that $\e=\frac12$ in $d=1$). 
    \end{proof}

\section{Tightness and identification of limit points} \label{sec:6}

In this section we will finally prove our main results (Theorems \ref{main1}, \ref{main2}, \ref{main3}, \ref{main4}, and \ref{main5}) using the estimates of the previous section. The main idea is as follows. We want to study the field $\mathfrak F_N$ appearing in the main results. In Section \ref{sec:4}, we derived a relation which approximately says ``$(\partial_t - \frac12 \Delta) \mathfrak F_N(t,x) = M^N(t,x)$" for some martingale field $M^N.$ In Section \ref{sec:5} we proved very strong bounds and limit formulas for the quadratic variations $Q_N$ of the martingales $M^N$. Those estimates will be used in this section to show that as $N\to \infty$, this random process $Q_N$ converges to a deterministic limit. Thus as $N\to\infty$, by the Dubins-Schwarz Theorem, the $M^N$ must converge to some Gaussian limit with explicit covariance since their quadratic variation is deterministic. Since $\mathfrak F_N$ and $M^N$ are related by a linear map, this will imply that $\mathfrak F_N$ is also converging to a Gaussian limit, and the covariance operator will be clear since the limit of the quadratic variations $Q_N$ in every regime is explicit from Section \ref{sec:5}. This method for studying the density profile in RWRE models was introduced by us in our previous works \cite{DDP23, DDP+, Par24} and was also used recently by \cite{Kot} to study the bulk regime of a Kraichnan-type model (Example \ref{diffrm}). An important goal of this paper is to illustrate that the same approach works in every regime up to the critical scale.

\subsection{Weighted H\"older spaces and Schauder estimates} 

    We now introduce various natural topologies for our field $\mathfrak{F}_N$ and its limit points. We then discuss how the heat flow affects these topologies and record Kolmogorov-type lemmas that will be key in showing tightness under these topologies. We begin by recalling many familiar and useful spaces of continuous and differentiable functions that have natural metric structures.

For $d\ge 1$, we denote by $C_c^{\infty}(\R^d)$ the space of all compactly supported smooth functions on $\R^d$. For a smooth function on $\R^d$, we define its $C^r$ norm as
\begin{align*}
    \|f\|_{C^r}:=\sum_{  \#k \le r} \sup_{\mathbf x\in \R^d} |D^kf(\mathbf x)|
\end{align*}
 where the sum is over all multi-indices $  k=(k_1,...,k_d)\in \mathbb{Z}_{\ge 0}^d$ with $\#k=\sum_{i=1}^d k_i\le r$ and $D^{  {k}}:=\partial_{x_1}^{k_1}\cdots \partial_{x_d}^{k_d}$ denotes the mixed partial derivative.

\smallskip

We now recall the definition of weighted H\"older spaces from \cite[Definitions 2.2 and 2.3]{HL16}.  For the remainder of this paper, we shall work with \textit{elliptic and parabolic} weighted H\"older spaces with polynomial weight function
 \begin{align*}
     w(x):=(1+x^2)^\tau
 \end{align*}
  for some fixed $\tau>1$. We introduce these weights because weighted spaces will be more convenient to obtain tightness estimates. We expect that it is possible to remove the weights throughout this section, but this would require more precise moment estimates than the ones we derived in previous sections, which take into account spatial decay of the fields.

	\begin{defn}[Elliptic H\"older spaces]\label{ehs}
		For $\alpha\in(0,1)$ we define the space $C^{\alpha,\tau}(\mathbb R^d)$ to be the completion of $C_c^\infty(\mathbb R)$ with respect to the norm given by $$\|f\|_{C^{\alpha,\tau}(\mathbb R^d)}:= \sup_{x\in\mathbb R^d} \frac{|f(x)|}{w(x)} + \sup_{|x-y|\leq 1} \frac{|f(x)-f(   y)|}{w(x)|x-y|^{\alpha}}.$$
		For $\alpha<0$ we let $r=-\lfloor \alpha\rfloor$ and we define $C^{\alpha,\tau}(\mathbb R^d)$ to be the completion of $C_c^\infty(\mathbb R)$ with respect to the norm $$\|f\|_{C^{\alpha,\tau}(\mathbb R^d)}:= \sup_{x\in\mathbb R} \sup_{\lambda\in (0,1]} \sup_{\phi \in B_r} \frac{(f,S^\lambda_{x}\phi)_{L^2(\mathbb R)}}{w(x)\lambda^\alpha}$$ where the scaling operators $S^\lambda_{x}$ are defined by 
  \begin{align}\label{escale}
  S^\lambda_{x}\phi (y) = \lambda^{-d}\phi(\lambda^{-1}(x-y)),\end{align} 
  and where $B_r$ is the set of all smooth functions of {$C^r$ norm} less than 1 with support contained in the unit ball of $\mathbb R^d$.
	\end{defn}

 One may verify that these spaces embed continuously into $\mathcal S'(\mathbb R^d)$.

 \begin{defn}[Function spaces] \label{fsp} Let $C^{\alpha,\tau}(\mathbb R^d)$ be as in Definition \ref{ehs}. We define $C([0,T],C^{\alpha,\tau}(\mathbb R^d))$ to be the space of continuous maps $H:[0,T]\to C^{\alpha,\tau}(\mathbb R^d),$ equipped with the norm $$\|g\|_{C([0,T],C^{\alpha,\tau}(\mathbb R^d))} := \sup_{t\in[0,T]} \|g(t)\|_{C^{\alpha,\tau}(\mathbb R^d)}.$$ 
\end{defn}

	\begin{defn}Here and henceforth we will define $\Psi_{[a,b]}:=[a,b]\times\mathbb R^d$ and $\Psi_T:=\Psi_{[0,T]}.$ 
    \end{defn}

\smallskip
 
	\begin{defn}[Parabolic H\"older spaces]
		We define $C_c^\infty(\Psi_T)$ to be the set of functions on $\Psi_T$ that are restrictions to $\Psi_T$ of some function in $C_c^\infty(\mathbb R^{d+1})$, in particular we do not impose that elements of $C_c^\infty(\Psi_T)$ vanish at the boundaries of $\Psi_T.$ 
  
    For $\alpha\in(0,1)$ we define the space $C^{\alpha,\tau}_\mathfrak s(\Psi_T)$ to be the completion of $C_c^\infty(\Psi_T)$ with respect to the norm $$\|f\|_{C^{\alpha,\tau}_\mathfrak s(\Psi_T)}:= \sup_{(t,x)\in \Psi_T} \frac{|f(t,x)|}{w(x)} + \sup_{|s-t|^{1/2}+|x-y|\leq 1} \frac{|f(t,x)-f(s,y)|}{w(x)(|t-s|^{1/2}+|x-y|)^{\alpha}}.$$
		For $\alpha<0$ we let $r=-\lfloor \alpha\rfloor$ and we define $C^{\alpha,\tau}_\mathfrak s(\Psi_T)$ to be the completion of $C_c^\infty(\Psi_T)$ with respect to the norm $$\|f\|_{C^{\alpha,\tau}_\mathfrak s(\Psi_T)}:= \sup_{(t,x)\in\Psi_T} \sup_{\lambda\in (0,1]} \sup_{\varphi \in B_r} \frac{(f,S^\lambda_{(t,x)}\varphi)_{L^2(\Psi_T)}}{w(x)\lambda^\alpha}$$ where the scaling operators are defined by 
		\begin{align}
			\label{scale}
			S^\lambda_{(t,x)}\varphi (s,y) = \lambda^{-d-2}\varphi(\lambda^{-2}(t-s),\lambda^{-1}(x-y)),
		\end{align}
		and where $B_r$ is the set of all smooth functions of $C^r$ norm less than 1 with support contained in the unit ball of $\mathbb R^{d+1}$.
	\end{defn}

    So far we have used $\phi,\psi$ for test functions on $\R^d$. To make the distinction clear between test functions on $\R^d$ and $\R^{d+1}$, we shall use variant Greek letters such as $\varphi, \vartheta, \varrho$ for test functions on $\R^{d+1}$. In many instances below, we will explicitly write $(f,\varphi)_{\mathrm{para}}$ or $(g,\phi)_{\mathrm{ell}}$ when we want to be clear about the space in which we are applying the $L^2$-pairing (parabolic vs. elliptic).

An important property of both the parabolic and elliptic spaces is that one has a continuous embedding $C^{\alpha,\tau} \hookrightarrow C^{\beta,\tau}$ whenever $\beta<\alpha.$ In fact this embedding is compact, though we will not use this. We also have the following embedding of function spaces inside parabolic spaces.

 \begin{lem}\label{embed}
     For $\alpha<0,\tau>0$ one has a continuous embedding $C([0,T],C^{\alpha,\tau}(\mathbb R^d))\hookrightarrow C^{\alpha,\tau}_\mathfrak s(\Psi_T)$ given by identifying $v=(v(t))_{t\in [0,T]}$ with the tempered space-time distribution given by $$(v,\varphi)_{\mathrm{para}} = \int_0^T (v(t) ,\varphi(t,\cdot))_{\mathrm{ell}}dt.$$ 
 \end{lem}

 The proof is straightforward from the definitions and is omitted.
 
\begin{rk}[Derivatives of distributions] \label{d/dx}  Let $\alpha<0$. By definition, any element $f\in C_\mathfrak s^{\alpha,\tau}(\Psi_T)$ admits an $L^2$-pairing with any smooth function $\varphi: \Psi_T\to \mathbb R$ of rapid decay (i.e., $\varphi$ is a Schwarz function restricted to $\Psi_T$). We will write this pairing as $(f,\varphi)_{\Psi_T}.$ Consequently there is a natural embedding $C_\mathfrak s^{\alpha,\tau}(\Psi_T)\hookrightarrow \mathcal S'(\mathbb R^{d+1})$ which is defined by formally setting $f$ to be zero outside of $[0,T]\times \mathbb R^d$. More rigorously, this means that the $L^2(\mathbb R^{d+1})$-pairing of $f$ with any $\varphi \in \mathcal S(\mathbb R^{d+1})$ is defined to be equal to $(f,\varphi|_{\Psi_T})_{\Psi_T}$.

 \smallskip
 
 The image of this embedding consists of some specific collection of tempered distributions that are necessarily supported on $[0,T]\times \mathbb R.$ Thus we can sensibly define $\partial_tf$ and $\partial_{x_j}f$ as elements of $\mathcal S'(\mathbb R^{d+1})$ whenever $f\in C_\mathfrak s^{\alpha,\tau}(\Psi_T)$, by the formulas$$(\partial_tf,\varphi)_{\Psi_T} := -(f,\partial_t\varphi)_{\Psi_T},\;\;\;\;\;\;\;\;\;\;(\partial_{x_j}f,\varphi)_{\Psi_T} := -(f,\partial_{x_j}\varphi)_{\Psi_T}$$for all smooth $\varphi:\Psi_T\to\mathbb R$ of rapid decay. This convention on derivatives will be useful for certain computations later. From the definitions, when $\alpha<0$ one can check that for such $f$ one necessarily has $\partial_tf \in C_\mathfrak s^{\alpha-2,\tau}(\Psi_T)$ and $\partial_x f \in C_\mathfrak s^{\alpha-1,\tau}(\Psi_T).$

\smallskip

The latter statement fails for $\alpha>0$. Indeed by our convention of derivatives, $\partial_tf$ may no longer be a smooth function (or even a function) even if $f\in C_c^\infty(\Psi_T)$. This is because such an $f$ gets extended to all of $\mathbb R^{d+1}$ by setting it to be zero outside $\Psi_T$. In particular, if $f$ does not vanish on the boundary of $\Psi_T$, then it may become a discontinuous function under our convention of extension to $\mathbb R^{d+1}$. Due to these discontinuities, the distributional derivative $\partial_tf$ can be a tempered distribution with singular parts along the boundaries. 
In our later computations, we never take derivatives of functions in $C_\mathfrak s^{\alpha,\tau}(\Psi_T)$ with $\alpha>0$. Sometimes we will write $\partial_s$ to mean the same thing as $\partial_t$.
\end{rk}

 We end this subsection by recording a Kolmogorov-type lemma for the three spaces introduced at the beginning of this subsection. It will be crucial in proving tightness in those respective spaces.
	
	\begin{lem}[Kolmogorov continuity lemma for distributions] \label{l:KC} 


		Let $(t,\phi) \mapsto V(t,\phi)$ be a map from $[0,T]\times \mathcal S(\mathbb R^d)$ into $L^2(\Omega,\mathcal F,\mathbb P)$ which is linear and continuous in $\phi$. Fix a non-negative integer $r$. Assume there exists some $\kappa>0, q>d/\kappa$ and $\alpha<-r$ and $C=C(\kappa,\alpha,q,{T})>0$ such that 
		\begin{align*}\mE[|V(t,S^\lambda_{x}\phi)|^q]^{1/q}&\leq C\lambda^{\alpha },\\  \mE[|V(t,S^\lambda_{x}\phi)-V(s,S^\lambda_{x}\phi)|^q]^{1/q}&\leq C\lambda^{\alpha -\kappa} |t-s|^{\kappa },
		\end{align*}uniformly over all smooth functions $\phi$ on $\mathbb R$ supported on the unit ball of $\mathbb R$ with {$\|\phi\|_{C^r}\leq 1$}, and uniformly over $\lambda\in(0,1]$ and $0\leq s,t\leq T$. Then for any $\tau>1$ and any $\beta<\alpha-\kappa$ there exists a random variable $\big(\mathscr V(t)\big)_{t\in[0,T]}$ taking values in $C([0,T],C^{\beta,\tau}(\mathbb R^d))$ such that $(\mathscr V(t),\phi)=V(t,\phi)$ almost surely for all $\phi$ and $t$. Furthermore, one has that $$\mE[\|\mathscr V\|^q_{C([0,T],C^{\beta,\tau}(\mathbb R^d))}]\leq C',$$ where $C'$ depends on the choice of $\alpha,\beta,q,\kappa,$ and the constant $C$ appearing in the moment bound above but not on $V,\Omega,\mathcal F,\mathbb P$.
 


	\end{lem}
	{A proof may be adapted from the proof of Lemma 9 in Section 5 of \cite{WM}.} 

 \subsection{Tightness}

 Throughout this section, we are going to fix a terminal time $T>0$. Let $\mu$ and $z_0$ be as in Assumption \ref{a1} Item \eqref{a22}. Recall that we denote $$\mu^{\varsigma}(\dr x):= e^{\varsigma x - \log M(\varsigma)} \mu(\dr x),$$ where $\mu$ as always is the annealed one-step measure of Assumption \ref{a1}. We claim that for any $f\in \mathcal S'(\mathbb R)$ and $|\varsigma| <z_0, $ the spatial convolution $g:= f*\mu^\varsigma$ given by 
 \begin{equation}
     \label{eq:convPhi} (g, \phi):= (f, \phi * \mu^\varsigma),\;\;\;\;\;\; \phi \in \mathcal S(\mathbb R^d),
 \end{equation}is well-defined as another element of $\mathcal S'(\mathbb R^d)$. To prove this, we must show that $\phi\mapsto \phi * \mu^\varsigma$ is a well-defined and continuous map from $\mathcal S(\mathbb R)$ into itself. By taking the Fourier transform (which is a linear isomorphism on $\mathcal S(\mathbb R^d)$), it suffices to show that the multiplication operator $T_\varsigma \psi(\xi) = \widehat{\mu^\varsigma}(\xi) \psi(\xi)$ is a continuous linear operator from $\mathcal S(\mathbb R^d)\to \mathcal S(\mathbb R^d)$. But this is clear from the fact that the Fourier transform $\widehat{\mu^\varsigma}$ (i.e., the ``characteristic function" of the measure $\mu^\varsigma$) is smooth with all derivatives bounded, since the measure $\mu^\varsigma$ has exponentially decaying tails. 

 Likewise if $f\in \mathcal S'(\mathbb R^{d+1})$ and one would like to define a tempered distribution $g\in \mathcal S'(\mathbb R^2)$ such that formally one has $g(t,x) = \int_{\mathbb R^d} f(t,x-y) \mu^\lambda(\dr y),$ the procedure is completely analogous, defining it via the Fourier transforms $\hat g(\xi_1,\xi_2):=  \widehat{\mu^\lambda}(\xi_2) \hat f(\xi_1,\xi_2).$
    
    \begin{defn}Define $$\Psi_{N,T}:= \{ (t,x)\in [0,T]\times \mathbb R^d : (Nt,N^{1/2}x) \in \mathbb Z\times I\}.$$ 
    \end{defn}

    Also recall the discrete lattice $$\Lambda_N:=\{(t,   x)\in \mathbb Z_{\ge 0}\times\mathbb R^d :   x+tN^{-1}   d_N\in \mathbb Z_{\ge 0}\times I\}$$ that was used in Section \ref{sec:3}.

    \begin{defn}\label{dnlnkn} For $(s,y) \in \mathbb Z\times I$ define $p_N(s,dy)$ to be the transition density at time $s$ and position $y$ of a random walker on $\Lambda_N$ with increment distribution $\mu^{   \varsigma_N}.$ Define the (macroscopic) linear operators $D_N,L_N,K_N 
    $ on $\mathcal S'(\mathbb R^{d+1})$ 
    by
    \begin{align*} D_Nf(t,x) &= N \big[ f(t+N^{-1},x)-f(t,x)\big],\\
        L_N f(t,x) &= N\bigg[f(t+N^{-1},x-N^{-1}   d_N) -\int_I f(t,x-N^{-1/2}y) \mu^{   \varsigma_N} (\dr y)\bigg],\\
        K_N f(t,x) &= N^{-1}\sum_{s\in [0,T]\cap (N^{-1}\mathbb Z_{\ge 0})} \int_{I} f(t-s,x-N^{-1/2}y)p_N(Ns,\dr y).
    \end{align*}
    These equalities should be understood by integration against smooth functions $\varphi \in \mathcal S(\mathbb R^{d+1})$, and the convolutions need to be understood using the explanation after \eqref{eq:convPhi}.
    \end{defn}

  $L_N$ is a diffusively rescaled version of the microscopic heat operator $\mathcal L_N$ from Section \ref{sec:3}, but which acts on tempered distributions rather than measures on $\mathbb Z\times I$. Indeed if $\varphi \in \mathcal S(\mathbb R^{d+1})$ then a second-order Taylor expansion shows that $L_N\varphi$ converges in $\mathcal S(\mathbb R^{d+1})$ as $N\to\infty$ to $(\partial_t-\frac12\partial_x^2)f$, i.e., $L_N$ approaches the continuum heat operator. Then $K_N$ is the inverse operator to $L_N$, in the following sense.
    
    \begin{lem}\label{inv} $L_NK_Nf=K_NL_N f=f$ whenever $f$ is a tempered distribution supported on $[a,b]\times \mathbb R^{d}$ with $b-a<T+1$.
    \end{lem}

    \begin{proof}
        Note that for each $N$ the operators $K_N,L_N$ are continuous on $\mathcal S'(\mathbb R^{d+1})$, as may be verified using the Fourier transform as explained after \eqref{eq:convPhi}.  Therefore it suffices to prove the claim for all smooth functions $f$ that have compact support contained in $(a,b)\times \mathbb R^d$, since these are dense in the subset of distributions supported on $[a,b]\times \mathbb R$, with respect to the topology of $\mathcal S'(\mathbb R^{d+1})$. The smooth analogue is true by direct calculations, since $p_N(s,\dr y)$ is the kernel for the inverse operator to the discrete heat operator $\mathcal L_N$ introduced in Section \ref{sec:3}.
    \end{proof}

    Recall the fields $M^N_t(\phi)$ and $Q^f_N(t,\phi)$ as defined in \eqref{mfield} and \eqref{qfield} respectively. We can now view these as random elements of $C([0,T+1],\mathcal S'(\mathbb R))$ by defining these fields by linear interpolation for $t\notin N^{-1}\mathbb Z_{\ge 0}.$
    

    For any $T$, note that $C([0,T],\mathcal S'(\mathbb R^d))$ embeds naturally into the linear subspace of $\mathcal S'(\mathbb R^{d+1})$ consisting of distributions supported on $[0,T]\times \mathbb R^d$, thus we can make sense of $D_Nf,L_Nf,K_Nf$ for all $f\in C([0,T],\mathcal S'(\mathbb R^d)),$ and these will be elements of $\mathcal S'(\mathbb R^{d+1})$ in general.

    \begin{defn}
        We will say that two tempered distributions $f,g \in \mathcal S'(\mathbb R^{d+1})$ agree on $[0,T]$ if there exists $\e>0$ such that $(f,\varphi) = (g,\varphi)$ for all $\varphi$ supported on $[-\varepsilon,T+\varepsilon]\times \mathbb R^d.$
    \end{defn}

    \begin{defn}
        Sample the environment $\omega$ and then define a collection of coefficients $a_N(s,\dr y)$ as $a_N(0,\dr y) = \nu_N$ and for $s\geq 1$, $$a_N(s,A) = \int_I p_N(s,A- y) \;\;(\nu_N^{\varsigma_N})(\dr x_0),$$ where $\nu_N^{\varsigma_N}$ is the $  \varsigma_N$-tilted initial state $(\nu_N^{\varsigma_N})(\dr x):= e^{  \varsigma_N\bullet x - \log\int_I e^{\varsigma_N \bullet a} \nu_N(\dr a)} \nu_N(\dr x)$ and $p_N$ is the same random walk kernel from Definition \ref{dnlnkn}. 
        Then define the distribution $\mathfrak G^N\in C([0,T+1],\mathcal S'(\mathbb R^d))$ for $t\in N^{-1}\mathbb Z_{\ge 0}$ by $$\mathfrak G^N(t,\phi) := \int_I  \phi(N^{-1/2}y) a_N(Nt,\dr y) = \mathbb E[ \mathfrak H^N(t,\phi)],$$and linearly interpolated for $t\notin N^{-1}\mathbb Z_{\ge 0}$.
    \end{defn}

    \begin{lem}[Discrete Duhamel formula]\label{u=kdm} 
        Let $M^N$ be as in \eqref{mfield}. Furthermore, let $\mathfrak H^N$ be as defined in \eqref{hn}, and define $$\mathfrak F^N:= \mathscr B_N\cdot (\mathfrak H^N-\mathfrak G^N).$$ Restrict $\mathfrak F^N$ to $[0,T]$ thus viewed as an element of $C([0,T],\mathcal S'(\mathbb R^d))$. Then $L_N\mathfrak F^N$ agrees with $D_NM^N$ on $[0,T]$. Consequently $\mathfrak F^N$ agrees with $K_N D_NM^N$ on $[0,T]$.
    \end{lem} 

    \begin{proof}  
    Firstly, notice that $L_N \mathfrak G^N (t,\cdot) = 0$ for $t\in N^{-1}\mathbb Z_{\ge 0}$. 
    
    Let $V^N := D_NM^N$. Then it is clear from \eqref{mfield} that 
    $$V^N(t,\phi) = N\big[ M^N_{t+N^{-1}}(\phi)-M^N_t(\phi)\big]=N\mathscr B_N\int_I \phi(N^{-1/2}x) (\mathcal L_NZ_N)(Nt,\dr x),$$ for all $0\le t\in N^{-1}\mathbb Z_{\ge 0}$ and $\phi\in \mathcal S(\mathbb R^d)$.
    Now, with $Z_N$ defined in \eqref{zn}, for $t\in N^{-1}\mathbb N$ we have $$\mathfrak H^N(t,\phi) = \int_I \phi(N^{-1/2}x)Z_N(Nt,\dr x) .$$ With $\mathcal L_N$ as defined in Section \ref{sec:3}, we than apply the convolution defining $L_N$ to both sides and it is then clear that for $t\in N^{-1}\mathbb Z_{\ge 0}$ the expression for $L_N\mathfrak H^N(t,\cdot)$ can be written as \begin{align*}  N \mathscr B_N \int_I \phi(N^{-1/2}x) (\mathcal L_N Z_N)(Nt,\dr x).
    \end{align*}
        which is the same as the expression for $V^N(t,\phi).$ 
        
        Now we need to show that equality holds even if $t\notin N^{-1}\mathbb Z_{\ge 0}$. Since linear operators respect linear interpolation, and since all of the fields $\mathfrak H^N, V_N, M^N$ are defined via linear interpolation, this is actually immediate.
        
        Finally, if we view the restriction of $V^N$ to $[0,T]$ as an element of $\mathcal S'(\mathbb R^{d+1})$ supported on $[0,T]\times \mathbb R$, then we can apply $K_N$ to both sides and we obtain that $(\mathfrak H^N-\mathfrak G^N)(t,\cdot) = K_N V^N(t,\cdot) = K_ND_NM^N(t,\cdot)$ for all $t\in [0,T]$ by Lemma \ref{inv}.
    \end{proof}

 \subsection{Identification of the correct noise coefficients for each regime}
 In this section, we specialize the function $f$ appearing in the quadratic variation field \eqref{qfield} to the specific correct choices of functions $\boldsymbol\eta_{\alpha,\beta}$, $\z$, $\boldsymbol{\vartheta}_{   v}$ from Definitions \ref{etadef}, \ref{z}, and \ref{def:vartheta} respectively.

 \begin{prop}[Tightness of all relevant processes]\label{mcts} The following are true.
 \begin{enumerate}
		\item The fields $M^N$ from \eqref{mfield} may be realized as an element of $C([0,T],C^{\alpha,\tau}(\mathbb R^d))$ for any $\alpha<-5$ and $\tau>1$. Moreover, they are tight with respect to that topology.  \item The fields $\mathfrak F_N:=\mathscr B_N\cdot (\mathfrak H^N-\mathbb E[\mathfrak H^N])$ from \eqref{hn} may be realized as an element of as an element of $C([0,T],C^{\alpha,\tau}(\mathbb R^d))$ for any $\alpha<-d-p-2$ and $\tau>1$. Moreover, they are tight with respect to that topology.
        \item For any $f:I\to \mathbb R$ of exponential decay, the fields $Q_N^f$ from Definition \ref{qfield} may be realized as an element of $C([0,T],C^{\gamma,\tau}(\mathbb R))$ for any $\gamma<-1$ and $\tau>1$. Moreover, they are tight with respect to that topology.
        
        \item Let $\alpha,\gamma<0$ be as in the previous items. Let $(M^\infty,Q^{f,\infty}, H^\infty)$ be a joint limit point of $(M^N,\hat Q_N^{f},\mathfrak H^N)$ in the space $C\big([0,T],C^{\alpha,\tau}(\mathbb R^d)\times C^{\gamma,\tau}(\mathbb R^d)\times C^{\alpha,\tau}(\mathbb R^d)\big),$ where $f$ could be any specific function on $I$. For all $\phi\in C_c^\infty(\mathbb R)$, the process $(M_t^\infty(\phi))_{t\in[0,T]}$ is a continuous martingale with respect to the canonical filtration on that space, and moreover its quadratic variation is given by
  \begin{equation}
      \label{e:mcts}
      \langle M^\infty(\phi)\rangle_t = \begin{cases} \sum_{\#(k_1)=p ,\#(k_2)=p} \frac{1}{k_1!k_2!} Q_t^{\boldsymbol\eta_{k_1,k_2},\infty}(\partial^{k_1} \phi\cdot \partial^{k_2}\phi), & N^{1/2} \varsigma_N\to 0 \\ Q_t^{\z_{  {\boldsymbol v}} ,\infty} (\phi^2) , & N^{-1}\psi_N(p,d)\gg |  \varsigma_N|\gg N^{-1/2} , \frac{   \varsigma_N}{|  \varsigma_N|}\to   {\boldsymbol v} \\ 
      Q_t^{\z_{  {\boldsymbol v}/|  {\boldsymbol v}|} ,\infty} (\phi^2) , & d=2 \text{ and } (\log N)^{1/2p}   \varsigma_N\to   {\boldsymbol v}\ne 0 \\ 
      |  {\boldsymbol v}|^{-2p} Q_t^{\boldsymbol\vartheta_{  {\boldsymbol v}} ,\infty} (\phi^2), & d\ge 3 \text{ and } \varsigma_N\to   {\boldsymbol v} \ne 0.
      \end{cases}
  \end{equation} Here $\boldsymbol\eta_{\alpha,\beta}$, $\z$, $\boldsymbol{\vartheta}_{   v}$ are the specific functions from Definitions \ref{etadef}, \ref{z}, and \ref{def:vartheta} respectively. Furthermore, the same joint limit holds if $f$ is replaced by $f_N$ in the prelimit, where $\sup_N f_N$ is dominated by some function of exponential decay and $f_N\to f$ uniformly.
  \end{enumerate}
	\end{prop}
 \begin{proof}
     Take any $\phi\in C_c^\infty(\mathbb R)$ with $\|\phi\|_{L^\infty}\leq 1.$ Recall $S_{(t,x)}^{\lambda}$ from \eqref{scale}. Using the first bound in Proposition \ref{tight1}, we have for fixed $f:I\to\mathbb R$ of rapid decay that$$\Ex[ (Q_N^f(t,S^\lambda_{x}\phi)-Q_N^f(s,S^\lambda_{x}\phi))^{k}]\leq C |t-s|^{k/2}\lambda^{-k},$$ uniformly over $x\in\mathbb R, \lambda\in (0,1], 0\leq s,t\leq T$ with $s,t\in N^{-1}\mathbb Z_{\ge 0}$. Since $Q_N(0,\phi)=0$ by definition, the assumptions of Lemma \ref{l:KC} are therefore satisfied for any $\kappa\leq 1/4$, any $p>1/\kappa$, and any $\alpha\leq -1$, and we conclude the desired tightness for $\hat Q_N=Q_N$. This proves Item \textit{(3)}.

  \smallskip
  
		Now we address the tightness of the $M^N$. Let $\e$ be as in Definition \ref{goodseq}. Using  Proposition \ref{optbound}, we have that \begin{equation}\label{m1}\Ex[ (M^N_t(\phi)-M^N_s(\phi))^{q}]^{1/q} 
        \leq C(t-s)^{\e} \|\phi\|_{C^{p+2}(\mathbb R^d)},\end{equation} where $\phi\in C_c^\infty(\mathbb R)$, $C=C(k,q)>0$ is free of $\phi,s,t,N.$
		This gives $$\Ex\big[ \big(M^N_t(S^\lambda_{x}\phi)-M^N_s(S^\lambda_{x}\phi)\big)^{2/\e}\big]^{\e/2}\leq C(t-s)^{\e} \lambda^{-d-p-2} $$
		uniformly over $x\in\mathbb R, \lambda\in (0,1], 0\leq s,t\leq T$, and $\phi\in C_c^\infty(\mathbb R^d)$ with $\|\phi\|_{C^{p+2}}\leq 1$ with support contained in the unit interval. Moreover $M^N(0,\phi)=0$ by definition, therefore the assumptions of Lemma \ref{l:KC} are satisfied with $\kappa=\epsilon$, any $q>1/\kappa$, and any $\alpha\leq -d-p-2$. Hence, we conclude the desired tightness for $M^N$. This proves Item \textit{(1)}.

  Now tightness for the fields $\mathfrak H^N-\mathfrak G^N$ is immediate from Lemma \ref{kndn}, since we know from Lemma \ref{u=kdm} that $\mathfrak H^N-\mathfrak G^N = K_ND_NM^N$ where we view $K_ND_N $ as a bounded operator from $C([0,T+1],C^{\alpha,\tau}(\mathbb R^d))\to C([0,T],C^{\alpha,\tau}(\mathbb R^d))$. 

  \smallskip
  
We next show that the limit point $M^{\infty}(\phi)$ is a martingale indexed by $t\in N^{-1}\mathbb Z_{\ge 0}$. Since $M^N_0(\phi)=0$, from \eqref{m1}, we see that $\sup_N \Ex[M^N_t(\phi)^{2k}]<\infty.$ Thus $M^{\infty}(\phi)$ is a martingale since martingality is preserved by limit points under the uniform integrability assumption. Continuity is guaranteed by the definition of the spaces in which we proved tightness. In the prelimit, we know $M_t^N(\phi)^2$ minus the objects in Proposition \ref{4.3} are martingales. 
By the latter proposition and the tightness estimates  \eqref{e.tight7} of $Q_N^f$, it follows that the error terms $V^j_N(t,\phi)$ and $\mathcal E_N^j(t,\phi)$ vanish in probability (in the topology of $C[0,T]$ for each $\phi \in C_c^\infty(\mathbb R)$), so we conclude (again by uniform $L^p$ boundedness guaranteed by Proposition \ref{tight1}) that the process given in the proposition statement is a martingale. This verifies \eqref{e:mcts} completing the proof of Item \textit{(4)}.
 \end{proof}

\subsection{Technical lemmas on the operators}

\begin{lemma}\label{kndn}
     Fix $\alpha<0,\tau>0$. Let $K_N,D_N$ be the operators from Definition \ref{dnlnkn}. Recall the function spaces from Definition \ref{fsp}.
     Then we have the operator norm bound $$\sup_{N\ge 1}\|K_ND_N\|_{C([0,T+1],C^{\alpha,\tau}(\mathbb R^d))\to C([0,T],C^{\alpha,\tau}(\mathbb R^d))}<\infty.$$
 \end{lemma}


 \begin{proof} The full proof can be found in \cite[Lemma 5.18]{Par24}. Note that that proof is for $d=1$ but readily extends to all dimensions. 
 \end{proof}

\begin{lem}[Controlling the difference between the discrete and continuum time-derivatives] \label{ds} Fix $\alpha<0, \tau>1$. The derivative operator $\partial_s : C^{\alpha,\tau}_\mathfrak s(\Psi_T) \to C_\mathfrak s^{\alpha-2,\tau}(\Psi_T)$ which was defined in Remark \ref{d/dx}, is a bounded linear map. Furthermore, let $D_N$ be as in Definition \ref{dnlnkn}. Then we have the operator norm bounds 
 \begin{align*}\sup_{N\ge 1}\|D_N\|_{C^{\alpha,\tau}_\mathfrak s(\Psi_T)\to C_\mathfrak s^{\alpha-2,\tau}(\Psi_T)}&<\infty. \\ \|D_N-\partial_s\|_{C^{\alpha,\tau}_\mathfrak s(\Psi_T) \to C^{\alpha-4,\tau}_\mathfrak s(\Psi_T)} &\leq \frac12 N^{-1}.
 \end{align*}
\end{lem}

 \begin{proof}The proof is fairly immediate from the definitions, see \cite[Lemma 6.21]{DDP23} for the complete argument. \end{proof}

 \begin{lem}[Controlling the difference between the discrete and continuum heat operators] \label{kn-k}
     Fix $\alpha<0, \tau>1$. Let $K_N$ be as in Definition \ref{dnlnkn}, and let $K : C^{\alpha,\tau}_\mathfrak s(\Psi_T) \to C^{\alpha+2,\tau}_\mathfrak s(\Psi_T)$ by \begin{align}
   \label{e:kf}
 			Kf(t,x):= \int_{\Psi_T} G_{t-s}(x-y)f(s,y)dsdy,
 		\end{align} where $G_t$ as usual is the deterministic keat kernel on $\mathbb R^d$. Then we have the operator norm bound $$\|K_N-K\|_{C^{\alpha,\tau}_\mathfrak s(\Psi_T) \to C^{\alpha-1,\tau}_\mathfrak s(\Psi_T)} \leq CN^{-1/4}.$$
  Here, $C$ is independent of $N$.
 \end{lem}

 The above bound is crude; we do not claim optimality of the H\"older exponents here.

 \begin{proof}
     The fact that $K$ as defined by \eqref{e:kf} actually defines a bounded linear map $C^{\alpha,\tau}_\mathfrak s(\Psi_T) \to C^{\alpha+2,\tau}_\mathfrak s(\Psi_T)$ is the so-called Schauder estimate, see e.g. \cite[Corollary 1.2]{HL16}. 
     For the operator norm bound, see \cite[Lemma 6.20]{DDP+}, and replace $\mathbb R$ by $\mathbb R^d$ throughout. When applying the KMT coupling there, replace the nearest neighbor random walk by that of the increment distribution $\mu$ instead, and the remaining details go through verbatim.
 \end{proof}

 \begin{cor} \label{cor:Kds}
     Let $\alpha<0$. With $K$ as in \eqref{e:kf} and $\partial_s$ as in Remark \ref{d/dx}, we have that $K\partial_s : C([0,T], C^{\alpha,\tau}(\mathbb R)) \to C([0,T], C^{\alpha,\tau}(\mathbb R)).$ Furthermore, $\|K_ND_N - K \partial_s \|_{C^{\alpha,\tau}_\mathfrak s(\Psi_T) \to C^{\alpha-5,\tau}_\mathfrak s(\Psi_T)} \to 0.$
 \end{cor}

 This is immediate from the result of Lemma \ref{kndn}, and the previous two corollaries. 

\subsection{Uniqueness of the martingale problem and proof of main results}

\begin{prop}\label{uniqueness_of_mart_prob}
    In any dimension, one has uniqueness of the martingale problem for the SPDEs appearing in Theorems \ref{main1} $-$ \ref{main4}. More precisely, we have the following result.
    
     Fix $\alpha<0$, $\tau\in \mathbb R$, $p\in\mathbb Z_{\ge 0},$ and coefficients $\gamma_{  k_1,k_2}\in\mathbb R$. Consider any $C([0,T], C^{\alpha,\tau}(\mathbb R^d))$-valued random process $u(t,\bullet)$ adapted to the canonical filtration, with $u(0,\bullet)\equiv 0$. If the processes 
\begin{align*}
    M_t(\phi)&:= u(t,\phi) - u(0,G_t*\phi)- \frac12 \int_0^t  u(s,\Delta \phi)\dr s \\ G_t(\phi) &:= M_t(\phi)^2 -  \sum_{\#(  k_1)=p, \#(  k_2)=p} \gamma_{  k_1,  k_2} \int_0^t \int_{\mathbb R^2} \partial^{  k_1 }\phi(a) \partial^{  k_2 }\phi(a)\cdot (\mathfrak H_0 *G_t)(a)^2\dr a \dr s
\end{align*}
are both martingales for every $\phi\in C_c^\infty(\mathbb R^d)$, then $u(t,\bullet)$ must have the same law as the solution of the additive-noise SPDE given by \begin{equation}\label{eq:extremal5}\partial_t \mathcal U(t,   x) = \frac12\Delta \mathcal U(t,   x) + \eta(t,x),
\end{equation}
where $\eta$ is the Gaussian noise on $\mathbb R^d\times \mathbb R_+$ that has covariance structure given by $$\mathbb E[(\eta(t,\bullet),\phi)_{L^2(\mathbb R^d)}^2] = \sum_{\#(  k_1)=p, \#(  k_2)=p} \gamma_{  k_1,  k_2} \int_{\mathbb R^2} \partial^{  k_1 }\phi(a) \partial^{  k_2 }\phi(a)\cdot (\mathfrak H_0 *G_t)(a)^2\dr a,\;\;\;\;\;\;\;\;\;\;\;t>0, \;\phi\in C_c^\infty(\mathbb R^d).$$
\end{prop}

In particular, if $p=0$, the result says the following. Consider any $C([0,T], C^{\alpha,\tau}(\mathbb R^d))$-valued random process $u(t,\bullet)$ adapted to the canonical filtration, with $u(0,\bullet)\equiv 0$. If the processes 
\begin{align*}
    M_t(\phi)&:= u(t,\phi) - u(0,G_t*\phi)- \frac12 \int_0^t  u(s,\Delta \phi)ds \\ G_t(\phi) &:= M_t(\phi)^2 - \gamma^2 \int_0^t \big(G(s,\bullet)^2, \phi^2 \big)_{L^2(\mathbb R^d)}\dr s
\end{align*}
are both martingales for every $\phi\in C_c^\infty(\mathbb R^d)$, then $u(t,\bullet)$ must have the same law as the solution of the additive-noise SPDE \begin{equation}\label{eq:extremal5a}\partial_t \mathcal U(t,   x) = \frac12\Delta \mathcal U(t,   x) + \gamma\cdot   (\mathfrak H_0 *G_t)(x)\; \xi(t,   x),
\end{equation} where $\mathcal U(t,0)=0$ and $\xi$ is a standard space-time white noise.

Thus, for $p=0$ the result is relevant to the extremal SPDE appearing in Theorems \ref{main1} $-$ \ref{main3}, whereas for $p\ge 1$ it is more relevant to the bulk SPDE. 

\begin{proof}
    This is quite standard from the theory of martingale problems (see \cite{SV}), so we only sketch the proof here. It is easily shown using the Dubins-Schwarz time change theorem that any continuous martingale with a \textit{deterministic} quadratic variation process is a Gaussian process. Using the Cramer-Wold device and linearity of $\phi\mapsto M_t(\phi)$, it follows that the entire family of random variables $(M_t(\phi))_{\phi\in C_c^{\infty}(\mathbb R^d), t\ge 0}$ is a jointly Gaussian process in the variables $t,\phi.$ Thus, using Lemma \ref{l:KC} and the Gaussianity of the marginals of $M$, the law of $M$ can be realized as a Gaussian measure on the Banach space $C([0,T],C^{\alpha,\tau}(\mathbb R))$. From here, notice that in the notation of Lemmas \ref{kndn} and \ref{ds} and \ref{kn-k}, the field $u$ is just $K\partial_s$ applied to $M$, where $K\partial_s : C([0,T],C^{\alpha,\tau}(\mathbb R))\to C([0,T],C^{\alpha,\tau}(\mathbb R))$ is a \textit{bounded linear map} (see Corollary \ref{cor:Kds}). In particular, the law of $u$ can also be realized as a Gaussian field on the same Banach space. Now one simply matches the covariances of $u$ to the correct Gaussian field, but this is immediate from the given expressions for the quadratic variation.
\end{proof}

 Now we are in a position to prove the main theorems of the paper.

\begin{proof}[Proof of Theorems \ref{main1}, \ref{main2}, \ref{main3}, and \ref{main4}] There are four different cases depending on the choice of location strength $  \varsigma_N$. Recall that we referred to these as \hyperref[eq:regimeA]{Regime A}, \hyperref[eq:regimeB]{Regime B}, \hyperref[eq:regimeC]{Regime C}, and \hyperref[eq:regimeD]{Regime D} in the introduction. We also recall the constant $c_1:=1 , c_2=2\pi, $ and $c_d:= d(2-d)/\Gamma(1+\frac{d}2)$. To prove the theorems, note from Proposition \ref{uniqueness_of_mart_prob} that we just need to show that the limiting quadratic variation field $Q^{f,\infty}$ constructed in Proposition \ref{mcts} is actually deterministic with the correct quadratic variation in each case. 
\\
\\
\textbf{Regime A. $N^{1/2} \varsigma_N \to \boldsymbol{\varsigma}$.} As explained in the proof of \ref{4.3} it suffices to consider only the case $\boldsymbol{\varsigma}=0.$ As guaranteed by Proposition \ref{mcts}, let $(M^\infty,Q^\infty, H^\infty)$ be a joint limit point of $(M^N,\hat Q_N^{\z},\mathfrak H^N)$ in the space $C([0,T],C^{\alpha,\tau}(\mathbb R^d)\times C^{\gamma,\tau}(\mathbb R)\times C^{\alpha,\tau}(\mathbb R^d)),$ where $\z$ is the specific function from Definition \ref{z}. On one hand, Proposition \ref{mcts} guarantees that the process $\langle M^\infty (\phi)\rangle_t = \sum_{\#(k_1)=p ,\#(k_2)=p} \frac{1}{k_1!k_2!} Q_t^{\eta_{k_1,k_2},\infty}(\partial^{k_1} \phi\cdot \partial^{k_2}\phi)$. On the other hand, we have by Lemmas \ref{kndn} and \ref{ds} and \ref{kn-k}, and the relation of Lemma \ref{u=kdm} that $H^\infty = K\partial_s M^\infty.$ Furthermore by Proposition \ref{gen.ic} and the uniform integrability guaranteed by Proposition \ref{tight1}, we also have that the process $$Q^{f,\infty}_t(\phi) - \pi^{\mathrm{inv}}(f) \cdot c_d\cdot \int_0^t \int_{\mathbb R^d} \big((G_s * \mathfrak H_0)(a)\big)^2 \phi(a)\dr a\dr s$$ is a martingale starting at 0, by using the Markov property and directly computing the conditional expectation. On the other hand, as long as $\phi \ge 0$, one verifies that it is a difference of two increasing processes, thus of finite variation. We conclude that the above process is in fact 0.
This identifies $H^\infty$ as an element of $C([0,T], C^{\alpha,\tau}(\mathbb R^d)) $ with the property that $M^\infty_t(\phi) = (H^\infty_t,\phi) -\frac12 \int_0^t (H^\infty_s,\Delta \phi)ds$ is a martingale with the deterministic quadratic variation given by $\int_0^t \int_{\mathbb R^d} A_p [ \phi,\phi] (a) (\mathfrak H_0 * G_s) (a) \dr s \dr a$. By the uniqueness of the martingale problem from \ref{uniqueness_of_mart_prob}, this uniquely identifies $H^\infty$ in law as the solution of the SPDE written in the first bullet points of Theorems \ref{main1}, \ref{main2}, and \ref{main3}.
\\
\\
\textbf{Regime B. $\psi_N(p,d)\gg  |  \varsigma_N|\gg N^{-1/2}$ and $  \varsigma_N/|  \varsigma_N|\to   {\boldsymbol v}$. } As guaranteed by Proposition \ref{mcts}, let $(M^\infty,Q^\infty, H^\infty)$ be a joint limit point of $(M^N,\hat Q_N^{\z},\mathfrak H^N)$ in the space $C([0,T],C^{\alpha,\tau}(\mathbb R^d)\times C^{\gamma,\tau}(\mathbb R)\times C^{\alpha,\tau}(\mathbb R^d)),$ where $\z$ is the specific function from Definition \ref{z}. On one hand, Proposition \ref{mcts} guarantees that $\langle M^\infty (\phi)\rangle_t = Q^{\z_{  {\boldsymbol v}},\infty}(\phi^2).$ On the other hand, we have by Lemmas \ref{kndn} and \ref{ds} and \ref{kn-k}, and  the relation of Lemma \ref{u=kdm} that $H^\infty = K\partial_s M^\infty.$ Furthermore by Proposition \ref{gen.ic} and the uniform integrability guaranteed by Proposition \ref{tight1}, we also have that the process $$Q^{f,\infty}_t(\phi) - \pi^{\mathrm{inv}}(f) \cdot c_d\cdot \int_0^t \int_{\mathbb R^d} \big((G_s * \mathfrak H_0)(a)\big)^2 \phi(a)\dr a\dr s$$ is a martingale starting at 0, by using the Markov property and directly computing the conditional expectation. On the other hand, as long as $\phi \ge 0$, one verifies that it is a difference of two increasing processes, thus of finite variation. We conclude that the above process is in fact 0.This identifies $H^\infty$ as an element of $C([0,T], C^{\alpha,\tau}(\mathbb R^d)) $ with the property that $M^\infty_t(\phi) = (H^\infty_t,\phi) -\frac12 \int_0^t (H^\infty_s,\Delta \phi)ds$ is a martingale with the correct quadratic variation. By the uniqueness of the martingale problem from \ref{uniqueness_of_mart_prob}, this uniquely identifies $H^\infty$ in law as the solution of the SPDE written in the second bullet points of Theorems \ref{main1}, \ref{main2}, and \ref{main3}.
\\
\\
\textbf{Regime C. $d=2$ and $(\log N)^{1/2p}   \varsigma_N\to   {\boldsymbol v}$. } As guaranteed by Proposition \ref{mcts}, let $(M^\infty,Q^\infty, H^\infty)$ be a joint limit point of $(M^N,\hat Q_N^{\z},\mathfrak H^N)$ in the space $C([0,T],C^{\alpha,\tau}(\mathbb R^d)\times C^{\gamma,\tau}(\mathbb R)\times C^{\alpha,\tau}(\mathbb R^d)),$ where $\z$ is the specific function from Definition \ref{z}. On one hand, Proposition \ref{mcts} guarantees that $\langle M^\infty (\phi)\rangle_t = Q_t^\infty(\phi^2).$ On the other hand, we have by Lemmas \ref{kndn} and \ref{ds} and \ref{kn-k}, and  the relation of Lemma \ref{u=kdm} that $H^\infty = K\partial_s M^\infty.$ Furthermore by Proposition \ref{gen.ic} and the uniform integrability guaranteed by Proposition \ref{tight1}, we also have that the process $$Q^{f,\infty}_t(\phi) - \pi^{\mathrm{inv}}(f)\frac{2}{2-\gamma_{
\mathrm{ext}}(
  {\boldsymbol v})}
 \cdot 2\pi \cdot \int_0^t \int_{\mathbb R^d} \big((G_s * \mathfrak H_0)(a)\big)^2 \phi(a)\dr a\dr s$$ is a martingale starting at 0, by using the Markov property and directly computing the conditional expectation. On the other hand, as long as $\phi \ge 0$, one verifies that it is a difference of two increasing processes, thus of finite variation. We conclude that the above process is in fact 0. This identifies $H^\infty$ as an element of $C([0,T], C^{\alpha,\tau}(\mathbb R^d)) $ with the property that $M^\infty_t(\phi) = (H^\infty_t,\phi) -\frac12 \int_0^t (H^\infty_s,\Delta \phi)ds$ is a martingale with the correct quadratic variation. By the uniqueness of the martingale problem from \ref{uniqueness_of_mart_prob}, this uniquely identifies $H^\infty$ in law as the solution of the SPDE written in the third bullet point of Theorems  \ref{main2}.
\\
\\
\textbf{Regime D. $d\ge 3$ and $  \varsigma_N\to    v.$} As guaranteed by Proposition \ref{mcts}, let $(M^\infty,Q^\infty, H^\infty)$ be a joint limit point of $(M^N,\hat Q_N^{\z},\mathfrak H^N)$ in the space $C([0,T],C^{\alpha,\tau}(\mathbb R^d)\times C^{\gamma,\tau}(\mathbb R)\times C^{\alpha,\tau}(\mathbb R^d)),$ where $\z$ is the specific function from Definition \ref{z}. On one hand, Proposition \ref{mcts} guarantees that $\langle M^\infty (\phi)\rangle_t = Q_t^\infty(\phi^2).$ On the other hand, we have by Lemmas \ref{kndn} and \ref{ds} and \ref{kn-k}, and  the relation of Lemma \ref{u=kdm} that $H^\infty = K\partial_s M^\infty.$ Furthermore by Proposition \ref{gen.ic} and the uniform integrability guaranteed by Proposition \ref{tight1}, we also have that the process $$Q^{f,\infty}_t(\phi) -\Theta_{\mathrm{eff}}^2(f;  {\boldsymbol v}) \cdot c_d\cdot \int_0^t \int_{\mathbb R^d} \big((G_s^{(  {\boldsymbol v})} * \mathfrak H_0)(a)\big)^2 \phi(a)\dr a\dr s$$ is a martingale starting at 0, by using the Markov property and directly computing the conditional expectation. On the other hand, as long as $\phi \ge 0$, one verifies that it is a difference of two increasing processes, thus of finite variation. We conclude that the above process is in fact 0. This identifies $H^\infty$ as an element of $C([0,T], C^{\alpha,\tau}(\mathbb R^d)) $ with the property that $M^\infty_t(\phi) = (H^\infty_t,\phi) -\frac12 \int_0^t (H^\infty_s,\Delta_v \phi)ds$ is a martingale with the correct quadratic variation (here $\Delta_v = \mathrm{div}(H_v\nabla))$ which is self-adjoint). By the uniqueness of the martingale problem from \ref{uniqueness_of_mart_prob}, this uniquely identifies $H^\infty$ in law as the solution of the SPDE written in the third bullet point of Theorems  \ref{main3} with the coefficient given by that of Theorem \ref{main4}.
\end{proof}

Finally, we prove Theorem \ref{main5}.

\begin{proof}[Proof of Theorem \ref{main5}]
    The proof is extremely similar to before: the main idea is to study the \textit{cross variations} of the martingales $M_t^N(\phi)$ for different initial data. 
    
        Thus we let $\nu^j_N$ $(1\le j \le n)$ be probability measures on $I$ with the property that the associated sequences of rescaled initial data $\mathfrak H^N_j(0,\phi)$ as derived from $\nu^j_N$ via \eqref{hn} are good sequences in the sense of Definition \ref{goodseq}, and furthermore converge weakly as $N\to \infty$ to some measures $\mathfrak H_j.$ Then we define for $1\le j \le n$ the martingales $M^{j,N}_t(\phi)$ to be the martingales from \eqref{mfield}, but with the property that the initial condition is replaced by $\nu_N^j$.
        
        For $i\ne j$, one needs to compute $\langle M^{i,N}(\phi), M^{j,N}(\phi)\rangle$. Redoing the calculation of Proposition \ref{4.3} for this case, one has that
        $$\langle M^{i,\infty} (\phi), M^{j,\infty}(\phi)\rangle_t = Q^{\bullet}_N(t,\phi; \nu^i_N\otimes \nu^j_N)+ \sum_{1,2} \mathcal E^a +\mathcal V^a$$ where $Q_N$ is exactly as in \eqref{qfield}, where the ``$\bullet$" depends on the regime in exactly the same way as in Proposition \ref{4.3}, and the error terms are defined likewise. Note the only difference here compared to earlier is that the initial condition for the random walkers is $\nu_N^i\otimes \nu_N^j$ rather than $\nu_N^{\otimes 2}$ as before. 
        
        Going through exactly the same calculation as in Theorem \ref{gen.ic} but replacing $\nu_N^{\otimes 2}$ by $\nu_N^i\otimes \nu_N^j$ throughout that proof, one 
        has the following limits:
     \begin{itemize}
         \item If $d=1$ and $N^{\frac1{4p}} \varsigma_N\to 0$ then $$\lim_{N \to \infty}\;  \Ex\big[ Q_N^{f}(t,\phi; \nu^i_N\otimes \nu^j_N)\big] =\pi^{\mathrm{inv}}(f)\int_0^t \int_{\mathbb R} (G_t* \mathfrak H_i)(a)(G_t* \mathfrak H_j)(a)\phi(a)\dr a\dr s.$$ 
         \item If $d=2$ and $(\log N)^{1/2p}   \varsigma_N\to   {\boldsymbol v}$ then we have 
            $$\lim_{N \to \infty}\;  \Ex\big[ Q_N^{f}(t,\phi; \nu^i_N\otimes \nu^j_N)\big] =\pi^{\mathrm{inv}}(f)\frac{
            2}{2- \gamma_{
            \mathrm{ext}}(  {\boldsymbol v})} \int_0^t \int_{\mathbb R^2} (G_t* \mathfrak H_i)(a)(G_t* \mathfrak H_j)(a)\phi(a)\dr a\dr s.$$ 
         \item If $d\ge 3$ and if $  \varsigma_N\to   {\boldsymbol v}$ then 
         $$\lim_{N \to \infty}\;  \Ex\big[ Q_N^f(t,\phi; \nu^i_N\otimes \nu^j_N)\big]= \Theta_{\mathrm{eff}}^2( f;  {\boldsymbol v}) \int_0^t \int_{\mathbb R^d} (G_t^{(\boldsymbol v)}* \mathfrak H_i)(a)(G_t^{(\boldsymbol v)}* \mathfrak H_j)(a) \phi(a)\dr a\dr s,$$ 
         where $\Theta_{\mathrm{eff}}^2( f;  {\boldsymbol v})$ is as in Proposition \ref{4.1c}.
     \end{itemize}
         We allow $\boldsymbol v$ to be 0 above, so that this covers all regimes. Now going through the proof of the uniqueness in \eqref{uniqueness_of_mart_prob}, one finds that for any joint limit point as guaranteed by Proposition \ref{mcts}, this then leads to the correct expressions to match the quadratic variations of the flow of solutions for the SPDEs written in the main theorems.
\end{proof}

\appendix

\section{Proof of Theorem \ref{inv00}: Invariance principle for SRI chains}\label{appendix:a}

These appendices are devoted to the proofs of the three major theorems in Section \ref{sec:4}. Most of the proofs will involve fairly standard techniques involving well-known bounds on martingales and additive functionals of time-homogeneous Markov chains. A few ideas from the theory of Fourier transforms will also be used. 

We start with the proof of Theorem \ref{inv00}. We first need three lemmas and another definition.

\begin{lem}\label{cool}
Consider a martingale $M$ defined on any filtered probability space $(\Omega, \mathcal F, P),$ and suppose that the increments of $M$ satisfy a uniform $q^{th}$ moment bound under the probability measure $P$: $$\mathcal S_q(M;P):=\sup_{n\ge 1}  E [|X_{n+1} -X_n|^q]<\infty,$$ where $q\ge 2$. Then $E [ |M_r-M_s|^q ] \leq C_q (r-s)^{\frac{q}2} \mathcal S_q(M;P),$ for some absolute constant $C_q$ that is uniform over all martingales $M$ and probability spaces $(\Omega, \mathcal F,P).$
 \end{lem}
    
    \begin{proof} We may apply Burkholder-Davis-Gundy, then Jensen, then the definition of $\mathcal S_q(M;P)$ in that order to obtain uniformly over $r\ge s\ge 0$
    \begin{align*}
        E [ |M_r-M_s|^q ] &\leq C_q  E \bigg[ \bigg( \sum_{n=s}^{r-1} (M_{n+1}-M_n)^2\bigg)^{q/2}\bigg] \\ &\leq C_q E \bigg[ (r-s)^{\frac{q}2 - 1} \sum_{n=s}^{r-1} |M_{n+1}-M_n|^q \bigg] \\ &\leq C_q (r-s)^{\frac{q}2} \mathcal S_q(M;P).
    \end{align*}
    Since the constant $C_q$ in Burkholder-Davis-Gundy is known to be universal, the claim follows.
\end{proof}

    \begin{lem}\label{tbb}
    Let $\nu_1,\nu_2$ be two probability measures on $\mathbb R^m$, and let $g:\mathbb R^m\to \mathbb R_+$ such that $G:=\int_{\mathbb R^m} g(x) (\nu_1(\dr    x) + \nu_2(\dr    x))<\infty $. Then for any measurable $f:\mathbb R^m\to\mathbb R$ such that $f(   x)^2\leq g(   x)$ we have $$\bigg|\int_{\mathbb R^m} f(x) \big((\nu_1(\dr    x)-\nu_2(\dr    x)\big)\bigg| \leq 2G^{1/2}\|\nu_1-\nu_2\|_{TV}^{1/2}$$
\end{lem}

Typically we will apply this lemma when $g(x)=|x|^q$ for some $q>1$, or where $g(x) = e^{\sigma|x|}$ for some $\sigma>0.$

\begin{proof}
    Let $\rho:= \|\nu_1-\nu_2\|_{TV}.$ Using the coupling definition of total variation distance, one may construct a probability measure $\omega$ on $\mathbb R^m\times \mathbb R^m$ with the property that $\omega(\{ (x,y): x\ne y\}) \leq 2\rho.$ Then 
    \begin{align*}
        \bigg|\int_{\mathbb R^m} f(   x) \big((\nu_1(\dr    x)-\nu_2(\dr    x)\big)\bigg|&= \bigg|\int_{\mathbb R^m\times \mathbb R^m} (f(   x)-f(   y)) \omega(\dr x,\dr y) \bigg|\\&\leq \int_{\{ (   x,   y):    x\ne   y\}} \big(|f(   x)|+|f(   y)| \big)\omega(\dr    x,\dr    y) \\ &\leq \omega(\{ (   x,   y):    x\ne   y\})^{1/2} \bigg(\int_{\mathbb R^m\times \mathbb R^m} \big(2f(   x)^2 +2f(   y)^2 \big)\omega(\dr    x,\dr    y)\bigg)^{1/2} \\ & \leq (2\rho)^{1/2} (2G)^{1/2}.
    \end{align*}
    In the second line, we used triangle inequality and the fact that $f(x)-f(   y)=0$ on the set $\{(x,y):x=y\}$. We used Cauchy-Schwartz in the third line, together with the fact that $(|a|+|b|)^2 \leq 2a^2+2b^2$.
\end{proof}

\begin{defn}
    Let $k \in \mathbb N$ and $1\leq i\neq j \leq k$. Consider a decreasing function $F:[0,\infty)\to[0,\infty)$ of exponential decay at infinity. We define the functional $V^{ij}(F;\bullet):(I^k)^{\mathbb Z_{\ge 0}} \to I^{\mathbb Z_{\ge 0}}$ by 
  \begin{align}
      \label{vijai}
      V^{ij}(F;r):= \sum_{s=0}^{r-1} F(|   R^i_{s}-   R^j_{s}|).
  \end{align}
We remark that the processes $V^{ij}(F;\bullet)$ are predictable with respect to the canonical filtration on $(I^k)^{\mathbb Z_{\ge 0}}$. 
\end{defn}

These processes $V^{ij}$ will play a crucial role going forward in this appendix. 

\begin{lem}\label{vfin}
Fix $k\in\mathbb N$, $q\ge 2$, $M>0$, a decreasing function $F:[0,\infty)\to[0,\infty)$ of exponential decay at infinity, and $1\leq i,j\leq k$. Let $u_\x^{ij}(\y):=|   y_i-   y_j-(   x_i-   x_j)|$, for $\x,\y\in I^k$. Here $|   u|$ as usual denotes the Euclidean norm of $   u\in \mathbb R^d$. Let $\br_0\in \Ek$ be centered. Then there exists $C>0$ such that uniformly over all $\br \in \Ek$ and $\x,\y\in I^k$ one has $$-C\cdot d_{\mathrm{SRI}}(\br,\br_0)^{1/4} F\big(\min_{i'<j'}|y_{i'}-y_{j'}|\big)^{1/4}\le (P^{\br}-\mathrm{Id}) u_\x^{ij}(\y) \leq C ,
$$
where $P^{\br}$ is the Markov operator associated to $P_{\br}$.
\end{lem}

\begin{proof}
    For the upper bound, we have that 
    \begin{align*}
        (P^{\br}-\mathrm{Id}) u_\x^{ij}(\y)&= \int_{I^k} \bigg(\big|a_i-a_j-(x_i-x_j)\big|-\big|y_i-y_j-(x_i-x_j)\big|\bigg) \br (\y,\mathrm d\bfa)  \\& \leq \int_{I^k} |a_i-a_j - (y_i-y_j)| \br(\y,\dr \bfa).
    \end{align*}
    From here, the uniform bound follows from $(q,M)$ being fixed.
    
    For the lower bound, use Jensen's inequality to say
    \begin{align*}
        P^{\br} u_{\x}(\y) & = \int_{I^k} |a_i-a_j-(x_i-x_j)| \br(\y,\dr\bfa) \\&\ge \bigg| \int_{I^k} (a_i-a_j)\br(\y,\dr\bfa) - (x_i-x_j)\bigg| \\&\ge |y_i-y_j - (x_i-x_j)| - \bigg| y_i-y_j - \int_{I^k} (a_i-a_j)\br(\y,\dr\bfa)\bigg| \\&= u_{\x}(\y) - \bigg| y_i-y_j - \int_{I^k} (a_i-a_j)\br(\y,\dr\bfa)\bigg|.
    \end{align*}
    Thus we just need to show that \begin{equation}\label{useful}\bigg| y_i-y_j - \int_{I^k} (a_i-a_j)\br(\y,\dr\bfa)\bigg|\leq C\cdot d_{\mathrm{SRI}}(\br,\br_0)^{1/4} F\big(\min_{i'<j'}|y_{i'}-y_{j'}|\big)^{1/4}.
    \end{equation}
    In fact, we will actually show the stronger claim that $$\bigg| y_i-y_j - \int_{I^k} (a_i-a_j)\br(\y,\dr\bfa)\bigg|\leq C\min \big\{  d_{\mathrm{SRI}}(\br,\br_0)^{1/2} , F\big(\min_{i'<j'}|y_{i'}-y_{j'}|\big)^{1/2}\big\},$$ from which \eqref{useful} will follow immediately using $\min\{a,b\}\leq a^{1/2}b^{1/2}.$
    For the upper bound of $d_{\mathrm{SRI}}(\br,\br_0)^{1/2}$, first note by the centered assumption that $y_i-y_j = \int_{I^k} (a_i-a_j) \br_0 (\y,\dr\bfa)$, then apply Lemma \ref{tbb}. For the upper bound of $F\big(\min_{i'<j'}|y_{i'}-y_{j'}|\big)^{1/2}$, let $\nu_{\br}$ denote the base measure of $\br$. Note that $\nu_{\br}^{\otimes k}$ is a permutation-invariant measure, thus $\int_{I^k} (a_i-a_j) \nu_{\br}^{\otimes k} (\dr \bfa)=0$. Then apply Lemma \ref{tbb}.
\end{proof}

The above bound will be quite useful in showing tightness of all of the relevant processes in the topology of $C[0,T]$. The following estimate is one of the key results of this appendix, and an important step towards proving both Theorem \ref{inv00} and Theorem \ref{anti}.

\begin{thm}\label{exp00}
    Fix $k\in \mathbb N$ and $q\ge 2$ and a decreasing function $F:[0,\infty)\to[0,\infty)$ such that $\sum_{\ell=0}^\infty \ell^{d-1}F(\ell)^{1/4}<\infty$. Let $\br\in \Ek$ be centered and $\delta$-repulsive.
    
    Also fix some decreasing function $H:[0,\infty)\to[0,\infty)$ such that $\sum_{\ell=0}^\infty \ell^{d-1}H(\ell)<\infty$. There exists some $C=C(k,F,H ,q)>0$ and $\epsilon=\epsilon(k,F,H,q)>0$ such that uniformly over all $\br\in \Ek$ with $d_{\mathrm{SRI}}(\br,\br_0)<\epsilon$, and all integers $r\ge r'\ge 0$ and $1\le i,j\le k$ we have the bounds 
    \begin{align}
        \;\sup_{\mathbf x\in I^k} \;\Ebb \big[ |(   R^i_r-   R^j_r)-(   R^i_{r'}-   R^j_{r'})|^q\big]^{1/q}& \leq C (r-r')^{1/2}. \label{e1''} \\  \;\sup_{\mathbf x\in I^k} \;\Ebb \big[ \big|V^{ij}(H;r)-V^{ij}(H;r')\big|^q\big]^{1/q}& \leq C  (r-r')^{1/2}.\label{e2''}
    \end{align}
\end{thm}


\begin{proof} By the Markov property and the fact that there is a supremum over $\x$, it suffices to prove the claim when $r'=0$, so that $   R^i_{r'}-   R^j_{r'} = x_i-x_j$ where $\x=(x_1,...,x_k)$ in coordinates. We can also assume without any loss of generality that $H\geq F^{1/4}$ because otherwise one may replace $G$ with $\tilde H:=\max\{H,F^{1/4}\}$, note that $\tilde H$ is still decreasing and summable, then prove the claim for $\tilde G$, and note that $V^{ij}(\tilde H;r)- V^{ij}(\tilde H;r')\ge V^{ij}(H;r)-V^{ij}( H;r')$ pathwise to immediately obtain the claim for the original function $H$.  

We now break the proof into five steps. 

\textbf{Step 1.} In this step, we show that there exists $\epsilon_0=\epsilon_0(q,M,\delta)>0$, so that $\br\in\Ek$ and $d_{\mathrm{SRI}}(\br,\br_0)<\epsilon_0$ together imply that $\br$ is $(\delta/2)$-repulsive. To prove this, notice by Lemma \ref{tbb} that $$\bigg|\int_{I^k} |y_i-y_j-(x_i-x_j)| \big( \br - \br_0 )(\x,\dr\y) \bigg| \leq C \cdot d_{\mathrm{SRI}}(\br,\br_0)^{1/2},$$
for some $C=C(q,M)>0.$ Thus taking $\epsilon_0:= \delta^2/(4C) $ gives the claim.

\textbf{Step 2.} Let $\epsilon_0$ be as in Step 1. In this step, we show that the collection of Markov kernels $\mathscr M:= \{ \br \in \Ek: d_{\mathrm{SRI}}(\br,\br_0)<\epsilon_0\}$ satisfies a certain coercivity condition on test functions given by absolute values of differences of coordinates, specifically \eqref{del} below. This will be extremely useful later.  For $\x,\y\in I^k$, let $$u_\x (\y):= |y_i-y_j-(x_i-x_j)| \;\;\;\;\text{and} \;\;\;\;v_\x^{\br}:= P^{\br} u_\x-u_\x,$$ so by Lemma \ref{vfin} 
we have that $v_\x^{\br}(\y)\ge -C\cdot d_{\mathrm{SRI}}(\br,\br_0)^{1/4} \cdot F\big(\min_{i'<j'}|y_{i'}-y_{j'}|\big)^{1/4}$.

We now claim a stronger bound: specifically, we claim that there exists $\delta'>0$ so that for all $\x,\y\in I^k$ and $\br\in \mathscr M$ one has that \begin{equation}\label{del}v_\x^{\br}(\y)>\delta' \ind_{B_{\delta'}(   0)}\big(y_i-y_j-(x_i-x_j)\big)-C\cdot d_{\mathrm{SRI}}(\br,\br_0)^{1/4} \cdot F\big(\min_{i'<j'}|y_{i'}-y_{j'}|\big)^{1/4}.\end{equation} To prove this, note that
    \begin{align*}
        |P^{\br} u_\x(\y) - P^{\br} u_\y(\y)| &= \bigg| \int_{I^k} |a_i-a_j-(x_i-x_j)|\br (\y, \mathrm d\bfa)-\int_{I^k} |a_i-a_j-(y_i-y_j)|\br (\y, \mathrm d\bfa)\bigg| \\ &\le \int_{I^k} \big| |a_i-a_j-(x_i-x_j)|-|a_i-a_j-(y_i-y_j)| \big| \br (\y, \mathrm d\bfa) \\ &\leq \int_{I^k}|x_i-x_j-(y_i-y_j)| \br (\y, \mathrm d\bfa)=|x_i-x_j-(y_i-y_j)|.
    \end{align*}
    Recall from Step 1 that $\inf_{\beta\in\ak} \inf_{\x\in I^k} P^{\br} u_\x(\x)\geq \delta/2$. Thus for all $\x,\y\in I^k$ we have $P^{\br} u_\y(\y)\ge \delta/2$, and from the above bound it then follows that $P^{\br} u_\x(\y)>  \delta/3$ if $|x_i-x_j-(y_i-y_j)|<\delta/6.$ Thus we see that $v_\x^{\br}(\y) = P^{\br} u_\x(\y) - |x_i-x_j-(y_i-y_j)| > \delta/6$ if $|x_i-x_j-(y_i-y_j)|<\delta/6$, thus proving the claim \eqref{del} with $\delta':=\delta/6.$

    \textbf{Step 3.}  In this step, we are going to find a useful family of martingales and derive a discrete ``Tanaka's formula," see \eqref{tanaka} below. That formula will be instrumental in the later analysis. We will also derive bounds related to these martingales, specifically \eqref{hb}, \eqref{lip}, and \eqref{go1} below. 
    
    If $   x\in \mathbb R^d$, let $   x^{(i)}$ denote the element of $\mathbb R^{kd}$ with $   x_i$ in the $i^{th}$ coordinate and $   0$ in all other coordinates. Equation \eqref{del} easily implies that for all $\y=(y_1,...,y_k)\in I^k$ one has
    \begin{align}\notag H(|y_i-y_j|) &\le (\delta')^{-1} \sum_{   m \in \mathbb Z^d}  \delta' \ind_{B_{\delta'}(   0)}\big(y_i-y_j-\delta'    m\big)\cdot G(\delta'|   m|) \\& \leq \inf_{\br \in \mathscr M}(\delta')^{-1}\sum_{   m\in \mathbb Z^d} G(\delta'  |   m|)\big[ v^{\br}_{\delta'    m^{(i)}}(\y) + C\cdot d_{\mathrm{SRI}}(\br,\br_0)^{1/4}F\big(\min_{i'<j'}|y_{i'}-y_{j'}|\big)^{1/4}\big].\label{hb} \end{align}
    For every $\x,\bfa\in I^k$, note that $u_\x(   R^i_r-   R^j_r)-\sum_{s=0}^{r-1} v_\x^{\br}(   R^i_s-   R^j_s)$ is a $\mathbf P^{\br}_\bfa$-martingale for every $\bfa$, just by definition of the Markov operator (these are the Dynkin martingales). Defining $f_H,g_{\br,H}:I^k\to\mathbb R$ by  $$f_H:= (\delta')^{-1}\sum_{   m\in \mathbb Z^d} H(\delta' |   m|) u_{\delta    m^{(i)}},\;\;\;\;\;\;\;\;\; g_{\br,H}:= (\delta')^{-1}\sum_{   m\in \mathbb Z^d}G(\delta' |   m|) v^\beta_{\delta    m^{(i)}},$$ we see that the process \begin{equation}\label{tanaka}\mathcal M_{\br,H}  (r):=-(f_G(   R^i_r-   R^j_r)- f_H (\x))+\sum_{s=0}^{r-1}g_{\br,H} (   R^i_r-   R^j_r)\end{equation} is a $\mathbf P^{\br}_\bfa$-martingale (for all $\y\in I^k)$ with $\mathcal M_{\br,H} (0)=0$. 
    Note that $f$ is globally Lipschitz, in fact 
    \begin{align} |f_F (\x)-f_F (\y)| &\leq \bigg((\delta')^{-1}\sum_{   m\in \mathbb Z^d} H(\delta' |   m|) \bigg)|   x_i-   x_j-(   y_i-   y_j)|\leq C|   x_i-   x_j-(   y_i-   y_j)|\label{lip},\end{align}
    where we absorbed all constants in the last bound. Note that the sum is convergent precisely because of the assumption that $\sum_{\ell=0}^\infty \ell^{d-1}H(\ell)<\infty$. We also claim that $g_{\br, G}$ is a globally bounded function. Indeed, Lemma \ref{vfin} gives that $\Gamma:= \sup_{\br \in \mathscr M} \sup_{\x,\y\in I^k} v^{\br}_\x(\y)    <\infty$, thus we see that 
    \begin{equation}\label{go1}\sup_{\br \in \mathscr M}\sup_{\x\in I^k}|g_{\br, H}(\x)| \leq \bigg((\delta')^{-1}\sum_{   m\in \mathbb Z^d} H(\delta' |   m|) \bigg) \cdot \Gamma <\infty.
    \end{equation}
    From \eqref{tanaka}, \eqref{lip}, and \eqref{go1} one sees that the increments of the martingales $\mathcal M_{\br, H}$ have $q^{th}$ moments bounded
    uniformly over $\br \in \mathscr M$ and all times $r\in \mathbb Z_{\ge 0}$, thus by Lemma \ref{cool} one obtains that 
    \begin{equation}\label{mamr}\sup_{\br \in \mathscr M} \Ebb [ |\mathcal M_{\br, H}(r)|^q]^{1/q} \leq  r^{1/2},
    \end{equation} where $C$ does not depend on $q$.
    
\textbf{Step 4.}  In this step, we will show an \textit{a priori} bound that the left side of \eqref{e2''} can actually be bounded above by some multiple of the left side of \eqref{e1''} plus $C(r-r')^{1/2},$ which will effectively reduce the problem to showing the first bound.  This bound is \eqref{preb1} below.

In \eqref{hb}, note by the decreasing property that $F\big(\min_{i'<j'}|x_{i'}-x_{j'}|\big)^{1/4} \leq \sum_{i'<j'} F(|x_{i'}-x_{j'}|)^{1/4}$. Furthermore, recall from the beginning of the proof our assumption (without loss of generality) that $H\ge F^{1/4}$, thus $F\big(|x_{i'}-x_{j'}|)^{1/4} \leq H(|x_{i'}-x_{j'}|)$. Thus from \eqref{hb} 
    we have the preliminary bound 
    \begin{align}\notag \Ebb[ V^{ij}(H;r)^q ]^{1/q} & \leq \Ebb \bigg[ \bigg( \sum_{s=0}^{r-1} g_{\br,H}(   R^i_s-   R^j_s) \bigg)^q\bigg]^{1/q} + C\cdot d_{\mathrm{SRI}}(\br,\br_0)^{1/4}  \sum_{1\le i'<j' \le k} \Eb [ V^{i'j'}(H;r)^q]^{1/q}.
    \end{align}
    Using \eqref{tanaka}, \eqref{lip}, \eqref{go1}, and \eqref{mamr}, we furthermore have 
    \begin{align*}\Ebb \bigg[ \bigg( \sum_{s=0}^{r-1} g_{\br,H}(   R^i_s-   R^j_s) \bigg)^q\bigg]^{1/q} &\leq \Ebb [ |\mathcal M_{\br,G}(r)|^q]^{1/q}+\Ebb [ |f_F (   R^i_r-   R^j_r) - f_F(\x)|^q ] ^{1/q}  \notag \\& \leq Cr^{1/2} + C \Ebb[ |R^i_r-   R^j_r-(x_i-x_j)|^q]^{1/q}.
    \end{align*}
    Combining the last two expressions and summing over all indices $(i,j)$ with $1\le i<j\le k$, then moving all instances of $\Ebb[ V^{ij}(G;r)^q ]^{1/q}$ to the left side of the inequality, yields 
    \begin{align*}\big(1- C\cdot d_{\mathrm{SRI}}(\br,\br_0)^{1/4} \cdot \tfrac12k(k-1)\big)& \sum_{1\le i<j \le k} \Ebb[ V^{ij}(H;r)^q ]^{1/q} \\ &\leq C r^{1/2} + C \sum_{1\le i<j\le k}\Ebb[ |R^i_r-   R^j_r-(x_i-x_j)|^q]^{1/q}. 
    \end{align*}
    Now choose $\epsilon<\epsilon_0$ small enough so that $ C\cdot \epsilon^{1/4} \cdot \tfrac12k(k-1) \leq \frac12.$ Then after absorbing all constants into some larger one, the last expression yields for $ d_{\mathrm{SRI}}(\br,\br_0) <\epsilon $ that
    \begin{equation}
        \label{preb1}  \sum_{1\le i<j \le k} \Eb[ V^{ij}(H;r)^q ]^{1/q} \leq C r^{1/2} + C \sum_{1\le i<j\le k}\Eb[ |R^i_r-   R^j_r-(x_i-x_j)|^q]^{1/q}. 
    \end{equation}
This bound will be a key step in proving the theorem shortly.

\textbf{Step 5.}  In this step, we will obtain a bound that is in a sense converse to \eqref{preb1}. 
Recall from \eqref{useful} that \begin{equation}\label{useful'} \bigg| y_i-y_j - \int_{I^k} (a_i-a_j)\br(\y,\dr\bfa)\bigg|\leq C\cdot d_{\mathrm{SRI}}(\br,\br_0)^{1/4} F\big(\min_{i'<j'}|y_{i'}-y_{j'}|\big)^{1/4}.\end{equation}
uniformly over $\beta\in \ak$ and $\x\in I^k.$ 
Let $   \pi_i:I^k\to I$ be the coordinate function $\x\mapsto x_i$, and define the process $$M^{i,j,\br}_r:= R^i_r-   R^j_r - (x_i-x_j)-\sum_{s=0}^{r-1} \big((P^{\br}-\mathrm{Id}) (   \pi_i-   \pi_j)\big)(\mathbf R_s).$$ This is a $\Pbb$-martingale for all $\x\in I^k$, which should \textit{not} be confused with the Tanaka martingales appearing in \eqref{tanaka}, as they are unrelated. Now we can use Lemma \ref{cool}, because \eqref{useful'} and the condition that $\br\in \Ek$ show that the increments of this martingale have uniformly bounded $q^{th}$ moments. Thus we see that $\Eb[ |M^{i,j,\br}_r|^q]^{1/q} \leq C r^{1/2},$ where $C$ does not depend on $\br\in \Ek$ and integers $r\ge 0$. Thus we find that \begin{equation}\label{ftrr}\Ebb[ |R^i_r-   R^j_r - (x_i-x_j)|^q]^{1/q} \leq C r^{1/2} + \Ebb\bigg[ \bigg( \sum_{s=0}^{r-1} \big((P^{\br}-\mathrm{Id}) (   \pi_i-   \pi_j)\big)(\mathbf R_s)\bigg)^q\bigg]^{1/q}.
\end{equation}
Equation \eqref{useful'} precisely says $|(P^{\br}-\mathrm{Id}) (   \pi_i-   \pi_j)\big)(\mathbf R_s)| \leq C\cdot d_{\mathrm{SRI}}(\br,\br_0)^{1/4} F\big(\min_{i'<j'}|y_{i'}-y_{j'}|\big)^{1/4},$ which can be further bounded above by $C \cdot d_{\mathrm{SRI}}(\br,\br_0)^{1/4} \sum_{1\le i'<j'\le k}F\big( |R^{i'}_s-R^{j'}_s|\big)^{1/4}.$ Using the definition \eqref{vijai} of the processes $V^{ij}$, this precisely means that $$\sum_{s=0}^{r-1} \big|\big(P^{\br}-\mathrm{Id}) (   \pi_i-   \pi_j)\big)(\mathbf R_s) \big| \leq Cd_{\mathrm{SRI}}(\br,\br_0)^{1/4}\sum_{1\le i'<j'\le k}V^{i'j'}(F^{1/4};r). $$
Plugging that bound into \eqref{ftrr}, and recalling $H\ge F^{1/4}$, we obtain 
\begin{equation}\label{back}\Ebb[ |R^i_r-   R^j_r - (x_i-x_j)|^q]^{1/q} \leq C r^{1/2}  + C\cdot d_{\mathrm{SRI}}(\br,\br_0)^{1/4}\sum_{1\le i'<j'\le k} \Ebb [ V^{i'j'}(H,r)^q]^{1/q}.
\end{equation}
uniformly over $\br\in \mathscr M$ and $\x\in I^k.$

\textbf{Step 6.} In this step, we will establish the theorem. Apply \eqref{back} and then \eqref{preb1} in that order, and one finds that 
\begin{align*}
    \Ebb[ |R^i_r-   R^j_r - (x_i-x_j)|^q]^{1/q} \leq C r^{1/2} + C\cdot d_{\mathrm{SRI}}(\br,\br_0)^{1/4} \sum_{1\le i'<j'\le k} \Ebb[ |R^{i'}_r-R^{j'}_r - (x_i-x_j)|^q]^{1/q} .
\end{align*}
Now sum the left side over all indices $(i,j)$ with $1\le i<j\le k$, and move all instances of $\Ebb[ |R^i_r-   R^j_r - (x_i-x_j)|^q]^{1/q}$ to the left side of the expression, and we obtain $$\big(1-\tfrac12 k(k-1) \cdot C\cdot d_{\mathrm{SRI}}(\br,\br_0)^{1/4} \big) \sum_{1\le i<j\le k}\Ebb[ |R^i_r-   R^j_r - (x_i-x_j)|^q]^{1/q} \leq Cr^{1/2} .$$ 
Now make $\epsilon$ smaller (if needed) so that $C \cdot \epsilon^{1/4} \cdot \frac12 k(k-1) \leq 1/2$. Then, for $d_{\mathrm{SRI}}(\br,\br_0)\leq \epsilon$, one obtains from the previous expression (after absorbing constants on the right side into some larger constant) that 
$$\sum_{1\le i<j\le k}\Ebb[ |R^i_r-   R^j_r - (x_i-x_j)|^q]^{1/q} \leq Cr^{1/2} .$$
This proves \eqref{e1''}, and then \eqref{e2''} follows immediately from \eqref{e1''} and \eqref{preb1}.
\end{proof}

Now we are in a position to prove Theorem \ref{inv00}.



\begin{proof}[Proof of Theorem \ref{inv00}] Without loss of generality, we can assume that the covariance matrix of the base measure $\nu_0$ of $\br_0$ is equal to the $d\times d$ identity matrix. If not, one can always replace $I$ by some linear transform of that lattice. We break the proof into several steps.

\textbf{Step 1. } In this step, we find a useful family of martingales associated to each of the coordinates of the process. Define the $j^{th}$ coordinate function $   \pi_j: I^k \to I$ by $(x_1,...,x_k)\mapsto x_j.$ Notice that for each $1\le j \le k$ that the process $$M^{j,\br}_r := R^j_r - \sum_{s=0}^{r-1} (P^{\br}-\mathrm{Id})   \pi_j(\mathbf R_s)$$ is a $\Pbb$-martingale for all $\x\in I^k$ and $\br\in \Ek$. Explicitly we have that \begin{equation}\label{expre}(P^{\br}-\mathrm{Id})   \pi_j(\x) = \int_{I^k} (y_j-x_j) \br(\x,\dr\y).\end{equation}

    \textbf{Step 2.} Fix $j\in \{1,...,k\}$, and let $R^j$ as usual denote the $j^{th}$ coordinate of the canonical process $\mathbf R$ on $(I^k)^{\mathbb Z_{\ge 0}}$. Recall the quantity $\mathcal S_q(M,P)$ from Lemma \ref{cool}. In this step we will show that for each $q\ge 1$, \begin{equation}\label{b4}\sup_{\br \in \Ek} \sup_{\x\in I^k} \mathcal S_q(\mathbf M^{\br};\Pbb)<\infty.\end{equation} First note that \begin{equation}\label{b4a}\sup_{\br \in \Ek}\sup_{\x\in I^k} |(P_{\br} -\mathrm{Id})   \pi_j(\x)| \leq \sup_{\br \in \Ek} \sup_{\x\in I^k} \int_{I^k} |y_j-x_j| \br(\x,\dr\y)<\infty.\end{equation}
    Then note by the Markov property that \begin{equation}\label{b4b}\sup_{r\ge 1} \sup_{\x\in I^k} \sup_{\br \in \Ek} \Ebb[ |R^j_r-   R^j_{r-1}|^q] \leq \sup_{\x'\in I^k} \sup_{\br \in \Ek} \int_{I^k} |y_j-x_j|^q \br(\x',\dr\y)<\infty .\end{equation} Combining \eqref{b4a} and \eqref{b4b}, we immediately obtain \eqref{b4}.
    
    \textbf{Step 3.} In this step we study the processes given by $$\mathcal B_N(t):= N^{-1/2}\mathbf s_Nt - N^{-1/2}\sum_{r=0}^{Nt-1} (P^{\brn}-\mathrm{Id})   \pi_j(\mathbf R_r).$$
    In particular we will show that if $q\ge 1$ then $\Ebbnn [ |\mathcal B_N(t)-\mathcal B_N(s)|^q]^{1/q} \leq Cd_{\mathrm{SRI}}(\brn,\br_0)^{1/4}|t-s|^{1/2},$ so that $\mathcal B_N$ converges to the zero process in the topology of $C[0,T]$. To prove this, let $    s(\nu):= \int_I    u \nu(\dr    u)$ so that $\mathbf s_N= (   s(\nu_N),...,   s(\nu_N)).$ First notice that by an extremely similar argument as that used in proving \eqref{useful}, one has \begin{equation}\label{useful''}\bigg|    s(\nu_{\br}) - \int_{I^k} (a_i-y_i)\br(\y,\dr\bfa)\bigg|\leq C\cdot d_{\mathrm{SRI}}(\br,\br_0)^{1/4} F\big(\min_{i'<j'}|y_{i'}-y_{j'}|\big)^{1/4}.
    \end{equation} uniformly over $\br\in \Ek$ and $\y\in I^k$. Here $\nu_{\br}$ denotes the base measure of $\br\in \Ek$. Thus for $t\in N^{-1}\mathbb Z_{\ge 0}$, we have
    \begin{align*}|\mathcal B_N(t+N^{-1})-\mathcal B_N(t)| &= N^{-1/2}|   s(\nu_N)- \int_{I^k} (y_j-   R^j_r)\br(\mathbf R_r,\dr\y)|\\ &\leq N^{-1/2} d_{\mathrm{SRI}}(\brn,\br_0)^{1/4} F\big(\min_{i'<j'}|R_r^{i'}-R_r^{j'}|\big)^{1/4} 
    \end{align*} 
    where $r:=Nt$. Thus using the bound $F(\min_\ell u_\ell)^{1/4} \leq \sum_\ell F(u_\ell)^{1/4}$ we find that 
    \begin{align*}
        |\mathcal B_N(t)-\mathcal B_N(s)| &\leq \sum_{u \in [s,t]\cap N^{-1}\mathbb Z_{\ge 0}} |\mathcal B_N(u+N^{-1})-\mathcal B_N(u)| \\ &\leq d_{\mathrm{SRI}}(\brn,\br_0)^{1/4}\sum_{r \in [Ns,Nt]\cap\mathbb Z_{\ge 0}} \sum_{1\le i'<j'\le k} F(|R^{i'}_r-R^{j'}_r| )^{1/4} 
    \end{align*}
    Thus taking expectation and applying \eqref{e2''} we find that \begin{equation}\label{bbd}\Ebbnn [ |\mathcal B_N(t)-\mathcal B_N(s)|^q]^{1/q} \leq Cd_{\mathrm{SRI}}(\brn,\br_0)^{1/4}|t-s|^{1/2} .\end{equation}
\textbf{Step 4.} Combining the results of \eqref{b4} and \eqref{bbd} and Lemma \ref{cool}, we immediately obtain tightness of the rescaled process $\mathbf X_N$ in the space $C[0,T]$. It remains to identify the limit point as Brownian motion. In this step, we will begin by showing that any limit point $\mathbf U = (U^1,...,U^k)$ is necessarily a martingale in the joint filtration of all coordinates, and moreover each coordinate $U^j$ is \textit{individually} distributed as a standard Brownian motion started from $x_j$ where $\boldsymbol x=(x_1,...,x_k)$ is as in the theorem statement.

Write $\mathbf X_N = (X_N^1,...,X_N^j)$ in coordinates. By \eqref{bbd}, the difference $X_N^j(t) - N^{-1/2} M^{j,\brn}_{Nt}$ converges in probability to zero in the topology of $C[0,T]$, thus it suffices to show that if $\mathbf U$ is any limit point of the processes $(N^{-1/2} M^{j,\brn}_{Nt})_{1\le j \le k,t\ge 0}$ then $\mathbf U$ is a martingale and each coordinate is individually a Brownian motion. The martingality of the limit point is clear, since martingality is preserved by limit points as long as one has uniform integrability as guaranteed by the $L^q$ bounds \eqref{b4}.

Now we will show that each coordinate of the limit point $\mathbf U$ is a standard Brownian motion. To do this, we need to study the quadratic variations of the martingales $(N^{-1/2} M^{j,\brn}_{Nt})_{t\ge 0}$. Fix some vector $\bfa := (   a_1,...,   a_k)\in \mathbb R^{kd}$ and define the martingale $$Z^{\bfa, \brn}_r := \sum_{j=1}^k    a_k \bullet M^{j,\brn}_r,$$ where $\bullet $ is the dot product in $\mathbb R^d$. Note that $(N^{-1/2} Z^{\bfa,\brn}_{Nt})_{t\ge 0}$ is a tight family of processes in $C[0,T]$ simply by Lemma \ref{cool} and \eqref{b4a}. 
Note that the process $$Y^{\bfa,\br}_r:= (Z^{\bfa,\br}_r)^2 - \sum_{s=0}^{r-1}\Ebb [ (Z^{\bfa,\br}_{s+1}-Z^{\bfa,\br}_s)^2|\mathcal F_s]$$ is also a $\Pbb$-martingale. We claim that $(N^{-1} Y^{\bfa,\brn}_{Nt})_{t\ge 0}$ is a tight family of processes in $C[0,T]$ as $N\to \infty$. Indeed, one easily verifies that $\big (N^{-1}(M^{j,\brn}_{Nt})^2\big)_{t\ge 0}$ is tight and satisfies the same $L^q$ estimates as the processes $\big (N^{-1/2}M^{j,\brn}_{Nt}\big)_{t\ge 0}$ on any compact time interval $t\in [0,T]$, simply because it is the square of a process satisfying such bounds as shown in Steps 2 and 3. On the other hand, by Minkowski's inequality, we also have that \begin{align*}\bigg\| N^{-1} \sum_{r=Ns}^{Nt} \Ebbnn [ (M^{j,\brn}_{r+1}-M^{j,\brn}_r)^2|\mathcal F_s]\bigg\|_{L^q(\mathbf P_{\x_N}^{(\beta_N,k)})}& \leq N^{-1} \sum_{r=Ns}^{Nt} \big\|M^{j,\brn}_{r+1}-M^{j,\brn}_r\big\|_{L^{2q}(\mathbf P_{\x_N}^{\brn})}\\ &\leq \bigg(\sup_{N\ge 1} \mathcal S_{2q}(M^{j,\brn};\mathbf P_{\x_N}^{\brn}) \bigg)|t-s|,\end{align*}
where the supremum is finite by \eqref{b4}. Summarizing these bounds, we have that for some $q> 2$ we have $\Ebbnn \big[ \big(N^{-1}| Y^{\bfa,\brn}_{Nt}- Y^{\bfa,\brn}_{Ns}|\big)^q\big]^{1/q} \leq C|t-s|^{1/2}$ for all $s,t$ in a compact interval, thus completing the proof that $(N^{-1} Y^{\bfa,\brn}_{Nt})_{t\ge 0}$ is a tight family of processes in $C[0,T]$ as $N\to \infty$. 

We will now study the joint limit as $N\to \infty$ under the measures $\mathbf P^{\brn}_{\x_N}$ of the pair of processes $(N^{-1/2} Z^{\bfa,\brn}_{Nt}, N^{-1} Y^{\bfa,\brn}_{Nt})_{t\ge 0}.$ Consider any joint limit point $(U,V).$ Then the pair are martingales in their joint filtration, since martingality is preserved by limit points as long as one has the $L^q$ bounds as shown above. Let us study how $U$ and $V$ must be related.  From \eqref{expre} and \eqref{useful''}, it is immediate that 
\begin{equation}\label{uu2}\Ebbnn \big[ \big(( P^{\brn} -\mathrm{Id})   \pi_j(\mathbf R_r)-   s(\nu_N)\big)^2 \big] \leq Cd_{\mathrm{SRI}}(\br_0,\brn)^{1/4},
\end{equation}
which implies from the definition of the martingales $M^{j,\br}$ that \begin{equation}\label{cob1}\Ebbnn \bigg[ N^{-1}\sum_{s=0}^{Nt-1} \bigg| \Ebbnn \bigg[ (Z_{r+1}^{\bfa,\brn} -Z_r^{\bfa,\brn})^2 - \bigg(\sum_{j=1}^k    a_j \bullet (   R^j_{r+1}-   R^j_r-   s(\nu_N))\bigg)^2\bigg| \mathcal F_r\bigg]\bigg| \bigg] \leq Cd_{\mathrm{SRI}}(\br_0,\brn)^{1/4} \stackrel{N\to\infty}{\longrightarrow} 0.\end{equation}
To prove this, one expands out \begin{align*}(Z_{r+1}^{\bfa,\brn}& -Z_r^{\bfa,\brn})^2 = \bigg(\sum_{j=1}^k    a_j \bullet (   R^j_{r+1}-   R^j_r-   s(\nu_N))\bigg)^2 \\&+ 2 \bigg(\sum_{j=1}^k    a_j \bullet (   R^j_{r+1}-   R^j_r-   s(\nu_N))\bigg)\big(( P^{\brn} -\mathrm{Id})   \pi_j(\mathbf R_r)-   s(\nu_N)\big)+\big(( P^{\brn} -\mathrm{Id})   \pi_j(\mathbf R_r)-   s(\nu_N)\big)^2,
\end{align*}
then one uses \eqref{uu2}. With \eqref{cob1} established, set $\mathrm{var}^{\bfa}_N:= \int_{I^k} \big( \sum_{j=1}^k    a_j \bullet (   u_j -    s(\nu_N)) \big)^2 \nu_N^{\otimes k}(\dr   u).$ Then using similar arguments to \eqref{useful''} we also have that \begin{align*}\bigg|\Ebbnn & \bigg[ \bigg(\sum_{j=1}^k    a_j \bullet (   R^j_{r+1}-   R^j_r-   s(\nu_N))\bigg)^2\bigg| \mathcal F_r\bigg]\;\; - \;\; \mathrm{var}^{\bfa}_N \bigg| \\&= \bigg|\int_{I^k} \bigg(\sum_{j=1}^k    a_j \bullet (y_j-   R^j_r-   s(\nu_N))\bigg)^2\brn(\mathbf R_r,\dr\y)\;\; - \;\; \mathrm{var}^{\bfa}_N \bigg| \le C\cdot d_{\mathrm{SRI}}(\brn,\br_0)^{1/4} F\big(\min_{i'<j'}|R_r^{i'}-R_r^{j'}|\big)^{1/4}.
\end{align*}
This implies that \begin{align}\notag \Ebbnn \bigg[ N^{-1}\sum_{s=0}^{Nt-1} \bigg|\Ebbnn & \bigg[ \bigg(\sum_{j=1}^k    a_j \bullet (   R^j_{r+1}-   R^j_r-   s(\nu_N))\bigg)^2\bigg| \mathcal F_r\bigg]\;\; - \;\; \mathrm{var}^{\bfa}_N \bigg|\bigg] \\ \notag & \leq C\cdot d_{\mathrm{SRI}}(\brn,\br_0)^{1/4} .\Ebbnn \bigg[ N^{-1}\sum_{s=0}^{Nt-1}\sum_{1\le i'<j'\le k}F\big( |R^{i'}_r-R^{j'}_r| \big)^{1/4} \bigg]\\&\stackrel{N\to\infty}{\longrightarrow} 0.\label{cob2}\end{align}
where we applied \eqref{e2''} and $d_{\mathrm{SRI}}(\br,\brn)\to 0$ in the last line. Combining \eqref{cob1} and \eqref{cob2}, and recalling that the covariance matrix of the base measure $\nu$ of $\br_0$ is $\mathrm{Id}_{d\times d}$, we have that $\mathrm{var}^{\bfa}_N \to |\bfa|^2$ as $N\to\infty$ we find that \begin{equation}\label{cob3}\lim_{N\to \infty} \Ebbnn \bigg[ \bigg||\bfa|^2 t\;\;\;-\;\;\;N^{-1}\sum_{s=0}^{Nt-1}  \Ebbnn \big[ (Z_{r+1}^{\bfa,\brn} -Z_r^{\bfa,\brn})^2\big| \mathcal F_r\big]\bigg| \bigg] =0.\end{equation}
Using \eqref{cob3}, one may immediately conclude that the joint limit point $(U,V)$ of the pair of processes $(N^{-1/2} Z^{\bfa,\brn}_{Nt}, N^{-1} Y^{\bfa,\brn}_{Nt})_{t\ge 0}$ must satisfy $V_t=U_t^2-|\bfa|^2t$. Since both are continuous martingales as explained above, and since this holds true for all choices of $\bfa \in \mathbb R^{kd}$, it must be true that $(U^1,...,U^k)$ is a standard Brownian motion in $\mathbb R^{kd}$.
\end{proof}

\section{Proof of Theorem \ref{anti}: Heat kernel bounds for SRI chains}\label{appendix:b}

\begin{lem}\label{rwbd}
    Let $m\ge 1$ and $z_0>0$ and $K>0$. Let $\mathscr A_\gamma (m)$ denote the set of all probability measures $\rho$ on $\mathbb R^d$ with mean zero, with covariance matrix $\mathrm{Id}_{d\times d}$, and satisfying $\int_{\mathbb R^d} |   x|^3\rho(\dr    x)<\gamma $. There exists a constant $C=C(d,z_0,K)>0$ such that for all $r\in \mathbb Z_{\ge 0}$ and $   y\in\mathbb R^d$ one has
    $$\sup_{\rho\in \mathscr A_\gamma (d)} \nu^{* r}(\{   x: |   x-   y|\leq 1\}) \leq C r^{-d/2} \big( 1+r^{-1/2} |   y-   x|\big)^{-K}.$$ Here $\rho^{* r}$ denotes the $r$-fold convolution of $\rho$ with itself.
\end{lem}

\begin{proof} We will effectively follow the proof of ``Esseen's concentration inequality," which is actually an anti-concentration inequality. Choose $\gamma=\gamma(z_0,M)>0$ (small enough) so that $|\hat \rho(\xi)|\leq e^{-\frac14|\xi|^2}$ for $|\xi|<\gamma.$ This is permissible since the covariance matrix of $\rho$ is $\mathrm{Id}_{d\times d}$. Next, choose a smooth even function $H$ supported on $[-\gamma,\gamma]$, and let $U$ be the Fourier transform of $H*H$. By Plancherel and then a change of variable, we have that $$\int_I U(x-y) \rho^{*r}(\dr x) = \int_{\mathbb R^d} e^{i\xi \bullet y} H(\xi)^2 \hat\rho(\xi)^r  \dr\xi =r^{-d/2} \int_{\mathbb R^d} e^{i\xi\bullet y/\sqrt{r}} |H(\xi/\sqrt r)|^2 \hat \rho (\xi/\sqrt r)^r \dr\xi.$$We need to show that the integral is bounded above by $C\big( 1+ r^{-1/2}|y|\big)^{-K}$ for some large enough constant $C=C(M,z_0,K).$ To do this, we will actually show that the integral is bounded above by $C\min\{1, r^{K/2}|y|^{-K}\},$ which is an equivalent bound. The upper bound of 1 is clear since $H$ being supported on $[-\gamma,\gamma]$ implies $|H(\xi/\sqrt r)|^2 |\hat \rho (\xi/\sqrt r)|^r \leq Ce^{-\frac14|\xi|^2}.$ For the upper bound of $r^{K/2}|y|^{-K}$, integrate by parts exactly $K$ times to yield that 
    $$\int_{\mathbb R^d} e^{i\xi\bullet y/\sqrt{r}} |H(\xi/\sqrt r)|^2 \hat \rho (\xi/\sqrt r)^r \dr\xi = \bigg(-\frac{i\sqrt r}{|y|}\bigg)^K \int_{\mathbb R^d}e^{i\xi\bullet y/\sqrt{r}} \partial_\xi^K\bigg( |H(\xi/\sqrt r)|^2 \hat \rho (\xi/\sqrt r)^r \bigg)\dr\xi.$$
    Now we will take absolute values on both sides. If we can prove that $\big|\partial_\xi^K\big( |H(\xi/\sqrt r)|^2 \hat \rho (\xi/\sqrt r)^r \big)\big| \leq C (1+|\xi|^m) e^{-\frac18 |\xi|^2}, $ for some constant $C=C(K,z_0,M)>0$, then this would complete the proof of the claim. Since $H$ is smooth and compactly supported, it is clear that the first $K$ derivatives of the function $\xi \mapsto |H(\xi/\sqrt r)|^2$ are bounded independently of $\xi$ and $r$, thus by the product rule for derivatives, we just need to obtain a sufficiently strong bound on $\big|\partial_\xi^K\big( \hat \rho (\xi/\sqrt r)^r \big)\big|$, that is independent of $r\ge 0$ and $|\xi|<\gamma\sqrt r$. 
    
    We will now show that $\big|\partial_\xi^K\big( \hat \rho (\xi/\sqrt r)^r \big)\big| \leq C (1+|\xi|^K) e^{-\frac18 |\xi|^2} $ for $|\xi|\leq \gamma$, and then we will be finished. When $K=0$, we already know this to be true by construction. Abbreviate $\phi:=\hat \rho$, and $\phi^{(\ell)} :=\partial_\xi^\ell \phi$. Since $\rho$ is assumed to have moments of all orders, we have $\max_{1\le \ell\le m} \|\phi^{(\ell)}\|_{L^\infty}\leq C.$ Thus at each instance of differentiation, notice that the only seemingly problematic terms that could arise are derivatives of terms of the form $\phi(\xi/\sqrt r)^{r-\ell}$ with $\ell< K$. The derivative of such a term is $(r-\ell)\phi(\xi/\sqrt r)^{r-\ell-1} \phi'(\xi/\sqrt r)\cdot r^{-1/2}.$ Since the measure $\rho$ has mean zero, one has that $|\phi'(\xi)|\leq C |\xi| $ for $|\xi|<\gamma$. We thus obtain that $(r-\ell)\phi(\xi/\sqrt r)^{r-\ell-1} \phi'(\xi/\sqrt r)\cdot r^{-1/2}\leq C|\xi| \phi(\xi/\sqrt r)^{r-\ell-1}\leq C|\xi| e^{-\frac{r-\ell-1}{r} \cdot\frac14 |\xi|^2}$ for $|\xi|< \gamma\sqrt r$. Without loss of generality, we can assume that $r\ge 2K$ since controlling some finite set of $r$-values is vacuous (just making $C$ larger if necessary), which means $\frac{r-\ell-1}{r}\geq \frac12$ for $\ell<K,$ so that $e^{-\frac{r-\ell-1}{r} \cdot\frac14 |\xi|^2}\leq e^{-\frac18|\xi|^2}.$
\end{proof}

\begin{lem}\label{cool2}Consider any adapted sequence of random variables $(X_n)_{n\ge 0}$ defined on any filtered probability space $(\Omega, (\mathcal F_r)_{r\in \mathbb Z_{\ge 0}}, P),$ with $M_0=0$. Assume that we have $\mathcal E_{\lambda_0} (M;P):=\sup_{n\ge 1} \| E [e^{4\lambda_0|X_{n+1} -X_n|}|\mathcal F_n] \big\|_{L^\infty(\Omega,\mathcal F,P)}<\infty,$ for some $\lambda_0>0$. Then for $|\lambda|<\lambda_0$ we have that 
\begin{equation}\label{mart1}E[ e^{\lambda X_r}] \leq E\bigg[ \prod_{s=0}^{r-1} E[e^{2\lambda (X_{s+1}-X_s)}|\mathcal F_s]\bigg]^{1/2}.
\end{equation}
If in particular $X$ is a martingale, then for $|\lambda|<\lambda_0$ we have that 
\begin{align}
    \label{mart2} \notag E[ e^{\lambda X_r}] &\leq E\bigg[ \prod_{s=0}^{r-1} \big( 1+4 \lambda^2 \mathcal E_{\lambda_0}(M,P)^{1/2}  E[|X_{s+1}-X_s|^4|\mathcal F_s]^{1/2}\big)\bigg]^{1/2} \\& \leq E\bigg[e^{4 \lambda^2 \mathcal E_{\lambda_0}(M,P)^{1/2}\sum_{s=0}^{r-1}  E\big[|X_{s+1}-X_s|^4\big|\mathcal F_s\big]^{1/2}}\bigg]^{1/2}
\end{align}
\end{lem}

\begin{proof}
    Fix $\lambda>0$, and let $W_r:=\sum_{n=0}^{r-1} \log E[e^{2\lambda (X_{n+1}-X_n)}|\mathcal F_n],$ and notice that $e^{2\lambda X_r - W_r}$ is a $P$-martingale. Therefore by Cauchy-Schwartz we have that $$E[ e^{\lambda X_r}] =E[e^{\lambda X_r -\frac12 W_r} e^{\frac12W_r}] \leq E[e^{2\lambda X_r-W_r}]^{1/2} E[e^{W_r}]^{1/2} = 1\cdot E[e^{W_r}]^{1/2}, $$ which already proves \eqref{mart1}. To prove the first bound in \eqref{mart2}, we just need to show that for $|\lambda|<\lambda_0$ one has $$E[e^{2\lambda (X_{s+1}-X_s)}|\mathcal F_s] \leq 1 +4\lambda^2 E[ e^{4|\lambda||X_{s+1}-X_s|}|\mathcal F_s]^{1/2} \cdot E[|X_{s+1}-X_s|^4|\mathcal F_s]^{1/2}.$$ By Cauchy-Schwartz it suffices to show the weaker bound $$E[e^{2\lambda (X_{s+1}-X_s)}|\mathcal F_s] \leq 1 +4\lambda^2 E[ (X_{s+1}-X_s)^2 e^{2|\lambda (X_{s+1}-X_s)| }|\mathcal F_s].$$ But this is immediate from the elementary inequality $|e^u-1-u| \leq u^2 e^{|u|},$ and the assumption that $E[X_{s+1}-X_s|\mathcal F_s]=0.$ 
    
    The second bound of \eqref{mart2} then just follows from the trivial bound $1+a\leq e^a,$ applied to each term in the product.
\end{proof}

\begin{lem} \label{lemm:antid1} Consider decreasing functions $F, H:[0,\infty)\to[0,\infty)$ of exponential decay at infinity. Let $q\ge 2, M>0$ and let $\br_0\in\Ek$ be centered and $\delta$-repulsive for some $\delta>0$. Then there exists $\epsilon>0$ such that for all $\lambda>0$ one has that
\begin{equation}
    \label{f2b'} \sup_{\substack{\br \in \Ek\\d(\br,\br_0)<\epsilon }} \sup_{\x\in I^k} \sup_{r\in\mathbb Z_{\ge 0}} \Ebb [ e^{ \lambda r^{-1/2} \sum_{s=0}^{r-1} \sum_{1\le i'<j'\le k} H(|R^{i'}_s-R^{j'}_s|) }]<\infty.
\end{equation} 
\end{lem}

We remark that this will only serve as an initial bound, not the optimal one by any means (as is indicated by the fact that the bound does not depend on $d$ in any way). More precisely, it is only optimal in $d=1$, or more generally when the covariance matrix of $\nu$ has rank 1, i.e., $\nu$ is supported on a dimension-one subgroup of $I$.

\begin{proof}
     One first Taylor expands the exponential as an infinite series. It then suffices to show that there exists $C>0$ such that uniformly over all $d(\br,\br_0)<\epsilon$, all $m\in \mathbb N$, and all integers $r\ge 0$ one has \begin{equation}\label{f2b}\sup_{\x \in I^k}\Ebb \bigg[ \bigg(\sum_{s=0}^{r-1} \sum_{1\le i'<j'\le k} H(|R^{i'}_s-R^{j'}_s|)\bigg)^m \bigg] \leq C^m \sqrt{m!} \cdot  r^{m/2} , 
\end{equation}
with the understanding that $s_0=0$. Multiplying both sides by $r^{-m/2}\gamma^{2m}/m!$ and summing over $m$ will then imply the claim.

To prove \eqref{f2b}, use induction on $m$. For the $m=1$ case, the bound in \eqref{f2b} is immediate from \eqref{e2''} with $q=1$. 

For the inductive step, let us define $\varrho_m$ to be the optimal constant so that the left side of \eqref{f2b} is bounded above by $\varrho_m r^{m/2}$ uniformly over all $r\ge 0$. We have 
\begin{align*}
           h_m(r;\x)&:= \sum_{s_1,s_2,...,s_m=0}^{r-1} \Ebb \bigg[ \prod_{j=1}^{m} H(|R^{i'}_{s_j}-   R^{j'}_{s_j}|) \bigg] \\&=  \sum_{\substack{0\le s_1,s_2,...,s_m<r\\ \mathrm{no\;repeated\;index}}} \Ebb \bigg[ \prod_{j=1}^{m} H(|R^{i'}_{s_j}-   R^{j'}_{s_j}|) \bigg] +\sum_{\substack{0\le s_1,s_2,...,s_m<r\\\mathrm{repeated\;index}}} \Ebb \bigg[ \prod_{j=1}^{m} H(|R^{i'}_{s_j}-   R^{j'}_{s_j}|) \bigg] \\&\leq m! \cdot \Ebb \bigg[ \sum_{0\leq s_1 < ...<s_m< r} \prod_{j=1}^{m} H(|R^{i'}_{s_j}-   R^{j'}_{s_j}|) \bigg] + \|H\|_{\infty} h_{m-1}(r;\x) \\&= m! \cdot  \Ebb \bigg[ \sum_{s_1=0}^{r-1} H(|R^{i'}_{s_1}-   R^{j'}_{s_1}|)\cdot \frac1{(m-1)!} h_{m-1}(r-s_1;\mathbf R_{s_1})] \bigg] + \|H\|_{\infty} h_{m-1}(r;\x) \\ &\le m\varrho_{m-1}   \Ebb \bigg[ \sum_{s_1=0}^{r-1} H(|R^{i'}_{s_1}-   R^{j'}_{s_1}|) \cdot  (r-s_1)^{(m-1)/2}\bigg] +  \|H\|_{\infty} h_{m-1}(r;\x) 
        \end{align*}
        We used the Markov property in the fourth line. Now for the first term on the right side, use summation by parts 
        and we obtain that $$\sum_{s=0}^{r-1} H(|R^{i'}_{s}-   R^{j'}_{s}|)\cdot  (r-s)^{(m-1)/2} = \sum_{s=0}^{r-1} \big[ (r-s)^{(m-1)/2}-(r-1-s)^{(m-1)/2}\big]\cdot \bigg(\sum_{t=0}^{s-1} H(|R^{i'}_{t}-   R^{j'}_{t}|)\bigg).$$ Now note that $(r-s)^{(m-1)/2}-(r-1-s)^{(m-1)/2}$ can be bounded above by $C (\frac{m-1}{2}) (r-s)^{(m-3)/2} $, then apply the expectation $\Ebb$ over the last expression, then use the $m=1$ case once again, and we see that the expectation is bounded above by $C (\frac{m-1}{2})\cdot \sum_{s=0}^{r-1} (r-s)^{(m-3)/2}s^{1/2}  $ which is equal to $r^{m/2} $ multiplied by a Riemann sum approximation for $\int_0^1 (1-u)^{(m-3)/2}u^{1/2}du $. Notice that the latter is $O(m^{-3/2})$ as $m\to \infty.$ This whole discussion yields $$\Ebb \bigg[ \sum_{s_1=0}^{r-1} H(|R^{i'}_{s_1}-   R^{j'}_{s_1}|) \cdot  (r-s_1)^{(m-1)/2}\bigg] \leq Cm^{-1/2}r^{m/2}$$ for some absolute constant $C$ not depending on $m,r$.

        As for the second term $ \|H\|_{\infty} h_{m-1}(r;\x) $, we can absorb $\|H\|_{\infty}$ into some absolute constant $C$ since $F$ is just a fixed function. By the inductive hypothesis, $ C h_{m-1}(r;\x) $ may then be bounded above by $ C\varrho_{m-1} r^{(m-1)/2} \leq C \varrho_{m-1} r^{m/2}. $ 
        
        The entire above discussion yields the relation $\varrho_m \leq Cm^{1/2} \varrho_{m-1}$ for some absolute constant $C$ independent of $m\in\mathbb N,$ which easily implies \eqref{f2b}. 
\end{proof}

\begin{defn}
    Let $k,d\in\mathbb N$. Suppose that $\br\in \Ekk$, and define 
    \begin{equation}
        \mathrm{Dr}_{\br}(\x):= \bigg|\int_{I^k} (\y-\x) \br(\x,\dr\y) - \mathbf s(\nu_{\br})\bigg|.
    \end{equation}
    Note that this is 0 for any centered chain $\br$. 
    Here $\mathbf s(\nu_{\br}) = \int_{I^k} \mathbf u \nu_{\br} (\dr u)$ is the deterministic drift vector as used in Theorem \ref{inv00}. 
\end{defn}

Typically, the underlying chain will be clear, and we will just write $\mathrm{Dr}_{\br}(\x) = \mathrm{Dr}(\x)$. 

The following proposition will be an intermediate step towards the proof of Theorem \ref{anti}.

\begin{prop}\label{trunc}
    Fix $K,\delta,\beta >0$, and let $\br_0 \in \Ekk$ be centered and $\delta$-repulsive. Then there is a $\epsilon=\epsilon(F,z_0,M,\delta)>0$ and $C=C(F,\sigma,\delta,K)>0$ such that uniformly over $\x\in I^k$ and $\y\in \mathbb R^{kd}$ and integers $r\ge 0$ one has the bound: We have $$\Ebb \bigg[ \bigg( 1+r^{-\beta  } |\mathbf R_r-r\mathbf s_N-\y| \bigg)^{-K} e^{-r^{-\beta}\sum_{s=0}^{r-1}\mathrm{Dr}(\mathbf R_s)}\bigg] \leq C^m r^{-\frac{kd}2 \big( 1-2\beta \big)} \big( 1+r^{-1/2} |\y-\x|\big)^{-K}.$$
\end{prop}

Note that if $\beta=0$, this would prove the first bound in Theorem \ref{anti}, but unfortunately, we cannot obtain this strong of a bound just yet. The point is that we can still get close, with $\beta$ as close to 0 as we want.

Note that in $d=1$, we want $\beta >1/2$ on the left-hand side to kill the sum in the exponential, which is of order $r^{1/2}$. However, such a $\beta$ gives a bad bound on the RHS.

\begin{proof} Throughout the proof, we will assume as always that $q,M,F$ as well as $\br_0 \in \Ekk$ are fixed, with $\br_0$ centered and $\delta$-repulsive. 

\textbf{Step 1.} In this step, we construct a coupled Markov chain under which the first marginal is the SRI chain $\br$ of interest, and the second marginal just a random walk with increment law $\nu_{\br}^{\otimes k}$, and the increments of both coordinates are equal with high probability away from the boundaries of the ``$(k-1)d$-dimensional Weyl chamber."

For now, fix some $\br\in \Ek$. Using the total variation bound in tshe definition of SRI chain, consider any Markov kernel $\boldsymbol{v}_{\br}((\x_1,\x_2),(\dr\y_1,\dr\y_2))$ on $I^{2k}$ with the following properties:
\begin{itemize}
    \item The first marginal of $\boldsymbol{v}_{\br}((\x_1,\x_2),\bullet)$ is equal to $\br(\x_1,\bullet)$.
    \item The second marginal of $\boldsymbol{v}_{\br}((\x_1,\x_2),\bullet)$ is $\nu_{\br}^{\otimes k}(\x_2-\bullet).$
    \item We have that $\boldsymbol{v}_{\br}((\x_1,\x_2), \{ (\y_1,\y_2): \y_2-\y_1\neq \x_2-\x_1 \}) \leq 2F(\min_{i<j} |   x_1^i-   x_1^j|)$ where $\x_1=(   x^1_1,    x_1^2,...,   x_1^k)$ in coordinates.
\end{itemize}
Such a coupling is possible simply by the coupling definition of total variation distance and the conditions in the definition of SRI chains. Denote by $(\mathbf R_r,\mathbf S_r)$ the canonical process on $(I^{2k})^{\mathbb Z_{\ge 0}},$ and let $\mathbf E^{\boldsymbol{v}_{\br}}_{(\x_1,\x_2)}$ denote the canonical measure on that space, under which $(\mathbf R_r,\mathbf S_r)$ is distributed as the Markov chain started from $(\x_1,\x_2)$ with transition kernels $\boldsymbol{v}_{\br}((\x_1,\x_2),\bullet).$ Under $\mathbf E^{\boldsymbol{v}_{\br}}_{(\x_1,\x_2)}$, notice that the marginal law of $\mathbf R_r$ is simply that of the SRI chain associated to $\br$ started from $\x_1$, and the marginal law of $\mathbf S_r$ is that of a random walk on $I^k$ with increment law $\nu_{\br}^{\otimes k}$. 

\textbf{Step 2.} Henceforth denote $$\mathbf X_r:=\mathbf R_r-\mathbf S_r.$$ If $q \ge 1$ and $s\in\mathbb Z_{\ge 0}$, we may use Cauchy-Schwartz to say that \begin{align}
    \notag \mathbf E^{\boldsymbol{v}_{\br}}_{(\x_1,\x_2)}\big[ |\mathbf X_{s+1} - \mathbf X_{s}|^{q/2}\big|\mathcal F_s \big] &\leq \notag  \mathbf E^{\boldsymbol{v}_{\br}}_{(\x_1,\x_2)} \big[|\mathbf X_{s+1}- \mathbf X_s|^q|\mathcal F_s]^{1/2} \mathbf P^{\boldsymbol{v}_{\br}}_{(\x_1,\x_2)}\big(\mathbf X_{s+1}-\mathbf X_s\ne 0\big|\mathcal F_s\big)^{1/2} \\& \leq (Cq)^{q/2}  \notag \mathbf P^{\boldsymbol{v}_{\br}}_{(\x_1,\x_2)}\big(\mathbf X_{s+1}-\mathbf X_s\ne 0\big|\mathcal F_s\big)^{1/2} \\& \leq \notag (Cq)^{q/2} F\big(\min_{i'<j'} |R^{i'}_s-R^{j'}_s|\big)^{1/2} \\ &\leq (Cq)^{q/2} \sum_{1\le i'<j'\le k}  F\big( |R^{i'}_s-R^{j'}_s|\big)^{1/2} .\label{db3}
\end{align}
Here $C$ is a large constant that independent of $q\ge 1$, $s\in \mathbb Z_{\ge 0}$, $\br \in \Ekk, $ and $\x_1,\x_2\in I^k.$ Note here that we are using the exponential tail bound inherent in Definition \ref{ekk} to bound the $q^{th}$ moment by $(Cq)^q$ for some absolute constant $C.$ In the last line, we are using that $F$ is decreasing.

\textbf{Step 3.}  In this step, we define some useful martingales and obtain some preliminary bounds associated to them. We can define the martingale $\mathbf M_r:= \mathbf X_r-\mathbf D_r,$ where $$\mathbf D_r:= \sum_{s=0}^{r-1} \mathbf E^{\boldsymbol{v}_{\br}}_{(\x_1,\x_2)} \big[\mathbf X_{s+1}- \mathbf X_s\big|\mathcal F_s\big].$$
Note the decomposition 
\begin{equation}\label{rsmd}\mathbf R_r = \mathbf S_r + \mathbf M_r + \mathbf D_r
\end{equation}
which will be crucial to our analysis.

Since $\mathbf M$ is a martingale, Lemma \ref{cool2} (together with the exponential moment assumption inherent in Definition \ref{ekk}) yields for $\lambda>0$ that 
\begin{equation}
    \label{cde}\sup_{\x_1,\x_2\in I^k} \sup_{r\in\mathbb Z_{\ge 0}} \mathbf E^{\boldsymbol{v}_{\br}}_{(\x_1,\x_2)} [ e^{\lambda r^{-1/4} |\mathbf M_r|}] \leq \sup_{\x_1,\x_2\in I^k} \sup_{r\in\mathbb Z_{\ge 0}} \mathbf E^{\boldsymbol{v}_{\br}}_{(\x_1,\x_2)} [ e^{C \lambda^2 r^{-1/2} \sum_{s=0}^{r-1} E[ |\mathbf M_{s+1} -\mathbf  M_s|^4|\mathcal F_s] }]<\infty,
\end{equation}
where the finiteness of the last supremum follows from \eqref{db3} (with $q=8$) and \eqref{f2b'}, and the trivial fact that by definition, one has $|\mathbf M_{s+1}-\mathbf M_s|^4 \leq 2^4 |\mathbf X_{s+1}-\mathbf X_s|^4 + 2^4 \mathbf E^{\boldsymbol{v}_{\br}}_{(\x_1,\x_2)}[|\mathbf X_{s+1}-\mathbf X_s|^4|\mathcal F_s].$

\textbf{Step 4.} In this step, we obtain a preliminary anticoncentration estimate for $|\mathbf R_r|$ under the measure $\mathbf E^{\boldsymbol{v}_{\br}}_{(\x_1,\x_2)}$. In the notation of Theorem \ref{anti}, notice that $|\mathbf D_{s+1}-\mathbf D_s| = \mathrm{Dr}(\mathbf R_s).$ Using the results of the previous steps, we thus have uniformly over all $d_{\mathrm{SRI}}(\br,\br_0)<\epsilon$ and $r\ge 0$ and $   y\in \mathbb R^d$ and $\x_1,\x_2\in I^k$ one has
\begin{align}
    \notag  \\ \notag \mathbf E^{\boldsymbol{v}_{\br}}_{(\x_1,\x_2)}&
    \bigg[ e^{-r^{-1/3}|\mathbf R_r-r\mathbf s_N-\y| - r^{-1/3}\sum_{s=0}^{r-1}\mathrm{Dr}(\mathbf R_s)}\bigg] 
    \\ \notag &\leq \mathbf E^{\boldsymbol{v}_{\br}}_{(\x_1,\x_2)}\bigg[ e^{-r^{-1/3}\big(|\mathbf R_r-r\mathbf s_N-\y| + |\mathbf D_r|\big)}\bigg]  
    \\ \notag &\leq \mathbf E^{\boldsymbol{v}_{\br}}_{(\x_1,\x_2)}
    \bigg[ e^{-r^{-1/3}\big(|\mathbf S_r-r\mathbf s_N-\y| - |\mathbf M_r|\big)}\bigg]\\&= \mathbf E^{\boldsymbol{v}_{\br}}_{(\x_1,\x_2)}
    \bigg[ e^{-r^{-1/3}\big(|\mathbf S_r-r\mathbf s_N-\y| - |\mathbf M_r|\big)}\ind_{\{ |\mathbf M_r| \le r^{1/3}\}} \bigg] + \mathbf E^{\boldsymbol{v}_{\br}}_{(\x_1,\x_2)}
    \bigg[ e^{-r^{-1/3}\big(|\mathbf S_r-r\mathbf s_N-\y| - |\mathbf M_r|\big)}\ind_{\{ |\mathbf M_r| > r^{1/3}\}}\bigg]\notag \\ &\le e \cdot \mathbf E^{\boldsymbol{v}_{\br}}_{(\x_1,\x_2)}
    \big[ e^{-r^{-1/3}|\mathbf S_r-r\mathbf s_N-\y| } \big] + \mathbf E^{\boldsymbol{v}_{\br}}_{(\x_1,\x_2)}
    \big[ e^{2r^{-1/3}|\mathbf M_r|}\big]^{1/2} \mathbf P^{\boldsymbol{v}_{\br}}_{(\x_1,\x_2)} \big(|\mathbf M_r| > r^{1/3}\big)^{1/2}\notag
    \\&\le Cr^{-kd/6}\big( 1+r^{-1/2} |\y-\x|\big)^{-K} + C e^{-r^{\frac1{12}}}.\label{c6a}
    \end{align}
Indeed the second line has already been explained. The fourth line follows from \eqref{rsmd}. The fifth line is just Cauchy-Schwartz. The final line follows from Lemma \ref{rwbd} and the result \eqref{cde} of Step 3.

Now we are going to argue why the stretch-exponential term can be disregarded. Assume without loss of generality that $K>10kd$ since the theorem for some $K$ implies the theorem for any smaller $K$ (and we are claiming that the bound in the theorem holds for all $K>0$). If $|\x-\y|< r^{K}$ then it is clear that $ e^{-r^{\frac1{12}}} \leq Cr^{-kd/6} r^{-(K-\frac12)} \le C'r^{-kd/6} (1+r^{-1/2}|\x-\y|)^{-K}$. On the other hand, if $|\x-\y|\ge r^K$ then  $\Ebb [ e^{-r^{-1/3} |\mathbf R_r-r\mathbf s_N-\y|} ] \leq e^{-r^{-1/3}|\x-\y|} \Ebb [ e^{+r^{-1/3} \sum_{s=0}^{r-1} |\mathbf R_{s+1}-\mathbf R_s - \mathbf s_N|}] \leq e^{-r^{-1/3}|\x-\y| + Cr^{2/3}},$ where the expectation being bounded above by $e^{Cr^{2/3}}$ can be deduced by iteratively applying the Markov property and the definition of what it means for $\br$ to be an element of $\Ekk$. Now $|\x-\y|>r^K$ and $K>2$ means that $Cr^{2/3} = Cr^{-4/3} \cdot r^2 \leq Cr^{-4/3} \cdot |\x-\y|.$ Thus the last expression can be bounded above by $e^{-r^{-1/3} |\x-\y| +Cr^{-4/3}|\x-\y|} \leq C'e^{-\frac12 r^{-1/3} |\x-\y|}.$ Again using $|\x-\y| >r^K$ it is not hard to see that this can be bounded above by $Cr^{-kd/6} \big( 1+r^{-1/2} |\x-\y| \big)^{-K} ,$ for example since one has that $e^{-\frac12 r^{-1/3} |\x-\y|} \leq e^{-\frac12 r^{-1/3} |\x-\y|^{2/3}} $ and then using the fact that the exponential decays faster than any polynomial. 

The last two paragraphs combined with \eqref{c6a} yield the bound that $$\Ebb \bigg[ e^{-r^{-1/3}|\mathbf R_r-r\mathbf s_N-\y| - r^{-1/3}\sum_{s=0}^{r-1}\mathrm{Dr}(\mathbf R_s)}\bigg]\le Cr^{-kd/6}\big( 1+r^{-1/2} |\y-\x|\big)^{-K}.$$
Assuming without loss of generality that $K>kd$, we have that $$(1+r^{-1/3} |\y-\x|)^{-K} \leq C \sum_{\bfa\in \mathbb Z^{kd} } e^{ -r^{-1/3} |\bfa-\x|} (1+|\bfa-\y|)^{-K},$$ and for similar reasons we also have $$\sum_{\bfa \in \mathbb Z^{kd} }(1+|\bfa -\y|)^{-K} (1 +r^{-1/2} |\x-\bfa|)^{-K} \leq C (1+r^{-1/2} |\x-\y|)^{-K},$$ thus from the last expression we obtain that 
\begin{equation}
    \label{iter} \Ebb \bigg[ \bigg( 1+r^{-1/3} |\mathbf R_r-r\mathbf s_N-\y| \bigg)^{-K} e^{-r^{-1/3}\sum_{s=0}^{r-1}\mathrm{Dr}(\mathbf R_s)}\bigg]\le Cr^{-kd/6}\big( 1+r^{-1/2} |\y-\x|\big)^{-K}.
\end{equation}

\textbf{Step 5.} Note that \eqref{iter} yields a bound that is \textit{optimal} at scales of size $r^{1/3}$ from the origin, which is strictly better than one can do just using the invariance principle of Theorem \ref{inv00} which only yields good bounds at scales $r^{1/2}.$ In particular since $1/3$ is smaller than $1/2$, the bound \eqref{iter} can be repeatedly \textit{iterated} using the Markov property to yield bounds at better and better scales, approaching (but not quite reaching) scales of size 1 which is what we ultimately want. 

To illustrate the iteration, note that $\frac29 = \frac13 \cdot \frac23$, thus we can use the Markov property and apply \eqref{iter} twice to see that 
 \begin{align*}
     &\Ebb \bigg[ \bigg( 1+r^{-2/9} |\mathbf R_r-r\mathbf s_N-\y| \bigg)^{-K} e^{-r^{-2/9}\sum_{s=0}^{r-1}\mathrm{Dr}(\mathbf R_s)}\bigg] \\ 
     &= \Ebb \bigg[ e^{-r^{-2/9}\sum_{s=0}^{r-r^{2/3}-1 } \mathrm{Dr}(\mathbf R_s)   } \\  & \ \ \ \ \ \ \ \ \cdot \mathbf E^{\br}_{\mathbf R_{r-r^{2/3}}} \bigg[ e^{-r^{-2/9} \sum_{s=0}^{r^{2/3}-1 } \mathrm{Dr}(\tilde{\mathbf R}_s)}\bigg( 1+ r^{-2/9} |\tilde{\mathbf R}_{r^{2/3}} -r^{2/3}\mathbf s_N - ((r-r^{2/3})\mathbf s_N +\y)|\bigg)^{-K} \bigg] \bigg] \\
     &\leq C \Ebb \bigg[ e^{-r^{-1/3}\sum_{s=0}^{r-r^{2/3}-1 } \mathrm{Dr}(\mathbf R_s)   } \cdot (r^{2/3})^{-kd/6} \bigg( 1+ r^{-1/3} |\mathbf R_{r-r^{2/3}} - (r-r^{2/3})\mathbf s_N - \y | \bigg)^{-K} \bigg] \\& \leq C^2 r^{-\frac{kd}6 \big( 1+ \frac23\big)} \big( 1+r^{-1/2} |\y-\x|\big)^{-K}.
 \end{align*}
 Above $\tilde{\mathbf R}$ is an independent copy of the Markov chain. This is now the optimal bound on scale $r^{2/9}$, which is better than the $r^{1/3}$ from \eqref{iter}. Keep doing this iteration repeatedly, and it will yield that for all $m\in \mathbb N$ one has 
 \begin{align*}
     \Ebb \bigg[ \bigg( &1+r^{-\frac13 \big(\frac23\big)^m } |\mathbf R_r-r\mathbf s_N-\y| \bigg)^{-K} e^{-r^{-\frac13 \big(\frac23\big)^m}\sum_{s=0}^{r-1}\mathrm{Dr}(\mathbf R_s)}\bigg] \leq C^m r^{-\frac{kd}6 \big( 1+ \frac23+...+\big(\frac23\big)^m \big)} \big( 1+r^{-1/2} |\y-\x|\big)^{-K}.
     \end{align*}
This implies the proposition for $\beta$ of the form $\frac13 \big(\frac23\big)^m$, with $m\in \mathbb N$, and the case of general $\beta$ follows easily. 
\end{proof}

Notice that we would really like to take $m=\infty$ in the last bound, because $\frac16\big( 1+ \frac23+...+\big(\frac23\big)^m+... \big)=\frac12 $ which is the optimal bound as claimed in the theorem, however the problem is that the constant $C^m$ blows up. In fact, one may convince themselves that only $m=\log\log r$ iterations are needed, but even then $C^m$ blows up like $(\log r)^{C'}$ which is bad despite being extremely close to optimal. Thus some new idea is needed, where we can actually truncate the above bound at some \textbf{fixed} value of $m$, then yield the optimal bound immediately afterwards. Using this idea, we now prove Theorem \ref{anti}.

\begin{proof}[Proof of 
Theorem \ref{anti}]
    \textbf{Step 1.} Fix arbitrary $\beta>0$. Combining the results of Lemma \ref{rwbd} and Proposition \ref{trunc} yields the relation 
    \begin{align}\label{fast}\Ebb [E[e^{-|\mathbf R_r + \mathbf S_{r^{2\beta}}- (r+r^{2\beta})\mathbf s_N -\y|}]e^{ -r^{-\beta} \sum_{s=0}^{r-1} \mathrm{Dr}(\mathbf R_s)}]\leq Cr^{-kd/2} (1+r^{-1/2}|\y-\x|)^{-K},  
    \end{align}
    where $\mathbf S$ is a random walk with increment law $\nu^{\otimes k}$ that is independent of $\mathbf R$, and $E$ denotes an expectation over $\mathbf S$ only (conditional on $\mathbf R)$. In other words, appending a random walk of length $r^{2\beta}$ to the \textit{end} of the trajectory of $\mathbf R$ yields a bound of the desired type. To prove this, condition on $\mathbf R$ and take the expectation just over $\mathbf S$ first, then Lemma \ref{rwbd} shows that $E[e^{-|\mathbf R_r + \mathbf S_{r^{2\beta}}- (r+r^{2\beta})\mathbf s_N -\y|}] \leq Cr^{-\beta k d} ( 1+r^{-\beta} | \mathbf R_r - r\mathbf s_N-\y|)^{-K}.$ Then one applies Proposition \ref{trunc}.

    We claim that appending a random walk of length $r^{2\beta}$ to the \textit{beginning} of the trajectory of $\mathbf R$ also yields a bound of the desired type, as long as the initial point $\x$ is far from the boundary of the Weyl chamber. More precisely, if $\x=(x_1,...,x_k)$ and if $\min_{i<j} |x_i-x_j| >r^{2\beta}$ we claim that 
    \begin{align}\label{fast3} E[ \mathbf E_{\x+\mathbf S_{r^{2\beta}}}^{\br} [ e^{-|\mathbf R_r - (r+r^{2\beta}) \mathbf s_N -\y| - r^{-\beta}\sum_{s=0}^{r-1} \mathrm{Dr}(\mathbf R_s)}] ]\leq Cr^{-kd/2} (1+r^{-1/2}|\y-\x|)^{-K}.
    \end{align}
    To prove this, define the function $h_r(\x,\y):= \Ebb [ e^{-|\mathbf R_r -r\mathbf s_N -\y| -r^{-\beta}\sum_{s=0}^{r-1} \mathrm{Dr}(\mathbf R_s)}],$ so by Proposition \ref{trunc} and the fact that $e^{-|\mathbf R_r-r\mathbf s_N-\y|} \leq e^{-r^{-\beta}|\mathbf R_r-r\mathbf s_N-\y|}, $ we get that $h_r(\x,\y) \leq Cr^{-\frac{kd}2 (1-2\beta)}  (1+r^{-1/2} |\x-\y|)^{-K},$ and thus we can write the left side of \eqref{fast3} as $$E[ h_r (\x+\mathbf S_{r^{2\beta}}, \y+r^{2\beta} \mathbf s_N)] \leq  Cr^{-\frac{kd}2 (1-2\beta)} E[ (1 +r^{-1/2} |\x-\y +\mathbf S_{r^{2\beta}} -r^{2\beta}\mathbf s_N|)^{-K}].$$
    Since we can take $K$ as large as we like, the result of Lemma \ref{rwbd} implies that the expectation on the right side is bounded above by $C(r^{2\beta})^{-kd/2} (1+ r^{-1/2} |\x-\y|)^{-K},$ which in turn implies \eqref{fast3} as desired.

    \textbf{Step 2.} Fix $\beta>0$. We now claim that if $\min_{i<j} |x_i-x_j| >r^{2\beta}$\textbf{ or} if $\min_{i<j} |y_i-y_j| >r^{2\beta}$ then one necessarily has $$\Ebb [ e^{-|\mathbf R_r- r\mathbf s_N -\y| - r^{-\beta}\sum_{s=0}^{r-1} \mathrm{Dr}(\mathbf R_s)} ] \leq Cr^{-kd/2} (1+r^{-1/2} |\x-\y| )^{-K},$$ in other words the optimal bound does indeed hold as long as \textit{at least one} of $\x$ or $\y$ are sufficiently far from the boundary of the Weyl chamber (where sufficiently far means of distance at least $r^{2\beta}$ at time $r$).

    First consider the case where $\min_{i<j} |x_i-x_j| >r^{2\beta}$. In this case, just using the total variation bound in the definition of SRI chain, one can use a coupling to replace the trajectory $\mathbf R|_{[0,r^{2\beta}]}$ with its initial condition plus the trajectory of the random walk $\mathbf S|_{[0,r^{2\beta}]}$, and the probability of any disagreement can be bounded above by $Ce^{-cr^{\beta}}$. Just as we explained after \eqref{c6a}, the stretch exponential terms can be disregarded as they decay sufficiently rapidly. Then \eqref{fast3} immediately gives the claim.

    Now consider the case where $\min_{i<j} |y_i-y_j| >r^{2\beta}$. Break the expectation into two pieces according to $\{|\mathbf R_r- r\mathbf s_N -\y| \ge \frac12 r^{2\beta}\}$ and $\{|\mathbf R_r- r\mathbf s_N -\y| < \frac12 r^{2\beta}\}$. The first event has probability bounded above by by $Ce^{-cr^{\beta}}$. Just as we explained after \eqref{c6a}, the stretch exponential terms can be disregarded as they decay sufficiently rapidly. For the second term one can use a coupling to replace the trajectory $\mathbf R|_{[r-r^{2\beta},r]}$ with its initial condition plus the trajectory of the random walk $\mathbf S|_{[r-r^{2\beta},r]}$, and the probability of any disagreement can be bounded above by $Ce^{-cr^{\beta}}$. Thus \eqref{fast} immediately gives the claim.

    \textbf{Step 3.} 
    Every point in $\mathbb R^{kd}$ is within distance 1 of some point in the lattice $\Lambda := (kd)^{-1/2}\mathbb Z^{kd}$. Thus apply a union bound, and then use the Markov property to write 
    \begin{align*}
        \Ebb [ e&^{-|\mathbf R_{2r} - 2r\mathbf s_N-\y| - r^{-\beta} \sum_{s=0}^{2r-1} \mathrm{Dr}(\mathbf R_s)} ] \leq \sum_{\bfa \in \Lambda } \Ebb [ \ind_{\{|\mathbf R_r - r\mathbf s_N- \bfa |\le 1 \}} e^{-|\mathbf R_{2r} -\y| - r^{-\beta}\sum_{s=0}^{2r-1} \mathrm{Dr}(\mathbf R_s)} ]] \\&=  \sum_{\bfa \in \Lambda } \bigg(\Ebb \big[ \ind_{\{|\mathbf R_r - r\mathbf s_N- \bfa |\le 1 \}} e^{-r^{-\beta}\sum_{s=0}^{r-1} \mathrm{Dr}(\mathbf R_s)} \mathbf E^{\br}_{\mathbf R_s} [ e^{-|\tilde{\mathbf R}_r - 2r\mathbf s_N-\y| - r^{-\beta}\sum_{s=0}^{r-1} \mathrm{Dr}(\tilde{\mathbf  R}_s)}]\big]\bigg)\\&  \leq \sum_{\bfa \in \Lambda } \bigg( \Ebb[\ind_{\{|\mathbf R_r- r\mathbf s_N - \bfa |\le 1 \}} e^{-r^{-\beta}\sum_{s=0}^{r-1} \mathrm{Dr}(\mathbf R_s)} ] \cdot \sup_{\mathbf u: |\mathbf u- r\mathbf s_N-\bfa |\leq 1} \mathbf E^{\br}_{\mathbf u} [e^{-|\tilde{\mathbf R}_r - 2r\mathbf s_N-\y| - r^{-\beta}\sum_{s=0}^{r-1} \mathrm{Dr}(\tilde{\mathbf  R}_s)}]\bigg) \\&\leq e \sum_{\bfa \in \Lambda } \bigg( \Ebb[e^{-|\mathbf R_r - r\mathbf s_N- \bfa |} e^{-r^{-\beta}\sum_{s=0}^{r-1} \mathrm{Dr}(\mathbf R_s)} ] \cdot \sup_{\mathbf u: |\mathbf u- r\mathbf s_N-\bfa |\leq 1} \mathbf E^{\br}_{\mathbf u} [e^{-|\tilde{\mathbf R}_r - 2r\mathbf s_N-\y| - r^{-\beta}\sum_{s=0}^{r-1} \mathrm{Dr}(\tilde{\mathbf  R}_s)}]\bigg)
    \end{align*}
    Writing $\bfa = (a_1,...,a_k)$, split the sum according to $\min_{i<j} |a_i-a_j| >r^{2\beta}$ or $\min_{i<j}|a_i-a_j| \le r^{2\beta}$. Write the corresponding sums as $I_1$ and $I_2$. Note that $\mathbf s_N$ is a vector of the form $(   s_N,...,   s_N)$ for some $   s_N\in \mathbb R^d$, thus using the result of Step 2, we have that $$I_1 \leq e \sum_{\bfa \in \Lambda} \bigg(r^{-kd/2} (1+r^{-1/2} |\x-\bfa | )^{-K}\bigg)\bigg(r^{-kd/2} (1+r^{-1/2} |\bfa -\y| )^{-K}\bigg) \leq Cr^{-kd/2} (1+r^{-1/2} |\x-\y| )^{-K},$$ which is the required bound (note here: $K>kd$ without loss of generality). Now let us bound $I_2$ which is where we will need to make an explicit (and very small) choice of $\beta>0$. We have by Proposition \ref{trunc} that 
    \begin{align*}
        I_2 &\le e \sum_{\substack{\bfa \in \Lambda \\ \min_{i<j}|a_i-a_j| \le r^{2\beta}}} \bigg(r^{-\frac{kd}{2} (1-2\beta) } (1+r^{-1/2} |\bfa -\x|) \bigg)\bigg(r^{-\frac{kd}{2} (1-2\beta) } (1+r^{-1/2} |\bfa -\y|) \bigg) \\&\leq C r^{3\beta kd} r^{(k-1)d/2} r^{-kd} \big( 1+r^{-1/2}|\x-\y|\big)^{-K} \\&= r^{3\beta kd -\frac{d}2} r^{-kd/2}\big( 1+r^{-1/2}|\x-\y|\big)^{-K} .
    \end{align*}
    Now take any $\beta<1/(6k)$ to obtain the bound 
    \begin{equation}\label{pprx}
        \Ebb [ e^{-|\mathbf R_{2r} - 2r\mathbf s_N-\y| - r^{-\beta} \sum_{s=0}^{2r-1} \mathrm{Dr}(\mathbf R_s)} ] \leq Cr^{-kd/2}\big( 1+r^{-1/2}|\x-\y|\big)^{-K} ,
    \end{equation}
    where $C$ may depend on $K,\beta$ but nothing else.

    \textbf{Step 4.} In this step, we prove the theorem. For $d=1$ note that \eqref{pprx} already implies the given bound in the second bullet point in Theorem \ref{anti}, since $\mathrm{Dr}$ vanishes for a centered chain. Thus we focus on $d\ge 2$. 
    
    Henceforth fix $\beta$ as appearing in \eqref{pprx}. From \eqref{pprx} and the decay conditions on $\mathrm{Dr}$ we have $$\Ebb \bigg[ \mathrm{Dr}(\mathbf R_r) e^{-r^{-\beta} \sum_{s=0}^{r-1} \mathrm{Dr}(\mathbf R_s)} \bigg] \leq Cr^{-d/2}, $$ for $r\ge 1$. Thus if $0\le s_1\le ... \le s_m\le r$ then by iterating this bound and using the Markov property, we have that
    \begin{equation}\label{betarbound}\Ebb \bigg[ \prod_{j=1}^m \mathrm{Dr}(\mathbf R_{s_j})e^{-r^{-\beta} \sum_{s=0}^{r-1} \mathrm{Dr}(\mathbf R_s)}\bigg]\le C^m \prod_{j=0}^{m-1} 1\wedge (s_{j+1}-s_j)^{-d/2}. \end{equation}
    The remainder of the proof will be split into cases depending on the dimension.
    \\
    \\
    If $d\ge 2$, then \eqref{betarbound} implies that e.g. $$\Ebb \bigg[ \bigg( r^{-\beta/2} \sum_{s=0}^{r-1} \mathrm{Dr}(\mathbf R_s)\bigg)^m e^{-r^{-\beta} \sum_{s=0}^{r-1} \mathrm{Dr}(\mathbf R_s)} \bigg] \leq m! (C'r^{-\beta/2}\log r)^m . $$
    Divide both sides by $m!$ and sum over $m$, and we will obtain that $$\sup_{r\ge 0} \Ebb \bigg[ e^{(r^{-\beta/2}-r^{-\beta}) \sum_{s=0}^{r-1} \mathrm{Dr}(\mathbf R_s)}\bigg]<\infty.$$
    Note that $r^{-\beta/2} - r^{-\beta} > \frac12 r^{-\beta/2}$ for sufficiently large $r\ge 0$, so by Markov's inequality, the last expression implies that $$\Pbb \bigg( \sum_{s=0}^{r-1} \mathrm{Dr}(\mathbf R_s) > r^{\beta} \bigg) =  \Pbb \bigg( \frac12 r^{-\frac \beta2} \sum_{s=0}^{r-1} \mathrm{Dr}(\mathbf R_s) >\frac12  r^{\beta/2} \bigg) \leq Ce^{-\frac12 r^{\beta/2}}.$$
    Going back to \eqref{pprx}, we find that 
    \begin{align*}
        \Ebb [ e^{-|\mathbf R_{2r} - 2r\mathbf s_N-\y|}] &\le \Ebb [ e^{-|\mathbf R_{2r} - 2r\mathbf s_N-\y|} \ind_{\{\sum_{s=0}^{2r-1} \mathrm{Dr}(\mathbf R_s) >  r^{\beta} \}}] + \Ebb [ e^{-|\mathbf R_{2r} - 2r\mathbf s_N-\y|} \ind_{\{\sum_{s=0}^{2r-1} \mathrm{Dr}(\mathbf R_s) \le  r^{\beta} \}}] \\&\leq Ce^{-\frac12 r^{\beta/2}} + e\cdot  \Ebb [ e^{-|\mathbf R_{2r} - 2r\mathbf s_N-\y| - r^{-\beta} \sum_{s=0}^{2r-1} \mathrm{Dr}(\mathbf R_s)} ] \\&\leq Ce^{-\frac12 r^{\beta/2}}+ Cr^{-kd/2}\big( 1+r^{-1/2}|\x-\y|\big)^{-K}.
    \end{align*}
    Just as we explained after \eqref{c6a}, the stretch exponential terms can be disregarded as they decay sufficiently rapidly, which gives the claim.
\end{proof}

\section{Proof of Theorem \ref{thm:invMeasure}: Invariant measures for SRI chains}\label{appendix:c}

We will see that the existence and uniqueness proofs of $\pi^{\mathrm{inv}}$ will be different in dimensions $d=1,2$ versus $d\ge 3.$ This is because recurrence essentially gives automatic existence and uniqueness, while in the case of a transient SRI chain it is somewhat trickier to show both, since transient chains need not have invariant measures. Again we emphasize that $\pi^{\mathrm{inv}}$ will be an infinite measure, not a probability measure (it is not an invariant \textit{distribution}).

\begin{lem}
    \label{buo} Under the assumptions of Theorem \ref{thm:invMeasure}, the limiting chain $\br$ is $\delta$-repulsive for some $\delta>0,$ and centered. 
\end{lem}

\begin{proof}
    The centeredness is clear. To show $\delta$-repulsivity, notice from the symmetry assumptions that we just need to show $$\inf_{x\in I} \int_I|y-x|\bdif(x,\dr y)>0.$$ By Lemma \ref{tbb}, one has that $\big| \int_I |y-x| \bdif (x,\dr y) - \int_I |u|\nu(\dr u)\bigg| \le 2F(|x|),$ and since $\nu$ has covariance matrix $\mathrm{Id}_{d\times d}$ this rules out the possibility of some sequence $x_n\in I$ with $|x_n|\to\infty$ such that $\int_I |y-x_n| \bdif(x_n,\dr y)\to 0.$ Thus we just need to show that for each compact set $K\subset I$ one has $\inf_{x\in K} \int_I |y-x| \bdif(x,\dr y)>0$. By the regularity assumptions on $\bdif$, the latter is actually a continuous function of $x$, thus it achieves its infimum on each compact set. Thus if the infimum was zero, then there would be an absorbing state, contradicting the topological irreducibility assumption on $\bdif$ in Assumption \ref{a1} \eqref{a16}.
\end{proof}


\begin{prop}[Existence/Uniqueness/nontriviality for $d=1,2$---recurrent case] \label{d1}
    Let $d\le 2$. With $\bdif$ as in Theorem \ref{thm:invMeasure}, there is a unique nontrivial invariant measure $\pi^{\mathrm{inv}}$ up to scalar multiple. 
\end{prop}

\begin{proof}
    Let $d\le 2$. The proof will proceed by construction of Lyapunov functions to show recurrence. By the continuity of $x\mapsto \bdif(x,\bullet)$ in total variation norm, the Markov kernel $\pdif$ is \textit{Strong Feller,} which means that the associated Markov operator $\Pdif$ sends bounded measurable functions to bounded continuous functions. In particular $\bdif$ is a ``$T$-chain" in the sense of \cite[Definition 6.0.0 (iii)]{MTbook}. 

    Furthermore $\bdif$ is ``open-set irreducible" in the sense of \cite[Section 6.1.2]{MTbook} by the assumptions of Theorem \ref{thm:invMeasure}. In particular, by \cite[Proposition 6.1.5]{MTbook} the chain is ``$\psi$-irreducible" where the maximal irreduciblilty measure $\psi$ has full support due to the open-set irreducibility and the Strong Feller property.

    We will now show that for all $x\in I$, there exists a compact set $K\subset I$ such that the Markov chain $(X_r)_{r\ge 0}$ visits $K$ infinitely often almost surely when started from $x$. In other words, the Markov kernel $\pdif$ is ``non-evanescent" in the sense of \cite[Section 9.2.1]{MTbook}. Now by \cite[Theorem  9.2.2 (ii)]{MTbook} any non-evanescent $\psi$-irreducible $T$-chain is automatically \textit{Harris recurrent}. 
    
    By \cite[Theorem 17.3.2]{MTbook}, Harris recurrent chains can be characterized as exactly those chains for which a unique invariant measure $\pi^{\mathrm{inv}}$ exists uniquely (up to scalar multiple). By \cite[Theorem 10.4.9]{MTbook}, the measure $\pi^{\mathrm{inv}}$ is equivalent to the maximal irreducibility measure $\psi$, which has already been explained to have full support, in particular it is a nonzero measure.

    The manner in which we will whow the non-evanescence of the chain $\bdif$ is by showing that there exists a Lyapunov drift function $V$ satisfying the so-called \textit{(V1)} drift condition, that is the function has compact sublevel sets and furthermore one has $\Pdif V(x)\leq V(x)$ for all $x$ outside some compact set. By \cite[Theorem 9.4.1]{MTbook}, this would imply the non-evanescence.
\\
\\
\textbf{Existence of a Lyapunov drift function in $d=1$:} We set $V(x):= \sqrt{|x|+1},$ and we let $\Pdif$ denote the Markov operator associated to the Markov kernel $\bdif$. We will show that $\Pdif V(x)\leq V(x)$ for all positive and sufficiently large $x$, and the proof for negative $x$ would be completely symmetric. 

    For $x>2$ and $a\in (x/2,3x/2)$ we can write \begin{align*}V(a)-V(x) &= V'(x) (a-x) + \frac12 V''(x) (a-x)^2 +R(a,x) \\ &= \tfrac12 (x+1)^{-1/2} (a-x) -\frac14 (x+1)^{-3/2} (a-x)^2 + R(a,x),
    \end{align*} where by Taylor's remainder theorem one has uniformly over all $x>2$ and $a\in (x/2,3x/2)$ the bound $$|R(a,x)| \leq \frac16 |a-x|^3 \sup_{b\in [x/2,3x/2]}|V'''(b)| \le \frac3{48} (x/2)^{-5/2} |a-x|^3 . $$ Now recall that $\int_I (a-x)\bdif(x,\dr a) = 0$ for all $x\in I$, which is immediate from the definitions. Consequently we can disregard the first-order term when calculating $\Pdif V$ and we find that \begin{align}\notag \Pdif V(x)-V(x) &= \int_I (V(a)-V(x)) \bdif (x,\dr a) \\ \notag &\leq -\frac14 (x+1)^{-3/2}\int_I (a-x)^2 \bdif(x,\dr a) + \frac3{48}(x/2)^{-5/2} \int_I |a-x|^3 \bdif(x,\dr a) \\ \label{b8}&\;\;\;\;\;\;\;\;\;\;\;\;\;\;\;\;\;+ \int_I |V(a)-V(x)| \ind_{\{|x-a|>x/2\}} \bdif(x,\dr a).\end{align}
    Let $\sigma$ be as the definition of the class $\Ekk$, and let us define constants
    \begin{align*}
        \omega_{\mathrm{exp}}: &= \sup_{y\in I} \int_I e^{\frac{\sigma}4|a-y|} \bdif(y,\dr a),\\
        \omega_2:&= \sup_{y\in I} \int_I (a-y)^2 \bdif(y,\dr a),\\ \omega_3 :&= \sup_{y\in I} \int_I |a-y|^3 \bdif(y,\dr a), \\ \delta_2 :&= \inf_{y\in I} \int_I (a-y)^2 \bdif(y,\dr a).
    \end{align*}
    The first three constants are all finite by e.g. Lemma \ref{grow}, and the last constant is strictly positive by e.g. Lemma \ref{buo}. Notice that $V$ is a globally Lipchitz function on $I$ with Lipchitz constant 1/2, i.e., $|V(x)-V(a)|\leq \frac12 |x-a|,$ thus using Cauchy-Schwarz and Markov's inequality we see that 
    \begin{align*}
        \int_I |V(a)-V(x)| \ind_{\{|x-a|>x/2\}} \bdif(x,\dr a) &\leq \bigg(\int_I (V(a)-V(x))^2\bdif(x,\dr a)\bigg)^{1/2} \bigg(\int_I \ind_{\{|x-a|>x/2\}} \bdif(x,\dr a)\bigg)^{1/2} \\ &\leq  \bigg(\frac14 \int_I (a-x)^2\bdif(x,\dr a)\bigg)^{1/2} \bigg(\frac{\int_I e^{\frac{\sigma}4|x-a|} \bdif(x,\dr a)}{e^{\frac{\sigma}4(x/2)}}\bigg)^{1/2} \\ &\leq \frac12 \omega_2^{1/2} \omega_{\mathrm{exp}}^{1/2}\cdot  e^{-\frac{\sigma x}{16}}.
    \end{align*}
    Plugging this bound back into \eqref{b8}, we find that for $x>2$ one has $$\Pdif V(x) - V(x) \leq -\frac{\delta_2} 4 (x+1)^{-3/2} + \frac{3\omega_3}{48} (x/2)^{-5/2} + \frac12 \omega_2^{1/2} \omega_{\mathrm{exp}}^{1/2} \cdot e^{-\frac{\sigma x}{16}}.$$ The right side is clearly negative for sufficiently large $x>2$, thus proving the claim.
\\
\\
    \textbf{Existence of a Lyapunov drift function in $d=2$:} We set $V(x):= \sqrt{\log(1+|x|^2)}.$ A direct calculation reveals that for $|x|$ large one has $\Delta V(x) \leq -\frac{1}{2|x|^2 \log^{3/2}|x|}$. 
    
    We will show that $\Pdif V(x)\leq V(x)$ for all sufficiently large $|x|$. For $|x|>2$ and $a$ within a ball of radius $|x|/2$ of $x$, we can write \begin{align*}V(a)-V(x) &= \sum_{j=1,2} V_j(x) (a_j-x_j) + \frac12 \sum_{i,j=1,2} V_{ij}(x) (a_i-x_i)(a_j-x_j) +R(a,x), 
    \end{align*} where by Taylor's remainder theorem one has uniformly over all $x>2$ and $a\in (x/2,3x/2)$ the bound $$|R(a,x)| \leq \frac16 |a-x|^3 \sup_{b: |b-x|<|x|/2}\sum_{i,j,k=1,2} |V_{ijk}(b)| \le C |x|^{-3} |a-x|^3 , $$ where $C$ is some large enough constant as one may verify by direct calculation. Now recall that $\int_I (a-x)\bdif(x,\dr a) = 0$ for all $x\in I$, which is immediate from the definitions. 
    
    Consequently we can disregard the first-order term when calculating $\bdif V$ and we find that \begin{align}\notag \Pdif V(x)-V(x) &= \int_I (V(a)-V(x)) \bdif (x,\dr a) \\ \notag &\leq  \frac12 \sum_{i,j=1,2} V_{ij}(x) \int_I (a_i-x_i)(a_j-x_j) \bdif(x,\dr a) + C|x|^{-3} \int_I |a-x|^3 \bdif(x,\dr a) \\ \label{b9}&\;\;\;\;\;\;\;\;\;\;\;\;\;\;\;\;\;+ \int_I |V(a)-V(x)| \ind_{\{|x-a|>x/2\}} \bdif(x,\dr a).\end{align}
    We also define the measure $\nu_{\mathrm{dif}}$ by $\int_I f(u)\nu_{\mathrm{dif}}(\dr u):= \int_{I^2} f(x-y) \nu(\dr x)\nu(\dr y),$ and we recall from the assumptions in Theorem \ref{thm:invMeasure}  that the covariance matrix of $\nu_{\mathrm{dif}}$ is $\mathrm{Id}_{d\times d},$ thus $\int_I u_iu_j \nu_{\mathrm{dif}}(\dr u) = \delta_{ij}$ for $i,j=1,2$. Since $\bdif$ is an SRI we have by Lemma \ref{tbb} 
    $$\bigg|\int_I u_iu_j \nu_{\mathrm{dif}}(\dr u)(\dr u) - \int_I (a_i-x_i)(a_j-x_j) \bdif(\x,\dr \bfa)\bigg| \leq C F(|x|)^{1/2}.$$
    One furthermore directly verifies that $|V_{ij}(x)| \leq C|x|^{-2}$ for some sufficiently large $C>0$. This means that 
    \begin{equation}
        \bigg| \frac12\Delta V(x) -\frac12 \sum_{i,j=1,2} V_{ij}(x) \int_I (a_i-x_i)(a_j-x_j) \bdif(x,\dr a)\bigg| \leq C|x|^{-2} F(|x|)^{1/2} \leq C|x|^{-3},
    \end{equation}
    where we are using the fact that the given conditions on $F$ clearly imply that $F(|x|)$ decays faster than $|x|^{-1}$ (in fact exponentially).

    Furthermore, just like in the $d=1$ case above, one easily shows that $$\int_I |V(a)-V(x)| \ind_{\{|x-a|>x/2\}} \bdif(x,\dr a) \leq Ce^{-c|x|}$$ for some $C,c>0$. In summary, we have $\Pdif V(x)-V(x) \leq \frac12\Delta V(x) + C|x|^{-3}. $ As noted above, we have $\Delta V(x) \leq -\frac{1}{2|x|^2 \log^{3/2}|x|}$ for sufficiently large $x$, which proves the claim.
\end{proof}

\begin{prop}[Existence/Uniqueness/nontriviality for $d\ge 3$---transient case]\label{d2}
    Let $d\ge 3$. With $\bdif$ as in Theorem \ref{thm:invMeasure} , there is a unique nontrivial invariant measure $\pi^{\mathrm{inv}}$ up to scalar multiple. The measure $\pi^{\mathrm{inv}}$ may be a signed measure, but its positive and negative parts both have constant growth at infinity.
\end{prop}

We can rule out the possibility of a signed measure using the results of Theorems \ref{inv01} and \ref{inv02} which are proved Appendix \ref{appendix:d}, showing that $\pi^{\mathrm{inv}}$ is actually a nonnegative measure, but for now this is the best we can say.

\begin{proof}
    In $d\ge 3$, one cannot rely on recurrence techinques to prove the existence of invariant measures. This is because the chain is simply \textit{not} recurrent. However, there is an easy alternative in higher dimensions: an explicit formula for invariant measure.

    Recall that $I$ is a locally compact additive subgroup of $\mathbb R^d$, thus $I$ has a Haar measure which we write $\dr x$. In fact one may write $I=\Sigma (\mathbb Z^m \times \mathbb R^{d-m})$ for some $d\times d$ matrix $\Sigma$ and $m\in \{0,...,d\}$, and then the measure $\dr x$ can just be identified with the product of counting measure and Lebesgue measure on $\mathbb Z^m$ and $\mathbb R^{d-m}$ respectively.
    
    Let $\Pdif$ be the Markov operator for the Markov kernel $\bdif$, and let $P_{\mathrm{RW}}$ be the Markov operator for the simple random walk with increment law $\nu_{\mathrm{dif}}$. Then we claim that for $f:I\to \mathbb R$ of exponential decay, the invariant measure acts on $f$ explicitly by $$\int_I f\; \dr \pi^{\mathrm{inv}} = \int_I \big(P_{\mathrm{RW}}-\mathrm{Id}\big) \bigg( \sum_{m=0}^\infty \Pdif^m f\bigg) (x) \dr x.$$
    We need to justify why the integral converges if $f$ is of exponential decay. We have the bound
    \begin{align}
        \notag \bigg|\big(P_{\mathrm{RW}}-\mathrm{Id}\big) \bigg( \sum_{m=0}^\infty \Pdif^mf \bigg) (x)\bigg| &\notag \le \bigg|\big(\Pdif-\mathrm{Id}\big) \bigg( \sum_{m=0}^\infty \Pdif^mf \bigg) (x)\bigg| + \bigg|\big(P_{\mathrm{RW}}-\Pdif\big) \bigg( \sum_{m=0}^\infty \Pdif^m f\bigg) (x)\bigg| \\ \notag &= |f(x)| + \bigg|\big(P_{\mathrm{RW}}-\Pdif \big) \bigg( \sum_{m=0}^\infty \Pdif^m f \bigg) (x)\bigg| \\&\leq |f(x)| + F(|x|) \cdot \bigg\| \sum_{m=0}^\infty \Pdif^m f\bigg\|_{L^\infty(I)},\label{cgai}
    \end{align}where the last bound follows from the definition of the SRI chain having decay function $F$. The exponential decay condition on $F$ imply that it is indeed integrable with respect to the Haar measure on $I$. The fact that the infinite series converges and yields a bounded function is clear from the anticoncentration estimate in Theorem \ref{anti}, since it implies that $\|\Pdif^m f\| \leq Cm^{-d/2} \|f\|_{L^\infty(I)}$ which is summable over $m\in \mathbb N$ as long as $d\ge 3$. That theorem is applicable because the topological irreducibility plus regularity conditions on $\bdif$ imply that the chain $\br$ on $I^k$ is indeed centered and $\delta$-repulsive for some $\delta>0$, see Lemma \ref{buo}. This justifies finiteness of the integral defining $\pi^{\mathrm{inv}}$.
    
    Then it is clear that $$\int_I \Pdif f \; \dr\pi^{\mathrm{inv}} = \int_I \big(P_{\mathrm{RW}}-\mathrm{Id}\big) \bigg( \sum_{k=1}^\infty \Pdif^m \bigg) f(x) \dr x = \int_I f\dr\pi^{\mathrm{inv}} - \int_I\big(P_{\mathrm{RW}}-\mathrm{Id}\big)f(x)\dr x =\int_I f\dr\pi^{\mathrm{inv}},$$
    because we know that $\int_I\big(P_{\mathrm{RW}}-\mathrm{Id}\big)f(x)\dr x = 0$ since the Haar measure $\dr x$ is easily checked to be invariant for $P_{\mathrm RW}.$ Thus we see that $\pi^{\mathrm{inv}}$ is indeed an invariant measure for $\bdif,$ and \eqref{cgai} easily implies that its positive and negative parts have constant growth at infinity.

    Now we just need to justify the uniqueness. Consider two different invariant measures $\pi_1,\pi_2$, each of which may be signed measures whose positive and negative parts have constant growth at infinity. Define measures $\lambda_i$ for $i=1,2$ by the formula $$\int_I f\; \dr \lambda_i = \int_I \big(\Pdif-\mathrm{Id}\big) \bigg( \sum_{m=0}^\infty P_{\mathrm{RW}}^m f\bigg) (x) \;\;\pi_i(\dr x),\;\;\;\; i=1,2.$$Since this is just an ``inversion" of the above formula for $\pi^{\mathrm{inv}}$, one then easily shows that the measures $\lambda_1,\lambda_2$ must both be invariant for the random walk on $I$ with increment distribution $\nu_{\mathrm{dif}}$. Thus, we just need to show that the random walk on $I$ with increment distribution $\nu_{\mathrm{dif}}$ has a unique invariant measure (up to scalar multiple) given by the Haar measure on $I$, which is a simpler problem. 
    
     Let $\lambda$ be an invariant (signed) measure for the random walk with increment law $\nu_{\mathrm{dif}}$, of constant growth at infinity. We need to show that $\lambda$ is a multiple of Haar measure. As a first step, we claim that the random walk with increment law $\nu_{\mathrm{dif}}$ is not supported on a proper closed subgroup of $I$. This is where we need to use the irreducibility assumptions on $\bdif$. 
     
     We assumed that $\bdif$ is continuous in TV norm in the $x$ variable, thus so is the underlying random walk under translations. This means that $\nu_{\mathrm{dif}}$ must have a $L^1$ density with respect to the Haar measure on $I$, so that it cannot be supported on a proper closed subgroup. 
     
     Assuming without loss of generality that $I=\mathbb Z^m \times \mathbb R^{d-m} $ for some $m\in\{0,...,d\}$ we can then speak of the Fourier transform $\phi_{\nu}$ of the measure $\nu$ which is a complex-valued function on $\mathbb T^m \times \mathbb R^{d-m}$. Since $\nu_{\mathrm{dif}}$ is not supported on a proper subgroup of $I$, the same holds true for $\nu$, and thus $|\phi_\nu(\xi)| = 1$ implies that $\xi = (0,...,0).$ Now if we write the relation for $\lambda $ to be an invariant measure (for the random walk with increment law $\nu_{\mathrm{dif}}$) in Fourier space, it reads $(|\phi_\nu (\xi)|^2 -1) \hat\lambda(\xi)=0$ for all $\xi$, where one needs to interpret $\hat \lambda$ as a tempered distribution on $\mathbb T^m \times \mathbb R^{d-m}$ (this is permissible because the constant growth condition on $\lambda$ implies that it does have a finite value when paired against smooth functions of rapid decay on $\mathbb T^m \times \mathbb R^{d-m}$). The relation $(|\phi_\nu (\xi)|^2 -1) \hat\lambda(\xi)=0$ means that $\hat\lambda$ must be a tempered distribution that is supported at the origin of $\mathbb T^m \times \mathbb R^{d-m}$. Since $\lambda$ is assumed to be a signed measure of constant growth at infinity, this means that $\hat\lambda$ can only be $c\delta_0$ for some $c\in \mathbb R$, which means by inverting the Fourier transform that $\lambda$ itself must be $c$ times Haar measure. This proves the uniqueness.

    Nontriviality of $\pi^{\mathrm{inv}}$ just follows from the explicit formula above. Indeed, take any strictly positive function $h$ of exponential decay at infinity, and note that if we set $f:= (\Pdif - \mathrm{Id}) \sum_{m=0}^\infty P_{\mathrm{RW}}^m  h,$ then one has $\int_I f \dr\pi^{\mathrm{inv}} = \int_I h(x)\dr x>0,$ so that $\pi^{\mathrm{inv}}$ cannot be the zero measure.
\end{proof}

Combining the results of Propositions \ref{d1} and \ref{d2} immediately proves Theorem \ref{thm:invMeasure}. 

\section{Proof of Theorems \ref{inv01}-\ref{inv02}: Expectation limits for SRI chains}\label{appendix:d}

Now we are in a position to complete the proof of Theorem \ref{inv01}. The key ingredients to proving this theorem will be the invariance principle of Theorem \ref{inv00} together with the anticoncentration result of Theorem \ref{anti} and the uniqueness of $\pi^{\mathrm{inv}}$ (up to scalar multiple) guaranteed by Propositions \ref{d1} and \ref{d2}. We shall start with the case $\ell=1$, then proceed inductively.

\begin{proof}[Proof of Theorem \ref{inv01}] 
    Existence and uniqueness are immediate from the previous two propositions. It remains to prove the limit theorem for the expectations. For $d=1,2$ it also remains to prove constant growth at infinity of $\pi^{\mathrm{inv}}$, but note that this is immediate from the expectation limit theorem and from the bound in Theorem \ref{anti}. For $d\ge 3$ we must also rule out the possibility that $\pi^{\mathrm{inv}}$ is a signed measure, but again this can be easily ruled out by the expectation limit theorem. 

Thus it only remains to prove the expectation limit theorem. We break the proof into five steps. As always let $\Pdif$ be the Markov operator associated to $\bdif$.

\textbf{Step 1.} We first reduce to the case where all $f_N$ are equal to some fixed function $f$. This is fairly straightforward given Theorem \ref{anti}. Indeed since we have by assumption $|f(x)|\leq H(|x|)$ and $H$ is of exponential decay, we have $|f_N(x)-f(x)| \leq \min\{ \|f_N-f\|_\infty, 2H(|x|)\}\leq \|f_N-f\|_\infty^{1/2}( 2H(|x|))^{1/2}$, and it follows from the latter theorem that $$N^{\frac{d-2}2} \sum_{r=0}^{Nt} \Ebbnn[ |(f_N-f)(   R^1_r-   R^2_r)| ] \leq C \|f_N-f\|_{L^\infty(I)}^{1/2}  \cdot \bigg(N^{\frac{d-2}2}\sum_{r=0}^{Nt} \Ebbnn[ H(|   R^1_r-   R^2_r|)^{1/2}]\bigg)\to 0,$$ which completes this reduction.

\textbf{Step 2.} 
Take any function $f:I\to \R$ such that $|f(x)| \leq H(|x|)$ for some decreasing $H:[0,\infty)\to[0,\infty)$ of exponential decay at infinity. In this step, we will show that $$\sup_{N\ge 1} N^{\frac{d-2}2} \sum_{s=0}^{Nt-1}\Ebbnn [ f(   R^1_s-   R^2_s)] \leq C(H), 
$$ where the constant $C(H)$ may depend on $H$ but not on $N$ or on any $f$ dominated by $x\mapsto H(|x|)$ as above. By the assumptions of Theorem \ref{inv01}, we have that $|x^i_N-x^j_N| \geq cN^{1/2}$ for $N\ge N_0$, where $N_0\in\mathbb N$ and $c>0$. Thus, taking $K:=d$ in Theorem \ref{anti}, we have for $N\ge N_0$ that 
\begin{align*}
    \int_I f\;\dr\gamma_N & \leq N^{\frac{d-2}2} \sum_{s=1}^{Nt-1} r^{-d/2} \big( 1+ N^{1/2}cr^{-1/2} \big)^{-d} \\&= N^{-1} \sum_{s=1}^{Nt-1} \bigg( \frac{N}r\bigg)^{d/2} \bigg( 1+c\sqrt{\frac{ N}r}\bigg)^{-d}.
\end{align*}
The latter is a Riemann sum approximation for $\int_0^t u^{-d/2} (1+cu^{-1/2})^{-d} \dr u,$ which is finite since $(1+cu^{-1/2})^{-d}\leq c^{-d} u^{d/2}.$ Consequently, the last sum remains bounded in $N$.

\textbf{Step 3.} We now claim that in Theorem \ref{inv01}, the quantity $\phi(N^{-1/2} (   R^1_r -    s_N r)\big)$ can be replaced by $\phi(\frac12 N^{-1/2} (   R^1_r+R^2_r-2    s_N r)\big)$ without changing the value of the limit. To prove this, notice that $$|\phi\big(N^{-1/2} (   R^1_r -    s_N r)\big)-\phi\big(\tfrac12 N^{-1/2} (   R^1_r +R^2_r-2    s_N r)\big)\big| \leq \tfrac12 \|\phi\|_{C^1} N^{-1/2} |   R^1_r-   R^2_r|.$$ Thus setting $\tilde f_N(x) = |x||f_N(x)|,$ we find that 
\begin{align*}
     N^{\frac{d-2}2} &\bigg| \sum_{r=0}^{Nt} \Ebbnn [ \phi(N^{-1/2} (   R^1_r -    s_N r)\big)f_N(   R^1_r-   R^2_r) ] - \sum_{r=0}^{Nt}\Ebbnn [ \phi\big(\tfrac12 N^{-1/2} (   R^1_r +R^2_r-2    s_N r)\big)f_N(   R^1_r-   R^2_r) ] \big)\bigg| \\&\leq N^{-1/2} \|\phi\|_{C^1} \cdot N^{\frac{d-2}2} \sum_{r=0}^{Nt} \Ebbnn [ \tilde f_N(   R^1_r-   R^2_r) ]  \\ &\leq C N^{-1/2} \|\phi\|_{C^1},
\end{align*}
where we use the result of Step 1 in the last step, noting that $|x|\sup_N |f_N(x)|$ still decays exponentially. Thus the difference tends to zero, which means we now need to show that 
\begin{align}\notag\lim_{N\to\infty} N^{\frac{d-2}2} \sum_{r=0}^{Nt} \Ebbnn &[ \phi\big(\tfrac12 N^{-1/2} (   R^1_r +R^2_r-2    s_N r)\big)f_N(   R^1_r-   R^2_r) ] \\&= \bigg( \int_0^t \int_{\mathbb R^d} G(2s,x_i-x_j) G(2s,y) \phi(\tfrac12(y+x_i+x_j)) \dr y\dr s\bigg)\int_I f \;\dr\pi^{\mathrm{inv}},\label{eqnn}
\end{align}
where $R^1_r$ appearing inside $\phi$ in the original formulation has now been replaced by $\tfrac12 (   R^1_r+R^2_r).$
We will see shortly that the latter formulation has certain advantages.

\textbf{Step 4.\label{step4}} The remainder of the proof will focus on proving \eqref{eqnn}. Henceforth abbreviate 
$$X^N_r:= \tfrac12 (   R^1_r+R^2_r-2 N^{-1}   s_N),\;\;\;\;\;\;\; Y_r:= R^1_r-   R^2_r .$$Consider the sequence $\gamma_N$ of measures on $I$ defined by $$N^{\frac{d-2}2} \sum_{s=0}^{Nt-1}\Ebbnn [ \phi\big(N^{-1/2} X^N_s\big)f(Y_s)].$$ 
We will show that any subsequence $\gamma_{N_k}$ of this sequence of measures has a further subsequence $\gamma_{N_{k_j}}$ converging as $N\to\infty$ to an invariant measure $\gamma$ for the Markov kernel $\pdif$. By convergence, we mean that for all continuous $f:I\to \mathbb R$ such that $|f(x)|\leq H(|x|)$, we have $\int_I f\;\dr\gamma_N \to \int_I f\;\dr\gamma$ along this subsequence $\gamma_{N_{k_j}}$. By the uniqueness of $\pi^{\mathrm{inv}}$ proved in Propositions \ref{d1} and \ref{d2}, this uniquely identifies any subsequential limit of $\gamma_N$ as a scalar multiple of $\pi^{\mathrm{inv}}$. Notice that the ultimate goal of the theorem can be formulated as establishing convergence of the \textit{entire} sequence $\gamma_N$ to $ \pi^{\mathrm{inv}}$ multiplied by the constant $\big( \int_0^t \int_{\mathbb R^d} G(s,x_1-x_2)G(s,y)\phi(y) \dr y\dr s\big),$ which will be accomplished in the next step by identifying the correct constant.

Thus consider a subsequence $\gamma_{N_k}$. There is indeed a further convergent subsequence $\gamma_{N_{k_j}}$. This is because we can apply the Banach-Alaoglu theorem to the sequence of measures $\chi_N:= H(|x|) \gamma_N(\dr x)$. By the result of Step 2, we have that  
\begin{align*}
    \sup_{N \geq 1} \|\chi_N\|_{TV} &= \sup_{N\ge 1} N^{\frac{d-2}2} \sum_{s=0}^{Nt-1}\Ebbnn [ H(|   R^1_s-   R^2_s|)] < \infty 
\end{align*}so that $\|\gamma\|_{TV} <\infty.$ 

Let us call this subsequential limit $\gamma$. To show that $\gamma$ is an invariant measure, we need to show that for all continuous $f:I\to \mathbb R$, say of exponential decay at infinity, one has $\int_I \Pdif f \;\dr\gamma = \int_I f\;\dr\gamma.$ To prove this, consider such $f$, say $|f(x)| \leq e^{-\theta|x|}$ where $\theta\leq z_0/4$, 
and define $\bar f: I^k\to\mathbb R$ by $\bar f(\x) = f(x_1-x_2)$. To show that $\int_I \Pdif f \;\dr\gamma = \int_I f\;\dr\gamma,$ it suffices to show that \begin{equation}\label{cesro}\lim_{N\to\infty}N^{\frac{d-2}2} \sum_{s=0}^{Nt-1} \Ebbnn [ |(P^{\br}-P^{\brn}) \bar f(\mathbf R_s)|] = 0.
\end{equation}
To prove \eqref{cesro}, we claim by the exponential decay assumption on $f$, one has the bound 
\begin{equation}\label{prvces}|(P^{\br}-P^{\brn}) \bar f(\x)| \leq C\min\{ d_{\mathrm{SRI}}(\br,\brn), e^{-\frac\theta{2} |x_1-x_2|} \}\leq Cd_{\mathrm{SRI}}(\br,\brn)^{1/2}   e^{-\frac\theta{4} |x_1-x_2|}. 
\end{equation}
The second inequality is immediate from the first one using $\min\{u,v\}\leq u^{1/2}v^{1/2}$. To prove the first inequality, the upper bound of $Cd_{\mathrm{SRI}}(\br,\brn)$ is clear simply from the definition of $d_{\mathrm{SRI}}$, since $|\bar f|\leq 1$. For the upper bound of $Ce^{-\frac\theta{2} |x_1-x_2|}$, we can show that both of the terms $P^{\brn} \bar f(\x)$ and $P^{\br} \bar f(\x)$ are bounded above by such a quantity. To show this, write the definition $P^{\brn} \bar f(\x):= \int_{I^k}  f(y_i-y_j) \brn(\x,\dr\y), $ then split the integral into two parts: $\{\y: |\y-\x| \leq \frac12|x_1-x_2|\}  $ and $\{\y: |\y-\x| > \frac12|x_1-x_2|\}  $, where $|\x|=\sum_{j=1}^k |x_j|$. On the first set, use the exponential decay bound on $f$ to bound $|f(\y) |\leq Ce^{-\frac\theta 2|x_1-x_2|}$. On the second set, use the exponential moment bounds inherent in the definition of $\Ekk$ together with Markov's inequality (and $|f|\leq 1$) to obtain an upper bound of $Ce^{-\frac{\sigma}{8} |x_1-x_j|} \leq Ce^{-\frac\theta2 |x_1-x_2|}.$

With \eqref{prvces} proved, the claim \eqref{cesro} follows immediately from Theorem \ref{anti}. Now let us show why \eqref{cesro} implies that the subsequential limit $\gamma$ is indeed an invariant measure. We have  $P^{\br} \bar f(\mathbf R_s) = \Pdif f(   R^1_s-   R^2_s)$ so that 
\begin{equation}\label{cesr1}N^{\frac{d-2}2} \sum_{s=0}^{Nt-1}\Ebbnn [ \phi\big(N^{-1/2} X^N_s)\big) P^{\br} \bar f(\mathbf R_s)] \to \int_I \Pdif f \;\dr\gamma
\end{equation}
along the subsequence $N_{k_j}$. Moreover, by the Markov property and a summation by parts, we have
\begin{align} \notag &\sum_{s=0}^{Nt-1}\Ebbnn [\phi(N^{-1/2} X^N_s) (P^{\brn} -\mathrm{Id})\bar f(\mathbf R_s)] = \sum_{s=0}^{Nt-1}\Ebbnn [\phi(N^{-1/2} X^N_s) (\bar f(\mathbf R_{s+1}) -\bar f(\mathbf R_s))] \\&= \notag \Ebbnn [\phi(N^{-1/2} X_{Nt}) \bar f(\mathbf R_{Nt})] - \phi(N^{-1/2} \x_N) \bar f(\x_N) \\
&- \label{cestr} \sum_{s=0}^{Nt-1} \Ebbnn[ (\phi(N^{-1/2} X^N_{s+1}) - \phi(N^{-1/2}X^N_s)) \bar f(\mathbf R_{s+1}) ] 
\end{align}We claim that all three terms on the right side go to zero when multiplied by $N^{\frac{d-2}{2}}.$ The first two terms go to zero because Theorem \ref{anti} implies that $N^{\frac{d-2}2} \Ebbnn [f(   R^1_{Nt} -   R^2_{Nt})]\to 0$, and because the assumption $N^{-1/2}|x^1_N-x^2_N|\not\to 0$ guarantees that $N^{\frac{d-2}2} f(x^1_N-x^2_N)\to 0$. The third term tends to zero because by conditioning at time $s$ and then using Theorem \ref{anti} we have the bound 
\begin{align*}\Ebbnn[ (\phi(N^{-1/2} X_{s+1}) - \phi(N^{-1/2}X^N_s)) \bar f(\mathbf R_{s+1}) ] &\leq \|\phi\|_{C^1(\mathbb R^d)} N^{-1/2} \Ebbnn [ |X^N_{s+1}-X^N_s|^{q'}]^{1/q'} \Ebbnn[ \bar f(\mathbf R_s)^q]^{1/q} \\&\leq C \|\phi\|_{C^1(\mathbb R^d)} N^{-1/2} \Ebbnn [ \bar f^q(\mathbf R_s)]^{1/q}\\&\leq  \|\phi\|_{C^1(\mathbb R^d)} N^{-1/2} s^{-\frac{d}{2q}} \big(1+s^{-1/2} |x_N^1-x_N^2|)^{-K} \\&\le \|\phi\|_{C^1(\mathbb R^d)} N^{-1/2} s^{-\frac{d}{2q}} \big(1+cs^{-1/2} N^{1/2})^{-K}.
\end{align*}
Here $q,q'$ are conjugate Holder exponents and this is valid for any $q$. In the last line, we used $|x^1_N-x^2_N| >cN^{1/2}$ by assumption. Now we need to multiply by $N^{\frac{d-2}2} $ and show that the last expression summed from $s=1$ to $s=Nt$ tends to 0. Taking $q<\frac{d}{d-1}$ will show that it tends to 0 using a Riemann sum interpretation.


Since \eqref{cestr} has been shown to vanish, \eqref{cesr1} thus shows that along the subsequence $N_{k_j}$ we have that \begin{equation}\label{cesr2}N^{\frac{d-2}2} \sum_{s=0}^{Nt-1}\Ebbnn [ \phi\big(N^{-1/2} X^N_s)\big) P^{\brn} \bar f(\mathbf R_s)] \to \int_I f \;\dr\gamma.
\end{equation}
Combining \eqref{cesr1} and \eqref{cesr2} with \eqref{cesro}, we do indeed see that $\int_I f\;\dr\gamma = \int_I \Pdif f\;\dr\gamma,$ proving that $\gamma$ is indeed an invariant measure.

\textbf{Step 5.} Recall the measures $\gamma_N$ from the previous step. Define the function $\mathfrak f:=(\Pdif - \mathrm{Id})u$ where $u(x) := \mathrm{sign}(2-d)\cdot (1+|x|)^{2-d}$ for $d\ne 2$ and $u(x) = \log(1+|x|)$ for $d=2$. Recall we defined the \textit{unit normalization} of $\pi^{\mathrm{inv}}$ as the one in which $\int_I \mathfrak f\; \dr\pi^{\mathrm{inv}} =1.$ Also recall that we define the dimension-dependent constants $c_1 =1, c_2=2\pi,$ and $c_d= d(d-2) /\Gamma(1+\frac{d}2)$ for $d\ge 3$.

We need to check that $\int_I \mathfrak f\; \dr\pi^{\mathrm{inv}}< \infty$ to ensure that the unit normalization of $\dr\pi^{\mathrm{inv}}$ is well-defined. This will follow from showing that 
\begin{equation} \label{eq:intfdu}
    |\mathfrak f(x)| \leq C|x|^{-d-1}
\end{equation}for $|x|>1$. A direct calculation with the Laplacian formula in polar coordinates reveals that 
$\Delta u(x) =  |x|^{-1} (1+|x|)^{-2}$ in $d=2$ and $\Delta u(x) = (2-d)(1-d) |x|^{-1} (1+|x|)^{-d}$ for $d \neq 2$ and $x\ne 0$, so in particular we have the bound $|\Delta u(x) | \leq C_d |x|^{-d-1}$ for $|x|>1$ in all dimensions $d\ge 1$. The third-order derivatives also satisfy $|u_{ijk}(x)| \leq C|x|^{-d-1}.$ From here and Lemma \ref{tbb} and the fact that the covariance matrix of $\nu$ is $\mathrm{Id}_{d\times d}$, one may show (very similarly to the proof of existence of Lyapunov functions in $d=2$ above) by third order Taylor expansion that $$|\mathfrak f(x)| \leq C|x|^{-d-1}$$ for $|x|>1$. This implies that $\mathfrak f$ is indeed in $L^1(\pi^{\mathrm{inv}})$.


We will prove in this step that $\int_I \mathfrak f\;\dr\gamma_N \to c_d \cdot \int_0^t \int_{\mathbb R^d} G(s,x_1-x_2) G(s,y) \phi(y)\dr y\dr s$. If we can prove this, then the theorem would be proved, because by the uniqueness proved in Propositions \ref{d1} and \ref{d2}, together with the result of Step 4, this would mean that any subsequential limit $\gamma$ of the sequence $\gamma_N$ must be a constant multiple of $\pi^{\mathrm{inv}}$. But the convergence of $\int_I \mathfrak f\;\dr\gamma_N $ to the correct quantity uniquely identifies the constant as the one appearing in the theorem statement. If every subsequence of $\gamma_N$ has a further subsequence converging to some fixed measure, then the sequence $\gamma_N$ must itself converge to that fixed measure, thus completing the proof. 

Letting $\bar u(\x):= u(x_1-x_2)$ and $\bar f_N(\x):= (P^{\brn} -\mathrm{Id})\bar u$, both of which are functions on $I^2$, we claim that 
\begin{equation}\label{ano}|\mathfrak f(x_1-x_2) - \bar{\mathfrak  f}_N(\x)|\leq d_{\mathrm{SRI}}(\br,\brn)^{1/4} F( |x_1-x_2| )^{1/4}.
\end{equation} for some large enough constant $C>0$.
This is indeed fairly immediate from Lemma \ref{tbb} and the definition of SRI chains, see \eqref{useful} for the proof of a similar claim.

With \eqref{ano} established, 
we find that $$\Ebbnn \bigg[ N^{\frac{d-2}2} \sum_{r=0}^{Nt-1} |\mathfrak f(R_r^1-R_r^2)-\bar{\mathfrak  f}_{N} (\mathbf R_r)|\bigg] \leq \Ebbnn \bigg[ N^{\frac{d-2}2} d_{\mathrm{SRI}}(\br,\brn)^{1/4}\sum_{r=0}^{Nt-1} F(|R^1_r-   R^2_r|)^{1/4}\bigg].$$ 
Using Theorem \ref{anti} and the conditions on $F$, we may immediately conclude that the above expectation is bounded above by $C d_{\mathrm{SRI}}(\br,\brn)^{1/4}$ and thus tends to zero. Thus, in order to complete this step, we need to calculate the limit  
\begin{equation}
    \label{qnm0} \lim_{N\to\infty}\Ebbnn \bigg[ N^{\frac{d-2}2} \sum_{r=0}^{Nt-1} \phi( N^{-1/2}X^N_s)\bar{\mathfrak  f}_{N}(\mathbf R_r)\bigg],
\end{equation}
where we remind the reader that $X^N_r:= \tfrac12 (   R^i_r+R^j_r-2 N^{-1}   s_N)$ is a shorthand notation. Recalling that $\bar{\mathfrak  f}_{N} = (P^{\brn}-\mathrm{Id}) \bar u$, we have by the Markov property and a summation by parts that
\begin{align*}N&^{\frac{d-2}2} \sum_{s=0}^{Nt-1}\Ebbnn [\phi(N^{-1/2} X^N_s) \bar{\mathfrak  f}_{N}(\mathbf R_s)] = N^{\frac{d-2}2} \sum_{s=0}^{Nt-1}\Ebbnn [\phi(N^{-1/2} X^N_s) (P^{\brn} -\mathrm{Id})\bar u(\mathbf R_s)]\\& = N^{\frac{d-2}2} \sum_{s=0}^{Nt-1}\Ebbnn [\phi(N^{-1/2} X^N_s) (\bar u(\mathbf R_{s+1}) -\bar u(\mathbf R_s))] \\&= N^{\frac{d-2}2} \bigg( \Ebbnn [\phi(N^{-1/2} X_{Nt}) \bar u(\mathbf R_{Nt})] - \phi(N^{-1/2} \x_N) \bar u(\x_N) \\ &\;\;\;\;\;\;\;\;\;\;\;\;\;\;\;\;\;\;\;\;\;\;\;- \sum_{s=0}^{Nt-1} \Ebbnn[ (\phi(N^{-1/2} X^N_{s+1}) - \phi(N^{-1/2}X^N_s)) \bar u(\mathbf R_{s}) ] \\& \;\;\;\;\;\;\;\;\;\;\;\;\;\;\;\;\;\;\;\;\;\;\;\;\;\;\;\;\;\;\;\;\;\;\;\;\;\;\;\;\;\;\;\;\;\;- \sum_{s=0}^{Nt-1} \Ebbnn [(\phi(N^{-1/2} X^N_{s+1}) - \phi(N^{-1/2}X^N_s)) (\bar u(\mathbf R_{s+1})-\bar u(\mathbf R_s)) ] \bigg).
\end{align*}
Now, we need to take the limit of each of the four terms of the last expression. Let us first deal with the first two terms at once. For this, it will be useful to define the function $$V_t(x) := \begin{cases}
    \mathbf E_{BM}[|B_t-x|], & d=1\\ \mathbf E_{BM}[\log|B_t-x|] , & d=2 \\ -\mathbf E_{BM}[|B_t-x|^{2-d}] , &d\ge 3
\end{cases}$$
where $x\in \mathbb R^d, t\ge 0$, and the expectation is with respect to the standard Brownian motion of rate 2 on $\mathbb R^d$, started from the origin. Then by the invariance principle of Theorem \ref{inv00}, and the assumptions on the sequence $\x_N$, we have that
\begin{align}
    \notag \lim_{N\to\infty} &N^{\frac{d-2}2} \bigg( \Ebbnn [\phi(N^{-1/2} X_{Nt}) \bar u(\mathbf R_{Nt})] - \phi(N^{-1/2} \x_N) \bar u(\x_N)\bigg)\\& \notag = \mathbf E_{\boldsymbol x}^{BM^{\otimes k}} [ \phi( \tfrac12(B^i_t+B^j_t)) V_0(B^i_t-B^j_t)] - \phi(\tfrac12(x_1+x_2)) V_0(x_1-x_2) \\&= \bigg( \int_{\mathbb R^d} G(2t ,y) \phi(\tfrac12 (y+x_1+x_2))\dr y\bigg) V_t(x_1-x_2)  - \phi(\tfrac12(x_1+x_2)) V_0(x_1-x_2).\label{step1}
\end{align}
All expectations are with respect to a standard $2d$-dimensional Brownian motion started from $\boldsymbol x=(x_1,x_2)$. In the third line we are using the independence of $B^1+B^2$ from $B^1-B^2$, which implies that $\mathbf E_{BM^{\otimes k}} [ \phi( \tfrac12(B^1_t+B^2_t)) V_0(B^1_t-B^2_t)] = \mathbf E_{BM^{\otimes k}} [ \phi( \tfrac12(B^1_t+B^2_t))]\mathbf E_{BM^{\otimes k}}[ V_0(B^1_t-B^2_t)] = \mathbf E_{BM^{\otimes k}}[\phi( \tfrac12(B^1_t+B^2_t))] V_t(x_1-x_2)$. The reason that the invariance principle applies in this situation is because e.g. in $d\ge 3$ we can write $N^{\frac{d-2}2} \bar u(\mathbf R_{Nt}) = (N^{-1/2} +N^{-1/2} |R^1_{Nt}-   R^2_{Nt}|)^{2-d}$, and we can therefore use the fact that the sequence of functions $|N^{-1/2}+x|^{2-d}$ converges for $x\ne 0$ to the function $|x|^{2-d}=V_0(x)$. Notice that when $d\ge 2$, the function $V_0$ does have a singularity at the origin, but fortunately the uniform integrability implied by Theorem \ref{anti} (which guarantees for instance that $\sup_N \Ebbnn \big[ \big| N^{\frac{d-2}2}(1+\big| R^1_{Nt} -   R^2_{Nt}|)^{2-d}\big|^q\big]<\infty $ for any $q \in [1,\frac{d}{d-2})$ if $d\ge 3$) does imply convergence of the associated expectations.

Now we will take the limit of the third term $N^{\frac{d-2}2} \sum_{s=0}^{Nt-1} \Ebbnn[ (\phi(N^{-1/2} X^N_{s+1}) - \phi(N^{-1/2}X^N_s)) \bar u(\mathbf R_{s}) ]$. Apply a third-order Taylor expansion of $\phi$, and we see that 
\begin{align*}
    \Ebbnn[ (\phi(N^{-1/2} X^N_{s+1}) - \phi(N^{-1/2}X^N_s))  |\mathcal F_s ] &= N^{-1} \Delta \phi(X_s) + \mathbf{Er}_N(\phi,s),
\end{align*}
where the error term satisfies $$|\mathbf{Er}_N(\phi,s)| \leq C\|\phi\|_{C^3} N^{-3/2}+ C N^{-1/2}\|\phi\|_{C^2} d_{\mathrm{SRI}}(\br,\br_0)^{1/4} F(R^1_s-R^2_s)^{1/4}.$$
Indeed the first order term $N^{-1/2} \nabla \phi (X_s)\bullet  \Ebbnn[ X^N_{s+1} - X^N_s |\mathcal F_s ] $ can be absorbed into $\mathbf E_N$ because $\brn$ is converging to a centered chain (for which the conditional expectation would be exactly zero), thus one has that $|\Ebbnn[ X^N_{s+1} - X^N_s |\mathcal F_s ]| \leq C d_{\mathrm{SRI}} (\br,\br_0)^{1/4} F( |R^1_s-R^2_s|)^{1/4} $ by e.g. \eqref{useful}. Likewise all of the ``off-diagonal" second partials involving terms of the form $\partial_{12} \phi(X_s^N)$ would be absorbed into $\mathbf{Er}_N$ for a very similar reason, noting that the covariance matrix of the base measure of the limiting chain $\br$ is $\mathrm{Id}_{d\times d}$ (see the proof of Proposition \ref{d1} above for a very similar calculation). The remainder terms involving third derivatives all come with a factor of $N^{-3/2}$. 

Using Theorem \ref{anti} and the preceding bound, one may convince oneself that $\mathbf{Er}_N$ does not contribute to the limit, so that by the invariance principle (Theorem \ref{inv00}), one then has
\begin{align}
    \notag \lim_{N\to \infty} N^{\frac{d-2}2} \sum_{s=0}^{Nt-1} &\Ebbnn[ (\phi(N^{-1/2} X^N_{s+1}) - \phi(N^{-1/2}X^N_s)) \bar u(\mathbf R_s) ] = \lim_{N\to\infty} N^{\frac{d-2}2} \sum_{s=0}^{Nt-1} N^{-1} \Ebbnn[ \Delta \phi(X_s) \bar u(\mathbf R_s) ] \\ \notag &= \int_0^t \mathbf E_{BM^{\otimes 2} } [ \Delta \phi (B^1_s+B^2_s) V_0(B^i_s-B^j_s)]\dr s \\&= \int_0^t \bigg( \int_{\mathbb R^d} G(2s ,y) (\Delta \phi)(\tfrac 12(y+x_1+x_2))\dr y\bigg) V_s(x_1-x_2)\dr s\notag \\&= \int_0^t \partial_s \bigg( \int_{\mathbb R^d} G(2s ,y) \phi(\tfrac 12(y+x_1+x_2))\dr y\bigg) V_s(x_1-x_2)\dr s \notag \\ &= \bigg( \int_{\mathbb R^d} G(2t ,y) \phi(\tfrac12 (y+x_1+x_2))\dr y\bigg) V_t(x_1-x_2)  - \phi(\tfrac12(x_1+x_2)) V_0(x_1-x_2) \notag \\&\;\;\;\;\;\;\;\;- \int_0^t  \bigg(\int_{\mathbb R^d} G(2s ,y) \phi(\tfrac 12(y+x_1+x_2))\dr y\bigg)  (\partial_s V_s)(x_1-x_2)\dr s. \label{step2}
\end{align}
In the third line we are using the independence of $B^1+B^2$ from $B^1-B^2$, which implies that $\mathbf E_{BM^{\otimes 2}} [ \phi( \tfrac12(B^1_t+B^2_t)) V_0(B^1_t-B^2_t)] = \mathbf E_{BM^{\otimes 2}} [ \phi( \tfrac12(B^1_t+B^2_t))]\mathbf E_{BM^{\otimes 2}}[ V_0(B^1_t-B^2_t)] = \mathbf E_{BM^{\otimes 2}}[\phi( \tfrac12(B^1_t+B^2_t))] V_t(x_1-x_2)$ since $G_s*V_0 = V_s$. In the fifth line, we simply integrated by parts.

Next we argue that the fourth term goes to 0, i.e., $$\lim_{N\to \infty} N^{\frac{d-2}2} \sum_{s=0}^{Nt-1} \Ebbnn [(\phi(N^{-1/2} X^N_{s+1}) - \phi(N^{-1/2}X^N_s)) (\bar u(\mathbf R_{s+1})-\bar u(\mathbf R_s)) ].$$
The proof uses H\"older's inequality is essentially verbatim with the ``Term 4" argument in the proofs of Theorem \ref{2.11} and \ref{2.12}. We do not repeat it here.

Subtracting the results of \eqref{step1} and \eqref{step2}, we have obtained that 
\begin{align*}
    \lim_{N\to\infty} \int_I \mathfrak f\; \dr\gamma_N = \int_0^t  \bigg(\int_{\mathbb R^d} G(2s ,y) \phi(\tfrac 12(y+x_1+x_2))\dr y\bigg)  (\partial_s V_s)(x_1-x_2)\dr s.
\end{align*}
Now, simply note that $\partial_s V_s = c_d \cdot G(2s,\bullet)$ for the universal constant $c_d$ defined earlier. In particular, the subsequential limits of $\gamma_N$ are unique, as it must be the unique scalar multiple $c\cdot \pi^{\mathrm{inv}}$ under which $c \int_I \mathfrak f \dr \pi^{\mathrm{inv}} = c_d \int_0^t \int_I G(2s, y) \phi(\tfrac12 (y+x_1+x_2)) G(2s,x_1-x_2) \dr y \dr s.$ Recalling that $\int_I\mathfrak f\; \dr \pi^{\mathrm{inv}}=1$ gives the result.
\end{proof}

\begin{proof}[Proof of Theorem \ref{inv02}]

The $d=1$ statement follows from exactly the same proof as that of Theorem \ref{inv01}, with no changes necessary. Thus, we focus on the $d=2$ statement, which needs the appropriate modifications. In particula,r we need to explain why the correct factor is $\frac12$ multiplying $\pi^{\mathrm{inv}}$, which is the substance of the theorem. The proof starts similarly to that of Theorem \ref{inv01}, using the same subsequence trick.

   As in Step 1 in the proof of Theorem \ref{inv01}, it suffices to show that one has $$(\log N)^{-1}\sum_{r=0}^N \Ebbnn[f(   R^1_r-   R^2_r)] \to \frac12 \int_I f\dr\pi^{\mathrm{inv}}
    $$ for a fixed function $f$, rather than a sequence $f_N$ as in Assumption \ref{ass:SRI}. Consider the sequence of measures given by $$\gamma_N(f):= (\log N)^{-1}\sum_{r=0}^N \Ebbnn[f(   R^1_r-   R^2_r)].$$

    First we show that every subsequence $\gamma_{N_k}$ has a further converging subsequence $\gamma_{N_{k_j}}$ converging as $N\to\infty$ to a measure $\gamma$ which is a scalar multiple of the invariant measure $\pi^{\mathrm{inv}}$. As in the previous theorem, by convergence, we mean that for all continuous $f:I\to \mathbb R$ such that $|f(x)|\leq H(|x|)$ for some decreasing $H:[0,\infty)\to[0,\infty)$ such that $H(t) \leq Ae^{-bt}$ with $A,b>0$, we have $\int_I f\;\dr\gamma_N \to \int_I f\;\dr\gamma$ along this subsequence $\gamma_{N_{k_j}}$. 

Consider a subsequence $\gamma_{N_k}$. Consider the sequence $\chi_N:= H(|x|) \gamma_N(\dr x)$ for any fixed $H:[0,\infty)\to[0,\infty)$ that is decreasing and satisfies $H(t) \leq Ae^{-bt}$ with $A,b>0$. By the Banach-Alaoglu Theorem, there is a further subsequence $\gamma_{N_{k_j}}$ of $\gamma_{N_k}$ that converges. Let us call this subsequential limit $\gamma$. The fact that $\sup_N \|\chi_N\|_{TV} < \infty$ (so that $\gamma$ also has a finite total variation norm, and hence is a locally finite measure) is due to the fact that by Corollary \ref{cor:expectationBoundGeneral}, one has
\begin{align*}
    \|\chi_N\|_{TV} &= \int_I  H(|x|)d\gamma_N = (\log N)^{-1}\sum_{r=0}^N \Ebbnn[H(|   R^1_r-   R^2_r|)]
\leq  (\log N)^{-1}\sum_{r=0}^N C \cdot Ab^{-d} \cdot r^{-1}  \leq C \cdot Ab^{-d}.
\end{align*}

    Next, we show that any subsequential limit $\gamma$ must be an invariant measure. To prove this, we need to show that for all continuous $f:I\to \mathbb R$, say of exponential decay at infinity, one has $\int_I \Pdif f \;\dr\gamma = \int_I f\;\dr\gamma.$ To prove this, consider $f$ such that $|f(x)| \leq e^{-\theta|x|}$ where $\theta\leq z_0/4$. We first show that  \begin{equation}\lim_{N\to\infty}(\log N)^{-1} \sum_{r=0}^{N} \Ebbnn [ |(P^{\br}-P^{\brn})f(   R^1_r - R^2_r)|] = 0.
\end{equation}

Exactly as in \eqref{prvces}, we have that
\begin{equation}\label{c15}|(P^{\br}-P^{\brn})  f(\x)| \leq C\min\{ d_{\mathrm{SRI}}(\br,\brn), e^{-\frac\theta{2} |x_i-x_j|} \}\leq Cd_{\mathrm{SRI}}(\br,\brn)^{1/2}   e^{-\frac\theta{4} |x_i-x_j|}. 
\end{equation}

We therefore have 
\begin{align*}
    (\log N)^{-1} \sum_{r=0}^{N} \Ebbnn [ |(P^{\br}-P^{\brn})f(   R^1_r - R^2_r)|] \leq (\log N)^{-1} \sum_{r=0}^{N} \Ebbnn [ |Cd_{\mathrm{SRI}}(\br,\brn)^{1/2}   e^{-\frac\theta{4} |R_r^1-R_r^2|}|].
\end{align*}

It follows from Corollary \ref{cor:expectationBoundGeneral} that $ \Ebbnn [ |e^{-\frac\theta{4} |R_r^1-R_r^2|}|] \leq C r^{-1}$ so we have 
\begin{align*}
(\log N)^{-1} \sum_{r=0}^{N} \Ebbnn [ |Cd_{\mathrm{SRI}}(\br,\brn)^{1/2}   e^{-\frac\theta{4} |R_r^1-R_r^2|}|] &\leq (\log N)^{-1} \sum_{r=0}^{N} Cd_{\mathrm{SRI}}(\br,\brn)^{1/2} r^{-1} \\
&\leq C d_{\mathrm{SRI}}(\br,\brn)^{1/2}.
\end{align*}
This last term goes to $0$ as $N \to \infty$ due to the assumption that $\brn$ converges to $\br$.

We have  $P^{\br} f(   R^1_r - R^2_r) = \Pdif f(   R^1_r-   R^2_r)$ so that 
\begin{equation}\label{c16}(\log N)^{-1} \sum_{r=0}^{N}\Ebbnn [ P^{\br} f(   R^1_r - R^2_r)] \to \int_I \Pdif f \;\dr\gamma
\end{equation}
along the subsequence $N_{k_j}$.

Finally, by the Markov property, we have
\begin{align*}(\log N)^{-1} \sum_{r=0}^{N}\Ebbnn [(P^{\brn} -\mathrm{Id})f(   R^1_r - R^2_r)] &= (\log N)^{-1} \sum_{r=0}^{N}\Ebbnn [(f(   R^1_{r+1} - R^2_{r+1}) - f(   R^1_r - R^2_r))] \\&= (\log N)^{-1} \big(\Ebbnn [f(   R^1_{N+1} - R^2_{N+1})] - f(x^1_N-x^2_N) \big).
\end{align*}

This again goes to $0$ by Corollary \ref{cor:expectationBoundGeneral}, implying that along the subsequence $N_{k_j}$ we have that \begin{equation}(\log N)^{-1} \sum_{r=0}^{N}\Ebbnn [ P^{\brn} f(   R^1_r - R^2_r)] \to \int_I f \;\dr\gamma. \label{c17}
\end{equation}
Combining \eqref{c15}, \eqref{c16}, \eqref{c17}, we have $\int_I f\;\dr\gamma = \int_I \Pdif f\;\dr\gamma,$ proving that $\gamma$ is an invariant measure.

    Finally, we show that all of these subsequential limits agree, in other words, the invariant measures have the same normalizing factors. To show this we take the particular choice $\mathfrak f=(\Pdif-\mathrm{Id}) u(\x)$ where $u(\x):= \log(1+|x_i-x_j|).$ In this case, the invariance principle of Theorem \ref{inv00} easily implies that 
    \begin{align}
        \notag (\log N)^{-1}\sum_{r=0}^{N-1} \Ebbnn[f(   R^1_r-   R^2_r)] &= (\log N)^{-1} \sum_{r=0}^{N-1} \Ebbnn [ \log (1 + |R^i_{r+1} - R^j_{r+1}|) - \log ( 1+ |R^i_r-R^j_r|) ] \\ &= \notag \frac{ \Ebbnn[\log (1 + |R^1_{N}-R^2_{N}| )] - \log(1+ |x^i_N-x^j_N| ) }{\log N} \\&= \notag \frac12 + \frac{ \Ebbnn[\log (N^{-1/2} + N^{-1/2}|R^1_{N}-R^2_{N}| )] - \log(1+|x^i_N-x^j_N| ) }{\log N} \\& \stackrel{N\to\infty}{\longrightarrow} \frac12.
    \end{align} 
    In the last line, we used the invariance principle and the assumption that $|x^i_N-x^j_N|$ remains bounded to conclude that the numerator remains bounded. 
    
    Note that $\frac{1}{2} = \frac12 \int_I \mathfrak f \dr \pi^{\mathrm{inv}}$ when $\pi^{\mathrm{inv}}$ is under the unit normalization. 
    This means that every subsequence of $\gamma_N$ has a further subsequence converging to $\frac12 \pi^{\mathrm{inv}}$. Thus the sequence $\gamma_N$ must itself converge to $\frac12 \pi^{\mathrm{inv}}$. 
\end{proof}

Of course, the above proof is only valid in $d=2$, although at first glance, this logarithm trick appears to work in every dimension. However, for $d=1$, $\gamma_N$ has no subsequential limits in the first place since \eqref{eq:intfdu} would not hold for $d=1$ if we replaced $u(x)=1+|x|$ with $\log(1+|x|)$. Similarly, for $d>2$ the function $u$ does not decay sufficiently fast at infinity to be integrable with respect to $\pi^{\mathrm{inv}}$, thus $d=2$ is the only case where $\int_I (\Pdif -\mathrm{Id}) \log(1+|\bullet|) \dr\pi^{\mathrm{inv}}$ is actually guaranteed to exist.

\bibliographystyle{alpha}
\bibliography{ref.bib}

\begin{thebibliography}{HCGCC23}

\bibitem[AdHR11]{ADR}
L.~Avena, F.~den Hollander, and F.~Redig.
\newblock Law of large numbers for a class of random walks in dynamic random environments.
\newblock {\em Electron. J. Probab.}, 16:no. 21, 587--617, 2011.

\bibitem[AHC25]{ark2025universalfluctuationstailprobability}
Franscesca Ark, Jacob~B. Hass, and Eric~I. Corwin.
\newblock Universal fluctuations in the tail probability for d=2 random walks in space-time random environments, 2025.

\bibitem[BBMP98]{BMP3}
M.~S. Bernabei, C.~Boldrighini, R.~A. Minlos, and A.~Pellegrinotti.
\newblock Almost-sure central limit theorem for a model of random walk in fluctuating random environment.
\newblock {\em Markov Process. Related Fields}, 4(3):381--393, 1998.

\bibitem[BC17]{bc}
Guillaume Barraquand and Ivan Corwin.
\newblock Random-walk in beta-distributed random environment.
\newblock {\em Probability Theory and Related Fields}, 167(3-4):1057--1116, 2017.

\bibitem[BC25]{bc25}
Guillaume Barraquand and Francesco Casini.
\newblock Convergence of the {KMP} model to the {KPZ} equation.
\newblock 2025.

\bibitem[BG97]{BG97}
Lorenzo Bertini and Giambattista Giacomin.
\newblock Stochastic {B}urgers and {KPZ} equations from particle systems.
\newblock {\em Comm. Math. Phys.}, 183(3):571--607, 1997.

\bibitem[BLD20]{bld}
Guillaume Barraquand and Pierre Le~Doussal.
\newblock Moderate deviations for diffusion in time dependent random media.
\newblock {\em Journal of Physics A: Mathematical and Theoretical}, 53(21):215002, 2020.

\bibitem[BMP97]{BMP1}
C.~Boldrighini, R.~A. Minlos, and A.~Pellegrinotti.
\newblock Almost-sure central limit theorem for a {M}arkov model of random walk in dynamical random environment.
\newblock {\em Probab. Theory Related Fields}, 109(2):245--273, 1997.

\bibitem[BMP00]{BMP4}
C.~Boldrighini, R.~A. Minlos, and A.~Pellegrinotti.
\newblock Random walk in a fluctuating random environment with {M}arkov evolution.
\newblock In {\em On {D}obrushin's way. {F}rom probability theory to statistical physics}, volume 198 of {\em Amer. Math. Soc. Transl. Ser. 2}, pages 13--35. Amer. Math. Soc., Providence, RI, 2000.

\bibitem[BMP04]{BMP5}
C.~Boldrighini, R.~A. Minlos, and A.~Pellegrinotti.
\newblock Random walks in quenched i.i.d.\ space-time random environment are always a.s.\ diffusive.
\newblock {\em Probab. Theory Related Fields}, 129(1):133--156, 2004.

\bibitem[BP01]{BMP2}
C.~Boldrighini and A.~Pellegrinotti.
\newblock {$T^{-1/4}$}-noise for random walks in dynamic environment on {$\Bbb Z$}.
\newblock {\em Mosc. Math. J.}, 1(3):365--380, 470--471, 2001.

\bibitem[BR20]{mark}
Guillaume Barraquand and Mark Rychnovsky.
\newblock Large deviations for sticky {B}rownian motions.
\newblock {\em Electronic Journal of Probability}, 25, 2020.

\bibitem[BRAS06]{timo2}
M{\'{a}}rton Bal{\'{a}}zs, Firas Rassoul-Agha, and Timo Seppäläinen.
\newblock The random average process and random walk in a space-time random environment in one dimension.
\newblock {\em Communications in Mathematical Physics}, 266(2):499--545, may 2006.

\bibitem[BW22]{dom}
Dom Brockington and Jon Warren.
\newblock At the edge of a cloud of {B}rownian particles.
\newblock {\em arXiv preprint arXiv:2208.11952}, 2022.

\bibitem[BZ06]{BZ}
Antar Bandyopadhyay and Ofer Zeitouni.
\newblock Random walk in dynamic {M}arkovian random environment.
\newblock {\em ALEA Lat. Am. J. Probab. Math. Stat.}, 1:205--224, 2006.

\bibitem[CC22]{CAR}
Francesco Caravenna and Francesca Cottini.
\newblock {G}aussian limits for subcritical chaos.
\newblock {\em Electronic Journal of Probability}, 27:1--35, 2022.

\bibitem[CET21]{u6}
Giuseppe Cannizzaro, Dirk Erhard, and Fabio Toninelli.
\newblock Weak coupling limit of the anisotropic {KPZ} equation.
\newblock {\em arXiv preprint arXiv:2108.09046}, 2021.

\bibitem[CG17]{gu}
Ivan Corwin and Yu~Gu.
\newblock {Kardar--Parisi--Zhang equation and large deviations for random walks in weak random environments}.
\newblock {\em Journal of Statistical Physics}, 166:150--168, 2017.

\bibitem[CH49]{MR29488}
K.~L. Chung and G.~A. Hunt.
\newblock On the zeros of {$\sum^n_1\pm 1$}.
\newblock {\em Ann. of Math. (2)}, 50:385--400, 1949.

\bibitem[Com17]{Comets}
F.~Comets.
\newblock {\em {Directed Polymers in Rendom Environments (Springer-Verlag)}}.
\newblock 2017.

\bibitem[Cor12]{Cor12}
Ivan Corwin.
\newblock The {K}ardar{--P}arisi{--Z}hang equation and universality class.
\newblock {\em Random Matrices: Theory Appl.}, 1(01):1130001, 2012.

\bibitem[CS20]{CS20}
Ivan Corwin and Hao Shen.
\newblock Some recent progress in singular stochastic partial differential equations.
\newblock {\em Bulletin of the American Mathematical Society}, 57(3):409--454, 2020.

\bibitem[CSZ17]{marginal}
Francesco Caravenna, Rongfeng Sun, and Nikos Zygouras.
\newblock {Universality in marginally relevant disordered systems}.
\newblock {\em The Annals of Applied Probability}, 27(5):3050 -- 3112, 2017.

\bibitem[CSZ20]{u5}
Francesco Caravenna, Rongfeng Sun, and Nikos Zygouras.
\newblock The two-dimensional {KPZ} equation in the entire subcritical regime.
\newblock {\em Ann. Probab.}, 48(3), 2020.

\bibitem[CSZ23]{CSZ_2d}
Francesco Caravenna, Rongfeng Sun, and Nikos Zygouras.
\newblock The critical 2d stochastic heat flow.
\newblock {\em Inventiones mathematicae}, 233(1):325--460, 2023.

\bibitem[CW17]{CW17}
Ajay Chandra and Hendrik Weber.
\newblock Stochastic {PDE}s, regularity structures, and interacting particle systems.
\newblock In {\em Annales de la facult{\'e} des sciences de Toulouse Math{\'e}matiques}, volume 26(4), pages 847--909, 2017.

\bibitem[DDP24a]{DDP23}
Sayan Das, Hindy Drillick, and Shalin Parekh.
\newblock K{PZ} equation limit of sticky {B}rownian motion.
\newblock {\em J. Funct. Anal.}, 287(10):Paper No. 110609, 90, 2024.

\bibitem[DDP24b]{DDP+}
Sayan Das, Hindy Drillick, and Shalin Parekh.
\newblock Multiplicative {SHE} limit of random walks in space–time random environments.
\newblock {\em Probability Theory and Related Fields}, December 2024.

\bibitem[DG22a]{MR4413215}
Alexander Dunlap and Yu~Gu.
\newblock A forward-backward {SDE} from the 2{D} nonlinear stochastic heat equation.
\newblock {\em Ann. Probab.}, 50(3):1204--1253, 2022.

\bibitem[DG22b]{ew6}
Alexander Dunlap and Yu~Gu.
\newblock A quenched local limit theorem for stochastic flows.
\newblock {\em Journal of Functional Analysis}, 282(6):109372, 2022.

\bibitem[DG25]{dunlap2025edwardswilkinsonfluctuationssubcritical2d}
Alexander Dunlap and Cole Graham.
\newblock {Edwards-Wilkinson fluctuations in subcritical 2D stochastic heat equations}, 2025.

\bibitem[DGRZ20]{u2}
Alexander Dunlap, Yu~Gu, Lenya Ryzhik, and Ofer Zeitouni.
\newblock Fluctuations of the solutions to the {KPZ} equation in dimensions three and higher.
\newblock {\em Probability Theory and Related Fields}, 176(3-4):1217--1258, 2020.

\bibitem[DK57]{MR84222}
D.~A. Darling and M.~Kac.
\newblock On occupation times for {M}arkoff processes.
\newblock {\em Trans. Amer. Math. Soc.}, 84:444--458, 1957.

\bibitem[DKL08]{Dolgo1}
Dmitry Dolgopyat, Gerhard Keller, and Carlangelo Liverani.
\newblock Random walk in {M}arkovian environment.
\newblock {\em Ann. Probab.}, 36(5):1676--1710, 2008.

\bibitem[DL09]{Dolgo2}
Dmitry Dolgopyat and Carlangelo Liverani.
\newblock Non-perturbative approach to random walk in {M}arkovian environment.
\newblock {\em Electron. Commun. Probab.}, 14:245--251, 2009.

\bibitem[Dot24]{phys3}
Victor Dotsenko.
\newblock Scaling properties of (2+1) directed polymers in the low temperature limit.
\newblock {\em 2406.07998}, 2024.

\bibitem[DOV22]{dov}
Duncan Dauvergne, Janosch Ortmann, and Balint Virag.
\newblock The directed landscape.
\newblock {\em Acta Mathematica}, 229(2), 2022.

\bibitem[DPZ92]{DPZ}
Giuseppe Da~Prato and Jerzy Zabczyk.
\newblock {\em Stochastic equations in infinite dimensions}, volume~44 of {\em Encyclopedia of Mathematics and its Applications}.
\newblock Cambridge University Press, Cambridge, 1992.

\bibitem[Duc25]{duch}
Pawe{\l} Duch.
\newblock Flow equation approach to singular stochastic {PDEs}.
\newblock {\em Probability and Mathematical Physics}, 6(2):327--437, 2025.

\bibitem[ET60]{MR121870}
P.~Erd\"os and S.~J. Taylor.
\newblock Some problems concerning the structure of random walk paths.
\newblock {\em Acta Math. Acad. Sci. Hungar.}, 11:137--162. (unbound insert), 1960.

\bibitem[FS10]{FS10}
Patrik~L Ferrari and Herbert Spohn.
\newblock Random growth models.
\newblock {\em arXiv:1003.0881}, 2010.

\bibitem[GIP15]{GIP15}
Massimiliano Gubinelli, Peter Imkeller, and Nicolas Perkowski.
\newblock Paracontrolled distributions and singular {PDE}s.
\newblock In {\em Forum of Mathematics, Pi}, volume~3. Cambridge University Press, 2015.

\bibitem[GJ14]{GJ14}
Patr{\'\i}cia Gon{\c{c}}alves and Milton Jara.
\newblock Nonlinear fluctuations of weakly asymmetric interacting particle systems.
\newblock {\em Arch. Ration. Mech. Anal.}, 212(2):597--644, 2014.

\bibitem[GP17]{GP17}
Massimiliano Gubinelli and Nicolas Perkowski.
\newblock {KPZ} reloaded.
\newblock {\em Commun. Math. Phys.}, 349(1):165--269, 2017.

\bibitem[GP18]{GP18}
Massimiliano Gubinelli and Nicolas Perkowski.
\newblock Energy solutions of {KPZ} are unique.
\newblock {\em J. Amer. Math. Soc.}, 31(2):427--471, 2018.

\bibitem[GRZ18]{GRZ}
Yu~Gu, Lenya Ryzhik, and Ofer Zeitouni.
\newblock The {E}dwards–{W}ilkinson limit of the random heat equation in dimensions three and higher.
\newblock {\em Communications in Mathematical Physics}, 363(2):351–388, July 2018.

\bibitem[Gu20]{u4}
Yu~Gu.
\newblock Gaussian fluctuations from the {2D KPZ} equation.
\newblock {\em Stochastics and Partial Differential Equations: Analysis and Computations}, 8:150--185, 2020.

\bibitem[Hai13]{Hai13}
Martin Hairer.
\newblock Solving the {KPZ} equation.
\newblock {\em Annals of Mathematics}, pages 559--664, 2013.

\bibitem[Hai14]{Hai14}
Martin Hairer.
\newblock A theory of regularity structures.
\newblock {\em Invent Math.}, 198(2):269--504, 2014.

\bibitem[Has25]{hass2025superuniversalbehavioroutliersdiffusing}
Jacob Hass.
\newblock Super-universal behavior of outliers diffusing in a space-time random environment, 2025.

\bibitem[HCC23]{hass2023b}
Jacob~B Hass, Ivan Corwin, and Eric~I Corwin.
\newblock First passage time for many particle diffusion in space-time random environments.
\newblock {\em arXiv preprint arXiv:2308.01267}, 2023.

\bibitem[HCGCC23]{hass23}
Jacob~B Hass, Aileen~N Carroll-Godfrey, Ivan Corwin, and Eric~I Corwin.
\newblock Anomalous fluctuations of extremes in many-particle diffusion.
\newblock {\em Physical Review E}, 107(2):L022101, 2023.

\bibitem[HDCC24]{hass2024extreme}
Jacob~B Hass, Hindy Drillick, Ivan Corwin, and Eric~I Corwin.
\newblock Extreme diffusion measures statistical fluctuations of the environment.
\newblock {\em Physical Review Letters}, 133(26):267102, 2024.

\bibitem[HDCC25]{hass2025universal}
Jacob Hass, Hindy Drillick, Ivan Corwin, and Eric Corwin.
\newblock Universal {KPZ} fluctuations for moderate deviations of random walks in random environments.
\newblock {\em arXiv preprint arXiv:2504.00266}, 2025.

\bibitem[HH12]{phys4}
Timothy Halpin-Healy.
\newblock (2+1)-dimensional directed polymer in a random medium: Scaling phenomena and universal distributions.
\newblock {\em Physical review letters}, 109(17):170602, 2012.

\bibitem[HH13]{phys5}
Timothy Halpin-Healy.
\newblock Extremal paths, the stochastic heat equation, and the three-dimensional {Kardar-Parisi-Zhang} universality class.
\newblock {\em Physical Review E—Statistical, Nonlinear, and Soft Matter Physics}, 88(4):042118, 2013.

\bibitem[HH14]{Hall}
Peter Hall and Christopher~C Heyde.
\newblock {\em Martingale limit theory and its application}.
\newblock Academic press, 2014.

\bibitem[HHP14]{phys6}
Timothy Halpin-Healy and George Palasantzas.
\newblock Universal correlators and distributions as experimental signatures of 2+ 1 {Kardar-Parisi-Zhang} growth.
\newblock {\em arXiv preprint arXiv:1403.7509}, 2014.

\bibitem[HKLD23]{ldb}
Alexander~K Hartmann, Alexandre Krajenbrink, and Pierre Le~Doussal.
\newblock Probing the large deviations for the {B}eta random walk in random medium.
\newblock {\em arXiv preprint arXiv:2307.15041}, 2023.

\bibitem[HL15]{HL16}
Martin Hairer and Cyril Labb{\'e}.
\newblock A simple construction of the continuum parabolic {A}nderson model on $\mathbb{R}^2$.
\newblock {\em Electronic Communications in Probability}, 20:1--11, 2015.

\bibitem[HW09a]{HW09}
Chris Howitt and Jon Warren.
\newblock Consistent families of {B}rownian motions and stochastic flows of kernels.
\newblock {\em The Annals of Probability}, 37(4), jul 2009.

\bibitem[HW09b]{HW09b}
Chris Howitt and Jon Warren.
\newblock Dynamics for the {B}rownian web and the erosion flow.
\newblock {\em Stochastic Processes and their Applications}, 119(6):2028--2051, 2009.

\bibitem[JL24]{Junk}
Stefan Junk and Hubert Lacoin.
\newblock Strong disorder and very strong disorder are equivalent for directed polymers.
\newblock 2024.

\bibitem[JRAS19]{timo}
Mathew Joseph, Firas Rassoul-Agha, and Timo Sepp\"{a}l\"{a}inen.
\newblock Independent particles in a dynamical random environment.
\newblock In {\em Probability and analysis in interacting physical systems}, volume 283 of {\em Springer Proc. Math. Stat.}, pages 75--121. Springer, Cham, 2019.

\bibitem[KLM25]{Kot}
Sotirios Kotitsas, Dejun Luo, and Mario Maurelli.
\newblock {Edwards-Wilkinson} limit for a stochastic advection-diffusion {PDE}.
\newblock 2025.

\bibitem[KMP82]{MR656869}
C.~Kipnis, C.~Marchioro, and E.~Presutti.
\newblock Heat flow in an exactly solvable model.
\newblock {\em J. Statist. Phys.}, 27(1):65--74, 1982.

\bibitem[Kot25]{kotitsas2025heatequationtimecorrelatedrandom}
Sotirios Kotitsas.
\newblock The heat equation with time-correlated random potential in d=2: {Edwards-Wilkinson} fluctuations, 2025.

\bibitem[KPZ86]{kpz}
Mehran Kardar, Giorgio Parisi, and Yi-Cheng Zhang.
\newblock Dynamic scaling of growing interfaces.
\newblock {\em Physical Review Letters}, 56(9):889, 1986.

\bibitem[KR53]{MR56233}
G.~Kallianpur and H.~Robbins.
\newblock Ergodic property of the {B}rownian motion process.
\newblock {\em Proc. Nat. Acad. Sci. U.S.A.}, 39:525--533, 1953.

\bibitem[KS88]{KS88}
Norio Konno and Tokuzo Shiga.
\newblock Stochastic partial differential equations for some measure-valued diffusions.
\newblock {\em Probability theory and related fields}, 79(2):201--225, 1988.

\bibitem[Kun94]{kun94a}
Hiroshi Kunita.
\newblock Generalized solutions of a stochastic partial differential equation.
\newblock {\em Journal of Theoretical Probability}, 7:279--308, 1994.

\bibitem[Kun97]{kun94b}
Hiroshi Kunita.
\newblock {\em Stochastic flows and stochastic differential equations}.
\newblock Cambridge university press, 1997.

\bibitem[LBB97]{phys1}
Erik Luijten, Henk~WJ Bl{\"o}te, and Kurt Binder.
\newblock Crossover scaling in two dimensions.
\newblock {\em Physical Review E}, 56(6):6540, 1997.

\bibitem[LD23]{lda}
Pierre Le~Doussal.
\newblock Dynamics at the edge for independent diffusing particles.
\newblock {\em arXiv preprint arXiv:2308.16709}, 2023.

\bibitem[LDT17]{ldt}
Pierre Le~Doussal and Thimoth{\'e}e Thiery.
\newblock {Diffusion in time-dependent random media and the Kardar-Parisi-Zhang equation}.
\newblock {\em Physical Review E}, 96(1):010102, 2017.

\bibitem[LJR02]{ljr02}
Yves Le~Jan and Olivier Raimond.
\newblock Integration of {B}rownian vector fields.
\newblock {\em The Annals of Probability}, 30(2):826--873, 2002.

\bibitem[LJR04]{lejan}
Yves Le~Jan and Olivier Raimond.
\newblock Flows, coalescence and noise.
\newblock {\em Ann. Probab.}, 32(2):1247--1315, 2004.

\bibitem[LL10]{lawler2010random}
G.F. Lawler and V.~Limic.
\newblock {\em Random Walk: A Modern Introduction}.
\newblock Cambridge Studies in Advanced Mathematics. Cambridge University Press, 2010.

\bibitem[LZ23]{MR4616647}
Dimitris Lygkonis and Nikos Zygouras.
\newblock Moments of the 2{D} directed polymer in the subcritical regime and a generalisation of the {E}rd{\"o}s-{T}aylor theorem.
\newblock {\em Comm. Math. Phys.}, 401(3):2483--2520, 2023.

\bibitem[LZ24]{MR4750558}
Dimitris Lygkonis and Nikos Zygouras.
\newblock A multivariate extension of the {E}rd\"os-{T}aylor theorem.
\newblock {\em Probab. Theory Related Fields}, 189(1-2):179--227, 2024.

\bibitem[MQR21]{MQR}
Konstantin Matetski, Jeremy Quastel, and Daniel Remenik.
\newblock The {KPZ} fixed point.
\newblock {\em Acta Mathematica}, 227(1):115--203, 2021.

\bibitem[MT93]{MTbook}
S.~Meyn and R.L. Tweedie.
\newblock {\em {Markov Chains and Stochastic Stability (Springer-Verlag)}}.
\newblock 1993.

\bibitem[MU18]{u3}
Jacques Magnen and J{\'e}r{\'e}mie Unterberger.
\newblock The scaling limit of the {KPZ} equation in space dimension 3 and higher.
\newblock {\em Journal of Statistical Physics}, 171:543--598, 2018.

\bibitem[MW17]{WM}
Jean-Christophe Mourrat and Hendrik Weber.
\newblock Global well-posedness of the dynamic $\phi^{4}$ model in the plane.
\newblock {\em The Annals of Probability}, 45(4):2398--2476, 2017.

\bibitem[Par24]{Par24}
Shalin Parekh.
\newblock A hierarchy of {KPZ} equation scaling limits arising from directed random walk models in random media.
\newblock 2024.

\bibitem[Par25]{Par+}
Shalin Parekh.
\newblock Intermediate disorder for directed polymers with space-time correlations.
\newblock 2025.

\bibitem[QS15]{QS15}
Jeremy Quastel and Herbert Spohn.
\newblock The one-dimensional {KPZ} equation and its universality class.
\newblock {\em J. Stat. Phys.}, 160(4):965--984, 2015.

\bibitem[Qua11]{Qua11}
Jeremy Quastel.
\newblock Introduction to {KPZ}.
\newblock {\em Current developments in mathematics}, 2011(1), 2011.

\bibitem[RAS05]{timo3}
Firas Rassoul-Agha and Timo Seppäläinen.
\newblock An almost sure invariance principle for random walks in a space-time random environment.
\newblock {\em Probability Theory and Related Fields}, 133(3):299–314, February 2005.

\bibitem[RV13]{Redig}
Frank Redig and Florian V\"ollering.
\newblock Random walks in dynamic random environments: a transference principle.
\newblock {\em Ann. Probab.}, 41(5):3157--3180, 2013.

\bibitem[SSS09]{sss0}
Emmanuel Schertzer, Rongfeng Sun, and Jan Swart.
\newblock {Special points of the Brownian net}.
\newblock {\em Electronic Journal of Probability}, 14(none):805 -- 864, 2009.

\bibitem[SSS14]{sss}
Emmanuel Schertzer, Rongfeng Sun, and Jan Swart.
\newblock Stochastic flows in the {B}rownian web and net.
\newblock {\em Mem. Amer. Math. Soc.}, 227(1065):vi+160, 2014.

\bibitem[SSS17]{sss2}
Emmanuel Schertzer, Rongfeng Sun, and Jan Swart.
\newblock {The Brownian web, the Brownian net, and their universality}.
\newblock {\em Advances in disordered systems, random processes and some applications}, pages 270--368, 2017.

\bibitem[SV97]{SV}
Daniel Stroock and Srinivasa Varadhan.
\newblock {\em Multidimensional Diffusion Processes}.
\newblock Springer, 1997.

\bibitem[Tao24]{MR4709542}
Ran Tao.
\newblock Gaussian fluctuations of a nonlinear stochastic heat equation in dimension two.
\newblock {\em Stoch. Partial Differ. Equ. Anal. Comput.}, 12(1):220--246, 2024.

\bibitem[TFW92]{phys2}
Lei-Han Tang, Bruce~M Forrest, and Dietrich~E Wolf.
\newblock Kinetic surface roughening. ii. hypercube-stacking models.
\newblock {\em Physical Review A}, 45(10):7162, 1992.

\bibitem[TLD16]{ldt2}
Thimoth{\'e}e Thiery and Pierre Le~Doussal.
\newblock Exact solution for a random walk in a time-dependent 1d random environment: the point-to-point {B}eta polymer.
\newblock {\em Journal of Physics A: Mathematical and Theoretical}, 50(4):045001, dec 2016.

\bibitem[Tsa24]{Tsa24}
Li-Cheng Tsai.
\newblock Stochastic heat flow by moments, 2024.

\bibitem[TW94]{TW}
Craig~A Tracy and Harold Widom.
\newblock Level-spacing distributions and the {A}iry kernel.
\newblock {\em Communications in Mathematical Physics}, 159(1):151--174, 1994.

\bibitem[Wal86]{Wal86}
John~B Walsh.
\newblock An introduction to stochastic partial differential equations.
\newblock In {\em {\'E}cole d'{\'E}t{\'e} de Probabilit{\'e}s de Saint Flour XIV-1984}, pages 265--439. Springer, 1986.

\bibitem[Yu16]{yu}
Jinjiong Yu.
\newblock {Edwards-Wilkinson fluctuations in the Howitt-Warren flows}.
\newblock {\em Stochastic Processes and their Applications}, 126(3):948--982, 2016.

\end{thebibliography}

\end{document}